\documentclass[10pt,a4paper]{article}
\usepackage[nottoc]{tocbibind}
\usepackage{graphicx}
\usepackage{slashed}
\usepackage{enumerate}
\usepackage{amsmath}
\usepackage[bottom]{footmisc}
\usepackage[hidelinks]{hyperref}
{

}
\usepackage{amsfonts}
\usepackage{psfrag}
\usepackage{color}
\usepackage{mathtools}
\usepackage{pgf,tikz}
\usepackage{mathrsfs}
\usetikzlibrary{arrows}
\usepackage{fancyhdr}
\usepackage{amssymb}
\usepackage{slashed}
\usepackage{enumitem}

\usepackage[T1]{fontenc}
\usepackage[utf8]{inputenc}
\usepackage{authblk}

\newtheorem{remark}{Remark}

\newtheorem{proposition}{Proposition}[section]

\newtheorem{theorem}{Theorem}[section]

\newtheorem{lemma}{Lemma}[section]

\numberwithin{equation}{section}
\allowdisplaybreaks[4]

\newenvironment{proof}{\smallskip\noindent\emph{Proof.}\hspace{1pt}}%
{\hspace{-5pt}{\nobreak\quad\nobreak\hfill\nobreak$\square$\vspace{8pt}%
		\par}\smallskip\goodbreak}
\newcommand{\Lbar}{\underline{L}}
\newcommand{\Hbar}{\underline{H}}
\newcommand{\aln}{(\al \nabla)^i}
\newcommand{\tbeta}{\tilde{\beta}}
\newcommand{\tbetabar}{\tilde{\betabar}}
\newcommand{\omegabar}{\underline{\omega}}

\newcommand{\dubarprime}{\hspace{.5mm} \text{d} \ubar^{\prime}}
\newcommand{\intu}{\int_{u_{\infty}}^u}
\newcommand{\intubar}{\int_0^{\ubar}}

\newcommand{\chihat}{\hat{\chi}}

\newcommand{\upr}{\lvert u^\prime \rvert}
\newcommand{\al}{a^{\frac{1}{2}}}
\newcommand{\chibar}{\underline{\chi}}
\newcommand{\chibarhat}{\underline{\hat{\chi}}}
\newcommand{\ubar}{\underline{u}}
\newcommand{\e}{\mathrm{e}}
\newcommand{\be}{\begin{equation}}
\newcommand{\ee}{\end{equation}}

\newcommand{\bm}{\begin{align*}}
\newcommand{\enm}{\end{align*}}

\newcommand{\scaleinfinitySu}[1]{\lVert{#1} \rVert_{\mathcal{L}^\infty_{(sc)}(S_{u,\ubar})}}
\newcommand{\scaleinfinitySuprime}[1]{\lVert{#1} \rVert_{\mathcal{L}^\infty_{(sc)}(S_{u^{\prime},\ubar})}}
\newcommand{\ScaleinfinitySuprime}[1]{\big\lVert{#1} \big\rVert_{\mathcal{L}^\infty_{(sc)}(S_{u^{\prime},\ubar})}}
\newcommand{\scaleinfinitySuprimeubarprime}[1]{\lVert{#1} \rVert_{\mathcal{L}^\infty_{(sc)}(S_{u^{\prime},\ubar^\prime})}}

\newcommand{\tildetr}{\widetilde{\tr \chibar}}

\newcommand{\Y}{\mathrm{\Upsilon}}

\newcommand{\twoSu}[1]{\lVert{#1} \rVert_{L^2(S_{u,\ubar})}}
\newcommand{\pSu}[1]{\lVert{#1} \rVert_{L^p(S_{u,\ubar})}}
\newcommand{\inftySu}[1]{\lVert{#1} \rVert_{L^{\infty}(S_{u,\ubar})}}

\newcommand{\bespeq}{\begin{equation}\begin{split}}
\newcommand{\espeq}{\end{split}\end{equation}}
\newcommand{\scaletwoSu}[1]{\lVert{#1} \rVert_{\mathcal{L}^2_{(sc)}(S_{u,\ubar})}}
\newcommand{\scaleoneSu}[1]{\lVert{#1} \rVert_{\mathcal{L}^1_{(sc)}(S_{u,\ubar})}}
\newcommand{\ScaletwoSu}[1]{\big\lVert{#1} \big\rVert_{\mathcal{L}^2_{(sc)}(S_{u,\ubar})}}
\newcommand{\scaletwoSuprime}[1]{\lVert{#1} \rVert_{\mathcal{L}^2_{(sc)}(S_{u^\prime,\ubar})}}
\newcommand{\ScaletwoSuprime}[1]{\big\lVert{#1} \big\rVert_{\mathcal{L}^2_{(sc)}(S_{u^\prime,\ubar})}}
\newcommand{\scaleoneSuprimeubarprime}[1]{\lVert{#1} \rVert_{\mathcal{L}^1_{(sc)}(S_{u^\prime,\ubar^\prime})}}
\newcommand{\scaletwoSuprimeubarprime}[1]{\lVert{#1} \rVert_{\mathcal{L}^2_{(sc)}(S_{u^\prime,\ubar^\prime})}}
\newcommand{\ScaletwoSuprimeubarprime}[1]{\big\lVert{#1} \big\rVert_{\mathcal{L}^2_{(sc)}(S_{u^\prime,\ubar^\prime})}}
\newcommand{\scaletwoSuubarprime}[1]{\lVert{#1} \rVert_{\mathcal{L}^2_{(sc)}(S_{u,\ubar^\prime})}}
\newcommand{\ScaletwoSuubarprime}[1]{\Big \lVert{#1} \Big \rVert_{\mathcal{L}^2_{(sc)}(S_{u,\ubar^\prime})}}

\newcommand{\alphabar}{\underline{\alpha}}
\newcommand{\betabar}{\underline{\beta}}
\newcommand{\etabar}{\underline{\eta}}

\newcommand{\scaletwoHzero}[1]{\lVert{#1} \rVert_{\mathcal{L}^2_{(sc)}(H_{u_{\infty}}^{(0,\underline{u})})}}
\newcommand{\scaletwoHu}[1]{\lVert{#1} \rVert_{\mathcal{L}^2_{(sc)}(H_u^{(0,\underline{u})})}}
\newcommand{\scaletwoHbarzero}[1]{\lVert{#1} \rVert_{\mathcal{L}^2_{(sc)}(\underline{H}_{0}^{(u_{\infty},u)})}}
\newcommand{\scaletwoHbaru}[1]{\lVert{#1} \rVert_{\mathcal{L}^2_{(sc)}(\underline{H}_{\underline{u}}^{(u_{\infty},u)})}}

\newcommand{\omegad}{\omega^{\dagger}}
\newcommand{\omegabard}{\omegabar^{\dagger}}
\newcommand{\duprime}{\hspace{.5mm} \text{d}u^{\prime}}
\newcommand{\mubar}{\underline{\mu}}
\newcommand{\kappabar}{\underline{\kappa}}

\newcommand{\Hodge}[1]{\prescript{*}{}{#1}}

\newcommand{\intscaletwoSuubarprime}[1]{\int_0^{\ubar} \scaletwoSuubarprime{#1} \hspace{.5mm} \text{d} \ubar^\prime}
\newcommand{\intscaleTwoSuubarprime}[1]{\int_0^{\ubar} \ScaletwoSuubarprime{#1} \hspace{.5mm} \text{d} \ubar^\prime}

\newcommand{\ub}{\underline{u}}

\newcommand{\Lb}{\underline{L}}

\newcommand{\omb}{\underline{\omega}}

\newcommand{\chib}{\underline{\chi}}
\newcommand{\chih}{\hat{\chi}}
\newcommand{\chibh}{\hat{\chib}}

\newcommand{\beb}{\underline{\beta}}
\newcommand{\etb}{\underline{\eta}}

\def\a {\alpha}
\def\b {\beta}
\def\ab {\alphab}
\def\bb {\betab}
\def\nab {\nabla}

\def\ub {\underline{u}}
\def\th {\theta}
\def\Lb {\underline{L}}
\def\Hb {\underline{H}}
\def\chib {\underline{\chi}}
\def\chih {\hat{\chi}}
\def\chibh {\hat{\underline{\chi}}}
\def\omegab {\underline{\omega}}
\def\etab {\underline{\eta}}
\def\betab {\underline{\beta}}
\def\alphab {\underline{\alpha}}

\def\t {\tilde}

\def\d {\delta}
\def\f {\frac}
\def\i {\infty}
\def\l {\bigg(}
\def\r {\bigg)}
\def\S {S_{u,\underline{u}}}
\def\o{\omega}
\def\O{\Omega}

\def\M{\mathcal}
\def\p{\psi}

\def\Hu{H_u^{(0,\underline{u})}}

\def\at{a^{\f12}}

\def\p{\psi}

\def\ls{\leq}

\def\ls{\lesssim}

\renewcommand{\div}{\mbox{div }}
\newcommand{\curl}{\mbox{curl }}
\newcommand{\trchb}{\mbox{tr} \chib}
\def\trch{\mbox{tr}\chi}

\newcommand{\Ls}{{\mathcal L} \mkern-10mu /\,}

\newcommand{\tr}{\mbox{tr}}

\newcommand{\ui}{u_{\infty}}

\newcommand{\Rb}{\underline{\mathcal{R}}}
\newcommand{\tc}{\widetilde{\tr\chib}}

\newcommand\restri[2]{{
		\left.\kern-\nulldelimiterspace 
		#1 
		\right|_{#2} 
}}

\definecolor{ffqqqq}{rgb}{1.,0.,0.}
\definecolor{uuuuuu}{rgb}{0.26666666666666666,0.26666666666666666,0.26666666666666666}

\begin{document}

	\title{A scale-critical trapped surface formation criterion for the Einstein-Maxwell system}
	
	\author[a]{Xinliang An\thanks{matax@nus.edu.sg}}
	\author[b]{Nikolaos Athanasiou\thanks{Nikolaos.Athanasiou@maths.ox.ac.uk}}
	
	\affil[a]{Department of Mathematics, National University of Singapore, Singapore.}           
	\affil[b]{Mathematical Institute, University of Oxford, United Kingdom.}
	
	\renewcommand\Authands{ and }

		\maketitle

	\thispagestyle{plain}
	
	\vspace{-0.6cm}
	\begin{abstract} 
	In this study we show that, from arbitrarily dispersed initial data,  both the concentration of electromagnetic fields and the focusing of gravitational waves could lead to the formation of trapped surfaces. We establish a scale-critical semi-global existence result from the past null infinity for the Einstein-Maxwell system by assigning the signature for decay rates to geometric quantities and Maxwell components. This result generalizes an approach of the first author on studying Einstein vacuum equations by employing additional elliptic estimates and geometric renormalizations, and it also extends a result of Yu to the scale-critical regime.

	\vspace{5pt}
		
		\par \noindent 	
	\end{abstract}
	\vspace{0.2cm}
	
	\setlength{\voffset}{-0in} \setlength{\textheight}{0.9\textheight}
	
	\setcounter{page}{1} \setcounter{equation}{0}
	
	\section{Introduction}

	\subsection{Background} 
In this paper, we study the evolution of the Einstein--Maxwell system for a $(3+1)$--dimensional Lorentzian manifold $(\mathcal{M}, g)$ and an electromagnetic $2$--tensor $F_{\alpha \beta}$: 
\be\label{1.1}
	R_{\mu\nu}-\f12Rg_{\mu\nu}=T_{\mu\nu},
	\ee
	where
	\be \label{1.2} T_{\mu\nu}=F_{\mu\lambda}{F_{\nu}}^{\lambda}-\frac14 g_{\mu\nu}F_{\lambda\tau}F^{\lambda \tau}.\ee

\noindent Although a direct extension of Birkhoff's theorem implies that the $2$--parameter family of Reissner--N\"ordstrom electrovacuum spacetimes exhausts all possible spherically symmetric solutions to the Einstein-Maxwell system, in the absence of symmetry assumptions, the global dynamics of \eqref{1.1}-\eqref{1.2} are quite hard to study. \vspace{3mm}

\par \noindent For the Einstein vacuum equations (where $T_{\mu\nu}=0$), the study of \eqref{1.1}-\eqref{1.2} in the small data regime has been very successful. In a monumental work of Christodoulou and Klainerman \cite{Chr-Kl}, it is shown that the Minkowski spacetime is stable under small perturbations. Christodoulou and Klainerman showed that, for small perturbations of the trivial data, no singularity will form and all geodesics are complete.  Later, Zipser \cite{B-Z} extended this result to the Einstein--Maxwell system. 
\vspace{3mm}
\par \noindent One of the most fascinating aspects of the classical theory of general relativity is that it predicts the existence of a black hole. Historically, some notion of a black hole accompanies the theory of general relativity almost since its inception by Einstein in 1915. It was first encountered in an explicit solution to the Einstein vacuum equations and in particular the Schwarzschild solution $(\mathcal{M},g)_{Schw}$, communicated by Schwarzschild to Einstein in a letter about one month after the latter presented his field equations of general relativity at the Prussian academy of sciences. However, it was neither Schwarzschild nor Einstein who understood that what would come to be known as a black hole region featured prominently in the Schwarzschild solution. It was Lemaitre \cite{Lemaitre} who first observed, in 1932, that $(\mathcal{M},g)_{Schw}$ contains a  region $\mathcal{B}$ with the property that observers lying inside $\mathcal{B}$ cannot send signals to observers situated at an ideal conformal boundary \textit{at infinity} $\mathcal{I}^+$ (this being defined in a rigorous and appropriate way). In the case of Schwarzschild solution, the existence of a non-empty $\mathcal{B}$ is accompanied by another surprising, yet salient, feature:

\vspace{3mm}

\begin{center}
	\textit{Every observer in $\mathcal{B}$ lives for finite proper time (future geodesic incompleteness).}
\end{center}When physicists and mathematicians first realized these two properties, they were hoping to be able to associate to them the characterization of an accident; a non-generic pathology, present only due to the high degree of symmetry imposed a priori on the solutions and that, in general solutions to the equations, such phenomena would not arise. Much to the surprise of the community,  Penrose \cite{Penrose} in 1965 proved these hopes were ill--based through the following incompleteness theorem:

\begin{theorem}
For a spacetime $(\mathcal{M}, g)$ containing a non--compact Cauchy hypersurface and $g_{\mu\nu}, F_{\mu\nu}$ satisfying \eqref{1.1}--\eqref{1.2},  if $\mathcal{M}$ contains a compact trapped surface, then it is future causally geodesically incomplete. 
\end{theorem}

\par \noindent A trapped surface is a $2$--dimensional geometric object. Assume we are given a $(3+1)$--dimensional, time--oriented Lorentzian manifold $(\mathcal{M},g)$ and within it a closed, spacelike $2-$surface $S$. Since $S$ has co--dimension 2, the tangent space at a point $p$ on S, $T_p S$ has a 2--dimensional orthogonal complement in $T_p \mathcal{M}$. Let $l, \underline{l}$ denote a null basis\footnote{That means $g(l,l)=0, \,\, g(\underline{l}, \underline{l})=0, \,\, g(l, \underline{l})=-2$.} of this complement and extend $l,\underline{l}$ as vector fields. We define the following two fundamental forms $\chi, \underline{\chi}$ associated with the surface $S$:

\[  \chi(X,Y) := g(\nabla_X l,Y),  \quad \underline{\chi}(X,Y):= g(\nabla_X \underline{l},Y)       \]where $X$ and $Y$ are vector fields tangent to $S$. We look at the expansions $\tr\chi, \tr\underline{\chi}$. If both are pointwise negative on $S$, then the surface is called trapped. A trapped surface is, therefore, a surface for which the area decreases for arbitrary infinitesimal displacements along the null generators of both null geodesic congruences normal to $S$.  Penrose's theorem implies that the study of singularity formation for Einstein's equations can, in some generality, be reduced to the problem of trapped surface formation. This problem had, again, remained open for a long time.

\subsubsection{The Einstein vacuum case}
In a breakthrough work in 2008, Christodoulou solved this long--standing open problem (trapped surface formation for Einstein vacuum equations) with a 587--page monumental work \cite{Chr:book}. He designed an open set of large initial data, which encode a special structure, the \textit{short pulse}  ansatz.  In particular, this ansatz allows one to consider a hierarchy of large and small quantities, parametrized by a small parameter $\d$. For such initial data, these quantities behave differently, being of various sizes in term of $\d$. Moreover, their sizes form a hierarchy. But for each quantity, surprisingly, its size is almost preserved by the nonlinear evolution. 
Therefore, once this hierarchy is satisfied at the level of initial data, it persists for later time. With this philosophy, despite it being a large data problem, a long--time global existence theorem can be established. Moreover, these initial conditions indeed lead to trapped-surface formation in finite time.
\vspace{3mm}

\par \noindent Effort was consequently put towards simplifying Christodoulou's proof.  In \cite{KR:Trapped}, an ingenious systematical approach was introduced by Klainerman--Rodnianski \cite{KR:Trapped}. This approach was later extended by the first author in \cite{An:Trapped}. The Einstein vacuum equations are a nonlinear hyperbolic system, containing many unknowns. Christodoulou controlled all of them on a term-by-term basis. In \cite{KR:Trapped}, Klainerman and Rodnianski introduced a novel index $s_1$ which they termed \textit{signature for short pulse}. With this index, Klainerman and Rodnianski systematically tracked the $\d$-weights used in the estimates and gave a shorter, simplified proof of the almost--preservation of the $\d$--hierarchy in a finite region. In \cite{Chr:book}, besides $\d$--weights, Christodoulou also employed weights related to decay and proved his main theorem that \textit{a trapped surface could form dynamically with initial data prescribed arbitrary dispersed at past null infinity}.  In \cite{An:Trapped}, the first author introduced a new index $s_2$ called \textit{signature for decay rates}. With the help of this new index, the first author extended Klainerman and Rodnianski's result \cite{KR:Trapped} from a finite region to an infinite region and re-proved Christodoulou's main theorem in \cite{Chr:book} with around 120 pages.
\vspace{3mm}
\par \noindent Another important progress was made by the first author and Luk in \cite{A-L}. By designing and employing a different hierarchy, in \cite{A-L} they improved \cite{Chr:book} and proved the first scale--critical result for the Einstein vacuum equations. With the same small parameter $\d$, with relatively larger initial data, Christodoulou formed a trapped surface of radius 1; while with much smaller initial data, the first author and Luk formed a  trapped surface of radius $\d a$, where $a$ is a universal large constant like $1000$.\footnote{Letting $a=\d^{-1}$, in a finite region they recover Christodoulou's main result of \cite{Chr:book}. } In \cite{A-L} the first author and Luk want to form a tiny trapped surface with radius $\d a$, hence they have to deal with the region very close to the center. In this region all the geometric quantities have growth rates. To bound these growth rates, they employ weighted estimates as well as several crucial geometric {\color{black}\underline{renormalizations}}. 
\vspace{3mm}

\par \noindent Since \cite{A-L} is scale critical, one can keep $a$ as a universal constant and let $\d\rightarrow 0$. Hence a series of trapped surfaces (with radius shrinking to $0$) are obtained. In \cite{An:AH}, the first author further explored this idea. Together with an elliptic approach to identify the boundary, the first author showed that a whole black hole region could emerge dynamically from just a ``point'' $O$ in the spacetime. For an open set of initial data and appropriate control on all the derivatives of $\chihat_0$, this boundary (apparent horizon) is proved to be smooth except at $O$. 
\vspace{3mm}

\par \noindent In early 2019, the first author \cite{An:2019} gave a different, 55--page proof of a trapped surface formation theorem that sharpens the previous results both of Christodoulou \cite{Chr:book} and estimates in An--Luk \cite{A-L}. The argument in 
\cite{An:2019} is based on a systematic extension of the scale--critical arguments in \cite{A-L}, connecting Christodoulou's short--pulse method and Klainerman--Rodnianski's signature counting argument to the peeling properties previously used in small--data results such as Klainerman--Nicolo. This in particular allows the author to avoid elliptic estimates and geometric renormalizations, and gives new technical simplifications. 

\subsubsection{The Einstein--Maxwell case}
For the Einstein-Maxwell system \eqref{1.1}-\eqref{1.2}, important progress was made by Yu \cite{Yu1, Yu2}. In a finite region, Yu extended the result of Klainerman-Rodnianski \cite{KR:Trapped} and obtained trapped surface formation for the Einstein-Maxwell system by using signature for short-pulse. In the current paper, combining the new ingredients in \cite{An:2019}, we will extend Yu's results to obtain a scale-critical trapped surface formation criterion from past null infinity.

\subsubsection{The Einstein--scalar field case}
In a recent paper \cite{L-L},  Li and Liu studied the Einstein-scalar field system: 
\begin{equation*}
\begin{split}
&\mbox{Ric}_{\mu\nu}-\f12Rg_{\mu\nu}=2T_{\mu\nu},\\
&T_{\mu\nu}=\partial_{\mu}\phi \partial_{\nu}\phi-\f12g_{\mu\nu}\partial^{\sigma}\phi \partial_{\sigma}\phi,
\end{split}
\end{equation*}
and an almost scale-critical trapped surface formation criterion was obtained with singular initial data prescribed along the incoming null hypersurface. And renormalizations for scalar fields were used.

\subsection{Main Results}

We will introduce coordinates $u$ and $\ub$ in $(\mathcal{M}, g)$ through a \textit{double null foliation}\footnote{The detailed construction of double null foliation will be explained in Section \ref{secdnf}.}, where $H_u$ and $\Hb_{\ub}$ are incoming and outgoing characteristic cones, respectively. With coordinates $u, \ub$, characteristic initial data will be prescribed along incoming null hypersurface $\Hb_0$, where $\ub=0$, and outgoing null hypersurface $H_{u_{\infty}}$, where $u=u_{\infty}$.

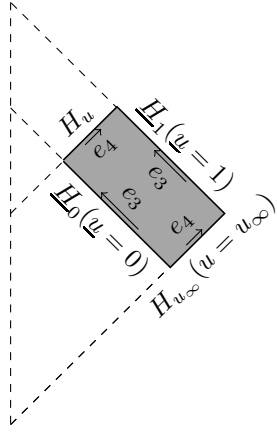
\begin{center}
\begin{minipage}[!t]{0.5\textwidth}
\begin{tikzpicture}[scale=0.7]
\draw [white](3,-1)-- node[midway, sloped, below,black]{$H_{u_{\infty}}(u=u_{\infty})$}(4,0);
\draw [white](1,1)-- node[midway,sloped,above,black]{$H_{u}$}(2,2);
\draw [white](2,2)--node [midway,sloped,above,black] {$\Hb_{1}(\ub=1)$}(4,0);
\draw [white](1,1)--node [midway,sloped, below,black] {$\Hb_{0}(\ub=0)$}(3,-1);
\draw [dashed] (0, 4)--(0, -4);
\draw [dashed] (0, -4)--(4,0)--(0,4);
\draw [dashed] (0,0)--(2,2);
\draw [dashed] (0,-4)--(4,0);
\draw [dashed] (0,2)--(3,-1);
\draw [very thick] (1,1)--(3,-1)--(4,0)--(2,2)--(1,1);
\fill[black!35!white] (1,1)--(3,-1)--(4,0)--(2,2)--(1,1);
\draw [->] (3.3,-0.6)-- node[midway, sloped, above,black]{$e_4$}(3.6,-0.3);
\draw [->] (1.4,1.3)-- node[midway, sloped, below,black]{$e_4$}(1.7,1.6);
\draw [->] (3.3,0.6)-- node[midway, sloped, below,black]{$e_3$}(2.7,1.2);
\draw [->] (2.4,-0.3)-- node[midway, sloped, above,black]{$e_3$}(1.7,0.4);
\end{tikzpicture}
\end{minipage}   
\end{center}
Our main results can be summarized in three Theorems. The first one is a global existence result.
	
	\begin{theorem}
		\label{main1} Given $\mathcal{I}^{(0)}$, there exists a sufficiently large $a_0=a_0(\mathcal{I}^{(0)})$ such that the following holds. For any $0< a_0<a$ and with initial data $(\chihat,\alpha_F)$ satisfying
		\begin{itemize}
			\item $\sum_{ i \leq 15, k\leq 3}a^{-\frac{1}{2}} \lVert \nabla_4^k \left(\lvert u_\infty\rvert \nabla\right)^i (\chihat,\alpha_F)\rVert_{L^\infty(S_{u_\infty,\ubar})}\leq \mathcal{I}^{(0)}$ along $u=u_\infty$,
			\item Minkowskian initial data along $\ubar=0$,
		\end{itemize}then the Einstein-Maxwell system admit a unique smooth solution in the region \[ u_\infty \leq u \leq -a/4, \hspace{2mm} 0\leq \ubar \leq 1.     \]
	\end{theorem}The second one steps directly on the first one and is a formation of trapped surfaces statement.  \begin{theorem}\label{main2} Given $\mathcal{I}^{(0)}$, there exists a sufficiently large $a_0=a_0(\mathcal{I}^{(0)})$ such that the following holds. For any $0< a_0<a$, the unique smooth solution $(\mathcal{M}, g)$ of the Einstein--Maxwell system from Theorem \ref{main1} with initial data satisfying
		\begin{itemize}
			\item $\sum_{ i \leq 15, k\leq 3}a^{-\frac{1}{2}} \lVert \nabla_4^k \left(\lvert u_\infty\rvert \nabla\right)^i (\chihat,\alpha_F)\rVert_{L^\infty(S_{u_\infty,\ubar})}\leq \mathcal{I}^{(0)}$ along $u=u_\infty$,
			\item Minkowskian initial data along $\ubar=0$,
			\item $\int_0^1 \lvert u_\infty \rvert^2 \left( \lvert \chihat_0 \rvert^2 + \lvert \alpha_{F0}\rvert^2 \right) (u_\infty, \ubar^\prime)\hspace{.5mm} \text{d}\ubar^\prime \geq a$ uniformly for every direction along $u=u_\infty$
		\end{itemize}has a trapped surface at $S_{-\alpha/4,1}$.
		
	\end{theorem}The third one is a rescaling of Theorem \ref{main2} and constitutes a criterion for trapped surface formaton of the Einstein--Maxwell system in the region close to the center.
	
	\begin{theorem}\label{main3} Given $\mathcal{I}^{(0)}$ and a fixed $\delta >0$, there exists a sufficiently large $a_0=a_0(\mathcal{I}^{(0)},\delta)$ such that the following holds. For any $0< a_0<a$, the unique smooth solution $(\mathcal{M}, g)$ of the Einstein-Maxwell equations from Theorem \ref{main1} with initial data satisfying
		\begin{itemize}
			\item $\sum_{ i \leq 15, k\leq 3}a^{-\frac{1}{2}} \lVert (\delta \nabla_4)^k \left(\lvert u_\infty\rvert \nabla\right)^i (\chihat,\alpha_F)\rVert_{L^\infty(S_{u_\infty,\ubar})}\leq \mathcal{I}^{(0)}$ along $u=u_\infty$,
			\item Minkowskian initial data along $\ubar=0$,
			\item $\int_0^{\delta} \lvert u_\infty \rvert^2 \left( \lvert \chihat_0 \rvert^2 + \lvert \alpha_{F0}\rvert^2 \right) (u_\infty, \ubar^\prime)\hspace{.5mm} \text{d}\ubar^\prime \geq a$ uniformly for every direction along $u=u_\infty$
		\end{itemize}has a trapped surface at $S_{-\delta \alpha/4,\delta}$.
	\end{theorem}

\subsection{New Ingredients}	Compared to the corresponding problem for the Einstein vacuum equations, the Einstein-Maxwell system gives rise to additional technical difficulties not encountered in the vacuum case. Here we list some important ones and outline the solutions.
\begin{enumerate}

\item We extend the systematical way of assigning signatures $s_2$ for the geometric quantities to the Maxwell field. To achieve this, we resort to the Bianchi equations, keeping the same $s_2$ values for the Ricci coefficients as those in the vacuum case. Crucially, the Maxwell equations expressed in a null frame have, in a sense, the same structure as the vacuum Bianchi equations. This means that such an assignment of signatures as in \cite{An:Trapped}, \cite{An:2019} can be carried through in a cohesive and coherent way to the Einstein--Maxwell system. 

\item The fact that the Ricci tensor is non--trivial for the Einstein--Maxwell system implies that the energy estimates will be best carried out using the Weyl tensor components instead of the Riemann tensor components. We thus work with the Weyl components $\alpha_W, \alphabar_W, \beta_W, \betabar_W, \rho_W, \sigma_W$ and re-express all equations with respect to them. For simplicity, we still use $\alpha, \alphabar, \beta, \betabar, \rho, \sigma$ to mean $\alpha_W, \alphabar_W, \beta_W, \betabar_W, \rho_W, \sigma_W$.
 
\item We introduce and employ crucial renormalizations for $\beta_a$ and $\beb_a$. The reason for this stems from the Bianchi equations. For example, in the Bianchi equation involving $\nabla_4 \beta$, the right--hand side contains the term $D_{4}R_{4A}$. This, in turn, contains the term $\nabla_4 \alpha_F$. When one attempts to estimate $\nabla_4 \alpha_F$, there is no available equation for it in the null Maxwell equations, thus making it difficult to estimate by itself. It is for this reason that we introduce the quantities

\[  \tbeta_A:= \beta_A - \frac{1}{2}R_{4A}, \quad \tbetabar_A:= \beta_A - \frac{1}{2}R_{3A}.  \] We then rewrite the entire Bianchi equations in terms of those renormalized quantities and use them to do energy estimates. In this way, all the terms can be estimated directly through the null Maxwell equations and null Bianchi equations. In particular, terms like $\nabla_4 \alpha_F$ and $\nabla_3 \alphabar_F$ no longer appear this time.

\item The above renormalizations force us to introduce and use elliptic estimates in the scale--invariant framework. This is achieved in Section 6. Its purpose is to allow us to close the energy estimates for up to $11$ derivatives of the Maxwell field components and up to $10$ derivatives of the Weyl curvature components. In the process, a control on $11$ derivatives of the Ricci coefficients is required. We find a non--trivial way (via energy estimates) to incorporate the elliptic estimates into the systematical approach via the signature for decay rates $s_2$.

\end{enumerate}

	\section{Setting, equations and notations}
	
	\subsection{Double Null Foliation}\label{secdnf}
	We construct a double null foliation in a neighbourhood of $S_{u_{\infty},0}$ as follows: 
	
	\begin{minipage}[!t]{0.35\textwidth}
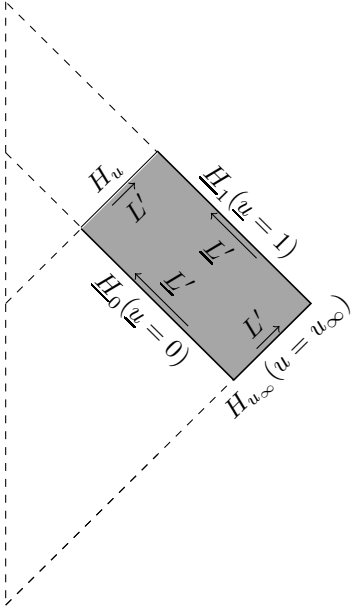
		\begin{tikzpicture}
		\draw [white](3,-1)-- node[midway, sloped, below,black]{$H_{u_{\infty}}(u=u_{\infty})$}(4,0);
		
		\draw [white](2,2)--node [midway,sloped,above,black] {$\Hb_1(\ub=1)$}(4,0);
		\draw [white](1,1)--node [midway,sloped, below,black] {$\Hb_{0}(\ub=0)$}(3,-1);
		\draw [dashed] (0, 4)--(0, -4);
		\draw [dashed] (0, -4)--(4,0)--(0,4);
		\draw [dashed] (0,0)--(2,2);
		\draw [dashed] (0,-4)--(2,-2);
		\draw [dashed] (0,2)--(3,-1);
		\draw [very thick] (1,1)--(3,-1)--(4,0)--(2,2)--(1,1);
		\fill [black!35!white]  (1,1)--(3,-1)--(4,0)--(2,2)--(1,1);
		\draw [white](1,1)-- node[midway,sloped,above,black]{$H_{u}$}(2,2);
		\draw [->] (3.3,-0.6)-- node[midway, sloped, above,black]{$L'$}(3.6,-0.3);
		\draw [->] (1.4,1.3)-- node[midway, sloped, below,black]{$L'$}(1.7,1.6);
		\draw [->] (3.3,0.6)-- node[midway, sloped, below,black]{$\Lb'$}(2.7,1.2);
		\draw [->] (2.4,-0.3)-- node[midway, sloped, above,black]{$\Lb'$}(1.7,0.4);
		\end{tikzpicture}
	\end{minipage}
	\hspace{0.02\textwidth} 
	\begin{minipage}[!t]{0.55\textwidth}
		\vspace{3mm}
		
		\par \noindent Starting from a point $p$ on $2$-sphere $S_{u_{\infty}, 0}$, in $2$-dimensional $T_p^{\perp}S_{u_{\infty}, 0}$, we could find two future-directed vectors $L'_p, \underline{L}'_p$ such that 
		$$g(L'_p, L'_p)=0, \, \, g(\underline{L}'_p, \underline{L}'_p)=0, \,  \, g(L'_p, \Lb'_p)=-2.$$ Note that $\{L'_p, \Lb'_p\}$ are uniquely determined up to a scaling factor $\lambda>0$: \quad $\{L'_p, \Lb'_p\}\rightarrow \{\lambda L'_p, \lambda^{-1}\Lb'_p\}$.  Emanating from  $p$ and initially tangent to $L'_p$, a unique geodesic $l_p$ is sent out.  We extend $L'$ along $l_p$ such that 
		$D_{L'}L'=0$. We then have that $l_p$ is null. This is because $g(L'_p, L'_p)=0$ and
		$$L'(g(L', L'))=2g(D_{L'}L', L')=0.$$
		We hence have $g(L',L')=0$ along $l_p$.
		Gathering all the $\{l_p\}$ together, we then have an outgoing null hypersurface called $H_{u_{\infty}}$. Similarly, we obtain the incoming null hypersurface $\Hb_0$ emanating from $S_{u_{\infty},0}$.\\
		
		Note that, by above construction, for each point $q$ on $H_{u_{\infty}}$ or $\Hb_0$, in $T_q H_{u_{\infty}}$ or $T_q \Hb_0$, there is a preferred null vector $L'_q$ or $\Lb'_q$ associated with $q$. \\
	\end{minipage}

\par \noindent	We proceed to define the function $\O$ to be 1 on $S_{u_{\infty},0}$ and extend $\O$ as a continuous function along $H_{u_{\infty}}$ and $\Hb_0$. \footnote{For a general double null foliation, we have the gauge freedom of choosing how to extend $\O$ along $H_{u_{\infty}}$ and $\Hb_0$. In this paper, we extend $\O\equiv 1$ on both $H_{u_{\infty}}$ and $\Hb_0$.} We consider vector fields 
	$$L=\O^2 L' \, \,\, \mbox{along}\, \, \, H_{u_{\infty}}, \,\, \mbox{and} \, \, \Lb=\O^2 \Lb' \, \, \,\mbox{along}\, \, \, \Hb_0$$
	and define functions
	$$\ub \,\, \mbox{on} \,\, H_{u_{\infty}}\,\,\, \mbox{satisfying}\, \, \, L\ub=1 \,\, \mbox{and} \,\, \ub=0\, \,\mbox{on}  \,\, S_{u_{\infty}, 0},$$
	$$u \, \, \mbox{on} \, \, \Hb_0 \, \, \, \mbox{satisfying} \, \, \,\Lb u=1 \, \, \mbox{and} \, \, u=u_{\infty}\, \,\mbox{on} \, \, S_{u_{\infty}, 0}.$$
	Let $S_{u_{\infty}, \ub'}$ be the embedded $2$-surface on $H_{u_{\infty}}$, such that $\ub=\ub'$. 
	Similarly, define $S_{u', 0}$ to be the embedded $2$-surface on $\Hb_0$, such that $u=u'$. 	At each point $q$ on $2$-surface $S_{u_{\infty}, \ub'}$, we already have the preferred outgoing null vector $L'_q$ tangent to $H_{u_{\infty}}$. Hence, at $q$, we can also fix a unique incoming null vector $\Lb_q'$ via requiring
	$$g(\Lb'_q, \Lb'_q)=0 \quad \mbox{and} \quad g(\Lb'_q, L'_q)=-2\O^{-2}|_q.$$
	There exists a unique geodesic $\underline{l}_q$ emitting from $q$ with direction $\Lb'$. We then extend $\Lb'$ along $\underline{l}_q$ by imposing $D_{\Lb'}\Lb'=0$.  Gathering all the $\{\underline{l}_q\}$ for $q\in S_{u_{\infty}, \ub'}$, we construct the incoming null hypersurface $\Hb_{\ub'}$ emanating from $S_{u_{\infty}, \ub'}$. Similarly, from $S_{u',0}$ we also construct the outgoing null hypersurface $H_{u'}$. We further define the $2$-spheres $S_{u', \ub'}:=H_{u'}\cap \Hb_{\ub'}$. 
	\vspace{3mm}
	
	\par \noindent At each point $p$ of $S_{u',\ub'}$, we define the positive-valued function $\O$ via 
	\begin{equation}\label{define Omega}
	g(L'_p, \Lb'_p)=:-2\O^{-2}| _p.
	\end{equation}
	Note that $L'_p$ is well-defined on $H_{u'}$, along an outgoing null geodesic $l$ passing through $p$;  $\Lb'_p$ is also well-defined on $\Hb_{\ub'}$, along an incoming null geodesic $\underline{l}$ crossing $p$. \\
	
\vspace{3mm} \par \noindent	These $3$-dimensional incoming null hypersurfaces $\{\Hb_{\ub'}\}_{0\leq \ub' \leq 1},$ along with the outgoing null hypersurfaces $\{H_{u'}\}_{u_{\infty}\leq u' \leq -a/4}$ and their pairwise intersections $S_{u', \ub'}=H_{u'}\cap \Hb_{\ub'}$ give us the so-called \textit{double null foliation}. \\
	
\vspace{3mm}	
\par \noindent	On $\S$, by (\ref{define Omega}) we have $g(L', \Lb')=-2\O^{-2}$. Thus, $g(\O L', \O \Lb')=-2$. Throughout this paper we will work with the normalized null pair $(e_3, e_4)$, namely
	$$e_3:=\Omega\Lb', \quad e_4:=\Omega L', \,\, \mbox{and} \quad g(e_3, e_4)=-2.$$
	
\vspace{3mm} \par \noindent 	Moreover, for the imposition of our characteristic initial data we choose the following gauge:
	$$\Omega\equiv 1 \quad\mbox{on $H_{u_{\infty}}$ and $\Hb_0$}.$$
	\begin{remark}
		The functions $u$ and $\ub$ defined above also satisfy the eikonal equations
		$$g^{\mu\nu}\partial_\mu u\partial_\nu u=0,\quad g^{\mu\nu}\partial_\mu\ub\partial_\nu \ub=0.$$
		And it is straight forward to check
		$$L'^\mu=-2g^{\mu\nu}\partial_\nu u,\quad \Lb'^\mu=-2g^{\mu\nu}\partial_\nu \ub, \quad  L\ub=1, \quad \Lb u=1.  $$
		Here $\Lb:=\Omega^2\Lb', \quad L:=\Omega^2 L'$ are also called equivariant vector fields. 
	\end{remark}

	\subsection{The Coordinate System}\label{coordinates}
	We shall use a coordinate system $(u,\ub, \theta^1, \theta^2)$. Here $u$ and $\ub$ are defined as above. To get $(\theta^1, \theta^2)$ for each point on $\S$, we follow the approach in Chapter 1 of \cite{Chr:book}. We first define a coordinate system $(\theta^1, \theta^2)$ on $S_{u_{\infty},0}$. Since $S_{u_{\infty},0}$ is the standard $2$-sphere in Minkowski spacetime, here we use the coordinates of stereographic projection. Then we extend this coordinate system to $\Hb_0$ by requiring
	$$\Ls_{\Lb} \theta^A=0\mbox{ on $\Hb_0$}. \footnote{On $\Hb_0$, we have $\O=1$ and $\Ls_{\Lb} \theta^A=\f{\partial}{\partial u}\theta^A$.}$$
	Here $\Ls_L$ is the restriction of the Lie derivative to $TS_{u,\ub}$. In other words,  given a point $p$ on $S_{u_{\infty}, 0}$, assuming $l_p$ is the incoming null geodesic on $\Hb_0$ emanating from $p$, then all the points along $l_p$ are assigned the same angular coordinate $(\theta^1, \theta^2)$.
	We further extend this coordinate system from $\Hb_0$ to the whole spacetime under the requirement
	$$\Ls_L \th^A=0,$$ 
	i.e. that all the points along the same outgoing null geodesics (along $L$) on $H_u$ have the same angular coordinate. We have thus established a coordinate system in a neighborhood of $S_{u_{\infty}, 0}$. With this coordinate system, we can rewrite $e_3$ and $e_4$ as 
	$$e_3=\Omega^{-1}\left(\frac{\partial}{\partial u}+b^A\frac{\partial}{\partial \th^A}\right), \hspace{.5mm} e_4=\Omega^{-1}\frac{\partial}{\partial \ub}.$$
	The Lorentzian metric $g$ takes the form
	\begin{equation}\label{equation g}
	g=-2\O^2(du\otimes d\ub+d\ub\otimes du)+\gamma_{AB}(d\theta^A-d^A du)\otimes(d\theta^B-d^B du).
	\end{equation}
	We require $b^A$ to satisfy $b^A=0$ on $\Hb_0$.
	
	\subsection{The equations}
	
	In this paper, we study the Einstein-Maxwell equations for a $4$-dimensional Lorentzian manifold $(\mathcal{M}, g)$ with signature $\{-,+,+,+\}$
	\be\label{em}
	R_{\mu\nu}-\f12Rg_{\mu\nu}=T_{\mu\nu},
	\ee
	where
	$$T_{\mu\nu}=F_{\mu\lambda}{F_{\nu}}^{\lambda}-\frac14 g_{\mu\nu}F_{\lambda\tau}F^{\lambda \tau}.$$ Here $F_{\alpha \beta}$ is an anti-symmetric $2-$tensor representing the electromagnetic field.
	We introduce null tetrads $\{e_a, e_b, e_3, e_4\}$ where $a,b=1,2$ and require 
	$$g(e_a, e_b)=\d_{ab}, \quad g(e_3, e_4)=-2, \quad g(e_a, e_3)=0, \quad g(e_a, e_4)=0.$$
	For the Weyl curvature $W_{\mu\nu\lambda\tau}$, we define the null Weyl curvature components:
	\begin{equation}
	\begin{split}
	\a_{ab}&=W(e_a, e_4, e_b, e_4),\quad \, \,\,   \ab_{ab}=W(e_a, e_3, e_b, e_3),\\
	\b_a&= \frac 1 2 W(e_a,  e_4, e_3, e_4) ,\quad \bb_a =\frac 1 2 W(e_a  e_3,  e_3, e_4),\\
	\rho&=\frac 1 4 W(e_4,e_3, e_4,  e_3),\quad \sigma=\frac 1 4  \,^*W(e_4,e_3, e_4,  e_3).
	\end{split}
	\end{equation}
	Here $\, ^*W$ is the Hodge dual of $W$. 
	\vspace{3mm}
\par	\noindent	For the Riemann curvature tensor $R_{\mu\nu\lambda\tau}$, we define the null Riemann curvature components:
	\begin{equation}
	\begin{split}
	{(\a_{\M R})}_{ab}&=R(e_a, e_4, e_b, e_4),\quad \, \,\,   (\ab_{\M R})_{ab}=R(e_a, e_3, e_b, e_3),\\
	(\b_{\M R})_a&= \frac 1 2 R(e_a,  e_4, e_3, e_4) ,\quad (\bb_{\M R})_a =\frac 1 2 R(e_a,  e_3,  e_3, e_4),\\
	\rho_{\M R}&=\frac 1 4 R(e_4,e_3, e_4,  e_3),\quad \sigma_{\M R}=\frac 1 4  \,^*R(e_4,e_3, e_4,  e_3).
	\end{split}
	\end{equation}
	Here $\, ^*R$ is the Hodge dual of $R$. 
	
\vspace{3mm}	\noindent Denote $D_A:=D_{e_{A}}$. We define the Ricci coefficients:
	
	\begin{equation}
	\begin{split}
	&\chi_{AB}=g(D_A e_4,e_B),\, \,\, \quad \chib_{AB}=g(D_A e_3,e_B),\\
	&\eta_A=-\frac 12 g(D_3 e_A,e_4),\quad \etab_A=-\frac 12 g(D_4 e_A,e_3),\\
	&\omega=-\frac 14 g(D_4 e_3,e_4),\quad\,\,\, \omegab=-\frac 14 g(D_3 e_4,e_3),\\
	&\zeta_A=\frac 1 2 g(D_A e_4,e_3).
	\end{split}
	\end{equation}
	\vspace{3mm} \par \noindent We decompose $\chi$ and $\chib$ into its trace and traceless parts. Denote by $\chih_{AB}$ and $\chibh_{AB}$ the traceless part of $\chi_{AB}$ and $\chib_{AB}$ respectively. \\

	\noindent		We further define
	$$(\a_F)_a=F_{a4}, \quad (\ab_F)_a=F_{a3}, \quad \rho_F=\f12F_{34}, \quad \sigma_F=F_{12}.$$
	
	\noindent Note that (\ref{em}) implies
	$$R_{\mu\nu}=T_{\mu\nu}  \quad \mbox{and} \quad R=0.$$

	\noindent		Expressed in this double null frame, we have
	\begin{gather}
	\nabla_4 \hspace{.5mm} \tr\chi + \frac{1}{2}\hspace{.5mm}  (\tr\chi)^2 = -\lvert \chihat \rvert^2- 2 \omega \hspace{.5mm} \tr\chi - \alpha_F^2,\\
	\nabla_4\hspace{.5mm}  \chihat + \tr\chi \hspace{.5mm} \chihat = -2\hspace{.5mm} \omega \hspace{.5mm} \chihat -\alpha,\\ \nabla_3\hspace{.5mm}  \tr\chibar + \frac{1}{2}\hspace{.5mm}  (\tr\chibar)^2 = -\lvert \chibarhat \rvert^2 -2\hspace{.5mm}  \omegabar \hspace{.5mm} \tr\chibar - \alphabar_F^2, \\ 
	\nabla_3 \hspace{.5mm}  \chibarhat + \tr\chibar \hspace{.5mm} \chibarhat = -2\omegabar \hspace{.5mm} \chibarhat- \alphabar, \\ 
	\nabla_4 \hspace{.5mm} \tr\chibar + \frac{1}{2}\hspace{.5mm} \tr\chi \hspace{.5mm} \tr\chibar = 2\hspace{.5mm} \omega \hspace{.5mm} \tr\chibar + 2\hspace{.5mm} \rho - \chihat\cdot \chibarhat +2 \hspace{.5mm} \text{div}\hspace{.5mm} \etabar +2 \lvert \etabar \rvert^2,\\ 
	(\nabla_4 \chibarhat)_{ab} + \frac{1}{2}\hspace{.5mm}  \tr\chi\hspace{.5mm}  \chibarhat_{ab} = (\nabla \widehat{\otimes} \etabar)_{ab} + 2\hspace{.5mm} \omega \hspace{.5mm} \chibarhat_{ab} -\frac{1}{2}\hspace{.5mm} \tr\chibar \hspace{.5mm} \chihat_{ab} +(\etabar \widehat{\otimes} \etabar)_{ab}-\frac12(\alpha_F\widehat{\otimes}\underline{\alpha}_F)_{ab},
	\\ \nabla_3 \tr\chi + \frac{1}{2}\hspace{.5mm} \tr\chi \hspace{.5mm} \tr\chibar = 2\hspace{.5mm} \omegabar \hspace{.5mm} \tr\chi +2 \hspace{.5mm} 		
	\rho- \chihat\cdot \chibarhat + 2 \hspace{.5mm} \text{div}\hspace{.5mm}  \eta + 2\hspace{.5mm} \lvert \eta \rvert^2,
	\\ 		
	(\nabla_3 \chihat)_{ab} + \frac{1}{2} \tr\chibar \hspace{.5mm}\chihat_{ab} = (\nabla \widehat{\otimes} \eta)_{ab} + 2\hspace{.5mm} \omegabar\hspace{.5mm} \chihat_{ab} - \frac{1}{2}\hspace{.5mm} \tr\chi \hspace{.5mm} \underline{\chihat}_{ab} + (\eta \widehat{\otimes} \eta)_{ab}-\frac12(\alpha_F\widehat{\otimes}\underline{\alpha}_F)_{ab}.
	\end{gather}

	\noindent Note 	that
	\[\beta_a-\frac12 R_{a4}=(\beta_{\M R})_a,\]
	\[\underline{\beta}_a+\frac12 R_{a3}=(\underline{\beta}_{\M R})_a,\]
	\[\rho-\frac12 R_{43}=\rho_{\M R}.\]
	Moreover,
	$$R_{11}=\f12\sigma_F^2+\f12\rho_F^2-(\a_F)_1(\ab_F)_1+\f12\a_F\cdot \ab_F,$$
	$$R_{22}=\f12\sigma_F^2+\f12\rho_F^2-(\a_F)_2(\ab_F)_2+\f12\a_F\cdot \ab_F,$$
	$$R_{4a}=R_{a4}=\rho_F (\alpha_F)_a-\sigma_F \epsilon_{ab}(\alpha_F)_b,$$
	$$R_{3a}=R_{a3}=-\rho_F (\ab_F)_a-\sigma_F \epsilon_{ab}(\ab_F)_b,$$
	$$R_{43}=\rho_F^2+\sigma_F^2, \quad R_{44}=\a_F\cdot \a_F, \quad R_{33}=\ab_F\cdot \ab_F.$$
	
	\vspace{3mm} \par \noindent The other components satisfy the following transport equations:

	\begin{gather}
	\nabla_4 \eta_a = - \chi_{ab} \cdot (\eta - \etabar)_b-\beta_a-\frac12 R_{a4}, \label{etastructureequation} \\ \label{etabarstructureequation}
	\nabla_3 \etabar_a = -\chibar_{ab} \cdot (\etabar - \eta)_b +\underline{\beta}_a-\frac12 R_{a3}, \\ \nabla_4 \omegabar = 2\omega \omegabar -\eta\cdot \etabar + \frac{1}{2} \lvert \eta \rvert^2 +\frac{1}{2} \rho+\frac14 R_{34}, \\ 
	\nabla_3 \omega = 2\omega \omegabar -\eta\cdot \etabar + \frac{1}{2} \lvert \etabar \rvert^2 +\frac{1}{2} \rho+\frac14 R_{34},
	\end{gather}		
	as well as the constraint equations
	\begin{gather} 
	\text{div}\hspace{.5mm}\chihat =  \frac{1}{2} \nabla \tr\chi - \frac{1}{2}(\eta - \etabar) \cdot (\chihat - \frac{1}{2} \tr\chi) -\beta_\mathcal{R}, \\
	\text{div}\hspace{.5mm}\chibarhat =  \frac{1}{2} \nabla \tr\chibar - \frac{1}{2}(\eta - \etabar) \cdot (\chibarhat - \frac{1}{2} \tr\chibar) -\betabar_\mathcal{R},\\ 
	\text{curl} \hspace{.5mm} \eta = - \text{curl}\hspace{.5mm} \etabar = \sigma + \frac{1}{2} \chibarhat \wedge \chihat, \\
	K = -\rho_\mathcal{R} - \frac{1}{4} \tr\chi\hspace{.5mm} \tr\chibar + \frac{1}{2} \hspace{.5mm}\chihat \cdot \chibarhat.
	\end{gather}
	
	\par \noindent	Here $K$ is the Gauss curvature of the spheres $S_{u,\ubar}$. The null curvature components satisfy the null Bianchi equations
	
	\begin{gather}
	\nabla_4 \beta + 2\tr\chi \beta = \text{div} \alpha - 2\omega \beta +\eta\cdot \alpha -\frac{1}{2}(D_A R_{44}-D_4 R_{4A}), \label{Bianchi beta}\\
	\nabla_3 \beta + \tr\chibar \beta = \nabla \rho + \prescript{*}{}{\nabla \sigma} +2\chihat\cdot \beta+2\omegabar \beta +3(\eta \rho + \prescript{*}{}{\eta}\sigma)  +\frac{1}{2}(D_4 R_{A3} -D_3R_{A4}),\\ 
	\nabla_4 \betabar +\tr\chi \betabar = -\nabla \rho + \prescript{*}{}{\nabla \sigma} +2\chibarhat\cdot \beta + 2\omega \betabar - 3(\etabar \rho - \Hodge{\etabar}\sigma)- \frac{1}{2}(D_4 R_{A3}-D_3R_{A4}),\\
	\nabla_3 \betabar +2\tr\chibar \hspace{.5mm}\betabar= -\text{div}\alphabar  -2 \omegabar \hspace{.5mm}\betabar +\etabar \alphabar + \frac{1}{2}(D_A R_{33}-D_{3}R_{3A}),\\ 
	\nabla_4 \alphabar +\frac{1}{2}\tr\chi \alphabar = -\nabla\widehat{\otimes} \betabar +4\omega \alphabar -3(\chibarhat\rho - \Hodge{\chibarhat}\sigma)+(\zeta - 4\etabar)\widehat{\otimes}\betabar +\frac{1}{4}(D_4 R_{33}-D_3 R_{34})g_{ab}, \\
	\nabla_3 \alpha + \frac{1}{2}\tr\chibar \alpha = \nabla \widehat{\otimes}\beta+ 4 \omegabar \alpha -3(\chihat \rho +\Hodge{\chihat}\sigma)+(\zeta + 4 \eta) \widehat{\otimes} \beta +\frac{1}{4}(D_3 R_{44}- D_4 R_{43}) g_{ab},\\ 
	\nabla_4 \rho + \frac{3}{2} \tr\chi \rho = \text{div} \beta - \frac{1}{2}\chibarhat \cdot \alpha +\zeta\cdot  \beta + 2\etabar \cdot \beta - \frac{1}{4}(D_3 R_{44} - D_4 R_{43}),\\ 
	\nabla_3 \rho + \frac{3}{2}\tr\chibar \rho = -\text{div} \betabar - \frac{1}{2}\chihat\cdot\alphabar+ \zeta\cdot \betabar - 2\eta\cdot \betabar + \frac{1}{4}(D_3 R_{34} - D_4 R_{33}),\\ 
	\nabla_4 \sigma +\frac{3}{2}\tr\chi \sigma = -\text{div}\Hodge{\beta} + \frac{1}{2}\hspace{.5mm}\hat{\chibar}\cdot \Hodge{\alpha}- \zeta \cdot \Hodge{\beta} -2 \etabar \cdot \Hodge{\beta} - \frac{1}{4}(D_\mu R_{4\nu} - D_\nu R_{4\mu}){\epsilon^{\mu \nu}}_{34}, \\ 
	\nabla_3 \sigma + \frac{3}{2} \tr\chibar \sigma  = -\text{div} \Hodge{\betabar} + \frac{1}{2} \chihat \cdot \Hodge{\alphabar} -\zeta \cdot \Hodge{\betabar} - 2\eta \cdot \Hodge{\betabar} +\frac{1}{4}(D_\mu R_{3\nu} - D_\nu R_{3 \mu}) {\epsilon^{\mu \nu}}_{34} \label{Bianchi rho}.\end{gather}Here, the Schouten tensor $S_{\mu \nu}$ is equal to the Ricci tensor $R_{\mu \nu}$ because of the special form of the electromagnetic field tensor, namely the fact that the Ricci scalar $R$ vanishes. Finally, the Maxwell equations are equivalent to the null Maxwell equations
	
	\begin{gather}
	\nabla_4 \alphabar_F+ \frac{1}{2}\tr\chi \alphabar_F = - \nabla \rho_F + \Hodge{\nabla} \sigma_F -2 \Hodge{\etabar}\cdot \sigma_F -2 \Hodge{\eta}\cdot \rho_F+2\omega \hspace{.5mm} \alphabar_F- \chibarhat \cdot \alpha_F, \\ 
	\nabla_3 \alpha_F+ \frac{1}{2} \tr\chibar \alpha_F =  \nabla \rho_F+ \Hodge{\nabla}\sigma_F -2\Hodge{\etabar} \cdot \sigma_F + 2 \etabar \cdot \rho_F +2\omegabar \alpha_F -\chihat \cdot \alphabar_F,\\ 
	\nabla_4 \rho_F = \text{div}\hspace{.5mm} \alpha_F - \tr\chi \rho_F - (\eta-\etabar) \cdot \alpha_F, \\
	\nabla_4 \sigma_F= - \text{curl}\hspace{.5mm} \alpha_F  -\tr\chi \sigma_F +(\eta-\etabar) \cdot \Hodge{\alpha_F}, \\
	\nabla_3 \rho_F+ \tr\chibar \rho_F =-\text{div} \alphabar_F +(\eta-\etabar)\cdot \alphabar_F, \\
	\nabla_3 \sigma_F +\tr\chibar \sigma_F = -\text{curl}\hspace{.5mm} \alphabar_F  + (\eta-\etabar) \cdot \Hodge{\alphabar_F}.
	\end{gather}

	\subsection{Integration}\label{24Integration} Let $U$ be a coordinate patch on $\S$ and let $p_U$ be the corresponding partition of unity. For a function $\phi$, we define its integral on $\S$ and along $H_u, \Hbar_{\ubar}$ by
	
	\[ \int_{\S} \phi := \sum_{U} \int_{-\infty}^{\infty} \int_{-\infty}^{\infty} \phi \cdot p_{{\color{black}U}} \cdot \sqrt{\det \gamma} \hspace{.5mm} \text{d} \theta^1 \hspace{.5mm} \text{d} \theta^2,    \] 	\[ \int_{\Hu} \phi := \sum_{U}\int_0^{\ubar} \int_{-\infty}^{\infty} \int_{-\infty}^{\infty} \phi \cdot 2\hspace{.5mm} p_{{\color{black}U}} \cdot \O \cdot \sqrt{\det \gamma} \hspace{.5mm} \text{d} \theta^1 \hspace{.5mm} \text{d} \theta^2 \hspace{.5mm} \dubarprime,    \] \[ \int_{\Hbar_{\ubar}^{(u_{\infty}, u)}  } \phi := \sum_{U}\int_{u_{\infty}}^{u} \int_{-\infty}^{\infty} \int_{-\infty}^{\infty} \phi \cdot 2\hspace{.5mm} p_{{\color{black}U}} \cdot \O \cdot \sqrt{\det \gamma} \hspace{.5mm} \text{d} \theta^1 \hspace{.5mm} \text{d} \theta^2 \hspace{.5mm} \duprime.    \]Let $\mathcal{D}_{u,\ubar}$ be the region $u_{\infty} \leq u^{\prime} \leq u$, $0 \leq \ubar^{\prime} \leq \ubar$. We define the integral of $\phi$ in the region $\mathcal{D}_{u,\ubar}$ by \[ \int_{\mathcal{D}_{u,\ubar}} \phi := \sum_U \int_{u_{\infty}}^u \int_0^{\ubar}   \int_{-\infty}^{\infty} \int_{-\infty}^{\infty} \phi \cdot p_U \cdot \O^2 \cdot \sqrt{-\det {\color{black}g}} \hspace{.5mm} \text{d}\theta^1 \hspace{.5mm} \text{d}\theta^2 \dubarprime \duprime.       \]We proceed to define, for $1\leq p < \infty$, the $L^p$-norms for an arbitrary tensorfield $\phi$:
	
	\[ \pSu{\phi}^p:= \int_{\S} \langle \phi,\phi \rangle_{\gamma}^{p/2},     \] \[ \lVert \phi \rVert_{L^p(H_u)}^p := \int_{H_u}    \langle \phi, \phi \rangle_{\gamma}^{p/2},    \]\[ \lVert \phi \rVert_{L^p(\Hbar_{\ubar})}^p := \int_{\Hbar_{\ubar}}    \langle \phi, \phi \rangle_{\gamma}^{p/2}.    \] When $p=\infty,$ we define the $L^{\infty}$ norm by \[  \inftySu{\phi} := \sup_{\theta \in \S} \langle \phi, \phi \rangle_{\gamma}^{\frac{1}{2}}(\theta).    \]
	\subsection{Definition of signatures}\label{definitionsignaturessection}
	
	We give the following table for signatures throughout this work:
	
	\begin{center}
		\begin{tabular}{|c|c|c|c|c|c|c|c|c|c|c|c|c|c|c|c|c|c|c|}
			\hline 
			& $\alpha$ & $\beta$ & $\rho$ & $\sigma$ & $\betabar$ & $\alphabar$ & $\chi$ & $\omega$ & $\zeta$ & $\eta$ & $\etabar$ & $\tr\chibar$ & $\chibarhat$ & $\omegabar$ & $\alpha_F$ & $\rho_F$ & $\sigma_F$ & $\alphabar_F$ \\ 
			\hline 
			$s_2$ & 0 & 0.5 & 1 & 1 & 1.5  & 2 & 0 & 0 & 0.5 & 0.5 & 0.5 & 1 &1  & 1 & 0 & 0.5 & 0.5 & 1 \\ 
			\hline 
		\end{tabular} 	
	\end{center}This comes from wanting to have $s_2(\alpha_F) = s_2(\chihat), s_2(\alphabar_F)=s_2(\chibarhat)$ and  $s_2(\rho_F,\sigma_F)$ such that the null Maxwell equations conserve signature.
	
	\subsection{Scale-invariant norms}
	For any horizontal tensor-field $\phi$ with signature $s_2(\phi)$, we define the following scale-invariant norms on $S_{u,\ubar}$:
	
	\[\scaleinfinitySu{\phi} :=  a^{-s_2(\phi)}  \lvert u \rvert^{2s_2(\phi)+1} \lVert \phi \rVert_{L^{\infty}(S_{u,\ubar})},    \]\[ \scaletwoSu{\phi} := a^{-s_2(\phi)}\lvert u \rvert^{2s_2(\phi)} \lVert \phi \rVert_{L^2(S_{u,\ubar})},   \]\[ \lVert \phi \rVert_{\mathcal{L}^{1}_{(sc)}(S_{u,\ubar})}  := a^{-s_2(\phi)}\lvert u \rvert^{2s_2(\phi)-1} \lVert \phi \rVert_{L^1(S_{u,\ubar})}.   \]Along $H_u^{(0,\ubar)}$ and $\Hbar_{\ubar}^{(u_\infty,u)}$ we also define scale-invariant norms along null hypersurfaces \[   \scaletwoHu{\phi}^2 := \int_{0}^{\ubar} \scaletwoSuubarprime{\phi}^2\hspace{.5mm} \text{d}\ubar^{\prime},     \] \[\scaletwoHbaru{\phi}^2   := \int_{u_{\infty}}^{u} \frac{a}{\lvert u^\prime \rvert^2} \scaletwoSuprime{\phi}^2\hspace{.5mm}\text{d}u^\prime. \]

	\subsection{Conservation of signatures}
	Notice that under the table of signatures in Section \ref{definitionsignaturessection} and the fact that the induced metric on a $2-$sphere $\gamma_{\alpha\beta}$ has $s_2(\gamma_{\alpha \beta})=0$, the following remarkable property follows for tensorfields $\phi_1$ and $\phi_2$:
	\[  s_2(\phi_1 \cdot \phi_2) = s_2(\phi_1)  +s_2(\phi_2).  \]This ensures \textit{signature conservation} for all null structure, Bianchi, constraint and null Maxwell equations. When working with scale-invariant norms, this key property
	enables us to treat all the terms on the right hand side as one term. For example, look at the null Maxwell equation for $\nabla_3 \alpha_F$:
	
	\[ 	\nabla_3 \alpha_F+ \frac{1}{2} \tr\chibar \alpha_F = - \nabla \rho_F+ \Hodge{\nabla}\sigma_F -2\Hodge{\etabar} \cdot \sigma_F + 2 \etabar \cdot \rho_F +2\omegabar \alpha_F -\chihat \cdot \alphabar_F. \]There holds \begin{itemize}
		\item $s_2(\nabla_3 \alpha_F) = s_2 \alpha_F +1 =1,$
		\item $s_2(\tr\chibar \hspace{.5mm} \alpha_F) = s_2(\tr\chibar )+s_2(\alpha_F) = 1$,
		
		\item $s_2(\nabla \rho_F, \Hodge{\nabla}\sigma_F) = \frac{1}{2} + s_2(\rho_F,\sigma_F)= \frac{1}{2}+ \frac{1}{2}=1$,
		\item  $s_2(\etabar \cdot \rho_F, \Hodge{\etabar}\cdot \sigma_F) = \frac{1}{2}+ \frac{1}{2}=1$,
		\item $s_2(\omegabar \hspace{.5mm} \alpha_F) =1+0=1,$
		\item $s_2(\chihat \cdot \alphabar_F) = 0+1=1$. 
	\end{itemize}Thus, throughout the equation, there is a balance of signature.
	\subsection{H\"older's inequality in scale-invariant norms}Any two tensorfields satisfy the following scale-invariant H\"older inequalities:
	
	\be \lVert \phi_1 \cdot \phi_2 \rVert_{\mathcal{L}^{1}_{(sc)}(S_{u,\ubar})} \lesssim \frac{1}{\lvert u \rvert}\scaletwoSu{\phi_1}\scaletwoSu{\phi_2},   \ee \be \lVert \phi_1 \cdot \phi_2 \rVert_{\mathcal{L}^{1}_{(sc)}(S_{u,\ubar})} \leq \frac{1}{\lvert u \rvert}\scaleoneSu{\phi_1}\scaleinfinitySu{\phi_2},   \ee \be  \scaletwoSu{\phi_1\cdot \phi_2} \leq \frac{1}{\lvert u \rvert} \scaletwoSu{\phi_1}\scaleinfinitySu{\phi_2},      \ee Also, the following inequality holds \be \scaleinfinitySu{\phi_1\cdot \phi_2} \leq \frac{1}{\lvert u \rvert} \scaleinfinitySu{\phi_1}\scaleinfinitySu{\phi_2}.  \ee Crucially, in the region of study $\frac{1}{\lvert u \rvert}\ll 1$. This means, when measuring the size of a product of terms in scale-invariant norms, this size is very small compared to the scale-invariant norms of the individual terms. Essentially, it is this crucial fact that allows us to close all bootstrap arguments throughout this paper. 
	
	\subsection{Norms}
	
	Let $\psi \in \begin{Bmatrix} \omega, \tr\chi, \eta, \etabar, \omegabar     \end{Bmatrix}, \hspace{.5mm} \Psi \in \begin{Bmatrix}
	\beta, \rho, \sigma, \betabar,\alphabar \end{Bmatrix}$ and $\Psi^\prime \in \begin{Bmatrix} \rho, \sigma, \betabar, \alphabar \end{Bmatrix}$. Denote $\widetilde{\tr \chibar} = \tr \chibar + \frac{2}{\lvert u\rvert}$. Also, let $\Y \in \begin{Bmatrix}
	\rho_F, \sigma_F, \alphabar_F
	\end{Bmatrix}$. For $0\leq i \leq 6$, we define 
	
	\begin{equation}
\begin{split} \label{Ricinfty} \mathcal{O}_{i,\infty}(u, \ubar) := \frac{1}{a^{\frac{1}{2}}} \scaleinfinitySu{ ( a^{\frac{1}{2}}\nabla )^i \chihat}     +\scaleinfinitySu{ ( a^{\frac{1}{2}}\nabla )^i \psi}  + \frac{a^{\frac{1}{2}}}{\lvert u \rvert } \scaleinfinitySu{( a^{\frac{1}{2}}\nabla )^i \chibarhat} \\ 
	+ \frac{a}{\lvert u \rvert^2} \scaleinfinitySu{ (a^{\frac{1}{2}}\nabla )^i \tr\chibar}  + \frac{a}{\lvert u \rvert} \scaleinfinitySu{ (a^{\frac{1}{2}}\nabla )^i      \widetilde{\tr\chibar}},\end{split}
\end{equation} \begin{gather*}
	\mathcal{R}_{i,\infty}(u,\ubar) = \frac{1}{a^{\frac{1}{2}}} \scaleinfinitySu{(a^{\frac{1}{2}}\nabla)^i \alpha} + \scaleinfinitySu{(a^{\frac{1}{2}}\nabla)^i \Psi}, \\
	\mathcal{F}_{i,\infty}(u,\ubar) =  \frac{1}{a^{\frac{1}{2}}} \scaleinfinitySu{(a^{\frac{1}{2}}\nabla)^i \alpha_F} + \scaleinfinitySu{(a^{\frac{1}{2}}\nabla)^i \Y}. \end{gather*} For $ 0 \leq i \leq 9$ and $0\leq j \leq 10$, we define \begin{align*}
	\mathcal{O}_{j,2}(u,\ubar) = \frac{1}{a^{\frac{1}{2}}} \scaletwoSu{ ( a^{\frac{1}{2}}\nabla )^j \chihat}     +\scaletwoSu{ ( a^{\frac{1}{2}}\nabla )^j \psi}  + \frac{a^{\frac{1}{2}}}{\lvert u \rvert } \scaletwoSu{( a^{\frac{1}{2}}\nabla )^j \chibarhat} \\ 
	+ \frac{a}{\lvert u \rvert^2} \scaletwoSu{ (a^{\frac{1}{2}}\nabla )^j \tr\chibar}  + \frac{a}{\lvert u \rvert} \scaletwoSu{ (a^{\frac{1}{2}}\nabla )^j      \widetilde{\tr\chibar}}, \end{align*} \begin{gather*} 
	\mathcal{R}_{i,2}(u,\ubar) = \frac{1}{a^{\frac{1}{2}}}
	\scaletwoSu{(a^{\frac{1}{2}}\nabla)^i \alpha} + \scaletwoSu{(a^{\frac{1}{2}}\nabla)^i \Psi}, \\ 
		\mathcal{F}_{j,2}(u,\ubar) =  
		\scaletwoSu{(a^{\frac{1}{2}})^{j-1}\nabla^j \alpha_F} + \scaletwoSu{(a^{\frac{1}{2}}\nabla)^j (\rho_F,\sigma_F,\alphabar_F)}.
	\end{gather*}For $0\leq i \leq 10$ and $0\leq j \leq 11$ we define \begin{gather*}
	\mathcal{R}_i(u,\ubar) =  \frac{1}{a^{\frac{1}{2}}}\scaletwoHu{(a^{\frac{1}{2}}\nabla)^i \alpha} + \scaletwoHu{(a^{\frac{1}{2}}\nabla)^i \Psi}, \\ 
	\underline{\mathcal{R}}_i(u,\ubar) = \frac{1}{a^{\frac{1}{2}}}\scaletwoHbaru{(a^{\frac{1}{2}}\nabla)^i \beta} + \scaletwoHbaru{(a^{\frac{1}{2}}\nabla)^i \Psi^\prime}, \\
	\mathcal{F}_j(u,\ubar) =  \frac{1}{a^{\frac{1}{2}}}\scaletwoHu{(a^{\frac{1}{2}})^{j-1}\nabla^j \alpha_F} + \scaletwoHu{(a^{\frac{1}{2}})^{j-1}\nabla^j (\rho_F,\sigma_F)},	 \\  
		\underline{\mathcal{F}}_j(u,\ubar) = \frac{1}{a^{\frac{1}{2}}}\scaletwoHbaru{(a^{\frac{1}{2}})^{j-1}\nabla^j (\rho_F, \sigma_F)} + \scaletwoHbaru{(a^{\frac{1}{2}})^{j-1}\nabla^j \alphabar_F}. 
	\end{gather*}We now set $\mathcal{O}_{i,\infty}, \mathcal{O}_{i,2}, \mathcal{R}_{i,\infty}, \mathcal{R}_{i,2}, \mathcal{F}_{i,\infty},\mathcal{F}_{i,2}$ to be the supremum over $u,\ubar$ in the spacetime region of the norms $\mathcal{O}_{i,\infty}(u,\ubar), \mathcal{O}_{i,2}(u,\ubar), \mathcal{R}_{i,\infty}(u,\ubar), \mathcal{R}_{i,2}(u,\ubar), \mathcal{F}_{i,\infty}(u,\ubar)$ and $\mathcal{F}_{i,2}(u,\ubar)$ respectively. Finally, define $\mathcal{O}, \mathcal{R}$ and $\mathcal{F}$:
	
	\[ \mathcal{O} = \sum_{i\leq 6}\left( \mathcal{O}_{i,\infty} + \mathcal{R}_{i,\infty} + \mathcal{F}_{i,\infty} \right)   +\sum_{0\leq i \leq 9} \mathcal{R}_{i,2} + \sum_{0 \leq j \leq 10} \left( \mathcal{O}_{j,2} + \mathcal{F}_{j,2}\right) ,       \]\[ \mathcal{R} = \sum_{0\leq i \leq 10}\left(  \mathcal{R}_i + \underline{\mathcal{R}}_i   \right),  \hspace{1mm} \mathcal{F}= \sum_{0\leq j \leq 11}\left( \mathcal{F}_j + \underline{\mathcal{F}}_j    \right).          \]
	\subsection{An explicit form of the Bianchi equations}
	
Recall  the Bianchi equations \eqref{Bianchi beta}-\eqref{Bianchi rho}. We will work on a term-by-term basis to give the explicit forms of the right hand sides (RHS) of these equations. We will use the property 	
	\[   (D_\alpha R)_{\beta \gamma} = D_\alpha(R_{\beta  \gamma}) - R(D_\alpha \e_\beta, \e_{\gamma}) - R(\e_{\beta}, D_\alpha \e_{\gamma}),       \] as well as the following identities:
	
	\begin{gather}
	D_A \e_{B} = \nabla_A \e_{B} +\frac12 \chi_{AB} \e_3 +\frac12\chibar_{AB} \e_4, \\ 
	D_3 \e_A = \nabla_3 \e_A +\eta_A \e_3, \hspace{2mm} D_4 \e_A = \nabla_4 \e_A +  \etabar_A \e_4, \\  
	D_A \e_3= {\chibar_A}^{\sharp B}\e_B +  \zeta_A \e_3, \hspace{2mm} D_A \e_4 = {\chi_A}^{\sharp B} \e_B - \zeta_A \e_4, \\   
	D_3 \e_4 = 2\eta^{\sharp A} \e_A +2\hspace{.5mm} \omegabar \hspace{.5mm}\e_4, \hspace{2mm} D_4 \e_3 = 2\etabar^{\sharp A} \e_A +2\hspace{.5mm}\omega \hspace{.5mm} \e_3, \\ 
	D_3 \e_3 = -2\omegabar \e_3, \hspace{2mm} D_4 \e_4 = -2 \omega \e_4.
	\end{gather}
	
\par \noindent 	We therefore compute 	
	\begin{align*}(D_A R)_{44} =& 2\alpha_F \cdot \nabla \alpha_F - 2R(D_A \e_4, \e_4)= 2\alpha_F \cdot \nabla \alpha_F -( \psi, \chihat)\cdot R(\e_A, \e_4) + \psi \cdot \alpha_F \cdot \alpha_F \\ =& 2\alpha_F \cdot \nabla \alpha_F + (\psi,\chihat)\cdot \Y \cdot \alpha_F + \psi \cdot \alpha_F^2 ,  \end{align*}
		
	\begin{align*}
	(D_4 R)_{4A} =& D_4(R_{4A}) - R(D_4 \e_4, \e_A) - R(e_4, D_4 \e_A) = D_4(R_{4(\cdot)})(e_A) +(\psi \cdot \Y \cdot \alpha_F) \quad  \\ =& D_4(R_{4(\cdot)})(e_A) + \psi \cdot \Y \cdot \alpha_F,  \quad   \end{align*}
	
	\begin{align*}
	(D_A R)_{43} = D_A (R_{43})- R(D_A \e_4, \e_3) - R(D_A \e_3, \e_4) = \Y \cdot \nabla \Y + (\psi,\chibarhat,\chihat,\tr\chibar)\cdot(\Y,\alpha_F) \cdot \Y,
	\end{align*}
	
	\begin{align*}
	(D_4 R)_{A3} =& D_4(R_{A3}) - R(D_4 \e_A , \e_3) - R(\e_A, D_4 \e_3) = \Y \cdot \nabla_4 \Y + (\psi \cdot \Y \cdot (\Y, \a_F))\\ =& \Y \cdot \nabla (\Y, \alpha_F) +  (\psi,\chibarhat)\cdot \Y \cdot(\Y,\alpha_F)    ,
	\end{align*}since
	
	\[ \nabla_4 \Y = \nabla(\Y,\alpha_F) + (\psi,\chibarhat)\cdot(\Y,\alpha_F).     \]

	\par \noindent	Continuing, we have
	
	\begin{align*}
	(D_3 R)_{A4} =& D_3 (R_{A4}) - R(D_3 \e_A, \e_4) - R(\e_A, D_3 \e_4) = \alpha_F \cdot \nabla_3 \Y + \Y \cdot \nabla_3 \alpha_F + \psi \cdot \Y \cdot (\Y, \alpha_F)\\ =& (\alpha_F, \Y) \cdot \nabla \Y +(\psi,\tr\chibar)\cdot \alpha_F \cdot \Y +  (\psi,\chibarhat, \chih)\cdot \Y \cdot \Y, 
	\end{align*}
	
\par \noindent since we have
	
	\[ \nabla_3 \Y = \nabla \Y + (\tr\chibar, \psi) \cdot \Y,      \]\[  \nabla_3 \alpha_F = \nabla \Y + (\psi,\tr\chibar) \alpha_F +(\psi,\chihat) \cdot \Y.               \]Continuing, we have
	
	\begin{equation}
	(D_A R)_{33} = D_A(R_{33}) -2 R(D_A \e_3, \e_3)= \Y \cdot \nabla \Y + (\psi,\chibarhat, \tr\chib)\cdot \Y \cdot \Y,
	\end{equation}
	
	\begin{equation}
	(D_3 R)_{3A} = D_3 (R_{3A}) - R(D_3 \e_3, \e_A) - R(\e_3, D_3 \e_A) = \nabla_3 (R_{3A}) + \psi \cdot \Y \cdot \Y ,
	\end{equation}
	
	\begin{align*}
	(D_4 R)_{33} =&2\alphabar_F \cdot \nabla_4 \alphabar_F - 2R(D_{4}e_3, e_3) =2\alphabar_F \cdot \nabla_4 \alphabar_F + \psi \cdot \Y \cdot \Y = \Y \cdot \left(\nabla \Y + (\psi,\chibarhat) \cdot (\Y, \alpha_F)\right) + \psi \cdot \Y \cdot \Y\\ =& \Y \cdot \nabla \Y + (\psi,\chibarhat) \cdot \Y \cdot (\Y, \alpha_F),
	\end{align*}
	
	\begin{equation}
	(D_3 R)_{34} = D_3(R_{34})-R({D_3 e_3, e_4})-R(e_3, D_3 e_4)=\Y \cdot \nabla_3 \Y + \psi \cdot \Y \cdot \Y = \Y \cdot \nabla \Y + (\tr\chibar, \psi) \cdot \Y \cdot \Y ,
	\end{equation}

	\begin{align*}
	(D_3 R)_{44} =&2\alpha_F \cdot \nabla_3 \alpha_F-2R(D_3 e_4, e_4)=2\alpha_F \cdot \nabla_3 \alpha_F + \psi \cdot \alpha_F\cdot (\alpha_F, \Y) \\=& 2\alpha_F\cdot \nabla \Y + (\psi,\tr\chibar, \chibh) \alpha_F^2 + (\psi,\chih)\cdot \alpha_F \cdot \Y + \psi \cdot \alpha_F \cdot (\alpha_F, \Y),  
	\end{align*}
	
	\begin{align*}
	(D_4 R)_{43} =&D_4(R_{43})-R(D_4 e_4, e_3)-R(e_4, D_4 e_3)= \Y \cdot \nabla_4 \Y + \psi \cdot \Y \cdot \Y+\psi \cdot \Y \cdot \a_F \\ =& \Y \cdot \nabla(\Y,\alpha_F) + \Y \cdot (\psi,\chibarhat) \cdot (\Y, \alpha_F) .
	\end{align*}

	\par \noindent  The expressions for $(D_A R)_{4B}$ and $(D_A R)_{3B}$ are as follows:		
	
	\begin{align*}
	(D_A R)_{4B} =& D_A(R_{4B})- R(D_A e_4,e_B) - R(D_A e_B, e_4) = \left( \alpha_F \cdot \nabla \Y + \Y \cdot \nabla \alpha_F \right)\\ &+ \left( (\psi,\chihat) \cdot \Y \cdot \Y + \psi \cdot \alpha_F \cdot \Y \right) +  \left((\psi,\chihat) \cdot \Y \cdot \Y +  (\psi,\chibarhat,\tr\chibar)\cdot \alpha_F \cdot \alpha_F \right)\\ =& \alpha_F \cdot \nabla \Y + \Y \cdot \nabla \alpha_F + (\psi,\chibarhat,\chihat,\tr\chibar)\cdot (\alpha_F,\Y)\cdot (\alpha_F,\Y),
	\end{align*}
	
	\begin{equation}
	(D_A R)_{3B} = D_A(R_{3B})- R(D_A e_3,e_B) - R(D_A e_B, e_3) = \Y \cdot \nabla \Y + (\psi,\chihat,\chibarhat,\tr\chibar) \cdot \Y \cdot \Y.
	\end{equation}

	\subsubsection{An important renormalization}

A novel ingredient in our analysis is the introduction of renormalized quantities for $\beta$ and $\betabar$. The motivation behind the introduction of these two quantities stems from the Bianchi equations.	Take for example the identity for $\beta$:
	
	\[\nabla_4 \beta + 2tr\chi \beta = \text{div} \alpha - 2\omega \beta +\eta\cdot \alpha -\frac{1}{2}(D_A R_{44}-D_4 R_{4A}).\]Strictly speaking, this is an equality of 1-forms. In particular, if we evaluate on a vector $\e_A$, we get
	
	\begin{align*} (\nabla_4 \beta)(\e_A) +  2 \tr\chi \beta_A = (\text{div}\alpha)(\e_A) - 2 \omega \beta_A +(\eta\cdot\a)_A- \frac{1}{2}(D_A R)_{44} + \frac{1}{2}(D_4 R)_{4 A} \Rightarrow \\
	(\nabla_4 \beta)(\e_A) - \frac{1}{2}\nabla_4 (R_{4(\cdot)})(e_A) = \nabla \alpha + \psi \cdot \Psi + \alpha_F \cdot \nabla \alpha_F + (\psi,\chihat)\cdot \Y \cdot \alpha_F+ \psi \cdot \alpha_F^2+\psi\cdot \a \Rightarrow\\
	\nabla_4( \beta - \frac{1}{2}R_{4(\cdot)}) =  \nabla \alpha + \psi \cdot (\Psi, \a) + \alpha_F \cdot \nabla \alpha_F + (\psi,\chihat)\cdot \Y \cdot \alpha_F+ \psi \cdot \alpha_F^2  .  \end{align*}

	\par \noindent	  This motivates us to define the normalized curvature component \be\label{tbetadef} \tilde{\beta} := \beta - \frac{1}{2} R_{4(\cdot)}. \ee
	Similarly, we need to define
	
	\be \label{tbetabardef}\tilde{\betabar} := \betabar + \frac{1}{2}R_{3(\cdot)}. \ee The gain from \eqref{tbetadef} and \eqref{tbetabardef} is that the Bianchi equations are now expressed in a way that the right-hand sides of the equations are controllable in terms of the Ricci coefficients and the curvature and Maxwell components.
	
	\subsubsection{The Bianchi equations in schematic form for the renormalized components}
	
	In this section we give, in schematic form, the Bianchi equations expressed in terms of $\begin{Bmatrix} \alpha, \alphabar, \tbeta, \tbetabar, \rho,\sigma \end{Bmatrix}$. We explain our way of obtaining these. Take for example the transport equation for $\beta$:
	
	\[  	\nabla_4 \beta + 2\tr\chi \beta = \text{div} \alpha - 2\omega \beta +\eta\cdot \alpha -\frac{1}{2}(D_A R_{44}-D_4 R_{4A}).    \]Then \[ \nabla_4 \tbeta + 2 \tr\chi \tbeta =  \left( \nabla_4 \beta + 2\tr\chi \beta  \right) -\frac{1}{2}\nabla_4 (R_{4(\cdot)})(e_A) - \tr\chi R_{4A}.\]
	
\par \noindent Working similarly, we obtain

	\begin{gather}
	\nabla_4 \tbeta + 2\tr \chi\tbeta = \nabla \alpha + \psi \cdot (\alpha, \Psi) + \alpha_F \cdot \nabla \alpha_F + \psi \cdot (\Y,\alpha_F)\cdot \alpha_F+(\psi, \chih)\cdot\Y\cdot\a_{F}, \\
	\nabla_3 \tbeta + \tr\chibar \tbeta = \nabla \Psi + (\psi,\chihat) \Psi + \Y \nabla (\Y,\alpha_F) +(\alpha_F,\Y)\nabla \Y + (\psi,\chibarhat,\tr\chibar, \chih)\cdot (\alpha_F,\Y)\cdot \Y, \\
	\nabla_4 \tbetabar + \tr\chi \tbetabar = \nabla \Psi + (\psi,\chibarhat)\Psi  + (\alpha_F, \Y) \nabla \Y + (\psi,\chibarhat,\tr\chibar, \chih)\cdot (\alpha_F,\Y)\cdot \Y, \\
	\nabla_3 \tbetabar + 2\tr\chibar\hspace{.5mm}\tbetabar = \nabla \Psi + \psi\cdot \Psi + \Y \nabla \Y + (\psi,\tr\chibar, \chibh)\cdot \Y \cdot \Y, \\
	\nabla_4 \alphabar + \frac{1}{2}\tr \chi \alphabar = \nabla \Psi + (\psi\chibarhat)\cdot \Psi + \Y \cdot \nabla \Y + (\psi,\tr\chibar)\cdot\Y\cdot\Y +(\psi, \chibh)\cdot\Y\cdot (\Y, \a_F),   \\ \label{renalphaeq}
	\nabla_3 \alpha + \frac{1}{2}\tr\chibar \alpha = \nabla \Psi + \alpha_F \cdot \nabla \Y + \Y \cdot (\nabla \alpha_F, \nab\Y) + (\psi,\chihat)\cdot \Psi + \psi\cdot \alpha  + (\psi,\chibarhat,\tr\chibar, \chih) \cdot (\alpha_F,\Y) \cdot (\alpha_F,\Y), \\
	\nabla_4 \rho + \frac{3}{2}\tr\chi \rho = \nabla \Psi + (\psi,\chibarhat)\cdot(\alpha,\Psi) + \alpha_F \cdot \nabla \Y + \Y \cdot (\nabla \alpha_F, \nab\Y) + (\psi,\chibarhat,\tr\chibar, \chih)\cdot (\alpha_F,\Y)\cdot (\alpha_F,\Y),  \\
	\nabla_3\rho + \frac{3}{2}\tr\chibar\rho = \nabla \Psi + \Y\cdot \nabla \Y + (\psi,\chihat)\cdot \Psi + (\psi,\tr\chibar)\cdot \Y \cdot \Y + (\psi,\chibarhat) \cdot (\Y,\alpha_F)\cdot \Y, \\
	\nabla_4 \sigma + \frac{3}{2}\tr\chi \sigma = \nabla \Psi + (\psi,\chibarhat)\cdot(\alpha,\Psi) + \alpha_F \cdot \nabla \Y + \Y \cdot \nabla \alpha_F + (\psi,\chibarhat,\tr\chibar)\cdot (\alpha_F,\Y)\cdot (\alpha_F,\Y),  \\
	\nabla_3\sigma + \frac{3}{2}\tr\chibar\sigma= \nabla \Psi + \Y\cdot \nabla \Y + (\psi,\chihat)\cdot \Psi + (\psi,\tr\chibar)\cdot \Y \cdot \Y + (\psi,\chibarhat) \cdot (\Y,\alpha_F)\cdot \Y. 
	\end{gather}
	\section{The preliminary estimates}
	\subsection{Setting up the bootstrap argument} 
	
	We shall employ a bootstrap argument to derive uniform upper bounds on $\mathcal{O}, \mathcal{R}, \underline{\mathcal{R}}, \mathcal{F}, \underline{\mathcal{F}}$ for the nonlinear Einstein-Maxwell equations. Along $H_{u_\infty}$ and $\Hbar_0$, by analysing the characteristic initial data, we can obtain the bounds
	
	\be\label{initialbounddata} \mathcal{O}^{(0)} + \mathcal{R}^{(0)} + \underline{\mathcal{R}}^{(0)} + \mathcal{F}^{(0)} + \underline{\mathcal{F}}^{(0)} \lesssim \mathcal{I}^{(0)}.        \ee Our goal is to show that in $\mathcal{D} = \begin{Bmatrix}
	(u,\ubar) \hspace{1mm} \mid \hspace{1mm} u_{\infty} \leq u \leq - a/4, 0\leq {\color{black}\ub} \leq 1 
	\end{Bmatrix}$ there holds
	
	\be \label{uniformestimate} \mathcal{O}(u,\ubar) + \mathcal{R}(u,\ubar) + \underline{\mathcal{R}}(u,\ubar) + \mathcal{F}(u,\ubar) + \underline{\mathcal{F}}(u,\ubar) \lesssim \left(\mathcal{I}^{(0)} \right)^4 + \left(\mathcal{I}^{(0)} \right)^2 +\mathcal{I}^{(0)} + 1.        \ee Once these uniform bounds are obtained, by a standard local existence result, the solutions can always
	be extended a bit towards the future direction of $u$. Hence, the uniform estimate \eqref{uniformestimate} for $u_{\infty} \leq u \leq - a/4$ would imply that a solution to the Einstein--Maxwell equations exists in the slab $\mathcal{D}$. To derive the uniform bound \eqref{uniformestimate},  we make the bootstrap assumptions
	
	\be  \label{bootstrapbounds}   \mathcal{O}(u,\ubar) \leq O, \hspace{2mm} \mathcal{R}(u,\ubar) + \underline{\mathcal{R}}(u,\ubar) \leq R, \hspace{2mm} \mathcal{F}(u,\ubar) + \underline{\mathcal{F}}(u,\ubar) \leq F,       \ee for large numbers $O, R$ and $F$ such that 
	\[   { \color{black}   \left(\mathcal{I}^{(0)} \right)^4 + \left(\mathcal{I}^{(0)} \right)^2 +\mathcal{I}^{(0)} + 1 \ll \min \begin{Bmatrix} O, R, F \end{Bmatrix}, }     \]but also such that \[ (O+R+F)^{20} \leq a^{\frac{1}{16}}.      \]Define the set $\mathcal{B} = \begin{Bmatrix}  u \hspace{1mm} \mid \hspace{1mm}  u_{\infty} \leq u \leq - a/4\hspace{1mm} \text{and} \hspace{1mm} \eqref{bootstrapbounds} \hspace{1mm} \text{holds for every}\hspace{1mm}  0\leq \ubar \leq 1        \end{Bmatrix}$. We are hoping to prove that $\mathcal{B}$ is in fact equal as a set to the entire interval $[u_{\infty}, -a/4]$. To do this, we take advantage of the topology of the unit interval. In particular, since it is connected, it suffices to show that the set $\mathcal{B}$ is both closed and open. 
	
	\vspace{3mm}
	
	\par \noindent By assumption, at $u= u_{\infty},$ we have \eqref{initialbounddata}. By continuity of solutions (via local existence), there exists a small $\epsilon > 0$ such that it holds for $u_{\infty} \leq u \leq u_{\infty}+\epsilon$ we have
	
	\[ \mathcal{O}^{(0)} \lesssim \mathcal{I}^{(0)} \ll O, \hspace{2mm} \mathcal{R}^{(0)} + \underline{\mathcal{R}}^{(0)} \lesssim \mathcal{I}^{(0)} \ll R , \hspace{2mm} \mathcal{F}^{(0)}+\underline{\mathcal{F}}^{(0)} \lesssim \mathcal{I}^{(0)}\ll F ,       \] \[  \mathcal{O}(u,\ubar)\lesssim 2 \mathcal{I}^{(0)}\ll O, \hspace{2mm} \mathcal{R}(u,\ubar) +\underline{\mathcal{R}}(u,\ubar)\lesssim 2\mathcal{I}^{(0)} \ll R, \hspace{2mm} \mathcal{F}(u,\ubar)+\underline{\mathcal{F}}(u,\ubar)\lesssim 2\mathcal{I}^{(0)}\ll F.         \]This implies in particular that $\mathcal{B}$ is not empty and in fact $[u_{\infty},u_{\infty} +\epsilon] \subseteq \mathcal{B}$. At the same time, naturally there holds $\mathcal{B}\subseteq[u_\infty,- \frac{a}{4}]$. If we are able to prove that $\mathcal{B}$ as a set is both open and closed, we can conclude that in fact $\mathcal{B} \equiv [u_\infty,- \frac{a}{4}]$. Indeed, in the remainder of the paper we show the following estimates
	
	\[ \mathcal{O}(u,\ubar) \lesssim \mathcal{I}^{(0)}+ \mathcal{R}(u,\ubar) + \underline{\mathcal{R}}(u,\ubar)+ \mathcal{F}(u,\ubar) + \underline{\mathcal{F}}(u,\ubar),      \]\[ \mathcal{F}(u,\ubar) +\underline{\mathcal{F}}(u,\ubar) \lesssim \mathcal{R}^2(u,\ubar) + \underline{\mathcal{R}}^2(u,\ubar) +\left(\mathcal{I}^{(0)} \right)^2 + \mathcal{I}^{(0)}+1,    \] \[ \mathcal{R}(u,\ubar)+\underline{\mathcal{R}}(u,\ubar)\lesssim \left(\mathcal{I}^{(0)}\right)^2 + \mathcal{I}^{(0)}+1.          \]These are improvements of the upper bounds in bootstrap assumptions. By the continuity of solutions and local existence arguments, $\mathcal{B}$ can be extended a bit towards larger $u$. This implies that $\mathcal{B}$ is open. Together with closedness or $\mathcal{B}$, we conclude that $\mathcal{B} \equiv [u_{\infty}, - a/4]$ and in $\mathcal{B}$ the desired bounds hold.
	\subsection{Estimates for the metric components}
	
	\begin{proposition}\label{Omega}
		Under the assumptions of Theorem \ref{main1} and the bootstrap assumptions \eqref{bootstrapbounds}, we have 
		
		\[ \lVert \Omega-1 \rVert_{L^\infty(S_{u,\ubar})} \lesssim \frac{O}{\lvert u \rvert} .   \]

	\end{proposition}
	\begin{proof}
		Consider the equation \[ \omega = -\frac{1}{2}\nabla_4 (\log \hspace{.5mm} \Omega) = \frac{1}{2}\frac{\partial}{\partial\ubar}(\Omega^{-1}). \]We integrate with respect to $\text{d} \ubar$. Since on $\Hbar_0$ we have $\Omega^{-1} = 1$, we can obtain
		
		\be \label{Omega-1} \lVert \Omega^{-1} -1 \rVert_{L^\infty(S_{u,\ubar})} \lesssim \int_0^{\ubar} \lVert \omega \rVert_{L^\infty(S_{u,\ubar^\prime})}  \text{d}\ubar^\prime \lesssim \frac{O}{\lvert u \rvert}.    \ee Here we have used the bootstrap assumption. Finally, notice that

		\[   \lVert \Omega - 1 \rVert_{L^\infty (S_{u,\ubar})}    \leq \lVert \Omega \rVert_{L^{\infty}(S_{u,\ubar})}\hspace{.5mm} \lVert \Omega^{-1}-1 \rVert_{L^{\infty}(S_{u,\ubar})}   \lesssim \frac{\frac{O}{\lvert u \rvert}}{1+ \frac{O}{\lvert u \rvert}} \lesssim \frac{O}{\lvert u \rvert}  .      \]
	\end{proof}We now control the induced metric $\gamma$ on $S_{u,\ubar}$:
	
	\begin{proposition}
		Under the assumptions of Theorem \ref{main1} and the bootstrap assumption \eqref{bootstrapbounds}, we have for the metric $\gamma$ on $S_{u,\ubar}$:\[  c^\prime \leq \det \gamma \leq C^\prime,      \]where the two constants depend only on the initial data. Moreover, in $\mathcal{D}$, there holds \[  \lvert \gamma_{AB} \rvert + \lvert (\gamma^{-1})^{AB}\rvert \leq C^\prime.    \] \label{prop32}
	\end{proposition}
	
	\begin{proof}
		We employ the first variation formula $\slashed{\mathcal{L}}_{L}\gamma = 2 \hspace{.5mm} \Omega \hspace{.5mm} \chi.$ In coordinates, this rewrites as \be   \frac{\partial}{\partial \ubar} \gamma_{AB}= 2 \hspace{.5mm} \Omega \hspace{.5mm} \chi_{AB}. \label{ub_gamma}   \ee This implies
		
		\[   \frac{\partial}{\partial \ubar} (\log(\det \gamma)) = 2\hspace{.5mm} \Omega \hspace{.5mm} \tr\chi.      \]Let $\gamma_0(u,\ubar, \theta^1,\theta^2) = \gamma(u,0,\theta^1,\theta^2)$. Then with $2\hspace{.5mm}\Omega \hspace{.5mm}\tr\chi \lesssim \frac{O}{\lvert u \rvert}$, we have
		\[  \frac{\det \gamma}{\det \gamma_0} = \e^{\int_0^{\ubar} 2 \Omega \tr \chi \text{d}\ubar^\prime } \lesssim    \e^{\frac{O}{a}}.        \]   Via Taylor expansion, this implies
		\be \label{detgaper} \lvert \det \gamma - \det \gamma_0 \rvert \lesssim \frac{O}{a}.        \ee This gives uniform upper and lower bounds for $\det \gamma$. Let $\Lambda$ be the larger eigenvalue of $\gamma$. We have \[ \Lambda \leq \sup \gamma_{AB}, \]\[ \sum_{A,B=1,2} \lvert \chi_{AB} \rvert \leq \Lambda\lVert \chi \rVert_{L^\infty(S_{u,\ubar})}, \]\[ \lvert \gamma_{AB} - (\gamma_{0})_{AB} \rvert \leq \int_{0}^{\ubar}\lvert \chi_{AB} \rvert \text{d}\ubar^\prime \leq \Lambda \frac{a^{\frac{1}{2}} }{\lvert u \rvert} O \lesssim \frac{O}{a^{\frac{1}{2}}}.   \]

	\end{proof}\par \noindent We will also need the following:
	\begin{proposition}\label{prop33}
		We continue to work under the assumptions of Theorem \ref{main1} and the bootstrap assumptions \eqref{bootstrapbounds}. Fix a point $(\ubar, \theta)$ on the initial hypersurface $\Hbar_0$. Along the outgoing null geodesics emanating from $(u,\theta)$, define $\Lambda(\ubar)$ and $\lambda(\ubar)$ to be the larger and smaller eigenvalue of $\gamma^{-1}(u,0,\theta)\hspace{.5mm} \gamma(u,\ubar,\theta).$ Then there holds \[\lvert \Lambda(\ubar)-1 \rvert +  \lvert \lambda(\ubar) -1 \rvert \lesssim \frac{1}{a^{\frac{1}{2}}}.  \] 
	\end{proposition}
	\begin{proof}
	        Define $\nu(\ub):=\sqrt{\frac{\Lambda(\ub)}{\lambda(\ub)}}.$ Following the derivation of (5.93) in \cite{Chr:book}, by \eqref{ub_gamma}, we have
			$$\nu(\ub)\leq 1+\int_0^{\ub}|\O\chihat(\ub')|_{\gamma} \nu(\ubar') {\color{black}\, d\ub'}.$$
			Via Gr\"onwall's inequality, this implies
			\begin{equation}\label{nu}
			|\nu(\ub)|\ls 1 \quad \mbox{ and } \quad |\nu(\ub)-1|\leq \f{\at\cdot O}{|u|^2}\leq \f{O}{a^{\f32}}\leq \f{1}{a}.
			\end{equation}
			The desired estimate follows from (\ref{detgaper}) and (\ref{nu}).
		\end{proof}
	\par \noindent The above two propositions also imply
	
	\begin{proposition}\label{areaprop}
		Under the assumptions of Theorem \ref{main1} and the bootstrap assumption \eqref{bootstrapbounds}, in the slab of existence $\mathcal{D}$ we have
		\[   \sup_{\ubar} \lvert \text{Area}(S_{u,\ubar})-\text{Area}(S_{u,0}) \rvert \lesssim \frac{O^{\frac{1}{2}}}{a^{\frac{1}{2}}} \lvert u \rvert^2.         \]
	\end{proposition}\begin{proof}
	This follows from the definitions in Subsection \ref{24Integration} and the estimate \eqref{detgaper}.
\end{proof}
	
	\subsection{Estimates for transport equations} In the sections to follow, we will employ the following propositions for the transport equations:
	
	\begin{proposition}
		Under the assumptions of Theorem \ref{main1} and the bootstrap assumption \eqref{bootstrapbounds}, for an arbitrary $S-$tangent tensor $\phi$ of arbitrary rank, we have
		\begin{gather}
		\lVert \phi \rVert_{L^2(S_{u,\ubar})} \lesssim \lVert \phi \rVert_{L^2(S_{u,\ubar^\prime})} + \int_{\ubar^\prime}^{\ubar} \lVert \nabla_4 \phi \rVert_{L^2(S_{u,\ubar^{\prime\prime}})}\text{d}\ubar^{\prime \prime},\label{transport1} \\ 
		\lVert \phi \rVert_{L^2(S_{u,\ubar})} \lesssim \lVert \phi \rVert_{L^2(S_{u^\prime,\ubar})} + \int_{u^\prime}^{u} \lVert \nabla_3 \phi \rVert_{L^2(S_{u^{\prime\prime},\ubar})}\text{d}u^{\prime \prime}\label{transport3}.
		\end{gather}
	\end{proposition}

	\begin{proof}
		Here we first prove (\ref{transport1}). For a scalar function $f$, by variation of area formula, we have
		\[
		\frac{d}{d\ub}\int_{\S} f=\int_{\S} \left(\frac{df}{d\ub}+\Omega \trch f\right)=\int_{\S} \Omega\left(e_4(f)+ \trch f\right).
		\]
		Taking $f=|\phi|_{\gamma}^2$, using Cauchy-Schwarz inequality on the sphere and $L^\infty$ bounds for $\Omega$ and $\trch$, we obtain
		$$2\|\phi\|_{L^2(\S)}\cdot \f{d}{d\ub}\|\phi\|_{L^2(\S)}\ls \|\phi\|_{L^2(\S)}\cdot \|\nab_4\phi\|_{L^2(\S)}+\f{O}{|u|}\|\phi\|^2_{L^2(\S)}.$$
		This implies
		$$\f{d}{d\ub}\|\phi\|_{L^2(\S)}\ls \|\nab_4\phi\|_{L^2(\S)}+\f{O}{|u|}\|\phi\|_{L^2(\S)}.$$
		And (\ref{transport1}) can be concluded by applying Gr\"onwall's inequality for $\ub$ variable. 
		\vspace{3mm}
		\par \noindent Inequality \eqref{transport3} can be proved in a similar fashion.  For a scalar function $f$, we arrive at
			\[
			\Lb\int_{\S} f=\int_{\S} \left(\Lb f+\Omega \tr\chib f\right)=\int_{\S} \Omega\left(e_3(f)+ \tr\chib f\right).
			\]
			Taking $f=|\phi|_{\gamma}^2$, using Cauchy-Schwarz inequality on the sphere and the fact $\O>0, \tr\chib<0$, we obtain
			$$2\|\phi\|_{L^2(\S)}\cdot \Lb\|\phi\|_{L^2(\S)}\ls \|\phi\|_{L^2(\S)}\cdot \|\nab_3\phi\|_{L^2(\S)}.$$
			This implies $ \Lb\|\phi\|_{L^2(\S)}\ls \|\nab_3\phi\|_{L^2(\S)}$ and (\ref{transport3}) follows.
		
	\end{proof}

We then rewrite the above inequalities in scale invariant norms:
	
	\begin{proposition}\label{prop36}
		There holds 
		
		\[ \scaletwoSu{\phi} \lesssim \lVert \phi \rVert_{\mathcal{L}^2_{(sc)}(S_{u,0})}  + \int_{0}^{\ubar}\lVert \nabla_4 \phi \rVert_{\mathcal{L}^2_{(sc)}(S_{u,\ubar^\prime})}\hspace{.5mm} \text{d}\ubar^\prime, \]\[  \scaletwoSu{\phi} \lesssim \lVert \phi \rVert_{\mathcal{L}^2_{(sc)}(S_{u_{\infty},\ubar})}  + \int_{u_{\infty}}^{u} \frac{a}{\lvert u^\prime \rvert^2}\lVert \nabla_3 \phi \rVert_{\mathcal{L}^2_{(sc)}(S_{u^\prime,\ubar})}\hspace{.5mm} \text{d}u^\prime.     \]
	\end{proposition}For equations along the incoming direction, sometimes the borderline terms necessitate more precise estimates. Typically,
	a borderline term contains $\tr\chibar$. It turns out that the coefficients in front of $\tr\chibar$ play an important role.
	
	\begin{proposition}\label{prop37}
		We continue to work under the assumptions of Theorem \ref{main1} and the bootstrap assumptions \eqref{bootstrapbounds}.	Let $\upsilon$ and $\Upsilon$ be $S_{u,\ubar}-$tangent tensor fields of rank $k$ satisfying the transport equation
		
		\[   \nabla_3 \upsilon_{A_1 \dots A_k} + \lambda_0 \tr\chibar \hspace{.5mm} \upsilon_{A_1 \dots A_k} = \Upsilon_{A_1\dots A_k}.      \]If we define $\lambda_1 = 2\lambda_0 -1$, we have \[  \lvert u \rvert^{\lambda_1} \lVert \upsilon \rVert_{L^2(S_{u,\ubar})}  \lesssim \lvert u_{\infty} \rvert^{\lambda_1}\lVert \upsilon \rVert_{L^{2}(S_{u_{\infty},\hspace{.5mm}\ubar})} +  \int_{u_{\infty}}^u \lvert u^\prime \rvert^{\lambda_1} \lVert \Upsilon \rVert_{L^{2}(S_{u^\prime,\hspace{.5mm}\ubar})} \hspace{.5mm} \text{d}u^\prime   \]where the implicit constant is allowed to depend on $\lambda_0$.
	\end{proposition}

\begin{proof}
	We use variation of area formula for equivariant vector $\Lb$ \footnote{Recall $\Lb=\O e_3$.} and a scalar function $f$: 
	\[
	\Lb\int_{\S} f=\int_{\S} \left(\Lb f+\Omega \trchb f\right)=\int_{\S} \Omega\left(e_3(f)+ \trchb f\right).
	\]
	With this identity, we obtain
	\begin{equation}\label{evolution.id}
	\begin{split}
	&\Lb(\int_{\S}|u|^{2\lambda_1 }|\phi|^{2})\\
	=&\int_{\S}\Omega\l -2\lambda_1 |u|^{2\lambda_1-1}(e_3 u)|\phi|^{2}+2|u|^{2\lambda_1}<\phi,\nab_3\phi>+ \tr\underline{\chi}|u|^{2\lambda_1}|\phi|^{2}\r\\
	=&\int_{\S}\Omega\l 2|u|^{2\lambda_1}<\phi, \nab_3\phi+{\lambda_0}\trchb\phi>\r\\
	&+\int_{\S}\Omega |u|^{2\lambda_1}\l -\f{2\lambda_1 (e_3u)}{|u|}+(1-2\lambda_0)\trchb\r|\phi|^2.
	\end{split}
	\end{equation}
	Observe that we have
	\begin{equation}\label{trchib additional}
	\begin{split}
	&-\f{2\lambda_1 (e_3u)}{|u|}+(1-2\lambda_0)\trchb\\
	= &-\f{2\lambda_1 \Omega^{-1}}{|u|}+(1-2\lambda_0)\trchb\\
	= &-\f{2\lambda_1 (\Omega^{-1}-1)}{|u|}+(1-2\lambda_0)(\trchb+\f{2}{|u|})-\f{2\lambda_1+2-4\lambda_0}{|u|}\\
	\ls &\f{O}{|u|^2}.
	\end{split}
	\end{equation}
	For the last inequality, we employ \eqref{Omega-1}, the bootstrap assumption and the fact that $\|\tr\chib+\f{2}{|u|}\|_{L^{\infty}(\S)}\leq \f{O}{|u|^2}$ and $\lambda_1=2(\lambda_0-1/2)$.
	
\vspace{3mm}
	
\par \noindent	Using Cauchy-Schwarz for the first term and applying Gr\"onwall's inequality for the second term, we obtain
	\begin{equation*}
	\begin{split}
	&|u|^{\lambda_1}\|\phi\|_{L^2(\S)}\\
	\ls &e^{O\|u^{-2}\|_{L^1_u}}\l|u_{\infty}|^{\lambda_1}\|\phi\|_{L^2(S_{u_{\infty},\underline{u}})}+\int_{u_{\infty}}^u |u'|^{\lambda_1}\|F\|_{L^2(S_{u',\underline{u}})}du'\r\\
	\ls &|u_{\infty}|^{\lambda_1}\|\phi\|_{L^2(S_{u_{\infty},\underline{u}})}+\int_{u_{\infty}}^u |u'|^{\lambda_1}\|F\|_{L^2(S_{u',\underline{u}})}du'.
	\end{split}
	\end{equation*}
	In the last step, we use $O\|u^{-2}\|_{L^1_u}\ls O/a\leq 1$.
\end{proof}

	\subsection{Sobolev embedding}
	
	With the derived estimates for the metric $\gamma$, we can obtain a bound on the isoperimetric constant for a $2-$sphere $S$:
	
	\[   I(S)= \sup_{\substack{U\subset S\\\partial U \in C^1}}      \frac{\min \begin{Bmatrix}\text{Area}(U), \text{Area}(U^c) \end{Bmatrix}}{\left( \text{Perimeter}(\partial U) \right)^2 }. \] 
	
	\begin{proposition}\label{propisoperimetric}
		Under the assumptions of Theorem \ref{main1} and the bootstrap assumption \eqref{bootstrapbounds}, the isoperimetric constant obeys the bound \[ I(S_{u,\ubar}) \leq \frac{1}{\pi}  \]
		{\color{black}for $u_{\infty}\leq u \leq -a/4$ and $0\leq \ub \leq 1$.}
	\end{proposition}
\begin{proof}
	Fix $u$. Given $U_{\ubar} \subset \S$, denote by $U_0 \subset S_{u,0}$ the pullback image of $U_{\ubar}$ under the diffeomorphism generated by the equivariant vector field $L$. Using Propositions \ref{prop32} and \ref{prop33} we can obtain that
	
	\[  \frac{\text{Perimeter}\hspace{.5mm}\left(\partial U_{\ubar}\right)}{\text{Perimeter}\left(\partial U_0\right)}  \geq \sqrt{\inf_{S_{u,0}} \lambda(\ubar)},  \]\[ \frac{\text{Area} \hspace{.5mm} U_{\ubar}}{\text{Area} \hspace{.5mm} U_{0}}  \leq \sup_{S_{u,0}} \frac{\det (\gamma_{\ubar})}{\det (\gamma_0)},\hspace{2mm}  \frac{\text{Area} \hspace{.5mm} U_{\ubar}^c}{\text{Area} \hspace{.5mm} U_{0}^{c}}  \leq \sup_{S_{u,0}} \frac{\det (\gamma_{\ubar})}{\det (\gamma_0)}  . \] Using the fact that $I(S_{u,0})= \frac{1}{2\pi}$, as it is the standard sphere in Minkowski spacetime, the bounds in Propositions \ref{prop32} and \ref{prop33} yield the conclusion.
\end{proof}We shall be employing an $L^2-L^{\infty}$ embedding statement in this paper quite often. To derive it, in addition to \ref{propisoperimetric}, we require the two propositions below, whose proof is found in  \cite{Chr:book}.

\begin{proposition}\label{ltwolp}
	Suppose $(S,\gamma)$ is a Riemannian $2-$manifold. There holds
	
	\begin{equation}
	\left( \text{Area}\hspace{.5mm}(S) \right)^{-\frac{1}{p}} \lVert \phi \rVert_{L^p(S)}\leq C_p  \sqrt{\max\begin{Bmatrix} I(S),1 \end{Bmatrix}}\left( \lVert \nabla \phi \rVert_{L^2(S)}+ \left(\text{Area}\hspace{.5mm}(S)\right)^{-\frac{1}{2}} \lVert \phi \rVert_{L^2(S)} \right),
	\end{equation}for any $2 < p < \infty$ and any tensor $\phi$.
\end{proposition}

\begin{proposition}\label{lplinfinity}
	Suppose $(S,\gamma)$ is a Riemannian $2-$manifold. There holds
	
	\begin{equation}
	 \lVert \phi \rVert_{L^{\infty}(S)}\leq C_p  \sqrt{\max\begin{Bmatrix} I(S),1 \end{Bmatrix}}\left( \text{Area}\hspace{.5mm}(S) \right)^{\frac{1}{2}- \frac{1}{p}} \left( \lVert \nabla \phi \rVert_{L^p(S)}+ \left(\text{Area}\hspace{.5mm}(S)\right)^{-\frac{1}{2}} \lVert \phi \rVert_{L^p(S)} \right),
	\end{equation}for any $2 < p < \infty$ and any tensor $\phi$.
\end{proposition}
Given Proposition \ref{areaprop}, we know that $\text{Area}(S_{u,\ubar}) \approx \lvert u \rvert^2$. Substituting this into Propositions \ref{ltwolp} and \ref{lplinfinity} and taking into account Proposition \ref{propisoperimetric}, we have the following $L^2-L^\infty$ Sobolev embedding inequality:
	
	\begin{proposition}\label{Sobolevembedding}
		Under the assumptions of Theorem \ref{main1} and the bootstrap assumption \eqref{bootstrapbounds}, it holds
		\[  \lVert \phi \rVert_{L^\infty (S_{u,\ubar})} \lesssim      \sum_{0 \leq i\leq 2} \big\lVert \lvert u \rvert^{i-1}\nabla^i \phi \big\rVert_{L^2(S_{u,\ubar})}.  \]In scale invariant norms:
		
		\[   \scaleinfinitySu{\phi} \lesssim \sum_{0\leq i\leq 2} \scaletwoSu{(a^{\frac{1}{2}}\nabla)^i \phi}.      \]
	\end{proposition}
	\subsection{Commutation formulae}
	We list some useful commutation formulae that shall be used to give a schematic representation of repeated commutations. 
	
	\begin{proposition}
		For a scalar function $f$, there holds 
		\[ [\nabla_4, \nabla]f = \frac{1}{2}(\eta+ \etabar) \nabla_4 f - \chi \cdot \nabla f, \]\[  [\nabla_3, \nabla] f = \frac{1}{2}(\eta+\etabar) \nabla_3 f - \chibar \cdot \nabla f.      \]
	\end{proposition}
	\begin{proposition}
		For an $S_{u,\ubar}-$tangent $1-$form $U_b$, there holds
		\[   [\nabla_4, \nabla_a] U_b = -\chi_{ac} \nabla_c U_b +\epsilon_{ac}\Hodge{\beta_{\mathcal{R}}}_b U_c + \frac{1}{2}(\eta_a + \etabar_a)\nabla_4 U_b -\chi_{ac}\etabar_b U_c + \chi_{ab} \etabar\cdot U,   \]\[  [\nabla_3, \nabla_a] U_b = -\chibar_{ac} \nabla_c U_b +\epsilon_{ac}\Hodge{\betabar_{\mathcal{R}}}_b U_c + \frac{1}{2}(\eta_a + \etabar_a)\nabla_3 U_b -\chi_{ac}\etabar_b U_c + \chibar_{ab} \eta\cdot U.     \]
	\end{proposition}
	\begin{proposition}
		For an $S_{u,\ubar}-$tangent $2-$form $V_{bc}$, there holds
		\begin{align*}
		[\nabla_4, \nabla_a] V_{bc} =&   \frac{1}{2}(\eta_a + \etabar_a)\nabla_4 V_{bc}- \etabar_b V_{dc} \chi_{ad} -\etabar_c V_{bd}\chi_{ad} - \epsilon_{bd}\Hodge{\beta_{\mathcal{R}}}_a       V_{dc} -\epsilon_{cd} \Hodge{\beta_{\mathcal{R}}}_c V_{bd}\\ 
		&+\chi_{ac}V_{bd} \etabar_{d} + \chi_{ab} V_{dc} \etabar_{d} - \chi_{ad} \nabla_d V_{bc},
		\end{align*}\begin{align*}
		[\nabla_3, \nabla_a] V_{bc} =&   \frac{1}{2}(\eta_a + \etabar_a)\nabla_3 V_{bc}- \eta_b V_{dc} \chibar_{ad} -\eta_c V_{bd}\chibar_{ad} - \epsilon_{bd}\Hodge{\betabar_{\mathcal{R}}}_a       V_{dc} -\epsilon_{cd} \Hodge{\betabar_{\mathcal{R}}}_c V_{bd}\\ 
		&+\chibar_{ac}V_{bd} \eta_{d} + \chibar_{ab} V_{dc} \eta_{d} - \chibar_{ad} \nabla_d V_{bc}.
		\end{align*}
	\end{proposition}
	\begin{proposition}
		Assume $\nabla_4 \phi = F_0$. Let $\nabla_4 \nabla^i \phi = F_i$. Then \begin{align*} F_i=& \sum_{i_1+ i_2 + i_3 =i} \nabla^{i_1} (\eta+\etabar)^{i_2} \nabla^{i_3}F_0 + \sum_{i_1+ i_2 + i_3 +i_4 =i-1} \nabla^{i_1} (\eta+\etabar)^{i_2} \nabla^{i_3} \beta_{\mathcal{R}} \nabla^{i_4}\phi   \\  
		&+       \sum_{i_1+ i_2 + i_3 +i_4 =i} \nabla^{i_1} (\eta+\etabar)^{i_2} \nabla^{i_3} \chi \nabla^{i_4}\phi.                \end{align*}Assume now that $\nabla_3 \phi =G_0$. Let $\nabla_3 \nabla^i \phi= G_i$. Then \begin{align*}
		G_i + \frac{i}{2}\tr \chibar \nabla^i \phi = &\sum_{i_1+ i_2 + i_3 =i}\nabla^{i_1} (\eta+\etabar)^{i_2} \nabla^{i_3}G_0 
		+ \sum_{i_1+ i_2 + i_3 +i_4 =i-1} \nabla^{i_1}(\eta+\etabar)^{i_2} \nabla^{i_3} \betabar_{\mathcal{R}} \nabla^{i_4} \phi \\ 
		+ &\sum_{i_1+ i_2 + i_3 +i_4 =i-1} \nabla^{i_1}(\eta+\etabar)^{i_2} \nabla^{i_3} (\chibarhat, \widetilde{\tr\chibar})\nabla^{i_4} \phi 
		+ \sum_{i_1+ i_2 + i_3 +i_4 =i-1} \nabla^{i_1}(\eta+\etabar)^{i_2+1} \nabla^{i_3} \tr\chibar \nabla^{i_4} \phi.
		\end{align*}
	\end{proposition}Finally, we can replace $\beta_{\mathcal{R}}, \betabar_{\mathcal{R}}$ by expressions involving Ricci coefficients, under the Codazzi equations:
	
	\begin{align*}   \beta_{\mathcal{R}} = - \text{div} \chihat + \frac{1}{2}\nabla \tr\chi - \frac{1}{2}(\eta-\etabar)\cdot (\chihat  - \frac{1}{2}\tr \chi), \\
	\betabar_{\mathcal{R}}=  \text{div} \chibarhat - \frac{1}{2}\nabla \tr \chi - \frac{1}{2}(\eta-\etabar)\cdot(\chibarhat - \frac{1}{2}\tr\chibar) .        \end{align*}That way, we arrive at the following:
	
	\begin{proposition}\label{commutationformulaeprop}
		Suppose $\nabla_4 \phi = F_0$. Let $\nabla_4 \nabla^i \phi = F_i$. Then \[ F_i = \sum_{i_1+ i_2 + i_3 =i}   \nabla^{i_1} \psi^{i_2} \nabla^{i_3}F_0 +   \sum_{i_1+ i_2 + i_3 +i_4 =i} \nabla^{i_1}\psi^{i_2} \nabla^{i_3}(\psi, \chihat) \nabla^{i_4}\phi.       \]Similarly, suppose $\nabla_3 \phi = G_0$. Let $\nabla_3\nabla^i \phi = G_i$. Then 
		
		\begin{align*}
		G_i + \frac{i}{2}\tr\chibar \nabla^i \phi =& \sum_{i_1+ i_2 + i_3 =i} \nabla^{i_1} \psi^{i_2} \nabla^{i_3} G_0 + \sum_{i_1+ i_2 + i_3 +i_4 =i} \nabla^{i_1} \psi^{i_2}\nabla^{i_3}(\psi, \chibarhat, \widetilde{\tr \chibar}) \nabla^{i_4}\phi \\
		&+ \sum_{i_1+ i_2 + i_3 +i_4 =i-1} \nabla^{i_1} \psi^{i_2+1} \nabla^{i_3} \tr \chibar \nabla^{i_4}\phi.
		\end{align*}
	\end{proposition}
	\section{$L^2(S_{u,\ubar})$-estimates for Ricci coefficients and Maxwell components}
	
	We start with some useful estimates. Let  
	
	\[ \psi \in \begin{Bmatrix} \frac{\chihat}{a^{\frac{1}{2}}}, \tr\chi, \omega, \eta, \etabar, \zeta, \omegabar, \frac{a}{\lvert u \rvert}\widetilde{\tr \chibar}, \frac{a^{\frac{1}{2}}}{\lvert u \rvert} \chibarhat, \frac{a}{\lvert u \rvert^2}\tr\chibar        \end{Bmatrix}, \Psi \in \begin{Bmatrix} \frac{\alpha}{a^{\frac{1}{2}}}, \beta, \rho, \sigma, \betabar, \alphabar       \end{Bmatrix}     \hspace{2mm}\text{and}\hspace{2mm} \Y \in \begin{Bmatrix}    \frac{\alpha_F}{a^{\frac{1}{2}}}, \rho_F, \sigma_F, \alphabar_F   \end{Bmatrix}.     \]
	\begin{proposition} \label{usefulstatements}
		Under the assumptions of Theorem \ref{main1} and the bootstrap assumption \eqref{bootstrapbounds}, we have
		
		\begin{gather*}
		\sum_{i_1 + i_2 \leq 10} \lVert (a^{\frac{1}{2}})^{i_1+i_2}\nabla^{i_1} \psi^{i_2} \rVert_{\mathcal{L}^2_{(sc)}(S_{u,\ubar})} \leq \lvert u \rvert, \\
		\sum_{i_1 + i_2 \leq 10} \lVert (a^{\frac{1}{2}})^{i_1+i_2}\nabla^{i_1} \psi^{i_2+1} \rVert_{\mathcal{L}^2_{(sc)}(S_{u,\ubar})} \leq O, \\ 
		\sum_{i_1 + i_2 \leq 10} \lVert (a^{\frac{1}{2}})^{i_1+i_2}\nabla^{i_1} \psi^{i_2+2} \rVert_{\mathcal{L}^2_{(sc)}(S_{u,\ubar})} \leq \frac{O^2}{\lvert u \rvert}, \\
		\sum_{i_1 + i_2 \leq 10} \lVert (a^{\frac{1}{2}})^{i_1+i_2}\nabla^{i_1} \psi^{i_2+3} \rVert_{\mathcal{L}^2_{(sc)}(S_{u,\ubar})} \leq \frac{O^3}{\lvert u \rvert^2}, \\ 
		\sum_{i_1 + i_2 +i_3 \leq 9} \lVert (a^{\frac{1}{2}})^{i_1+i_2+i_3 }\nabla^{i_1} \psi^{i_2} \nabla^{i_3} \Psi \rVert_{\mathcal{L}^2_{(sc)}(S_{u,\ubar})} \leq O, \\ 
		\sum_{i_1 + i_2 +i_3  \leq 9} \lVert (a^{\frac{1}{2}})^{i_1+i_2+i_3+1 }\nabla^{i_1} \psi^{i_2+1} \nabla^{i_3} \Psi \rVert_{\mathcal{L}^2_{(sc)}(S_{u,\ubar})} \leq \frac{a^{\frac{1}{2}}}{\lvert u \rvert }O^2, \\ 
		\sum_{i_1 + i_2 +i_3 \leq 9} \lVert (a^{\frac{1}{2}})^{i_1+i_2+i_3+2 }\nabla^{i_1} \psi^{i_2+2 } \nabla^{i_3} \Psi \rVert_{\mathcal{L}^2_{(sc)}(S_{u,\ubar})} \leq \frac{a}{\lvert u \rvert^2}O^3, \\ 
		\sum_{i_1 + i_2 \leq 10} \lVert (a^{\frac{1}{2}})^{i_1+i_2}\nabla^{i_1} \Y^{i_2} \rVert_{\mathcal{L}^2_{(sc)}(S_{u,\ubar})} \leq \lvert u \rvert,\\ 
		\sum_{i_1 + i_2 \leq 10} \lVert (a^{\frac{1}{2}})^{i_1+i_2}\nabla^{i_1} \Y^{i_2+1} \rVert_{\mathcal{L}^2_{(sc)}(S_{u,\ubar})} \leq O, \\ 
		\sum_{i_1 + i_2 \leq 10} \lVert (a^{\frac{1}{2}})^{i_1+i_2}\nabla^{i_1} \Y^{i_2+2} \rVert_{\mathcal{L}^2_{(sc)}(S_{u,\ubar})} \leq \frac{O^2}{\lvert u \rvert}, \\  \sum_{i_1+i_2+i_3\leq 10} \scaletwoSu{(\al)^{i_1+i_2+i_3} \nabla^{i_1} \psi^{i_2+1}     \nabla^{i_3}\Y    }\leq \frac{O^2}{\lvert u \rvert},  \\
		\sum_{i_1 + i_2 +i_3 + i_4 \leq 10} \lVert (a^{\frac{1}{2}})^{i_1+i_2+i_3 +i_4}\nabla^{i_1} \psi^{i_2} \nabla^{i_3}\Y \nabla^{i_4} \Y \rVert_{{\color{black}\mathcal{L}}^2_{(sc)}(S_{u,\ubar})} \leq \frac{O^2}{\lvert u \rvert}. \\ 
		\end{gather*}
	\end{proposition}
	\begin{proof}
		We focus on the last five statements. The ones before are similar and their proof can be found in Section 4 of \cite{An:2019}. 
		\begin{itemize}
			\item For the first one we distinguish two cases: If $i_2 =0$, then the result holds trivially as $\scaletwoSu{1} = \lvert u \rvert.$ For $i_2 \geq 1$, we can rewrite $\nabla^{i_1} \Y^{i_2}$ as a product of $i_2$ terms
			\[  \nabla^{i_1} \Y^{i_2} = \nabla^{j_1} \Y \dots \nabla^{j_{i_2}}\Y, \hspace{2mm} \text{with} \hspace{2mm} j_1 + \dots + j_{i_2} = i_1.        \]Assume that $j_{i_2}$ is the largest number. Then we rewrite \[ (a^{\frac{1}{2}})^{i_1+i_2} \nabla^{i_1} \Y^{i_2} = (a^{\frac{1}{2}})^{i_2} \cdot  (a^{\frac{1}{2}}\nabla)^{j_{i_2}}\Y  \cdot \prod_{k=1}^{i_2-1} (a^{\frac{1}{2}}\nabla)^{j_k}\Y.    \]Now we bound $(a^{\frac{1}{2}}\nabla)^{j_{i_2}}\Y $ in $\mathcal{L}^2_{(sc)}(S_{u,\ubar})$ and the rest of the terms in $\mathcal{L}^{\infty}_{(sc)}(S_{u,\ubar})$. We then have
			
			\begin{align*}  &\frac{1}{\lvert u \rvert} \sum_{i_1+i_2 \leq 10}  \lVert (a^{\frac{1}{2}})^{i_1+i_2}\nabla^{i_1} \Y^{i_2} \rVert_{\mathcal{L}^2_{(sc)}(S_{u,\ubar})}\\ \leq&  
			\frac{1}{\lvert u \rvert} \sum_{i_1 + i_2 \leq 10} \frac{(a^{\frac{1}{2}})^{i_2}}{\lvert u \rvert^{i_2 -1} }\scaletwoSu{(a^{\frac{1}{2}}\nabla)^{j_{i_2}}\Y}  \prod_{k=1}^{i_2-1} \scaleinfinitySu{(a^{\frac{1}{2}}\nabla)^{j_k}\Y}  \leq \frac{(a^{\frac{1}{2}})^{i_2}O^{i_2}}{\lvert u \rvert^{i_2}}\leq 1. \end{align*}
			
			\item For the second one, if $i_2 =0$, then the statement is true because of the definition of $O$. If $i_2 \geq 1$, then assume $i_1 = j_1+\dots+j_{i_2 +1}$. Assume $j_{i_2+1}$ is the largest.  Then, as above, 	\begin{align*}  &\sum_{i_1+i_2 \leq 10}  \lVert (a^{\frac{1}{2}})^{i_1+i_2}\nabla^{i_1} \Y^{i_2+1} \rVert_{\mathcal{L}^2_{(sc)}(S_{u,\ubar})}\\ \leq&  
			\sum_{i_1 + i_2 \leq 10} \frac{(a^{\frac{1}{2}})^{i_2}}{\lvert u \rvert^{i_2 } }\scaletwoSu{(a^{\frac{1}{2}}\nabla)^{j_{i_2+1}}\Y}  \prod_{k=1}^{i_2} \scaleinfinitySu{(a^{\frac{1}{2}}\nabla)^{j_k}\Y}  \leq \frac{(a^{\frac{1}{2}})^{i_2}O^{i_2+1}}{\lvert u \rvert^{i_2}}\leq O. \end{align*}
			
			\item We have \begin{align*} &\lvert u \rvert \sum_{i_1+i_2 \leq 10}  \lVert (a^{\frac{1}{2}})^{i_1+i_2}\nabla^{i_1} \Y^{i_2+2} \rVert_{\mathcal{L}^2_{(sc)}(S_{u,\ubar})}\\ \leq& 
			\lvert u \rvert \cdot \frac{1}{\lvert u \rvert}\cdot\sum_{i_1+i_2 \leq 10} \scaleinfinitySu{(a^{\frac{1}{2}}\nabla)^{i_3} \Y}\scaletwoSu{(a^{\frac{1}{2}})^{i_2+i_4} \nabla^{i_4} \Y^{i_2+1}}, \hspace{1mm}\text{where} \hspace{1mm} i_3+i_4 = i_1\hspace{1mm} \text{and} \hspace{1mm} i_3 \leq i_4 \\ 
			\leq& O\cdot O=O^2. \end{align*}
			\item  The proof of this is the same as the item above. If $i_3\leq 6$ we bound $(\al \nabla)^{i_3} \Y$ in $\mathcal{L}^{\infty}_{(sc)}$, otherwise we bound it in $\mathcal{L}^{2}_{(sc)}$ and the rest of the terms in $\mathcal{L}^{\infty}_{(sc)}$. 
			\item { \color{black}We have \begin{align*}
			&\lvert u \rvert \sum_{i_1+i_2+i_3+i_4 \leq 10}  \lVert (a^{\frac{1}{2}})^{i_1+i_2+i_3+i_4}\nabla^{i_1} \psi^{i_2} \nabla^{i_3} \Y \nabla^{i_4}\Y \rVert_{\mathcal{L}^2_{(sc)}(S_{u,\ubar})}\\  &\leq \lvert u \rvert \cdot \frac{1}{\lvert u \rvert }\cdot O \cdot \sum_{i_1+i_2+i_3\leq 10}  \lVert (a^{\frac{1}{2}})^{i_1+i_2+i_3}\nabla^{i_1} \psi^{i_2} \nabla^{i_3} \Y \rVert_{\mathcal{L}^2_{(sc)}(S_{u,\ubar})}
			\end{align*}since one of $i_3, i_4$, without loss of generality say $i_4$, has to be at most $6$ so that we can bound the term $\nabla^{i_4} \Y $ in $\mathcal{L}^{\infty}_{(sc)}$. We have
			
			\begin{equation*}
			    \begin{split}
			       & \sum_{i_1+i_2+i_3\leq 10}  \lVert (a^{\frac{1}{2}})^{i_1+i_2+i_3}\nabla^{i_1} \psi^{i_2} \nabla^{i_3} \Y \rVert_{\mathcal{L}^2_{(sc)}(S_{u,\ubar})} \\&\leq \frac{O}{\lvert u \rvert} \sum_{i_1 +i_2 \leq 10} \scaletwoSu{ (\al)^{i_1+i_2} \nabla^{i_1}\psi^{i_2}} \\ &+ \frac{O}{\lvert u \rvert} \sum_{i_1+i_2 \leq 3} \scaleinfinitySu{ (\al)^{i_1+i_2} \nabla^{i_1}\psi^{i_2}} \\ &\leq \frac{O}{\lvert u \rvert} \sum_{i_1 +i_2 \leq 10} \scaletwoSu{ (\al)^{i_1+i_2} \nabla^{i_1}\psi^{i_2}} \\ &+ \frac{O}{\lvert u \rvert} \sum_{i_1+i_2 \leq 5} \scaletwoSu{ (\al)^{i_1+i_2} \nabla^{i_1}\psi^{i_2}} \leq \frac{O}{\lvert u \rvert} \cdot \lvert u \rvert = O.
			    \end{split}
			\end{equation*}Here we have distinguished between the cases where $i_3$ is at most $6$, in which case we bound it in $\mathcal{L}^{\infty}_{(sc)}$ and the case where $7 \leq i_3 \leq 10$, in which case we bound the term $\nabla^{i_3} \Y$ in $\mathcal{L}^2_{(sc)}$ and use the Sobolev embedding theorem to bound $(\al)^{i_1+i_2}\nabla^{i_1}\psi^{i_2}$ in $\mathcal{L}^{\infty}_{(sc)}$. Putting everything together, we arrive at 
			
			\begin{equation*}
			    \sum_{i_1+i_2+i_3+i_4 \leq 10}  \lVert (a^{\frac{1}{2}})^{i_1+i_2+i_3+i_4}\nabla^{i_1} \psi^{i_2} \nabla^{i_3} \Y \nabla^{i_4}\Y \rVert_{\mathcal{L}^2_{(sc)}(S_{u,\ubar})} \leq \frac{O^2}{\lvert u \rvert}.
			\end{equation*}
			}
			
					\end{itemize}
	\end{proof}	\subsection{$L^2(S_{u,\ubar})$-estimates for the Ricci coefficients } 
	\begin{proposition}\label{omegaprop}
		Under the assumptions of Theorem \ref{main1} and the bootstrap assumptions \eqref{bootstrapbounds}, we have
		\[    \sum_{i\leq 10}   \scaletwoSu{(a^{\frac{1}{2}}\nabla)^i \omega} \lesssim \frac{a^{\frac{1}{2}}}{\lvert u \rvert^{\frac{1}{2}}}(\underline{\mathcal{R}}[\rho]+1).     \]
	\end{proposition}
	
	\begin{proof}
		We use the following schematic null structure equation for $\omega$:
		\[      \nabla_3 \omega = \frac{1}{2}\rho + \psi\psi + \Y \Y.    \]Commuting  $i$ times with $\nabla$ using Proposition \ref{commutationformulaeprop} , we arrive at
		\begin{align*}  &\nabla_3 \nabla^i \omega + \frac{i}{2}\tr\chibar \nabla^i \omega\\ = &  
		\sum_{i_1 + i_2 +i_3 = i} \nabla^{i_1} \psi^{i_2} \nabla^{i_3}(\rho + \psi \psi + \Y \Y)+ \sum_{i_1+ i_2 + i_3 +i_4 =i} \nabla^{i_1} \psi^{i_2} \nabla^{i_3}(\psi, \chibarhat, \widetilde{\tr\chibar}) \nabla^{i_4} \omega \\
		&+\sum_{i_1+ i_2 + i_3 +i_4 =i-1} \nabla^{i_1} \psi^{i_2+1}\nabla^{i_3} \tr\chibar \nabla^{i_4} \omega\\
=&		\nabla^{i} \rho + \sum_{i_1 + i_2 +i_3 +1= i} \nabla^{i_1} \psi^{i_2+1} \nabla^{i_3} \rho +\sum_{i_1 + i_2 +i_3 = i} \nabla^{i_1} \psi^{i_2} \nabla^{i_3}(\psi \psi )\\
&+ \sum_{i_1 + i_2 +i_3 = i} \nabla^{i_1} \psi^{i_2} \nabla^{i_3}(\Y \Y )+ \sum_{i_1+ i_2 + i_3 +i_4 =i} \nabla^{i_1} \psi^{i_2} \nabla^{i_3}(\psi, \chibarhat, \widetilde{\tr\chibar}) \nabla^{i_4} \omega \\
		&+\sum_{i_1+ i_2 + i_3 +i_4 =i-1} \nabla^{i_1} \psi^{i_2+1}\nabla^{i_3} \tr\chibar \nabla^{i_4} \omega. \end{align*}
		Now notice that for any $j$ and $S-$tangent tensorfields $\phi_1, \phi_2$ we have the schematic equality $\nabla^j (\phi_1 \cdot \phi_2) = \sum_{j_1 +j_2=j} \nabla^{j_1}\phi_1 \nabla^{j_2} \phi_2$. We can thus write
		
		\begin{align*} &\nabla_3 \nabla^i \omega + \frac{i}{2}\tr\chibar \nabla^i \omega\\ =&
		\nabla^i \rho + \sum_{i_1 + i_2 +i_3 +1= i} \nabla^{i_1} \psi^{i_2+1} \nabla^{i_3} \rho+ \sum_{i_1+ i_2 + i_3 +i_4 =i} \nabla^{i_1} \psi^{i_2} \nabla^{i_3} \Y \nabla^{i_4} \Y \\& 
		+ \sum_{i_1+ i_2 + i_3 +i_4 =i}  \nabla^{i_1} \psi^{i_2} \nabla^{i_3}(\psi, \chibarhat, \widetilde{\tr\chibar}) \nabla^{i_4} \psi +\sum_{i_1+ i_2 + i_3 +i_4 =i-1} \nabla^{i_1} \psi^{i_2+1}\nabla^{i_3} \tr\chibar \nabla^{i_4} \psi.   \end{align*}Rewrite the above as \[\nabla_3 \nabla^i \omega  + \frac{i}{2} \tr\chibar \nabla^i \omega = G.      \]Applying Proposition \ref{prop37}, there holds \[ \lvert u \rvert^{i-1} \lVert \nabla^i \omega \rVert_{L^{2}(S_{u,\ubar})}\leq \lvert u_{\infty}\rvert^{i-1} \lVert \nabla^i \omega \rVert_{L^{2}(S_{u_{\infty},\ubar})} + \int_{u_{\infty}}^u \lvert u^\prime\rvert^{i-1} \lVert G \rVert_{L^{2}(S_{u^\prime,\hspace{.5mm}\ubar})}\text{d}u^\prime. \]Multiplying both sides by $\lvert u \rvert$ and using $\lvert u \rvert \leq \lvert u^\prime \rvert, \lvert u \rvert \leq \lvert u_{\infty} \rvert$ we get

		\be \label{omegafirstequation} \lvert u \rvert^{i} \lVert \nabla^i \omega \rVert_{L^{2}(S_{u,\ubar})}\leq \lvert u_{\infty}\rvert^{i} \lVert \nabla^i \omega \rVert_{L^{2}(S_{u_{\infty},\ubar})} + \int_{u_{\infty}}^u \lvert u^\prime\rvert^{i} \lVert G \rVert_{L^{2}(S_{u^\prime,\hspace{.5mm}\ubar})}\text{d}u^\prime. \ee From the signature table we get \[ s_2(G) = s_2(\nabla_3\nabla^i \omega) = 1+ s_2(\nabla^i \omega) = \frac{i}{2}+1. \]Using the definition of the scale-invariant norms $\mathcal{L}^2_{(sc)}(S_{u,\ubar})$ we have
		\[\scaletwoSu{\phi} = a^{-s_2(\phi)}\lvert u \rvert^{2s_2(\phi)} \lVert \phi \rVert_{L^2(S_{u,\ubar})}          \]and thus \[ \scaletwoSu{\nabla^i \omega} = a^{-\frac{i}{2}}\lvert u \rvert^i \lVert \nabla^i \omega \rVert_{L^{2}(S_{u^,\ubar})}, \hspace{2mm} \scaletwoSu{G} = a^{-\frac{i}{2}-1}\lvert u \rvert^{i+2} \lVert G \rVert_{L^{2}(S_{u^,\ubar})}.   \]Equivalently, \[  \lvert u\rvert^i \lVert \nabla^i \omega \rVert_{L^{2}(S_{u,\ubar})} = \scaletwoSu{(a^{\frac{1}{2}}\nabla)^i \omega}, \hspace{2mm} \lvert u \rvert^i \lVert G \rVert_{L^{2}(S_{u^,\ubar})} = \frac{a}{\lvert u \rvert^2} \scaletwoSu{{\color{black}(a^{\frac{1}{2}})^i} G}.      \] We can now write \eqref{omegafirstequation} in scale-invariant norms as 
		
		\begin{align*}
		\scaletwoSu{(a^\frac{1}{2}\nabla)^i \omega} \leq& \lVert (a^\frac{1}{2}\nabla)^i \omega \rVert_{\mathcal{L}^{2}_{(sc)}(S_{u_{\infty},\ubar})} + \int_{u_{\infty}}^u \frac{a}{\lvert u^\prime \rvert^2} \lVert (a^{\frac{1}{2}}\nabla)^i \rho\rVert_{\mathcal{L}^{2}_{(sc)}(S_{u^\prime,\ubar})}\text{d}u^\prime \\
		&+ \int_{u_{\infty}}^u \frac{a}{\lvert u^\prime \rvert^2} \lVert \sum_{i_1 + i_2 +i_3 +1= i} (a^{\frac{1}{2}})^i \nabla^{i_1} \psi^{i_2+1}\nabla^{i_3} \rho \rVert_{\mathcal{L}^{2}_{(sc)}(S_{u^\prime,\ubar})}\text{d}u^\prime\\
		&+ \int_{u_{\infty}}^u \frac{a}{\lvert u^\prime \rvert^2} \lVert \sum_{i_1+ i_2 + i_3 +i_4 =i} (a^{\frac{1}{2}})^i \nabla^{i_1} \psi^{i_2} \nabla^{i_3}\Y \nabla^{i_4} \Y \rVert_{\mathcal{L}^{2}_{(sc)}(S_{u^\prime,\ubar})}\text{d}u^\prime \\
		&+ \int_{u_{\infty}}^u \frac{a}{\lvert u^\prime \rvert^2} \lVert \sum_{i_1+ i_2 + i_3 +i_4 =i}(\al)^i \nabla^{i_1}\psi^{i_2}  \nabla^{i_3}(\psi,{\color{black}\chibh},\widetilde{\tr\chibar})\nabla^{i_4} \psi      \rVert_{\mathcal{L}^2_{(sc)}(S_{u^\prime,\ubar})} \text{d} u^\prime\\ 
		&+\int_{u_{\infty}}^u \frac{a}{\lvert u^\prime \rvert^2} \lVert \sum_{i_1+ i_2 + i_3 +i_4+1 =i} (\al)^i \nabla^{i_1}\psi^{i_2+1}  \nabla^{i_3}\tr\chibar\nabla^{i_4} \psi      \rVert_{\mathcal{L}^2_{(sc)}(S_{u^\prime,\ubar})} \text{d} u^\prime.
		\end{align*}\par \noindent We look at each term separately. For the first one, since $\Omega\mid_{u=u_\infty} =1$, we note that $\omega = - \frac{1}{2}\nabla_4 (\log\hspace{.5mm} \Omega)$, we have $\lVert (\al \nabla)^i \omega \rVert_{{\color{black}\mathcal{L}}^{2}_{(sc)}(S_{u_{\infty},\ubar})} =0$. For the second and third terms, we have
		
		\begin{align*}
		&\int_{u_{\infty}}^u \frac{a}{\lvert u^\prime \rvert^2} \lVert (\al \nabla)^i \rho \rVert_{\mathcal{L}^{2}_{(sc)}(S_{u^\prime,\ubar})}\text{d}u^\prime + \int_{u_{\infty}}^u \frac{a}{\lvert u^\prime \rvert^2} \lVert \sum_{i_1 + i_2 +i_3 +1= i} (a^{\frac{1}{2}})^i \psi^{i_2+1}\nabla^{i_3} \rho \rVert_{\mathcal{L}^{2}_{(sc)}(S_{u^\prime,\ubar})}\text{d}u^\prime \\
		\leq& \left( \int_{u_{\infty}}^u \frac{a}{\lvert u^\prime \rvert^2}  \lVert (\al \nabla)^i \rho \rVert^2_{\mathcal{L}^{2}_{(sc)}(S_{u^\prime,\ubar})}\text{d}u^\prime       \right)^{\frac{1}{2}} \left(\int_{u_{\infty}}^u \frac{a}{\lvert u^\prime \rvert^2}  \text{d}u^\prime    \right)^{\frac{1}{2}} + \int_{u_{\infty}}^u \frac{a}{\lvert u^\prime \rvert^2}\cdot \frac{\al}{\lvert u \rvert}\cdot O^2 \text{d}u^\prime \\ 
		=& \scaletwoHbaru{(\al \nabla)^i \rho}\cdot \frac{\al}{\lvert u \rvert^{\frac{1}{2}}} + \frac{a^{\frac{3}{2}}}{\lvert u \rvert^2}O^2  \leq \frac{\al}{\lvert u \rvert^{\frac{1}{2}}}\left(\underline{\mathcal{R}}[\rho]+1 \right).
		\end{align*}For the next term, we have 
		\begin{align*}
		&\int_{u_{\infty}}^u \frac{a}{\lvert u^\prime \rvert^2} \lVert \sum_{i_1+ i_2 + i_3 +i_4 =i} (a^{\frac{1}{2}})^i \nabla^{i_1} \psi^{i_2} \nabla^{i_3}\Y \nabla^{i_4} \Y \rVert_{\mathcal{L}^{2}_{(sc)}(S_{u^\prime,\ubar})}\text{d}u^\prime  \\
		\leq& \int_{u_{\infty}}^u \frac{a}{\lvert u^\prime\rvert ^2} \cdot \frac{O^2}{\lvert u^\prime \rvert}\text{d}u^\prime= \frac{a}{\lvert u \rvert^2}O^2\leq \frac{\al}{\lvert u \rvert}O^2 \leq \frac{\al}{\lvert u \rvert^{\frac{1}{2}}}.
		\end{align*}
		For the last two terms, we have \begin{align*}
		&\int_{u_{\infty}}^u \frac{a}{\lvert u^\prime \rvert^2} \lVert \sum_{i_1+ i_2 + i_3 +i_4 =i} (\al)^i\nabla^{i_1}\psi^{i_2}  \nabla^{i_3}(\psi,{\color{black}\chibh},\widetilde{\tr\chibar})\nabla^{i_4} \psi      \rVert_{\mathcal{L}^2_{(sc)}(S_{u^\prime,\ubar})} \text{d} u^\prime \\
		\leq&  \int_{u_{\infty}}^u \frac{\al}{\lvert u^\prime \rvert} \lVert \sum_{i_1+ i_2 + i_3 +i_4 =i} (\al)^i \nabla^{i_1}\psi^{i_2}  \nabla^{i_3}(\frac{\al}{\lvert u^\prime \rvert} \psi,\frac{\al}{\lvert u^\prime \rvert}{\color{black}\chibh},\frac{\al}{\lvert u^\prime \rvert}\widetilde{\tr\chibar})\nabla^{i_4} \psi      \rVert_{\mathcal{L}^2_{(sc)}(S_{u^\prime,\ubar})} \text{d} u^\prime\\
		\leq& \int_{u_{\infty}}^u \frac{\al}{\lvert u^\prime \rvert} \cdot \frac{O^2}{\lvert u^\prime \rvert}\text{d}u^\prime \leq \frac{\al}{\lvert u \rvert}O^2 \leq \frac{\al}{\lvert u \rvert^{\frac{1}{2}}}.
		\end{align*}Moreover, \begin{align*}
		&\int_{u_{\infty}}^u \frac{a}{\lvert u^\prime \rvert^2} \sum_{i_1+ i_2 + i_3 +i_4+1 =i} \lVert(\al)^i  \nabla^{i_1}\psi^{i_2+1}  \nabla^{i_3}\tr\chibar\nabla^{i_4} \psi      \rVert_{\mathcal{L}^2_{(sc)}(S_{u^\prime,\ubar})} \text{d} u^\prime \\
		\leq& \int_{u_{\infty}}^u \al\sum_{i_1+ i_2 + i_3 +i_4+1 =i} \lVert(\al)^{i{\color{black}-1}}  \nabla^{i_1}\psi^{i_2+1}  \nabla^{i_3}(\frac{a}{\lvert u^\prime\rvert^2}\tr\chibar)\nabla^{i_4} \psi      \rVert_{\mathcal{L}^2_{(sc)}(S_{u^\prime,\ubar})}\text{d}u^\prime\\ 
		\leq& \int_{u_{\infty}}^u \al \cdot \frac{O^3}{\lvert u^\prime \rvert^2}\text{d}u^\prime \leq \frac{\al}{\lvert u \rvert}O^3 \leq \frac{\al}{\lvert u \rvert^{\frac{1}{2}}}.
		\end{align*}Gathering all the estimates above and letting $a$ be sufficiently large, we obtain \[  \sum_{i \leq 10} \scaletwoSu{(\al\nabla)^i \omega} \lesssim \frac{\al}{\lvert u \rvert^{\frac{1}{2}}}(\underline{\mathcal{R}}[\rho]+1).     \]
	\end{proof}
	
	\begin{proposition} \label{chibarhatproposition} 
		Under the assumptions of Theorem \ref{main1} and the bootstrap assumption \eqref{bootstrapbounds}, we have \[ \sum_{i \leq 10} \frac{\al}{\lvert u \rvert } \scaletwoSu{(\al \nabla)^i \chibarhat} \lesssim 1.      \]

	\end{proposition}
	
	\begin{proof}
		We look at the $\nabla_3-$equation for $\chibarhat$:
		\[  \nabla_3 \chibarhat + \tr\chibar \hspace{.5mm} \chibarhat = -2 \omegabar \chibarhat - \alphabar.       \]Commuting with $i$ angular derivatives and using Proposition \ref{prop36} we arrive at
		
		\begin{align*}
		&\nabla_3 \nabla^i \chibarhat + \frac{i+2}{2} \tr\chibar \nabla^i \chibarhat \\ 
		=& \nabla^i \alphabar + \sum_{i_1+i_2+i_3+1=i} \nabla^{i_1} \psi^{i_2+1} \nabla^{i_3}\alphabar + \sum_{i_1+i_2+i_3+i_4=i} \nabla^{i_1} \psi^{i_2} \nabla^{i_3}(\psi,\chibarhat, \widetilde{\tr\chibar})\nabla^{i_4}\chibarhat \\
		&+ \sum_{i_1 + i_2 +i_3 + i_4 +1 =i} \nabla^{i_1} \psi^{i_2+1} \nabla^{i_3} \tr\chibar \nabla^{i_4} \chibarhat. 
		\end{align*}
		Rewriting the above equation as
		
		\[ \nabla_3 \nabla^i  \chibarhat + \frac{i+2}{2}\tr \chibar \nabla^{i}\chibarhat = F,  \]an application of Proposition \ref{prop37} gives us
		
		\be \lvert u \rvert^{i+1} \twoSu{\nabla^i \chibarhat} \leq \lvert u_{\infty} \rvert^{i+1}  \lVert \nabla^i \chibarhat \rVert_{L^{2}(S_{u_{\infty},\ubar})} + \int_{u_\infty}^u \lvert u^\prime \rvert^{i+1} \lVert F \rVert_{L^2(S_{u^{\prime},\ubar})}  \hspace{.5mm} \text{d} u^\prime. \label{chibarhateq}    \ee Rewriting \eqref{chibarhateq} in scale-invariant norms, we arrive at
		
		\[ \frac{a}{\lvert u \rvert} \scaletwoSu{\aln \chibarhat} \leq \frac{a}{\lvert u_\infty \rvert} \lVert \aln \chibarhat \rVert_{\mathcal{L}^{2}_{(sc)}(S_{u_{\infty},\ubar})}  + \int_{u_{\infty}}^u \frac{a^2}{\lvert u^\prime \rvert^3} \scaletwoSuprime{(\al)^i F} \hspace{.5mm} \text{d} u^\prime.    \]Multiplying this equation by $a^{-\frac{1}{2}}$ we get
		
		\begin{align*} \frac{\al}{\lvert u \rvert} \scaletwoSu{\aln \chibarhat} \leq& \frac{\al}{\lvert u_\infty \rvert} \lVert \aln \chibarhat \rVert_{\mathcal{L}^{2}_{(sc)}(S_{u_{\infty},\ubar})}  + \int_{u_{\infty}}^u \frac{a^{\frac{3}{2}}}{\lvert u^\prime \rvert^3} \scaletwoSuprime{\aln\alphabar} \hspace{.5mm} \text{d} u^\prime \\
		&+ \int_{u_{\infty}}^u \frac{a^{\frac{3}{2}}}{\lvert u^\prime \rvert^3} \scaletwoSuprime{ \sum_{i_1 + i_2 + i_3+1 =i}    (\al)^i \nabla^{i_1} \psi^{i_2+1} \nabla^{i_3} \alphabar} \hspace{.5mm} \text{d} u^\prime   \\ 
		&+ \int_{u_{\infty}}^u \frac{a^{\frac{3}{2}}}{\lvert u^\prime \rvert^3} \scaletwoSuprime{ \sum_{i_1 + i_2 + i_3+i_4 =i}    (\al)^i \nabla^{i_1} \psi^{i_2} \nabla^{i_3} (\psi,\chibarhat, \widetilde{\tr\chibar})\nabla^{i_4}\chibarhat} \hspace{.5mm} \text{d} u^\prime \\ 
		&+ \int_{u_{\infty}}^u \frac{a^{\frac{3}{2}}}{\lvert u^\prime \rvert^3} \scaletwoSuprime{ \sum_{i_1 + i_2 + i_3+i_4+1 =i}    (\al)^i \nabla^{i_1} \psi^{i_2+1} \nabla^{i_3}\tr\chibar \nabla^{i_4} \chibarhat} \hspace{.5mm} \text{d} u^\prime. \end{align*}The initial data term is directly bounded by $\mathcal{I}^{(0)}(\ubar) \lesssim 1.$ For the terms containing $\alphabar$ , we have
		
		\begin{align*}
		&\int_{u_{\infty}}^u \frac{a^{\frac{3}{2}}}{\lvert u^\prime \rvert^3} \scaletwoSuprime{(\al \nabla)^i \alphabar} \hspace{.5mm} \text{d} u^\prime + \int_{u_{\infty}}^u \frac{a^{\frac{3}{2}}}{\lvert u^\prime \rvert^3} \scaletwoSuprime{ \sum_{i_1 + i_2 + i_3+1 =i}    (\al)^i \nabla^{i_1} \psi^{i_2+1} \nabla^{i_3} \alphabar} \hspace{.5mm} \text{d} u^\prime \\
		\leq& \scaletwoHbaru{\aln \alphabar} \cdot \frac{a}{\lvert u \rvert^{\frac{3}{2}}} + \frac{a^2 \cdot O^2}{\lvert u \rvert^3} \leq 1.
		\end{align*}The last two terms can be bounded as follows:
		
		\begin{align*}
		&\int_{u_{\infty}}^u \frac{a^{\frac{3}{2}}}{\lvert u^\prime \rvert^3} \scaletwoSuprime{ \sum_{i_1 + i_2 + i_3+i_4 =i}    (\al)^i \nabla^{i_1} \psi^{i_2} \nabla^{i_3} (\psi,\chibarhat, \widetilde{\tr\chibar})\nabla^{i_4}\chibarhat} \hspace{.5mm} \text{d} u^\prime \\
		&+ \int_{u_{\infty}}^u \frac{a^{\frac{3}{2}}}{\lvert u^\prime \rvert^3} \scaletwoSuprime{ \sum_{i_1 + i_2 + i_3+i_4+1 =i}    (\al)^i \nabla^{i_1} \psi^{i_2+1} \nabla^{i_3}\tr\chibar \nabla^{i_4} \chibarhat} \hspace{.5mm} \text{d} u^\prime 
		\leq \frac{O^2+O^3}{\al} \leq 1.
		\end{align*}
	\end{proof}

	\begin{proposition}\label{chihatproposition}
		Under the assumptions of Theorem \ref{main1} and the bootstrap assumptions \eqref{bootstrapbounds}, we have 
		\[  \sum_{i \leq 10} \frac{1}{\al} \scaletwoSu{(\al \nabla)^i \chihat}\lesssim \mathcal{R}[\alpha]+1.        \]
	\end{proposition}
	
	\begin{proof}
		We look at the schematic equation \be \label{chihateq}  \nabla_4 \chihat = \psi \cdot \chihat + \alpha.\ee Commuting \eqref{chihateq} with $i$ angular derivatives we arrive at
		
		\be \nabla_4 \nabla^i \chihat = \nabla^{i} \alpha + \sum_{i_1+i_2+i_3+1=i} \nabla^{i_1} \psi^{i_2+1} \nabla^{i_3} \alpha + \sum_{i_1 + i_2 + i_3+i_4 =i} \nabla^{i_1} \psi^{i_2} \nabla^{i_3}(\psi,\chihat) \nabla^{i_4} \chihat. \ee
		
		\par \noindent	We thus have, passing to scale-invariant norms,
		
		\begin{align*}
		&\frac{1}{\al} \scaletwoSu{\aln \chihat} \\ 
		\leq& \frac{1}{\al} \int_{0}^{\ubar} \scaletwoSuubarprime{\aln \alpha} \hspace{.5mm} \text{d}\ubar^{\prime} +  \sum_{i_1 + i_2 + i_3+1 =i} \frac{1}{\al} \int_0^{\ubar} \scaletwoSuubarprime{(\al)^i \nabla^{i_1} \psi^{i_2+1} \nabla^{i_3} \alpha } \hspace{.5mm} \text{d}\ubar^{\prime} \\ 
		&+\sum_{i_1+i_2+i_3 +i_4=i}  \frac{1}{\al} \int_0^{\ubar} \scaletwoSuubarprime{(\al)^i \nabla^{i_1} \psi^{i_2} \nabla^{i_3} (\psi,\chihat) \nabla^{i_4} \chihat} \hspace{.5mm}\text{d} \ubar^{\prime} \\  \leq& \frac{1}{\al}\left( \int_0^{\ubar} \scaletwoSuubarprime{\aln \alpha}^2 \hspace{.5mm} \text{d}\ubar^\prime         \right)^{\frac{1}{2}} \left( \int_0^{\ubar} 1 \hspace{.5mm} \text{d}\ubar^{\prime} \right)^{\frac{1}{2}} \\ &+ \sum_{i_1 + i_2 + i_3+1 =i} \int_0^{\ubar} \ScaletwoSuubarprime{(\al)^i \nabla^{i_1} \psi^{i_2+1} \nabla^{i_3}\left(\frac{\alpha}{\al}\right)}   \hspace{.5mm} \text{d}\ubar^{\prime}           \\  &+ \al  \sum_{i_1+i_2+i_3 +i_4=i} \int_0^{\ubar} \ScaletwoSuubarprime{(\al)^i \nabla^{i_1} \psi^{i_2} \nabla^{i_3} \left(\frac{\psi}{\al},\frac{\chihat}{\al}\right)  \nabla^{i_4} \left(\frac{\chihat}{\al} \right) }    \hspace{.5mm}\text{d} \ubar^{\prime} \\
		\leq& \frac{1}{\al} \scaletwoHu{\aln \alpha} + \frac{\al \cdot O^2}{\lvert u \rvert} \leq \mathcal{R}[\alpha]+ \frac{4 O^2}{\al} \leq \mathcal{R}[\alpha]+1.
		\end{align*}
		
	\par \noindent 	The result follows.
	\end{proof}\par \noindent We proceed with estimates for $\omegabar$.

	\begin{proposition}\label{omegabarprop}
		Under the assumptions of Theorem \ref{main1} and the bootstrap assumptions \eqref{bootstrapbounds}, there holds 
		
		\[ \sum _{i\leq 10} \scaletwoSu{\aln \omegabar} \lesssim \mathcal{R}[\rho]+1.   \]

	\end{proposition} \begin{proof}
		We have the schematic null structure equation 
		\[   \nabla_4 \omegabar    = \rho + \psi \cdot \psi + \Y \cdot \Y       \]Commuting this equation with $i$ angular derivatives, using Proposition \ref{commutationformulaeprop}, we obtain
		
		\begin{align*}   \nabla_4 \nabla^i \omegabar =& \nabla^i \rho + \sum_{i_1+i_2+i_3+1 =i}\nabla^{i_1} \psi^{i_2+1} \nabla^{i_3} \rho + \sum_{i_1+i_2+i_3+i_4=i} \nabla^{i_1} \psi^{i_2} \nabla^{i_3} \Y \nabla^{i_4} \Y  \\ 
		&+ \sum_{i_1+i_2+i_3+i_4=i} \nabla^{i_1} \psi^{i_2} \nabla^{i_3}(\psi,\chihat) \nabla^{i_4} \psi.    \end{align*}Multiplying by $(\al)^i$ and using Proposition \ref{prop36} we get

		\begin{align*}
		&\scaletwoSu{\aln \omegabar} \\
		\leq& \int_0^{\ubar} \scaletwoSuubarprime{\aln \rho} \hspace{.5mm} \text{d}\ubar^\prime + \sum_{i_1 + i_2 + i_3+1 =i} \int_0^{\ubar} \scaletwoSuubarprime{(\al)^i \nabla^{i_1} \psi^{i_2+1} \nabla^{i_3} \rho} \hspace{.5mm} \text{d}\ubar^\prime \\
		&+ \sum_{ i_1+i_2+i_3+i_4=i} \int_0^{\ubar} \scaletwoSuubarprime{(\al)^i \nabla^{i_1} \psi^{i_2} \nabla^{i_3} \Y \nabla^{i_4} \Y} \hspace{.5mm} \text{d}\ubar^\prime \\ 
		&+   \sum_{i_1+i_2+i_3+i_4=i}\int_0^{\ubar}  \scaletwoSuubarprime{(\al)^i \nabla^{i_1} \psi^{i_2} \nabla^{i_3} (\psi,\chihat)\nabla^{i_4} \psi} \hspace{.5mm} \text{d}\ubar^\prime \\  \leq& \left( \int_0^{\ubar} \scaletwoSuubarprime{\aln \rho}^2 \hspace{.5mm} \text{d}\ubar^\prime         \right)^{\frac{1}{2}} \left( \int_0^{\ubar} 1 \hspace{.5mm} \text{d}\ubar^{\prime} \right)^{\frac{1}{2}}         \\ &+ \sum_{i_1 + i_2 + i_3+1 =i} \int_0^{\ubar} \scaletwoSuubarprime{(\al)^i \nabla^{i_1} \psi^{i_2+1} \nabla^{i_3} \rho} \hspace{.5mm} \text{d}\ubar^\prime         \\ 	&+ \sum_{ i_1+i_2+i_3+i_4=i} \int_0^{\ubar} \scaletwoSuubarprime{(\al)^i \nabla^{i_1} \psi^{i_2} \nabla^{i_3} \Y \nabla^{i_4} \Y} \hspace{.5mm} \text{d}\ubar^\prime       \\    	&+ \sum_{i_1+i_2+i_3+i_4=i}  \int_0^{\ubar} {\color{black}\at} \ScaletwoSuubarprime{(\al)^i \nabla^{i_1} \psi^{i_2} \nabla^{i_3} \left(\frac{\psi}{\al},\frac{\chihat}{\al} \right)\nabla^{i_4} \psi} \hspace{.5mm} \text{d}\ubar^\prime                       \\
		\leq&  \scaletwoHu{\aln \rho} + \frac{\al \cdot O^2}{\lvert u \rvert} \leq \mathcal{R}[\rho]+1.
		\end{align*}Here and throughout we have made use of Proposition \ref{usefulstatements}. 
	\end{proof}
	
	\begin{proposition}
		Under the assumptions of Theorem \ref{main1} and the bootstrap assumptions \eqref{bootstrapbounds}, we have \[    \sum_{i \leq 10}  \scaletwoSu{(\al \nabla)^i \eta}\lesssim \mathcal{R}[\tbeta]+1.   \] 
	\end{proposition}
	
	\begin{proof}  We begin by recalling the structure equation \eqref{etastructureequation} for $\eta$: \[	\nabla_4 \eta_a = - \chi_{ab} \cdot (\eta - \etabar)_b-\beta_a-\frac12 R_{a4}. \]Also recall that $\tbeta = \beta - \frac{1}{2}R_{4(\cdot)}$.  We can therefore rewrite \eqref{etastructureequation} in terms of $\tbeta$ as follows: \[	\nabla_4 \eta_a = - \chi_{ab} \cdot (\eta - \etabar)_b-\tbeta_a- R_{a4}.  \]         
			This leads us to the following schematic null structure equation:  \[ \nabla_4 \eta = \tbeta  + \psi \cdot (\psi,\chihat) + (\rho_F,\sigma_F)\cdot \alpha_F.    \]Commuting with $i$ angular derivatives, using Proposition \ref{commutationformulaeprop}, we have
		\begin{align*} \nabla_4 \nabla^i \eta =& \nabla^i \tbeta   + \sum_{i_1 + i_2 +i_3 +1= i} \nabla^{i_1} \psi^{i_2+1} \nabla^{i_3}\tbeta    \\
		&+ \sum_{i_1+ i_2 + i_3 +i_4 =i} \nabla^{i_1} \psi^{i_2} \nabla^{i_3} (\psi,\chihat) \nabla^{i_4} \psi \\ 
		&+  \sum_{i_1+ i_2 + i_3 +i_4 =i} \nabla^{i_1} \psi^{i_2} \nabla^{i_3}(\rho_F, \sigma_F) \nabla^{i_4} \alpha_F.          \end{align*}Working in scale-invariant norms, we get
		
		\begin{align*}
		&\scaletwoSu{(\al \nabla)^i \eta}\\
		\leq& \int_0^{\ubar} \lVert (\al \nabla)^i \tbeta   \rVert_{\mathcal{L}^{2}_{(sc)}(S_{u,\ubar^\prime})} \text{d}\ubar^{\prime} + \sum_{i_1 + i_2 +i_3 +1= i} \lVert (\al)^i \nabla^{i_1} \psi^{i_2+1} \nabla^{i_3}\tbeta   \rVert_{\mathcal{L}^{2}_{(sc)}(S_{u,\ubar^\prime})}\text{d}\ubar^\prime \\ 
		&+ \sum_{i_1+i_2+i_3+i_4 =i} \al \big\lVert (\al)^i \nabla^{i_1} \psi^{i_2} \nabla^{i_3} \psi \nabla^{i_4}\left(\frac{\psi}{\al}, \frac{\chihat}{\al}\right) \big\rVert_{\mathcal{L}^{2}_{(sc)}(S_{u,\ubar^\prime})} \text{d}\ubar^\prime  \\
		&+ \sum_{i_1+i_2+i_3+i_4 =i} \al \big \lVert (\al)^i \nabla^{i_1} \psi^{i_2} \nabla^{i_3} \left(\frac{\alpha_F}{\al}\right) \nabla^{i_4}(\rho_F, \sigma_F) \big\rVert_{\mathcal{L}^{2}_{(sc)}(S_{u,\ubar^\prime})} \text{d}\ubar^\prime \\
		\leq& \mathcal{R}[\tbeta]+\frac{\al \cdot O^2}{\lvert u \rvert} \leq \mathcal{R}[\tbeta]+1.
		\end{align*}
	\end{proof}
		
	\begin{proposition}
		Under the assumptions of Theorem \ref{main1} and the bootstrap assumptions \eqref{bootstrapbounds}, we have \[  \sum_{i \leq 10} \scaletwoSu{(\al \nabla)^i \tr \chi} \lesssim (\mathcal{R}[\alpha]+\underline{\mathcal{F}}[\rho_F,\sigma_F]+1)^{{\color{black}2}}.      \]\label{trchiboundRicci}
	\end{proposition}
	
	\begin{proof}
		We again start by considering the schematic equation 
		\[ \nabla_4 \tr \chi =\chihat\cdot \chihat + \alpha_F \cdot \alpha_F + \psi \psi.       \]By commuting with $i$ angular derivatives, we arrive at
		
		\begin{align*} \nabla_4 \nabla^i \tr\chi =& \sum_{i_1+ i_2 + i_3 +i_4 =i} \nabla^{i_1}\psi^{i_2} \nabla^{i_3}\chihat \nabla^{i_4} \chihat + \sum_{i_1 + i_2 +i_3 + i_4 =i}  \nabla^{i_1}\psi^{i_2} \nabla^{i_3}\alpha_F \nabla^{i_4} \alpha_F  \\ 
		&+ \sum_{i_1 + i_2 +i_3 + i_4 =i} \nabla^{i_1} \psi^{i_2} \nabla^{i_3} (\psi, \chihat) \nabla^{i_4} \psi \\ =& \sum_{i_1+i_2 =i} \nabla^{i_1} \chihat \nabla^{i_2} \chihat + \sum_{i_1+i_2+i_3+i_4+1=i} \nabla^{i_1} \psi^{i_2+1} \nabla^{i_3}\chihat \nabla^{i_4}\chihat     \\  &+   \sum_{i_1+i_2 =i} \nabla^{i_1} \alpha_F \nabla^{i_2} \alpha_F + \sum_{i_1+i_2+i_3+i_4+1=i} \nabla^{i_1} \psi^{i_2+1} \nabla^{i_3}\alpha_F \nabla^{i_4}\alpha_F    \\  &+\sum_{i_1 + i_2 +i_3 + i_4 =i} \nabla^{i_1} \psi^{i_2} \nabla^{i_3} (\psi, \chihat) \nabla^{i_4} \psi.  \end{align*}Taking this into account\footnote{In the following, even though we do not encounter cross terms of the form $\nabla^{i_1}\chihat \nabla^{i_2} \alpha_F$, we do not lose any control on the inequality by grouping the terms together and controlling schematically terms of the form $\nabla^{i_1}(\chihat,\alpha_F)\nabla^{i_2}(\chihat,\alpha_F)$.}, we have
		
		\begin{equation}
\begin{split} \label{Kusefulequation}
		&\scaletwoSu{\aln \tr\chi}\\
		\leq& \sum_{i_1+i_2=i}  \int_{0}^{\ubar} a \big\lVert (\al)^i \nabla^{i_1}\left(\frac{\chihat}{\al}, \frac{\alpha_F}{\al}\right)\nabla^{i_2}\left(\frac{\chihat}{\al}, \frac{\alpha_F}{\al}\right)\big \rVert_{\mathcal{L}^2_{(sc)}(S_{u,\ubar^\prime})} \text{d}\ubar^\prime \\
		&+ \sum_{i_1 + i_2 +i_3 + i_4 +1 =i} \int_0^{\ubar} a \big\lVert  (\al)^i \nabla^{i_1} \psi^{i_2+1} \nabla^{i_3} \left(\frac{\chihat}{\al}, \frac{\alpha_F}{\al}\right)\nabla^{i_2}\left(\frac{\chihat}{\al}, \frac{\alpha_F}{\al}\right)\big \rVert_{\mathcal{L}^2_{(sc)}(S_{u,\ubar^\prime})} \text{d}\ubar^\prime \\
		&+  \sum_{i_1 + i_2 +i_3 + i_4 =i} \int_0^{\ubar}  \al \big\lVert  {\color{black}(\at)^i}\nabla^{i_1} \psi^{i_2} \nabla^{i_3} \left(\frac{\psi}{\al}, \frac{\chihat}{\al}\right) \nabla^{i_4} \psi\big\rVert_{\mathcal{L}^{2}_{(sc)}(S_{u,\ubar^\prime})}\text{d}\ubar^\prime \\
		\leq& \frac{a}{\lvert u \rvert}O_{2}[\chihat,\alpha_F]\cdot O_{\infty}[\chihat,\alpha_F] + \frac{a }{\lvert u \rvert^2}O^3+ \frac{\al}{\lvert u \rvert}O^2\\
		\leq& O_{2}[\chihat,\alpha_F]\cdot O_{\infty}[\chihat,\alpha_F] +1 \lesssim {\color{black}(\mathcal{R}[\alpha]+ \underline{\mathcal{F}}[\rho_F,\sigma_F]+1)^2},
		\end{split} 
\end{equation}
by using the estimates on $\chihat$ proved in Proposition \ref{chihatproposition} and the estimates on $\alpha_F$ that will be shown in Proposition \ref{alphaFproposition} in the following subsection.    
	\end{proof}\par \noindent We move on to estimates for $\tr\chibar$. 
	\begin{proposition} \label{trchibarboundRicci}
		Under the assumptions of Theorem \ref{main1} and the bootstrap assumptions \eqref{bootstrapbounds}, we have 
		
		\[   \sum_{i \leq 10} \frac{a}{\lvert u \rvert }  \big \lVert (\al \nabla)^i(\tr\chibar+ \frac{2}{\lvert u \rvert})\big \rVert_{\mathcal{L}^2_{(sc)}(S_{u,\ubar}) } \lesssim \mathcal{R}[\rho] + \underline{\mathcal{R}}[\rho]+1,\hspace{2mm}  \sum_{i \leq 10} \frac{a}{\lvert u \rvert^2}\scaletwoSu{(\al \nabla)^i \tr\chibar} \lesssim \f{\M R[\rho]+\underline{\M R}[\rho]}{|u|}+1.                \]\end{proposition}
	\begin{proof} 
For $\tildetr=\tr\chib+\f{2}{|u|}$, we have the schematic null structure equation     
		\[  \nabla_3 \widetilde{\tr\chibar} + \tr\chi \tildetr = \frac{2}{\lvert u \rvert^2}(\Omega^{-1}-1)+\tildetr \tildetr +\psi \tr\chibar - \lvert \chibarhat\rvert^2 - \lvert \alphabar_F \rvert^2.                   \]Commuting this equation with $i$ angular derivatives, we have 
		
		\begin{align*}  &\nabla_3 \nabla^i \tildetr + \frac{i+2}{2}\tr\chibar \tildetr\\ =& \sum_{i_1 + i_2 +i_3 = i}\nabla^{i_1} \psi^{i_2} \nabla^{i_3} \left(\frac{2}{\lvert u \rvert^2}(\Omega^{-1}-1)+\tildetr \tildetr +\psi \tr\chibar - \lvert \chibarhat \rvert^2 - \lvert \alphabar_F \rvert^2 \right)  \\ 
		&+ \sum_{i_1 + i_2 +i_3 + i_4 =i} \nabla^{i_1} \psi^{i_2} \nabla^{i_3}(\psi,\chibarhat, \tildetr)\nabla^{i_4} \tildetr       + \sum_{i_1 + i_2 +i_3 + i_4 +1 =i} \nabla^{i_1} \psi^{i_2+1} \nabla^{i_3} \tr\chibar \nabla^{i_4}\tildetr \\ :=& \tilde{F}.        \end{align*}Rewriting in terms of scale-invariant norms, \begin{align*} \frac{a}{\lvert u \rvert} \scaletwoSu{(\al \nabla)^i \tildetr } \leq& \frac{a}{\lvert u_\infty \rvert} \rvert \lVert (\al \nabla)^i \tildetr  \rVert_{\mathcal{L}^2_{(sc)}(S_{u_{\infty},{\color{black}\ub}})} +\int_{u_{\infty}}^{u}  \frac{a^2}{\lvert  u^\prime \rvert^3    } \lVert (\al)^i \tilde{F}\rVert_{\mathcal{L}^2_{(sc)}(S_{u^\prime, \ubar})}\text{d}u^\prime\\
		=&\frac{a}{\lvert u_\infty \rvert} \lVert (\al \nabla)^i \tildetr  \rVert_{\mathcal{L}^2_{(sc)}(S_{u_{\infty},{\color{black}\ub}})}  + \mathcal{I}_1+\mathcal{I}_2+\mathcal{I}_3+\mathcal{I}_4,        \end{align*}where \[  \frac{a}{\lvert u \rvert} \scaletwoSu{(\al \nabla)^i \tildetr } \leq \frac{a}{\lvert u_\infty \rvert} \rvert \lVert (\al \nabla)^i \tildetr  \rVert_{\mathcal{L}^2_{(sc)}(S_{u_{\infty},{\color{black}\ub}})} \lesssim 1,  \]\begin{align*}
		\mathcal{I}_1 =&\int_{u_{\infty}}^u  \frac{a^2}{\lvert u^\prime \rvert^3}\lVert (\al)^i \sum_{i_1 + i_2 +i_3 + i_4 =i} \nabla^{i_1} \psi^{i_2} \nabla^{i_3} (\psi , \tildetr, \chibarhat, \alphabar_F) \nabla^{i_4}(\psi, \tildetr, \chibarhat, \alphabar_F) \rVert_{\mathcal{L}^{2}_{(sc)}(S_{u^{\prime},\ubar})}\text{d}u^\prime\\ 
		=& \int_{u_{\infty}}^u  \frac{a}{\lvert u^\prime \rvert}\lVert (\al)^i \sum_{i_1 + i_2 +i_3 + i_4 =i} \nabla^{i_1} \psi^{i_2} \nabla^{i_3} (\frac{\al}{\lvert u^\prime \rvert}\psi , \frac{\al}{\lvert u^\prime \rvert}\tildetr,\frac{\al}{\lvert u^\prime \rvert} \chibarhat, \frac{\al}{\lvert u^\prime \rvert}\alphabar_F)
		\\ &\quad \quad \quad \quad \quad \quad \quad \quad \quad \quad \quad \quad \times \nabla^{i_4}(\frac{\al}{\lvert u^\prime \rvert}\psi, \frac{\al}{\lvert u^\prime \rvert}\tildetr, \frac{\al}{\lvert u^\prime \rvert}\chibarhat, \frac{\al}{\lvert u^\prime \rvert}\alphabar_F) \rVert_{\mathcal{L}^{2}_{(sc)}(S_{u^{\prime},\ubar})}\text{d}u^\prime
		\\ \lesssim& O^2[\chibarhat]+ \frac{a^2}{\lvert u \rvert^3}O^2[\alphabar_F]+1 \lesssim 1.
		\end{align*}There holds  \begin{equation*} 
		\begin{split}
		\mathcal{I}_2=&\int_{u_{\infty}}^u \f{a^2}{|u'|^3}\|(\at)^i \sum_{i_1+i_2+i_3+i_4=i}\nab^{i_1}\p^{i_2}\nab^{i_3}\omb\nab^{i_4}\tr\chib\|_{\mathcal{L}^2_{{\color{black}(sc)}}(S_{u',\ub})}du'\\
		=&\int_{u_{\infty}}^u \f{a}{|u'|^2}\|(\at)^{i} \sum_{i_1+i_2+i_3+i_4+1=i}\nab^{i_1}\p^{i_2}\nab^{i_3}\omb\nab^{i_4+1}(\f{a}{|u'|}\tc)\|_{\mathcal{L}^2_{{\color{black}(sc)}}(S_{u',\ub})}du'\\
		&+\int_{u_{\infty}}^u \f{a}{|u'|}\|(\at)^{i} \sum_{i_1+i_2+i_3=i}\nab^{i_1}\p^{i_2}\nab^{i_3}\omb\cdot(\f{a}{|u'|^2}\tr\chib)\|_{\mathcal{L}^2_{{\color{black}(sc)}}(S_{u',\ub})}du'\\
		\leq& \int_{u_{\infty}}^u \f{a}{|u'|}\cdot\f{1}{|u'|}\cdot \big (O[\omb]\cdot O_{\infty}[\tr\chibar]\big) \, du' \quad \\
		\ls& O[\omb]O_{\infty}[\tr\chibar]+1\ls \mathcal{R}[\rho]+1  \quad (\mbox{by Proposition \ref{omegabarprop} and the fact 
			 that $\frac{a}{\lvert u \rvert^2} \scaleinfinitySu{\tr\chibar} \lesssim 1$}).
		\end{split}
		\end{equation*}
	There also holds	\begin{equation*}
		\begin{split}
		\mathcal{I}_3=&\int_{u_{\infty}}^u \f{a^2}{|u'|^3}\|(\at)^i \sum_{i_1+i_2+i_3=i}\nab^{i_1}\p^{i_2}\nab^{i_3}(\f{\O^{-1}-1}{|u'|^2})\|_{\mathcal{L}^2_{{\color{black}(sc)}}{(S_{u',\ub})}}du'\\
		=&\int_{u_{\infty}}^u |u'|^{i+1} \|\sum_{i_1+i_2+i_3=i}\nab^{i_1}\p^{i_2}\nab^{i_3}(\f{\O^{-1}-1}{|u'|^2})\|_{L^2{(S_{u',\ub})}}du' \quad (\mbox{in standard norms})\\
		=&\int_{u_{\infty}}^u |u'|^{i+1} \|\sum_{i_1+i_2+i_3=i}\nab^{i_1}\p^{i_2}\nab^{i_3}(\f{\O^{-1}-1}{|u'|^2})\|_{L^2{(S_{u',\ub})}}du' \,\, (\mbox{Using} \,\, \f{\partial}{\partial \ub}\O^{-1}=2\o \Leftrightarrow \nab_4\O^{-1}=2\O^{-1}\o)\\
		=&\int_{u_{\infty}}^u |u'|^{i+1} \|\sum_{i_1+i_2+i_3=i}\nab^{i_1}\p^{i_2}\nab^{i_3}[\f{1}{|u'|^2}\cdot \int_0^{\ub}2\o(u',\ub',\theta^1, \theta^2)d\ub']\|_{L^2{(S_{u',\ub})}}du' \\
		=&\int_{u_{\infty}}^u |u'|^{i+1} \|\sum_{i_1+i_2+i_3=i}\nab^{i_1}\p^{i_2}[\f{1}{|u'|^2}\cdot \int_0^{\ub}2\nab^{i_3}\o(u',\ub',\theta^1, \theta^2)d\ub']\|_{L^2{(S_{u',\ub})}}du' \\
		\leq&  |u'|^{i+1} \|\sum_{i_1+i_2+i_3=i}\f{1}{|u'|^{i_1+i_2}}\cdot\f{1}{|u'|^2}\cdot\f{1}{|u'|^{i_3}}\cdot \f{\at}{|u'|^{\f12}}\cdot \big(\underline{\mathcal{R}}[\rho]+1\big)\, du'    \\
		\leq& \int_{u_{\infty}}^u \f{\at}{|u'|^{\f32}} \big(\underline{\mathcal{R}}[\rho]+1\big)\,du' \ls \underline{\mathcal{R}}[\rho]+1.
		\end{split}
		\end{equation*}\vspace{3mm}
		
		\par \noindent Finally, there holds \begin{equation*} 
		\begin{split}
		\mathcal{I}_4=&\int_{u_{\infty}}^u \f{a^2}{|u'|^3}\|(\at)^i \sum_{i_1+i_2+i_3=i-1}\nab^{i_1}\p^{i_2+1}\cdot\tr\chib\cdot\nab^{i_3}\tc\|_{\mathcal{L}^2_{sc}(S_{u',\ub})}du'\\
		=&\int_{u_{\infty}}^u \at \|(\at)^{i-1} \sum_{i_1+i_2+i_3=i-1}\nab^{i_1}\p^{i_2+1}\cdot\f{a}{|u'|^2}\tr\chib\cdot\nab^{i_3}(\f{a}{|u'|}\tc)\|_{\mathcal{L}^2_{sc}(S_{u',\ub})}du'\\
		\leq& \int_{u_{\infty}}^u \at\cdot \f{O^3}{|u'|^2} du'\leq 1 \quad  (\mbox{by Proposition \ref{omegafirstequation}}).
		\end{split}
		\end{equation*}
		In summary, we have obtained
		$$\sum_{i\leq 10} \f{a}{|u|}\|(\at\nab)^i (\tr\chib+\f{2}{|u|})\|_{{\color{black}\mathcal{L}^2_{(sc)}}(\S)}\ls \M R[\rho]+\underline{\M R}[\rho]+1,$$
		which implies
		 $$\sum_{i \leq 10} \frac{a}{\lvert u \rvert^2}\scaletwoSu{(\al \nabla)^i \tr\chibar} \lesssim \f{\M R[\rho]+\underline{\M R}[\rho]}{|u|}+1. $$

			\end{proof}
	
	\begin{proposition}
		Under the assumptions of Theorem \ref{main1} and the bootstrap assumptions \eqref{bootstrapbounds}, we have\[ \sum_{ i \leq 10}   \scaletwoSu{(\al \nabla)^i \etabar} \lesssim \underline{\mathcal{R}}[\tbetabar]+ \mathcal{R}[\tbeta]+1.  \] 
	\end{proposition}
	
	\begin{proof}
		We use the following schematic null structure equation for $\etabar$: 
		\[\nabla_3 \etabar + \frac{1}{2}\tr\chibar\hspace{.5mm} \etabar =  \tbetabar + \tr\chibar \eta + \chibarhat\cdot \psi + \Y \cdot \Y. \]Commuting with $i$ angular derivatives, we have \begin{align*}  &\nabla_3 \nabla^i \etabar + \frac{i+1}{2}\tr\chibar {\color{black}\nab^i \etb} \\
		=&  \nabla^i \tbetabar + \sum_{i_1+i_2+i_3+1=i}\nabla^{i_1}\psi^{i_2+1} \nabla^{i_3} \tbetabar +\tr\chibar \nabla^i \eta + \sum_{i_1+i_2+1= i} \nabla^{i_1+1} \tr\chibar \nabla^{i_2}(\eta,\etabar) \\
		&+  \sum_{i_1 + i_2 +i_3 + i_4 +1 =i} \nabla^{i_1} \psi^{i_2+1} \nabla^{i_3} \psi \nabla^{i_4} \tr\chibar + \sum_{i_1 + i_2 +i_3 + i_4 =i} \nabla^{i_1} \psi^{i_2} \nabla^{i_3}\psi \nabla^{i_4} (\chibarhat, \tildetr)\\
		&+ \sum_{i_1 + i_2 +i_3 + i_4 =i} \nabla^{i_1} \psi^{i_2} \nabla^{i_3}\Y \nabla^{i_4} \Y.   \end{align*} 
		By passing to scale-invariant norms we have 
		
		\begin{align*}
		&\frac{1}{\lvert u \rvert} \scaletwoSu{\aln \etabar} \\ \leq& \frac{1}{\lvert u_{\infty} \rvert} \lVert \aln \etabar \rVert_{\mathcal{L}^2_{(sc)}(u_{\infty},\ubar)} + \int_{u_{\infty}}^u  \frac{a}{\lvert u^\prime \rvert^3} \scaletwoSuprime{\aln  \tbetabar} \hspace{.5mm} \text{d} u^{\prime} \\ 
		&+ \int_{u_{\infty}}^u \frac{a}{\lvert u^{\prime}\rvert^3} \scaletwoSuprime{\sum_{i_1 + i_2 + i_3+1 =i} \aln \nabla^{i_1}\psi^{i_2+1} \nabla^{i_3} \tbetabar} \hspace{.5mm} \text{d} u^\prime + \int_{u_{\infty}}^u  \frac{a}{\lvert u^\prime \rvert^3}  \scaletwoSuprime{ \tr\chibar \aln \eta } \hspace{.5mm} \text{d}u^\prime  \\ 
		&+ \int_{u_{\infty}}^u  \frac{a}{\lvert u^\prime \rvert^3}  \scaletwoSuprime{ \sum_{i_1 + i_2 + 1=i} (\al\nabla)^{i_1+1} \tr\chibar (\al \nabla)^{i_2} \eta } \hspace{.5mm} \text{d}u^\prime \\ 
		&+ \int_{u_{\infty}}^u \frac{a}{\lvert u^\prime \rvert^3} \scaletwoSuprime{ \sum_{i_1 + i_2 + i_3+i_4+1 =i} (\al)^i \nabla^{i_1} \psi^{i_2+1} \nabla^{i_3} \psi \nabla^{i_4} \tr\chibar}\hspace{.5mm} \text{d}u^\prime \\ 
		&+ \int_{u_{\infty}}^u \frac{a}{\lvert u^\prime \rvert^3} \scaletwoSuprime{ \sum_{i_1 + i_2 + i_3+i_4 =i} (\al)^i \nabla^{i_1} \psi^{i_2} \nabla^{i_3} \psi \nabla^{i_4} (\chibarhat, \widetilde{\tr\chibar})}\hspace{.5mm} \text{d}u^\prime	\\
		&+ \int_{u_{\infty}}^u \frac{a}{\lvert u^\prime \rvert^3} \scaletwoSuprime{ \sum_{i_1 + i_2 + i_3+i_4 =i} (\al)^i \nabla^{i_1} \psi^{i_2} \nabla^{i_3} \Y \nabla^{i_4} \Y}\hspace{.5mm} \text{d}u^\prime \\ 
		\leq& \frac{1}{\lvert u_\infty \rvert} + \frac{\underline{\mathcal{R}}[\t\beta] + \mathcal{R}[\t\beta]+1}{\lvert u \rvert} + \int_{u_{\infty}}^u \frac{a}{\lvert u^\prime \rvert^3} \cdot \frac{O^2}{\lvert u^\prime \rvert}\hspace{.5mm} \text{d}u^\prime \lesssim  \frac{\underline{\mathcal{R}}[\t\beta] + \mathcal{R}[\t\beta]+1}{\lvert u \rvert}. 
		\end{align*}

	\end{proof}
	\par \noindent  This concludes the $L^2$-estimates on Ricci coefficients.
	\subsection{$L^2(S_{u,\ubar})$-estimates for the Maxwell components}
	
	\begin{proposition}
		Under the assumptions of Theorem \ref{main1} and the bootstrap assumption \eqref{bootstrapbounds}, we have
		\[ \sum_{i\leq 10} \frac{1}{\al}\scaletwoSu{(\al \nabla)^i \alpha_F} \leq  \underline{\mathcal{F}} [\rho_F,\sigma_F]+ 1.   \] \label{alphaFproposition}
	\end{proposition}
	\begin{proof}
		We have the schematic equation
		\[ \nabla_3 \alpha_F+ \frac{1}{2}\tr\chibar \alpha_F = \nabla (\rho_F,\sigma_F) + \psi \cdot (\rho_F, \sigma_F) + \chihat \cdot \alphabar_F + \psi\cdot \alpha_F.       \]Commuting with $i$ angular derivatives, we get
		
		\begin{align*}
		&\nabla_3 \nabla^i \alpha_F + \frac{i+1}{2}\tr\chibar \alpha_F\\
		=& \nabla^{i+1}(\rho_F,\sigma_F) + \sum_{i_1+i_2+i_3+1= i} \nabla^{i_1}\psi^{i_2+1}\nabla^{i_3+1}(\rho_F,\sigma_F) \\ &+ \sum_{i_1+i_2+1=i} \nabla^{i_1+1}\tr\chibar \nabla^{i_2}\alpha_F + \sum_{i_1+i_2=i}\nabla^{i_1}\chibarhat \nabla^{i_2}\alpha_F  + \sum_{i_1 + i_2 +i_3 + i_4 =i} \nabla^{i_1}\psi^{i_2}\nabla^{i_3} \psi\nabla^{i_4}(\rho_F,\sigma_F) \\ &+ \sum_{i_1 + i_2 +i_3 + i_4 =i} \nabla^{i_1}\psi^{i_2}\nabla^{i_3} \chihat \nabla^{i_4}\alphabar_F + \sum_{i_1+i_2+i_3+i_4=i} \nabla^{i_1}\psi^{i_2}\nabla^{i_3}\psi \nabla^{i_4}\alpha_F \\
		&+ \sum_{i_1 + i_2 +i_3 + i_4 +1 =i} \nabla^{i_1}\psi^{i_2+1}\nabla^{i_3}\tr\chibar \nabla^{i_4}\alpha_F + \sum_{i_1 + i_2 +i_3 + i_4=i} \nabla^{i_1}\psi^{i_2}\nabla^{i_3}(\psi,\chibarhat,\tildetr)\nabla^{i_4}\alpha_F.
		\end{align*} Denote the right-hand side of the above as $G$. We then have 
		
		\[a^{-\frac{1}{2}} \lvert u \rvert^i \lVert \nabla^i \alpha_F \rVert_{L^{2}(S_{u,\ubar})} \leq  a^{-\frac{1}{2}} \lvert u_{\infty} \rvert^i \lVert  \nabla^i \alpha_F \rVert_{L^{2}(S_{u_{\infty},\ubar})}  + \int_{u_{\infty}}^u a^{-\frac{1}{2}} \lvert u^\prime \rvert^i \lVert G \rVert_{L^{2}(S_{u^\prime,\ubar})}\text{d}u^\prime.         \]From the signature table one can read off \[  s_2(\alpha_F)=0 \Rightarrow s_2(\nabla^i \alpha_F)= 0+ i\cdot \frac{1}{2} = \frac{i}{2}.     \]By conservation of signatures, \[ s_2(G) = s_2(\nabla_3 \nabla^i \alpha_F) = \frac{i+2}{2}.    \]Taking into account now that \[ \scaletwoSu{\phi} := a^{-s_2(\phi)}\lvert u \rvert^{2s_2(\phi)} \lVert \phi \rVert_{L^{2}(S_{u,\ubar})},       \]we can conclude that \[ a^{-\frac{1}{2}} \lvert u \rvert^i \lVert \nabla^i \alpha_F \rVert_{L^{2}(S_{u,\ubar})} = a^{-\frac{1}{2}}\scaletwoSu{(\al \nabla)^i \alpha_F} , \hspace{2mm} a^{-\frac{1}{2}} \lvert u \rvert^i \lVert G \rVert_{L^{2}(S_{u,\ubar})} = \frac{\al}{\lvert u \rvert^2}\scaletwoSu{(\al)^i G}.   \]Therefore, \begin{align*} &a^{-\frac{1}{2}} \lvert u \rvert^i \lVert \nabla^i \alpha_F \rVert_{L^{2}(S_{u,\ubar})}  \\ 
		\leq& a^{-\frac{1}{2}} \lvert u \rvert^i \lVert \nabla^i \alpha_F \rVert_{L^{2}(S_{u_{\infty},\ubar})} + \int_{u_{\infty}}^u \frac{\al}{\lvert u^\prime \rvert^2} \scaletwoSuprime{(\al)^i \nabla^{i+1}(\rho_F,\sigma_F)} \text{d}u^\prime \\
		&+ \int_{u_{\infty}}^u  \frac{\al}{\lvert u^\prime \rvert^2} \scaletwoSuprime{\sum_{i_1+i_2+i_3+1= i} ( \al)^i \nabla^{i_1} \psi^{i_2+1}\nabla^{i_3+1}(\rho_F,\sigma_F)} \text{d}u^\prime \\
		&+ \int_{u_{\infty}}^u  \frac{\al}{\lvert u^\prime \rvert^2} \scaletwoSuprime{  \sum_{i_1+i_2+1=i} (\al)^i  \nabla^{i_1+1}\tr\chibar \nabla^{i_2} \alpha_F} \text{d}u^\prime \\
		&+\int_{u_{\infty}}^u  \frac{\al}{\lvert u^\prime \rvert^2} \scaletwoSuprime{  \sum_{i_1+i_2=i} (\al)^i  \nabla^{i_1}\chibarhat \nabla^{i_2} \alpha_F} \text{d}u^\prime\\  &+ \int_{u_{\infty}}^u  \frac{\al}{\lvert u^\prime \rvert^2} \scaletwoSuprime{ (\al)^i \sum_{i_1+i_2+i_3+i_4 =i} \nabla^{i_1} \psi^{i_2} \nabla^{i_3}\psi \nabla^{i_4} (\rho_F,\sigma_F) } \text{d}u^\prime \\  &+ \int_{u_{\infty}}^u  \frac{\al}{\lvert u^\prime \rvert^2} \scaletwoSuprime{ (\al)^i \sum_{i_1+i_2+i_3+i_4 =i} \nabla^{i_1} \psi^{i_2} \nabla^{i_3}\chihat \nabla^{i_4} \alphabar_F } \text{d}u^\prime\\
		&+ \int_{u_{\infty}}^u  \frac{\al}{\lvert u^\prime \rvert^2} \scaletwoSuprime{(\al)^i \sum_{i_1 + i_2 +i_3 + i_4 =i} \nabla^{i_1} \psi^{i_2} \nabla^{i_3}\psi \nabla^{i_4}\alpha_F     }\text{d}u^\prime \\
		&+ \int_{u_{\infty}}^u  \frac{\al}{\lvert u^\prime \rvert^2} \scaletwoSuprime{(\al)^i \sum_{i_1 + i_2 +i_3 + i_4 +1 =i} \nabla^{i_1} \psi^{i_2+1} \nabla^{i_3}\tr\chibar \nabla^{i_4}\alpha_F      }\text{d}u^\prime + \\
		&+ \int_{u_{\infty}}^u  \frac{\al}{\lvert u^\prime \rvert^2} \scaletwoSuprime{(\al)^i \sum_{i_1 + i_2 +i_3 + i_4 =i} \nabla^{i_1}\psi^{i_2}\nabla^{i_3}(\psi,\chibarhat, \tildetr)\nabla^{i_4}\alpha_F} \text{d}u^\prime.          \end{align*}For the first term, we have \[  a^{-\frac{1}{2}} \lvert u \rvert^i \lVert \nabla^i \alpha_F \rVert_{L^{2}(S_{u_{\infty},\ubar})} \leq \mathcal{I}^{(0)}(\ubar) \lesssim 1.    \]
		
		\par \noindent 	For the two terms involving the highest number of derivatives,  we have 
		
		\begin{align*}
		&\int_{u_{\infty}}^u \frac{\al}{\lvert u^\prime \rvert^2} \scaletwoSuprime{(\al)^i \nabla^{i+1}(\rho_F,\sigma_F)} \text{d}u^\prime  + \int_{u_{\infty}}^u  \frac{\al}{\lvert u^\prime \rvert^2} \scaletwoSuprime{\sum_{i_1+i_2+i_3+1= i} ( \al)^i \nabla^{i_1} \psi^{i_2+1}\nabla^{i_3+1}\Y} \text{d}u^\prime \\
		\leq&   \left(  \int_{u_{\infty}}^u \frac{a}{\lvert u^\prime \rvert^2} \scaletwoSuprime{(\al)^i \nabla^{i+1}(\rho_F,\sigma_F)}^2 \text{d}u^\prime          \right)^{\frac{1}{2}} \left( \int_{u_{\infty}}^u \frac{1}{ \lvert  u^\prime \rvert^2}\text{d}u^\prime    \right)^{\frac{1}{2}} + \int_{u_{\infty}}^u \frac{1}{\lvert u^\prime \rvert^2} \cdot \frac{\al}{\lvert  u^{\prime}\rvert}  \cdot O^2\hspace{.5mm} \text{d}u^\prime\\
		=& \scaletwoHbaru{(\al)^i \nabla^{i+1}(\rho_F,\sigma_F)}\cdot \frac{1}{ \lvert u \rvert^{\frac{1}{2}}} +\frac{\al}{\lvert u \rvert^2}\cdot O^2 \\
		=& a^{-\frac{1}{2}} \scaletwoHbaru{(\al)^i \nabla^{i+1}(\rho_F,\sigma_F)}\cdot \frac{\al}{\lvert u \rvert^{\frac{1}{2}}} + \frac{\al}{\lvert u \rvert^2}\cdot O^2\\  \leq&  \underline{\mathcal{F}}[\rho_F,\sigma_F]\cdot \frac{\al}{\lvert u \rvert^{\frac{1}{2}}}+1 \lesssim \underline{\mathcal{F}}[\rho_F,\sigma_F]+1.
		\end{align*}For the next two terms, we have \begin{align*} &\int_{u_{\infty}}^u  \frac{\al}{\lvert u^\prime \rvert^2} \ScaletwoSuprime{  \sum_{i_1+i_2+1=i} (\al)^i  \nabla^{i_1+1}\tr\chibar \nabla^{i_2} \alpha_F} \text{d}u^\prime   \\
		&+   \int_{u_{\infty}}^u  \frac{\al}{\lvert u^\prime \rvert^2} \ScaletwoSuprime{\sum_{i_1+i_2= i} ( \al)^i \nabla^{i_1} \chibarhat \nabla^{i_2}\alpha_F} \text{d}u^\prime \\
		\leq& \int_{u_{\infty}}^u \frac{1}{\lvert u^\prime \rvert} \ScaletwoSuprime{ \sum_{i_1+i_2+1=i}(\al)^{i}\nabla^{i_1+1}\left(\frac{a}{\lvert u^\prime \rvert}\big(\tr\chibar+ \frac{2}{\lvert u^\prime \rvert}\big)\right)  \nabla^{i_2}\left( \frac{\alpha_F}{\al}\right)   }\hspace{.5mm} \text{d}u^\prime \\
		&+ \int_{u_{\infty}}^u  \frac{\al}{\lvert u^\prime \rvert} \ScaletwoSuprime{  \sum_{i_1+i_2=i} (\al)^{i}  \nabla^{i_1}\big(\frac{\al}{\lvert u^\prime \rvert}\chibarhat \big)\nabla^{i_2} \left( \frac{\alpha_F}{\al}\right)} \text{d}u^\prime\\
		\leq&  \int_{u_{\infty}}^u  \frac{1}{\lvert u^\prime \rvert}\cdot \frac{O^2}{\lvert u^\prime \rvert}    \hspace{.5mm}\text{d}u^\prime + \int_{u_{\infty}}^u \frac{\al}{\lvert u^\prime \rvert}  \cdot \frac{O^2}{\lvert u^\prime \rvert } \text{d}u^\prime \leq  1. \end{align*}
		
		\par \noindent For the sixth term, notice

		\begin{align*}
		&\int_{u_{\infty}}^u  \frac{\al}{\lvert u^\prime \rvert^2} \scaletwoSuprime{ (\al)^i \sum_{i_1+i_2+i_3+i_4 =i} \nabla^{i_1} \psi^{i_2} \nabla^{i_3}(\psi,\chihat) \nabla^{i_4} \Y } \text{d}u^\prime \\ \leq&
		\int_{u_{\infty}}^u  \frac{a}{\lvert u^\prime \rvert^2} \scaletwoSuprime{ (\al)^{i} \sum_{i_1+i_2+i_3+i_4 =i} \nabla^{i_1} \psi^{i_2} \nabla^{i_3}(\frac{\psi}{\al},\frac{\chihat}{\al}) \nabla^{i_4} \Y } \text{d}u^\prime\\  \leq&
		\int_{u_{\infty}}^u  \frac{a}{\lvert u^\prime \rvert^2}\cdot \frac{\al}{\lvert u^\prime \rvert}O^2\hspace{.5mm} \text{d}u^{\prime} \leq \frac{a^{\frac{3}{2}} O^2}{\lvert u \rvert^2}\leq 1.
		\end{align*}The seventh term can be absorbed schematically under the last term.  For the last two terms, we can write

		\begin{align*} 
		&\int_{u_{\infty}}^u  \frac{\al}{\lvert u^\prime \rvert^2} \scaletwoSuprime{(\al)^i \sum_{i_1 + i_2 +i_3 + i_4 +1 =i} \nabla^{i_1} \psi^{i_2+1} \nabla^{i_3}\tr\chibar \nabla^{i_4}\alpha_F      }\text{d}u^\prime \\
		&+ \int_{u_{\infty}}^u  \frac{\al}{\lvert u^\prime \rvert^2} \scaletwoSuprime{(\al)^i \sum_{i_1 + i_2 +i_3 + i_4 =i} \nabla^{i_1}\psi^{i_2}\nabla^{i_3}(\psi,\chibarhat, \tildetr)\nabla^{i_4}\alpha_F} \text{d}u^\prime \\  =& \int_{u_{\infty}}^u   \ScaletwoSuprime{(\al)^i \sum_{i_1 + i_2 +i_3 + i_4 +1 =i} \nabla^{i_1} \psi^{i_2+1} \nabla^{i_3}\left( \frac{a}{\lvert u^\prime \rvert^2}\tr\chibar \right) \nabla^{i_4}\left(\frac{\alpha_F}{\al}     \right) }\text{d}u^\prime
		\\      &+ \int_{u_{\infty}}^u  \frac{\al}{\lvert u^\prime \rvert} \ScaletwoSuprime{(\al)^i \sum_{i_1 + i_2 +i_3 + i_4 =i} \nabla^{i_1}\psi^{i_2}\nabla^{i_3}\left(\frac{\al}{\lvert u^\prime \rvert}\psi,\frac{\al}{\lvert u^\prime \rvert}\chibarhat,\frac{\al}{\lvert u^\prime \rvert} \tildetr\right))\nabla^{i_4}\left(\frac{\alpha_F}{\al}\right)} \text{d}u^\prime  \\ \leq& \int_{u_{\infty}}^{u} \frac{O^3}{\lvert u^{\prime} \rvert^2} \hspace{.5mm} \text{d}u^{\prime} + \int_{u_{\infty}}^{u} \frac{\al}{\lvert u^{\prime}\rvert}\cdot \frac{O^2}{\lvert u^{\prime} \rvert}\hspace{.5mm} \text{d} u^{\prime}         \leq \frac{O^3}{\lvert u \rvert} + \frac{\al\cdot O^2}{\lvert u \rvert}\leq 1.
		\end{align*}
		\par 
		
	\end{proof}
	
	\begin{proposition}
		Under the assumptions of Theorem \ref{main1} and the bootstrap assumptions \ref{bootstrapbounds}, there holds \[ \sum_{i \leq 10} \scaletwoSu{\aln (\rho_F,\sigma_F)}\lesssim \underline{\mathcal{F}}[\alphabar_F]+1.      \] \label{rhoFsphereestimates}
	\end{proposition}

	\begin{proof}
		We have the following equations along the incoming direction:
		\begin{gather}
		\nabla_3 \rho_F +\tr\chibar \hspace{.5mm} \rho_F = \text{div}\hspace{.5mm} \alphabar_F+ (\eta-\etabar)\cdot \alphabar_F, \\	\nabla_3 \sigma_F +\tr\chibar \hspace{.5mm} \sigma_F = -\text{curl}\hspace{.5mm} \alphabar_F+ (\eta-\etabar)\cdot \Hodge{\alphabar_F}.
		\end{gather}Schematically, we can rewrite the above as \[ \nabla_3(\rho_F,\sigma_F)+\tr\chibar(\rho_F,\sigma_F) = \nabla \alphabar_F + \psi \cdot \Y.    \]Commuting with $i$ angular derivatives, we arrive at \begin{align*}
		&\nabla_3 \nabla^i (\rho_F,\sigma_F) + \frac{i+2}{2}\tr\chibar \nabla^{i}(\rho_F,\sigma_F) \\ =& \nabla^{i+1}\alphabar_F + \sum_{i_1+i_2+i_3+1 =i}\nabla^{i_1}\psi^{i_2+1}\nabla^{i_3+1}\alphabar_F  + \sum_{i_1+i_2+i_3+i_4=i}\nabla^{i_1}\psi^{i_2}\nabla^{i_3}\psi \nabla^{i_4}\Y  + \sum_{i_1+i_2+1=i}\nabla^{i_1+1}\tr\chibar \nabla^{i_2}\alphabar_F \\ &+\sum_{i_1+i_2+i_3+i_4+1 =i} \nabla^{i_1}\psi^{i_2+1}\nabla^{i_3}\tr\chibar \nabla^{i_4}(\rho_F,\sigma_F)+ \sum_{i_1+i_2+i_3+i_4=i}\nabla^{i_1}\psi^{i_2}\nabla^{i_3}(\psi,\chihat,\tildetr)\nabla^{i_4}(\rho_F,\sigma_F)\\ :=& G.
		\end{align*}Applying Proposition \ref{prop37} with $\lambda_1 = i+1$, we have
		
		\[ \lvert u \rvert^{i+1} \twoSu{\nabla^i (\rho_F,\sigma_F)} \lesssim \lvert u_{\infty} \rvert^{i+1}\lVert        \nabla^i (\rho_F,\sigma_F) \rvert_{L^2(S_{u_{\infty},\ubar})}   + \intu \lvert u^\prime \rvert^{i+1} \lVert G \rVert_{L^2(S_{u^\prime,\ubar})}\duprime .     \]We have \[ s_2\left(\nabla^{i}(\rho_F,\sigma_F)\right) = \frac{i+1}{2}, \hspace{.5mm} s_2(G) = \frac{i+3}{2}.   \]Therefore, \[ \scaletwoSu{\aln (\rho_F,\sigma_F)} = (\al)^i \cdot a^{-\frac{i+1}{2}} \lvert u \rvert^{i+1} \twoSu{\nabla^i (\rho_F,\sigma_F)},       \]so that \[    \lvert u \rvert^{i+1} \twoSu{\nabla^i (\rho_F,\sigma_F)}   = \al  \scaletwoSu{\aln (\rho_F,\sigma_F)},   \]as well as \[ \scaletwoSu{(\al)^i G} = (\al)^i a^{-\frac{i+3}{2}} \lvert u \rvert^{i+3} \twoSu{G},     \]whence we get \[  \lvert u \rvert^{i+1} \twoSu{(\al)^i G} = \frac{a^{\frac{3}{2}}}{\lvert u \rvert^2} \scaletwoSu{(\al)^i G}.      \]\par \noindent Passing therefore to scale-invariant norms we have \begin{align*}
		&\scaletwoSu{\aln (\rho_F,\sigma_F)}\\ \lesssim& \lVert \aln (\rho_F,\sigma_F) \rVert_{\mathcal{L}^{2}_{(sc)}(S_{u_{\infty},\ubar})} + \intu \frac{a}{\lvert u^\prime \rvert^2}\scaletwoSuprime{(\al)^i G} \duprime \\ \lesssim& \lVert \aln (\rho_F,\sigma_F) \rVert_{\mathcal{L}^{2}_{(sc)}(S_{u_{\infty},\ubar})} + \intu \frac{a}{\lvert u^\prime \rvert^2} \scaletwoSuprime{(\al)^i \nabla^{i+1}\alphabar_F} \duprime \\ &+ \intu \frac{a}{\lvert u^\prime \rvert^2}\ScaletwoSuprime{ (\al)^i \sum_{i_1+i_2+i_3+1=i} \nabla^{i_1}\psi^{i_2+1}\nabla^{i_3+1}\alphabar_F} \duprime \\ &+ \intu \frac{a}{\lvert u^\prime \rvert^2}\ScaletwoSuprime{ (\al)^i \sum_{i_1+i_2+i_3+i_4=i} \nabla^{i_1}\psi^{i_2}\nabla^{i_3}\psi \nabla^{i_4}\Y } \duprime \\ &+ \intu \frac{a}{\lvert u^\prime \rvert^2}\ScaletwoSuprime{ (\al)^i \sum_{i_1+i_2+1=i} \nabla^{i_1+1}\tr\chibar \nabla^{i_2}\alphabar_F} \duprime \\  &+ \intu \frac{a}{\lvert u^\prime \rvert^2}\ScaletwoSuprime{ (\al)^i \sum_{i_1+i_2+i_3+i_4+1=i} \nabla^{i_1}\psi^{i_2+1} \nabla^{i_3}\tr\chibar \nabla^{i_4}(\rho_F,\sigma_F)} \duprime \\ &+ \intu \frac{a}{\lvert u^\prime \rvert^2}\ScaletwoSuprime{ (\al)^i \sum_{i_1+i_2+i_3+i_4+1=i} \nabla^{i_1}\psi^{i_2} \nabla^{i_3}(\psi,\chihat,\tildetr) \nabla^{i_4}(\rho_F,\sigma_F)} \duprime\\ :=& J_1 + \dots +J_7.
		\end{align*}We treat $J_1,\dots, J_7$ on a term-by-term basis.
		
		\begin{itemize}
			\item The initial data term $J_1$ is bounded by $\mathcal{I}^{(0)}(\ubar) \lesssim 1$.
			
			\item We have \[ J_2 \lesssim \frac{a}{\lvert u \rvert} \scaletwoHbaru{(\al)^i \nabla^{i+1}\alphabar_F} \lesssim \frac{a}{\lvert u \rvert}\underline{\mathcal{F}}[\alphabar_F] \lesssim  \underline{\mathcal{F}}\hspace{.5mm}[\alphabar_F] .   \]
			
			\item For $J_3$ we have  \[ J_3    \lesssim \frac{\al \cdot O^2}{\lvert u \rvert} \lesssim 1.             \] \item For $J_4$ we can similarly bound the term by $1$. \item For $J_5$, there holds \[ J_5 = \intu \frac{1}{\lvert u^\prime \rvert} \ScaletwoSuprime{(\al)^i \sum_{i_1+i_2+1=i} \nabla^{i_1+1}\left(\frac{a}{\lvert u^\prime \rvert}\tildetr\right) \nabla^{i_2} \alphabar_F} \duprime  \lesssim \frac{O^2}{\lvert u \rvert}\lesssim 1.       \]Note that we have crucially used the fact that in $J_5$ there exists at least one derivative on $\tr\chi$, allowing us to rewrite the expression in terms of $\tildetr$, since $\nabla \tildetr = \nabla \tr\chibar$.
			
			\item There holds \[  J_6 \lesssim \frac{O^3}{a}\lesssim 1.    \] 
			\item There holds \[J_7 \lesssim \frac{\al\cdot O^2}{\lvert u \rvert} \lesssim 1.\] 
		\end{itemize}\par \noindent Combining these estimates together, we arrive at the desired result.
	\end{proof}

	\begin{proposition}\label{othermaxwellprop}
		Under the assumptions of Theorem \ref{main1} and the bootstrap assumption \eqref{bootstrapbounds}, we have 
		\[  \sum_{i \leq 10} \scaletwoSu{(\al \nabla)^i \alphabar_F} \lesssim \mathcal{F}[\rho_F,\sigma_F]+ \underline{\mathcal{F}}[\rho_F,\sigma_F] +1.      \]
	\end{proposition}
	\begin{proof}
		For $\Y \in \begin{Bmatrix}
		\rho_F, \sigma_F ,\alphabar_F 
		\end{Bmatrix}$ we have the following schematic equation \[ \nabla_4 \alphabar_F = \nabla (\rho_F,\sigma_F) + (\chibarhat, \psi)\cdot(\Y, \alpha_F).   \]Commuting this with $i$ angular derivatives, we get \begin{align*}
		&\nabla_4 \nabla^i \alphabar_F\\ =& \nabla^{i+1}(\rho_F, \sigma_F) + \sum_{i_1 + i_2 +i_3 +1= i} \nabla^{i_1}\psi^{i_2+1} \nabla^{i_3}(\rho_F,\sigma_F) \\
		&+ \sum_{i_1 + i_2 +i_3 + i_4 =i} \nabla^{i_1} \psi^{i_2} \nabla^{i_3} (\psi, \chihat) \nabla^{i_4} \alphabar_F + \sum_{i_1 + i_2 +i_3 + i_4 =i} \nabla^{i_1} \psi^{i_2} \nabla^{i_3} (\psi, \chibarhat) \nabla^{i_4} (\Y, \alpha_F). 
		\end{align*}By multiplying with $(\al)^i $ on both sides and passing to scale-invariant norms, we have
		
		\begin{align*}
		&\scaletwoSu{(\al \nabla)^i \alphabar_F}\\ \leq& \int_{0}^{\ubar} \scaletwoSuubarprime{(\al)^i \nabla^{i+1}(\rho_F,\sigma_F)} \hspace{.5mm} \text{d}\ubar^{\prime} \\
		&+ \int_{0}^{\ubar} \scaletwoSuubarprime{(\al)^i \sum_{i_1 + i_2 +i_3 +1 =i} \nabla^{i_1} \psi^{i_2+1}\nabla^{i_3+1}(\rho_F,\sigma_F) }\hspace{.5mm} \text{d}\ubar^{\prime} \\
		&+ \int_0^{\ubar} \scaletwoSuubarprime{(\al)^i \sum_{i_1 + i_2 +i_3 + i_4 =i} \nabla^{i_1} \psi^{i_2} \nabla^{i_3} (\psi, \chihat) \nabla^{i_4} \alphabar_F} \hspace{.5mm} \text{d}\ubar^{\prime} \\
		&+ \int_0^{\ubar} \scaletwoSuubarprime{(\al)^i  \sum_{i_1 + i_2 +i_3 + i_4 =i} \nabla^{i_1} \psi^{i_2} \nabla^{i_3} (\psi, \chibarhat) \nabla^{i_4} (\Y, \alpha_F) } \hspace{.5mm} \text{d}\ubar^{\prime} \\
		\leq&  \mathcal{F}[\rho_F,\sigma_F] 
		+ \int_{0}^{\ubar} \scaletwoSuubarprime{(\al)^{i} \sum_{i_1+i_2+i_3+1=i} \nabla^{i_1} \psi^{i_2+1} \nabla^{i_3+1}(\rho_F,\sigma_F)} \hspace{.5mm} \text{d}\ubar^{\prime} \\
		&+ \int_0^{\ubar} \scaletwoSuubarprime{(\al)^{i+1} \sum_{i_1 + i_2 +i_3 + i_4 =i} \nabla^{i_1} \psi^{i_2} \nabla^{i_3} (\frac{\psi}{\al}, \frac{\chihat}{\al}) \nabla^{i_4} \alphabar_F} \hspace{.5mm} \text{d}\ubar^{\prime} \\
		&+ \int_0^{\ubar} a^{-\frac{1}{2}} \lvert u \rvert  \scaletwoSuubarprime{(\al)^{i+1}  \sum_{i_1 + i_2 +i_3 + i_4+1 =i} \nabla^{i_1} \psi^{i_2+1} \nabla^{i_3} (\frac{\al}{\lvert u \rvert}\psi, \frac{\al}{\lvert u \rvert} \chibarhat) \nabla^{i_4} (\frac{\Y}{\al}, \frac{\alpha_F}{\al}) } \hspace{.5mm} \text{d}\ubar^{\prime} \\
		&+ \int_0^{\ubar} a^{-\frac{1}{2}} \lvert u \rvert  \scaletwoSuubarprime{(\al)^{i+1}  \sum_{i_1+i_2=i}\nabla^{i_1} (\frac{\al}{\lvert u \rvert}\psi, \frac{\al}{\lvert u \rvert} \chibarhat) \nabla^{i_2} (\frac{\Y}{\al}, \frac{\alpha_F}{\al}) }\hspace{.5mm} \text{d}\ubar^{\prime}\\ \leq&\mathcal{F}[\rho_F,\sigma_F] + O[\chibarhat]O[\alpha_F]+1 \lesssim \mathcal{F}[\rho_F,\sigma_F] + \underline{\mathcal{F}}[\rho_F,\sigma_F]+1.
		\end{align*}In the last inequality we have used the refined estimates on $\chibarhat$ and $\alpha_F$ from Propositions \ref{chibarhatproposition} and \ref{alphaFproposition} respectively.
	\end{proof}
	
	\par \noindent

	\section{$L^2(S_{u,\ubar})$-estimates for curvature}
	
	\begin{proposition}\label{alphaproposition}
		Under the assumptions of Theorem \ref{main1} and the bootstrap assumption \eqref{bootstrapbounds}, we have  
		\[\sum_{i \leq 9} \frac{1}{\al} \scaletwoSu{(\al \nabla)^i \alpha} \lesssim 1.    \]
	\end{proposition}
	
	\begin{proof}
			Reading off equation \eqref{renalphaeq}, we have
			
			\begin{align*}   \nabla_3 \alpha + \frac{1}{2}\tr\chibar \alpha =& \nabla \tbeta + \alpha_F \cdot \nabla \Y + \Y \cdot (\nabla \alpha_F,\nabla \Y)  \\ &+ (\psi,\chihat)\cdot \Psi + \psi\cdot \alpha + (\psi,\chibarhat,\tr\chibar,\chihat)\cdot(\alpha_F,\Y)\cdot(\alpha_F,\Y).  \end{align*}
			
			\par \noindent Commuting with $i$ angular derivatives, we obtain
			
			\begin{align*}
			&\nabla_3 \nabla^i \alpha + \frac{i+1}{2}\tr\chibar \nabla^i \alpha  \\
			=& \nabla^{i+1} \tbeta + \sum_{i_1 + i_2 +i_3 +1 =i} \nabla^{i_1} \psi^{i_2+1} \nabla^{i_3}\tbeta + \sum_{i_1+i_2+i_3+i_4=i} \nabla^{i_1}\psi^{i_2}\nabla^{i_3}\alpha_F \nabla^{i_4+1}\Y   \\&+  \sum_{i_1+i_2+i_3+i_4=i} \nabla^{i_1}\psi^{i_2}\nabla^{i_3}\Y \nabla^{i_4+1}(\alpha_F,\Y) + \sum_{ i_1+i_2+i_3+i_4=i} \nabla^{i_1} \psi^{i_2} \nabla^{i_3} (\psi,\chihat)\nabla^{i_4} \Psi              \\
			&+ \sum_{i_1+i_2+i_3+i_4+i_5=i} \nabla^{i_1} \psi^{i_2} \nabla^{i_3}(\psi,\chibarhat,\tr\chibar,\chihat) \nabla^{i_4}(\alpha_F,\Y) \nabla^{i_5}(\alpha_F,Y) \\
			&+ \sum_{i_1+i_2+i_3+i_4 = i} \nabla^{i_1} \psi^{i_2} \nabla^{i_3} (\psi,\chibarhat,\tildetr)\nabla^{i_4}\alpha + \sum_{i_1+i_2+i_3+i_4+1=i} \nabla^{i_1}\psi^{i_2+1}\nabla^{i_3}\tr\chibar\nabla^{i_4}\alpha .
			\end{align*} 
			
			\par \noindent 	Denote the above as \[ \nabla_3 \nabla^i \alpha + \frac{i+1}{2}\tr\chibar \nabla^i \alpha =G.       \]Using the definition of the $\mathcal{L}^2_{(sc)}(u,u)$-norms we have 
			
			\[ \scaletwoSu{\nabla^i \alpha} = a^{-\frac{i}{2}} \lvert u \rvert^i  \lVert \nabla^i \alpha \rVert_{L^2(S_{u,\ubar})} ,    \hspace{2mm}   \scaletwoSu{G} = a^{-\frac{i+2}{2}} \lvert u \rvert^{i+2} \lVert G \rVert_{L^{2}(S_{u,\ubar})} ,      \] which translates to  \[    a^{-\frac{1}{2}} \lvert u \rvert^i  \lVert \nabla^i \alpha \rVert_{L^2(S_{u,\ubar})}   = a^{-\frac{1}{2}}  \scaletwoSu{(\al \nabla)^i \alpha}, \hspace{2mm}   a^{-\frac{1}{2}} \lvert u \rvert^i  \lVert G \rVert_{L^2(S_{u,\ubar})} =\frac{\al}{\lvert u \rvert^2}    \scaletwoSu{(\al)^i G}      .            \]Hence we have 
			
			\begin{align*}
			&a^{-\frac{1}{2}} \scaletwoSu{(\al \nabla)^i \alpha} \\
			\leq&  a^{-\frac{1}{2}} \lVert (\al \nabla)^i \alpha \rVert_{L^{2}_{(sc)}(S_{u_{\infty},\ubar})} + \int_{u_{\infty}}^u \frac{1}{\lvert u^\prime \rvert^2} \scaletwoSuprime{(\al\nabla)^{i+1} \tbeta} \text{d}u^\prime \\
			&+ \int_{u_{\infty}}^u \frac{\al}{\lvert u^\prime \rvert^2} \scaletwoSuprime{ \sum_{i_1 + i_2 +i_3 +1 =i} (\al)^i \nabla^{i_1} \psi^{i_2 +1} \nabla^{i_3+1} \tbeta } \hspace{.5mm} \text{d} u^\prime \\
			&+ \int_{u_{\infty}}^u \frac{\al}{\lvert u^\prime \rvert^2}  \scaletwoSuprime{(\al)^i \sum_{i_1 + i_2 +i_3 +i_4= i} \nabla^{i_1} \psi^{i_2} \nabla^{i_3} (\alpha_F,\Y)  \nabla^{i_4+1} \Y                      }\hspace{.5mm} \text{d}u^{\prime} \\
			&+ \int_{u_{\infty}}^u \frac{\al}{\lvert u^\prime \rvert^2}  \scaletwoSuprime{(\al)^i \sum_{i_1 + i_2 +i_3 +i_4= i} \nabla^{i_1} \psi^{i_2} \nabla^{i_3} \Y \nabla^{i_4+1}   (\alpha_F,\Y)                     }\hspace{.5mm} \text{d}u^{\prime} \\ &+ \int_{u_{\infty}}^u \frac{\al}{\lvert u^\prime \rvert^2} \scaletwoSuprime{(\al)^i \sum_{i_1+i_2+i_3+i_4=i} \nabla^{i_1} \psi^{i_2} \nabla^{i_3} (\psi,\chihat) \nabla^{i_4} \Psi} \hspace{.5mm} \text{d}u^\prime \\
			&+  \int_{u_{\infty}}^u \frac{\al}{\lvert u^\prime \rvert^2}  \scaletwoSuprime{(\al)^i \sum_{i_1 + i_2 +i_3 +i_4+i_5= i} \nabla^{i_1} \psi^{i_2} \nabla^{i_3}(\psi, \chibarhat, \tr \chibar, \chihat)  \nabla^{i_4} (\alpha_F,\Y) \nabla^{i_5} (\alpha_F ,\Y)                    }\hspace{.5mm} \text{d}u^{\prime}  \\
			&+ \int_{u_{\infty}}^u \frac{\al}{\lvert u^\prime \rvert^2}  \scaletwoSuprime{(\al)^i \sum_{i_1+i_2+i_3+i_4=i} \nabla^{i_1}\psi^{i_2}\nabla^{i_3}(\psi,\chibarhat,\tildetr)\nabla^{i_4}\alpha}\hspace{.5mm} \text{d}u^{\prime} \\  &+ \int_{u_{\infty}}^u \frac{\al}{\lvert u^\prime \rvert^2}  \scaletwoSuprime{(\al)^i \sum_{i_1+i_2+i_3+i_4 +1=i} \nabla^{i_1}\psi^{i_2+1}\nabla^{i_3}\tr\chibar \nabla^{i_4}\alpha}\hspace{.5mm} \text{d}u^{\prime} \\ :=& T_1+T_2+\dots+T_9.
			\end{align*}\par \noindent The first term can be bounded by the initial data, since \[ T_1 = a^{-\frac{1}{2}} \lVert (\al \nabla)^i \alpha \rVert_{L^{2}_{(sc)}(S_{u_{\infty},\ubar})} \leq \mathcal{I}^{(0)}(\ubar) \lesssim 1.  \] For the terms involving $\tbeta$, we have   
		
		\begin{align*}
		T_2+T_3 =&  \int_{u_{\infty}}^u \frac{1}{\lvert u^\prime \rvert^2} \scaletwoSuprime{(\al\nabla)^{i+1} \tbeta} \text{d}u^\prime \\
		&+ \int_{u_{\infty}}^u \frac{\al}{\lvert u^\prime \rvert^2} \scaletwoSuprime{ \sum_{i_1 + i_2 +i_3 +1 =i} (\al)^i \nabla^{i_1} \psi^{i_2 +1} \nabla^{i_3+1} \tbeta } \hspace{.5mm} \text{d} u^\prime \\ \leq&  \left( \int_{u_{\infty}}^u \frac{a}{\lvert u^\prime \rvert^2} \scaletwoSuprime{(\al\nabla)^{i+1} \tbeta}^2 \text{d}u^\prime \right)^{\frac{1}{2}} \left( \int_{u_{\infty}}^u \frac{1}{a \lvert u^{\prime}\rvert^2}\hspace{.5mm}\text{d} u^{\prime} \right)^{\frac{1}{2}} +  \int_{u_{\infty}}^u \frac{\al}{\lvert u^{\prime} \rvert^2} \cdot \frac{O^2}{\lvert u^\prime \rvert} \hspace{.5mm} \text{d} u^\prime \\ \leq& a^{-\frac{1}{2}}\scaletwoHbaru{(\al \nabla)^{i+1}\tbeta}\cdot \frac{1}{\lvert u \rvert^{\frac{1}{2}}} + \frac{\al\cdot O^2}{\lvert u \rvert^2 }\lesssim \underline{\mathcal{R}}[\tbeta]\cdot \frac{1}{\lvert u \rvert^{\frac{1}{2}}} +\frac{\al\cdot O^2}{\lvert u \rvert^2 }\lesssim 1 .
		\end{align*}

\par \noindent  Notice that the curvature term actually vanishes. For the next two terms, we need to treat the  cases where all the weight falls on $i_4$ separately. Look, first, at 
		
		\begin{equation*}
		T_4 = \int_{u_{\infty}}^u \frac{\al}{\lvert u^\prime \rvert^2}  \scaletwoSuprime{(\al)^i \sum_{i_1 + i_2 +i_3 +i_4= i} \nabla^{i_1} \psi^{i_2} \nabla^{i_3} (\alpha_F,\Y)  \nabla^{i_4+1} \Y                      }\hspace{.5mm} \text{d}u^{\prime}.
		\end{equation*}If $i_4 =i$, then we can bound the term $(\alpha_F,\Y)$ below in $\mathcal{L}^{\infty}_{(sc)}$ to get
		
		\begin{align*} T_4 =& \int_{u_{\infty}}^u  \frac{1}{\lvert u^{\prime} \rvert^2} \scaletwoSuprime{(\alpha_F,\Y)\cdot (\al \nabla)^{i+1} \Y} \hspace{.5mm} \text{d} u^\prime \\ \leq& \int_{u_{\infty}}^u \frac{\al}{\lvert u^{\prime} \rvert^2} \cdot \frac{O}{\lvert u^\prime \rvert}\cdot \scaletwoSuprime{(\al\nabla)^{i+1}\Y }\hspace{.5mm}\text{d}u^{\prime}  \\\leq&  \intu \frac{\al}{\lvert u^\prime \rvert^2}\cdot \frac{O}{\lvert u^\prime \rvert}\cdot O \duprime \lesssim \frac{\al \cdot O^2}{\lvert u \rvert^2} \lesssim 1. \end{align*} If $i_4 < i \leq 9$ we distinguish two cases:
		
		\begin{itemize}
			\item There holds $i_4+1 \leq 6$. We then write \[(\al)^i \sum_{i_1 + i_2 +i_3 + i_4 =i} \nabla^{i_1} \psi^{i_2} \nabla^{i_3}(\alpha_F, \Y) \nabla^{i_4+1} \Y = (\al)^{i+1} \sum_{i_1 + i_2 +i_3 + i_4 =i} \nabla^{i_1} \psi^{i_2} \nabla^{i_3}(\frac{\alpha_F}{\al}, \frac{\Y}{\al}) \nabla^{i_4+1} \Y \]and bound $(\al \nabla)^{i_4+1} \Y$ in $\mathcal{L}^\infty_{(sc)}$. We have 
			
			\begin{align*}
			&\int_{u_{\infty}}^u \frac{\al}{\lvert u^\prime \rvert^2}  \scaletwoSuprime{(\al)^i \sum_{i_1 + i_2 +i_3 +i_4= i} \nabla^{i_1} \psi^{i_2} \nabla^{i_3} (\alpha_F,\Y)  \nabla^{i_4+1} \Y                      }\hspace{.5mm} \text{d}u^{\prime} \\ 
			\leq& \int_{u_{\infty}}^u \frac{\al}{\lvert u^\prime \rvert^2} \cdot \frac{1}{\lvert u^{\prime} \rvert} \scaleinfinitySuprime{(\al \nabla)^{i_4+1} \Y} \scaletwoSuprime{\sum_{i_1+i_2+i_3\leq 9} (\al)^{i_1+i_2+i_3} \nabla^{i_1} \psi^{i_2} \nabla^{i_3} \Y    } \hspace{.5mm} \text{d}u^{\prime}\\ 
			\leq& \frac{\al \cdot O^2}{\lvert u \rvert} \leq \frac{O^2}{\al}\leq 1. \end{align*} 
			
			\item There holds $7 \leq i_4+1 \leq 9$. We then bound $(\al \nabla)^{i_4+1} \Y$ in $\mathcal{L}^2_{(sc)}$ and the rest of the terms in $\mathcal{L}^\infty_{(sc)}$. This gives
			
			\begin{align*}
			&\int_{u_{\infty}}^u \frac{\al}{\lvert u^\prime \rvert^2}  \scaletwoSuprime{(\al)^i \sum_{i_1 + i_2 +i_3 +i_4= i} \nabla^{i_1} \psi^{i_2} \nabla^{i_3} (\alpha_F,\Y)  \nabla^{i_4+1} \Y                      }\hspace{.5mm} \text{d}u^{\prime} \\
			\leq& \int_{u_{\infty}}^u \frac{\al}{\lvert u^\prime \rvert^2} \cdot \frac{1}{\lvert u^{\prime} \rvert} \scaletwoSuprime{(\al \nabla)^{i_4+1} \Y} \scaleinfinitySuprime{\sum_{i_1+i_2+i_3\leq 9} (\al)^{i_1+i_2+i_3} \nabla^{i_1} \psi^{i_2} \nabla^{i_3} \Y    } \hspace{.5mm} \text{d}u^\prime \\
			\leq& \int_{u_{\infty}}^u \frac{\al}{\lvert u^\prime \rvert^2} \cdot \frac{1}{\lvert u\rvert^{\prime}} \cdot O \cdot \frac{(\al)^{i_1+i_2+i_3} O^{i_1+i_2+i_3}}{\lvert u^\prime \rvert^{i_1+i_2+i_3-1}}\hspace{.5mm} \text{d}u^\prime \leq \frac{\al O}{\lvert u^\prime \rvert}\leq 1.
			\end{align*}
		\end{itemize}We move on to $T_5$. If $i_4=i$, we have  \begin{align*} &\int_{u_{\infty}}^u \frac{\al}{\lvert u^\prime \rvert^2}  \scaletwoSuprime{(\al)^i \sum_{i_1 + i_2 +i_3 +i_4= i} \Y \hspace{.5mm} \nabla^{i+1} (\alpha_F  ,\Y)       }\hspace{.5mm} \text{d}u^{\prime}\\
		\leq& \int_{u_{\infty}}^u \frac{\al}{\lvert u^\prime \rvert^2}\cdot \frac{O}{\lvert u^\prime \rvert} \cdot \frac{1}{\al} \scaletwoSuprime{(\al \nabla)^{i+1} \alpha_F} \text{d}u^\prime \leq \frac{\al \cdot O^2}{\lvert u \rvert^2}
		\leq 1.   \end{align*} 
		
		\par \noindent	If $i_4 < i \leq 9$ we again distinguish two cases:
		
		\begin{itemize}
			\item There holds $i_4+1 \leq 6$. We then write \[(\al)^i \sum_{i_1 + i_2 +i_3 + i_4 =i} \nabla^{i_1} \psi^{i_2} \nabla^{i_3}\Y \nabla^{i_4+1} \alpha_F = (\al)^{i+1} \sum_{i_1 + i_2 +i_3 + i_4 =i} \nabla^{i_1} \psi^{i_2} \nabla^{i_3} \Y \nabla^{i_4+1} \left( \frac{\alpha_F}{\al}\right) \]and bound $\nabla^{i_4+1} \big( \frac{\alpha_F}{\al}\big)$ in $\mathcal{L}^\infty_{(sc)}$. We have 
			\begin{align*} &\int_{u_{\infty}}^u \frac{\al}{\lvert u^\prime \rvert^2}  \scaletwoSuprime{(\al)^i \sum_{i_1 + i_2 +i_3 +i_4= i} \nabla^{i_1} \psi^{i_2} \nabla^{i_3}\Y  \nabla^{i_4+1} \alpha_F                      }\hspace{.5mm} \text{d}u^{\prime}\\
			\leq& \int_{u_{\infty}}^u \frac{\al}{\lvert u^\prime \rvert^2} \cdot \frac{1}{\lvert u^\prime \rvert} \ScaleinfinitySuprime{(\al \nabla)^{i_4+1} \left( \frac{\alpha_F}{\al}    \right) } \scaletwoSuprime{\sum_{i_1 + i_2 +i_3 \leq 9} (\al)^{i_1+i_2+i_3} \nabla^{i_1}\psi^{i_2} \nabla^{i_3} \Y    } \hspace{.5mm} \text{d}u^\prime \\
			\leq& \int_{u_{\infty}}^u \frac{\al}{\lvert u^\prime \rvert^2} \cdot \frac{1}{\lvert u^\prime \rvert}\cdot O^2 \text{d}u^\prime  = \frac{\al O^2}{\lvert u \rvert} \leq 1.\end{align*}
			
			\item There holds $7 \leq i_4 +1 \leq 9$. We then bound $\nabla^{i_4+1}\big( \frac{\alpha_F}{\al} \big)$ in $\mathcal{L}^2_{(sc)}$ and the rest of the terms in $\mathcal{L}^{\infty}_{(sc)}$. We then have \begin{align*}&\int_{u_{\infty}}^u \frac{\al}{\lvert u^\prime \rvert^2}  \ScaletwoSuprime{(\al)^{i+1} \sum_{i_1 + i_2 +i_3 +i_4= i} \nabla^{i_1} \psi^{i_2} \nabla^{i_3}\Y  \nabla^{i_4+1} (\frac{\alpha_F}{\al} )                  }\hspace{.5mm} \text{d}u^{\prime}\\
			\leq& \int_{u_{\infty}}^u \frac{\al}{\lvert u^\prime \rvert^2} \cdot \frac{1}{\lvert u^\prime \rvert} \ScaletwoSuprime{(\al \nabla)^{i_4+1} \left( \frac{\alpha_F}{\al}    \right) } \scaleinfinitySuprime{\sum_{i_1 + i_2 +i_3 \leq 2} (\al)^{i_1+i_2+i_3} \nabla^{i_1}\psi^{i_2} \nabla^{i_3} \Y    } \hspace{.5mm} \text{d}u^\prime \\
			\leq& \int_{u_{\infty}}^u \frac{\al}{\lvert u^\prime \rvert^2} \cdot \frac{O}{\lvert u^\prime \rvert} \cdot \frac{(\al)^{i_1+i_2+i_3} O^{i_1+i_2+i_3}}{\lvert u^\prime \rvert^{i_1+i_2+i_3-1}} \text{d}u^\prime   \leq 1.\end{align*}
		\end{itemize}
		For $T_6$, we have \begin{align*}  &\int_{u_{\infty}}^u \frac{\al}{\lvert u^\prime \rvert^2}  \scaletwoSuprime{(\al)^i \sum_{i_1 + i_2 +i_3 +i_4= i} \nabla^{i_1} \psi^{i_2} \nabla^{i_3}(\psi, \chihat)  \nabla^{i_4} \Psi          }\hspace{.5mm} \text{d}u^{\prime}\\
		=& \int_{u_{\infty}}^u \frac{\al}{\lvert u^\prime \rvert^2 } \ScaletwoSuprime{ (\al)^{i+1} \sum_{i_1 + i_2 +i_3 +i_4= i} \nabla^{i_1} \psi^{i_2} \nabla^{i_3}\left(\frac{\psi}{\al}, \frac{\chihat}{\al }\right)  \nabla^{i_4} \Psi                      }\hspace{.5mm} \text{d}u^{\prime}  \\
		\leq& \int_{u_{\infty}}^u \frac{\al}{\lvert u^\prime \rvert^2 } \cdot \frac{\al \cdot O^2}{\lvert u^\prime \rvert} \hspace{.5mm} \text{d} u^{\prime} = \frac{a\cdot O^2}{\lvert u \rvert^2} \leq 1 . \end{align*}
		For the term $T_7$, which contains a triple anomaly, we have \begin{align*} &\int_{u_{\infty}}^u \frac{\al}{\lvert u^\prime \rvert^2}  \scaletwoSuprime{(\al)^i \sum_{i_1 + i_2 +i_3 +i_4+i_5= i} \nabla^{i_1} \psi^{i_2} \nabla^{i_3}(\psi, \chibarhat, \tr \chibar, \chihat)  \nabla^{i_4} (\alpha_F, \Y) \nabla^{i_5} (\alpha_F, \Y)                     }\hspace{.5mm} \text{d}u^{\prime}\\
		=&  \int_{u_{\infty}}^u \frac{\text{d}u^{\prime}}{\al}  \ScaletwoSuprime{ (\al)^{i+2} \sum_{i_1 + i_2 +i_3 +i_4+i_5= i} \nabla^{i_1} \psi^{i_2} \nabla^{i_3}\left(\frac{a}{\lvert u^\prime \rvert^2}(\psi,\chibarhat,\tr\chibar,\chihat) \right) \nabla^{i_4} \left(\frac{(\alpha_F, \Y)}{\al}\right) \nabla^{i_5} \left(\frac{(\alpha_F,\Y)}{\al}\right)} \\
		=& \int_{u_{\infty}}^u \al \scaletwoSuprime{ (\al)^{i} \sum_{i_1+i_2+i_3+i_4=i}\nabla^{i_1} \psi^{i_2+1} \nabla^{i_3}\Y \nabla^{ i_4 } \Y } \duprime \\ 
		\leq& \int_{u_{\infty}}^u \al \cdot \frac{ O^3}{\lvert u^\prime \rvert^2} \duprime \leq \frac{\al \cdot O^3}{\lvert u \rvert} \leq 1.                \end{align*}   For $T_8$, there holds
		
		\begin{align*}
		&\int_{u_{\infty}}^u \frac{\al}{\lvert u^\prime \rvert^2}  \scaletwoSuprime{(\al)^i \sum_{i_1 + i_2 +i_3 +i_4= i} \nabla^{i_1} \psi^{i_2} \nabla^{i_3}(\psi,\chibarhat,\tildetr)  \nabla^{i_4} \alpha                  }\hspace{.5mm} \text{d}u^{\prime}\\ =& \int_{u_{\infty}}^u \frac{1}{\lvert u^{\prime} \rvert} \ScaletwoSuprime{(\al)^{i+1} \sum_{i_1+i_2+i_3+i_4=i}\nabla^{i_1}\psi^{i_2}\nabla^{i_3}\left( \frac{\al}{\lvert u^{\prime}\rvert}\psi,\frac{\al}{\lvert u^{\prime}\rvert}\chibarhat, \frac{\al}{\lvert u^{\prime}\rvert}\tildetr \right) \nabla^{i_4} \left( \frac{\alpha}{\al}\right) } \hspace{.5mm} \text{d}u^{\prime} \\ \leq&\int_{u_{\infty}}^u \frac{1}{\lvert u^\prime \rvert} \cdot \frac{\al \cdot O^2}{\lvert u^\prime \rvert} \hspace{.5mm} \text{d}u^\prime = \frac{\al \cdot O^2}{\lvert u \rvert}
		\leq 1 . 
		\end{align*} Finally, we can bound $T_9$ as follows: 
		
		\begin{align*}
		&\int_{u_{\infty}}^u \frac{\al}{\lvert u^\prime \rvert^2}  \scaletwoSuprime{(\al)^i \sum_{i_1+i_2+i_3+i_4 +1=i} \nabla^{i_1}\psi^{i_2+1}\nabla^{i_3}\tr\chibar \nabla^{i_4}\alpha}\hspace{.5mm} \text{d}u^{\prime} \\ =& \int_{u_{\infty}}^u \frac{1}{a }\ScaletwoSuprime{(\al)^{i+2} \sum_{i_1+i_2+i_3+i_4 +1=i} \nabla^{i_1}\psi^{i_2+1}\nabla^{i_3}\left( \frac{a}{\lvert u^\prime \rvert^2}\tr\chibar \right) \nabla^{i_4}\left( \frac{\alpha}{\al}\right)}\hspace{.5mm} \text{d}u^{\prime} \\ \leq& \int_{u_{\infty}}^u \frac{1}{a}\cdot \frac{a\cdot O^3}{\lvert u^\prime \rvert^2} \hspace{.5mm} \text{d}u^{\prime} \leq 1.
		\end{align*}In the last inequality above we have used Proposition \ref{usefulstatements}. Putting all the estimates together, the result follows.

	\end{proof}
	\par \noindent
	
	\par \noindent We move on to estimates for the curvature components $\tbeta, \rho, \sigma, \tbetabar, \alphabar$.
	
	\begin{proposition}\label{L2curvature2}
		Let $\Psi \in \begin{Bmatrix} \tbeta, \rho ,\sigma, \tbetabar, \alphabar \end{Bmatrix}$. Under the assumptions of Theorem \ref{main1} and \eqref{bootstrapbounds}, we have
		\[ \sum_{i \leq 9} \hspace{.5mm}\scaletwoSu{(\al \nabla)^i \Psi} \leq \mathcal{R}[\alpha] + \underline{\mathcal{F}}[\rho_F,\sigma_F]  +1.      \]
	\end{proposition}
	
	\begin{proof} The terms $\Psi$ satisfy the following schematic equations:
		
			\begin{equation}
			\nabla_4 \Psi = \nabla( \Psi, \alpha) + (\psi,\chibarhat) \cdot (\Psi, \alpha) + (\alpha_F, \Y ) \cdot \nabla(\alpha_F, \Y)+ (\psi,\chibarhat,\tr\chibar,
			\chihat) \cdot (\alpha_F, \Y) \cdot (\alpha_F, \Y). \label{schematicPsi}
			\end{equation} Commuting \eqref{schematicPsi} with $i$ angular derivatives and using Proposition \ref{commutationformulaeprop}, we have
		
	\begin{align*}
			\nabla_4 \nabla^i \Psi =& \nabla^{i+1}(\Psi,\alpha) + \sum_{i_1+i_2+i_3+1=i} \nabla^{i_1} \psi^{i_2+1} \nabla^{i_3+1} (\Psi, \alpha) + \sum_{i_1+i_2+i_3+i_4=i} \nabla^{i_1} \psi^{i_2} \nabla^{i_3}(\psi,\chibarhat) \nabla^{i_4}(\Psi,\alpha) \\ 
			&+ \sum_{i_1+i_2+i_3+i_4=i} \nabla^{i_1} \psi^{i_2} \nabla^{i_3}(\alpha_F,\Y) \nabla^{i_4+1}(\alpha_F,\Y)\\
			&+ \sum_{i_1+i_2+i_3+i_4+i_5=i} \nabla^{i_1} \psi^{i_2} \nabla^{i_3}(\psi,\chibarhat,\tr\chibar,\chihat)\nabla^{i_4}(\alpha_F,\Y) \nabla^{i_5}(\alpha_F,\Y)\\
			&+ \sum_{i_1+i_2+i_3+i_4=i} \nabla^{i_1} \psi^{i_2} \nabla^{i_3}(\psi,\chihat) \nabla^{i_4} \Psi.
			\end{align*} 
		Applying Proposition \ref{prop36} and multiplying both sides by $(\al)^i$ we get
		
		\begin{align*}
		&\scaletwoSu{\aln \Psi} \\  
		\leq& \int_0^{\ubar} \scaletwoSuubarprime{(\al)^i \nabla^{i+1} (\Psi,\alpha)} \hspace{.5mm} \text{d} \ubar^{\prime}+ \sum_{i_1+i_2+i_3+1=i} \int_0^{\ubar} \scaletwoSuubarprime{(\al)^{i} \nabla^{i_1} \psi^{i_2+1} \nabla^{i_3+1}(\Psi,\alpha)} \hspace{.5mm} \text{d}\ubar^\prime\\
		&+ \sum_{i_1+i_2+i_3+i_4=i} \int_0^{\ubar}  \scaletwoSuubarprime{(\al)^i \nabla^{i_1} \psi^{i_2}\nabla^{i_3} (\psi,\chibarhat) \nabla^{i_4}(\Psi,\alpha) } \hspace{.5mm} \text{d}\ubar^{\prime} \\
		&+ \sum_{i_1+ i_2 + i_3 +i_4 =i} \int_0^{\ubar} \scaletwoSuubarprime{(\al)^i\nabla^{i_1} \psi^{i_2} \nabla^{i_3}(\alpha_F,\Y) \nabla^{i_4+1} (\alpha_F,\Y)} \hspace{.5mm} \text{d}\ubar^\prime  \\
		&+ \sum_{i_1+ i_2 + i_3 +i_4 +i_5 =i} \intscaletwoSuubarprime{(\al)^i\nabla^{i_1} \psi^{i_2} \nabla^{i_3} (\psi,\chibarhat,\tr\chibar)\nabla^{i_4}(\alpha_F,\Y) \nabla^{i_5}(\alpha_F,Y)} \\
		&+ \sum_{i_1 + i_2 +i_3 + i_4 =i} \intscaletwoSuubarprime{(\al)^i \nabla^{i_1}\psi^{i_2} \nabla^{i_3}(\psi,\chihat) \nabla^{i_4}\Psi}\\
		\leq& \mathcal{R}[\alpha]+ \frac{1}{\al}\mathcal{R}[\Psi] +  \sum_{i_1+i_2+i_3+1 =i} \intscaleTwoSuubarprime{(\al)^{i+1} \nabla^{i_1}\psi^{i_2+1} \nabla^{i_3+1}\left(\frac{\Psi}{\al}, \frac{\alpha}{\al}\right)}\\
		&+  \sum_{i_1+i_2+i_3+i_4=i} \int_0^{\ubar} \frac{\lvert u \rvert}{\al} \ScaletwoSuubarprime{(\al)^{i+1} \nabla^{i_1} \psi^{i_2} \nabla^{i_3}\left(\frac{\al}{\lvert u \rvert} \psi, \frac{\al}{\lvert u \rvert}\chibarhat \right) \nabla^{i_4}\left(\frac{\Psi}{\al}, \frac{\alpha}{\al}\right)} \hspace{.5mm} \text{d}\ubar^\prime \\
		&+  \sum_{i_1+i_2+i_3+i_4=i} \intscaleTwoSuubarprime{(\al)^{i+2}\nabla^{i_1} \psi^{i_2} \nabla^{i_3}\left(\frac{\alpha_F}{\al}, \frac{\Y}{\al}\right) \nabla^{i_4+1}\left(\frac{\alpha_F}{\al}, \frac{\Y}{\al}\right)}  \\
		&+  \sum_{i_1+i_2+i_3+i_4+i_5=i} \frac{\lvert u \rvert^2}{a} \int_0^{\ubar} \ScaletwoSuubarprime{(\al)^{i+2} \nabla^{i_1} \psi^{i_2} \nabla^{i_3}\left(\frac{a}{\lvert u \rvert^2}\psi, \frac{a}{\lvert u \rvert^2}\chibarhat, \frac{a}{\lvert u \rvert^2}\tr\chibar\right) \nabla^{i_4}\left(\frac{\alpha_F}{\al}, \frac{\Y}{\al}\right)   \nabla^{i_5}\left(\frac{\alpha_F}{\al}, \frac{\Y}{\al}\right) } \hspace{.5mm} \text{d}\ubar^\prime \\ 
		&+ \sum_{i_1+i_2+i_3+i_4=i} \intscaleTwoSuubarprime{(\al)^{i+1} \nabla^{i_1} \psi^{i_2} \nabla^{i_3}\left(\frac{\psi}{\al}, \frac{\chihat}{\al}\right)\nabla^{i_4} \Psi} \\ :=& \sum_{k=1}^7 J_k.
		\end{align*}We focus on each $J_k$-term separately. 
		
		\begin{itemize}
			\item We have $J_1+J_2 \leq \mathcal{R}[\alpha]+1$.
			\item We have \begin{equation}
			J_3 = \sum_{i_1+i_2+i_3+1 =i} \intscaleTwoSuubarprime{(\al)^{i+1} \nabla^{i_1}\psi^{i_2+1} \nabla^{i_3+1}\left(\frac{\Psi}{\al}, \frac{\alpha}{\al}\right)} \\ 
			\leq \frac{\al\cdot O^2}{\lvert u \rvert}.
			\end{equation} 
			\item We have \begin{align*}
			J_4 =& \sum_{i_1+i_2+i_3+i_4=i} \int_0^{\ubar} \frac{\lvert u \rvert}{\al} \ScaletwoSuubarprime{(\al)^{i+1} \nabla^{i_1} \psi^{i_2} \nabla^{i_3}\left(\frac{\al}{\lvert u \rvert} \psi, \frac{\al}{\lvert u \rvert}\chibarhat\right) \nabla^{i_4}\left(\frac{\Psi}{\al}, \frac{\alpha}{\al}\right)} \hspace{.5mm} \text{d}\ubar^\prime \\ 
			\leq& \frac{\lvert u \rvert}{\al}\cdot \frac{\al}{\lvert u \rvert}(O[\chibarhat]\cdot O[\alpha] +1) \leq O[\alpha] +1 \leq 1.
			\end{align*}Note that here we have made use of Proposition \ref{chibarhatproposition} and Proposition \ref{alphaproposition} and used the improved (compared to the bootstrap assumptions) bounds on $\chibarhat$ and $\alpha$.
			\item We have\begin{align*} &J_5\\=&\sum_{i_1+i_2+i_3+i_4=i} \intscaleTwoSuubarprime{(\al)^{i+2}\nabla^{i_1} \psi^{i_2} \nabla^{i_3}\left(\frac{\alpha_F}{\al}, \frac{\Y}{\al}\right) \nabla^{i_4+1}\left(\frac{\alpha_F}{\al}, \frac{\Y}{\al}\right)}     \\
			\leq& \int_0^{\ubar} \ScaletwoSuubarprime{(\al)^{i+2} \left(\frac{\alpha_F}{\al}\right) \cdot \nabla^{i+1} \left(\frac{\alpha_F}{\al}\right)    }  \hspace{.5mm} \text{d} \ubar^{\prime} \\
			&+ \sum_{\substack{i_1+i_2+i_3+i_4=i \\ i_4 \leq i-1}}   \intscaleTwoSuubarprime{(\al)^{i+2}\nabla^{i_1} \psi^{i_2} \nabla^{i_3}\left(\frac{\alpha_F}{\al}, \frac{\Y}{\al}\right) \nabla^{i_4+1}\left(\frac{\alpha_F}{\al}, \frac{\Y}{\al}\right)} \\ 
			\leq&   \frac{\al}{\lvert u \rvert} \cdot O \cdot    \intscaleTwoSuubarprime{(\al \nabla)^{i+1} \left(\frac{\alpha_F}{\al}, \frac{\Y}{\al}\right)} \\ 
			&+ \sum_{\substack{i_1+i_2+i_3+i_4=i \\ i_4 \leq i-1}}   \intscaleTwoSuubarprime{(\al)^{i+2}\nabla^{i_1} \psi^{i_2} \nabla^{i_3}\left(\frac{\alpha_F}{\al}, \frac{\Y}{\al}\right) \nabla^{i_4+1}\left(\frac{\alpha_F}{\al}, \frac{\Y}{\al}\right)} \\ 
			\leq& \frac{\al \cdot O}{\lvert u \rvert} \cdot \left(\mathcal{F}[\alpha_F], \frac{\mathcal{F}[\Y]}{\al}\right) \\
			&+  \sum_{\substack{i_1+i_2+i_3+i_4 = i, \\ 1 \leq i_4+1 \leq 6  }} \int_0^{\ubar} \frac{\al}{\lvert u \rvert} \cdot \big\lVert (\al \nabla)^{i_4+1} \left(\frac{\alpha_F} {\al}, \frac{\Y}{\al}\right) \big\rVert_{\mathcal{L}^{\infty}_{(sc)}(S_{u,\ubar^\prime)}} \ScaletwoSuubarprime{(\al)^{i_1+i_2+i_3}\nabla^{i_1}\psi^{i_2} \nabla^{i_3} \left(\frac{\alpha_F}{\al}, \frac{\Y}{\al} \right)   } \hspace{.5mm} \text{d} \ubar^{\prime} \\ 
			&+ \sum_{\substack{i_1+i_2+i_3+i_4 = i, \\ 6 <i_4+1 \leq 9  }} \int_0^{\ubar} \frac{\al}{\lvert u \rvert} \cdot \big\lVert (\al \nabla)^{i_4+1} \left(\frac{\alpha_F} {\al}, \frac{\Y}{\al}\right) \big\rVert_{\mathcal{L}^{2}_{(sc)}(S_{u,\ubar^\prime)}} \big\lVert (\al)^{i_1+i_2+i_3}\nabla^{i_1}\psi^{i_2} \nabla^{i_3} \left(\frac{\alpha_F}{\al}, \frac{\Y}{\al} \right)   \big\rVert_{\mathcal{L}^\infty_{(sc)}(S_{u,\ubar^\prime})} \hspace{.5mm} \text{d} \ubar^{\prime}\\ 
			\leq&  \frac{\al \cdot O}{\lvert u \rvert}\cdot\left(\mathcal{F}[\alpha_F], \frac{\mathcal{F}[\Y]}{\al}\right) + \frac{\al \cdot O^2}{\lvert u \rvert} + \frac{\al}{\lvert u \rvert}\cdot O \cdot \left(O + \sum_{1\leq i_2 \leq 8} \frac{(\al)^{i_2} \cdot O^{i_2+1}}{\lvert u \rvert^{i_2}}\right)  \leq 1. \end{align*}

			\item We have  \begin{align*}
			J_6=& 			\sum_{i_1+i_2+i_3+i_4+i_5=i} \frac{\lvert u \rvert^2}{a}\\
			&\quad \times \int_0^{\ubar} \ScaletwoSuubarprime{(\al)^{i+2} \nabla^{i_1} \psi^{i_2} \nabla^{i_3}\left(\frac{a}{\lvert u \rvert^2}\psi, \frac{a}{\lvert u \rvert^2}\chibarhat, \frac{a}{\lvert u \rvert^2}\tr\chibar\right)\nabla^{i_4}\left(\frac{\alpha_F}{\al}, \frac{\Y}{\al}\right)   \nabla^{i_5}\left(\frac{\alpha_F}{\al}, \frac{\Y}{\al}\right) } \hspace{.5mm} \text{d}\ubar^\prime \\
			\leq& \frac{\lvert u\rvert^2}{a} \cdot a \cdot \left( \frac{O[\tr\chibar]\cdot O_{i_4}\left[ \frac{\alpha_F}{\al} \right]\cdot O_{i_5}\left[ \frac{\alpha_F}{\al} \right]}{\lvert u \rvert^2}  \right)+1.			\end{align*}The logic behind the bound above is as follows. If the term we wish to bound, schematically, is not in the form of a triple anomaly, then the estimates are not borderline and the term is bounded above by $1$. The worst term is when we wish to bound
			
			\[ \sum_{i_1+i_2+i_3+i_4+i_5=i} \frac{\lvert u \rvert^2}{a} \int_0^{\ubar} \ScaletwoSuubarprime{(\al)^{i+2} \nabla^{i_1} \psi^{i_2} \nabla^{i_3}\left( \frac{a}{\lvert u \rvert^2}\tr\chibar\right) \nabla^{i_4}\left(\frac{\alpha_F}{\al}\right)   \nabla^{i_5}\left(\frac{\alpha_F}{\al}\right) } \hspace{.5mm} \text{d}\ubar^\prime .    \]This term can only be bounded by $O[\tr\chibar]\cdot O_{i_4}\left[ \frac{\alpha_F}{\al} \right]\cdot O_{i_5}\left[ \frac{\alpha_F}{\al} \right]$. We now use the improved bounds from Propositions \ref{trchibarboundRicci} and \ref{alphaFproposition} to bound $O[\tr\chibar] \leq 1$ and\footnote{We need to improve Proposition \ref{alphaFproposition} like Proposition \ref{alphaproposition}.} $O_{i_4}\left[ \frac{\alpha_F}{\al} \right] \cdot O_{i_5}\left[ \frac{\alpha_F}{\al} \right] \leq \left(\underline{\mathcal{F}}[\rho_F,\sigma_F]+1\right)\cdot 1$. This is because at least one of the indices $i_4, i_5$ will not be of top order, hence the estimate from Proposition \ref{alphaFproposition} for that term will be better. Combining these estimates, we arrive at \[J_6 \leq \underline{\mathcal{F}}[\rho_F,\sigma_F]+1.\]
			
			\item The final term $J_7$ is handled as follows:

			\begin{align*}
			&\sum_{i_1+i_2+i_3+i_4=i} \intscaleTwoSuubarprime{(\al)^{i+1} \nabla^{i_1} \psi^{i_2} \nabla^{i_3}\left(\frac{\psi}{\al}, \frac{\chihat}{\al}\right)\nabla^{i_4} \Psi} \\ \leq& \int_0^{\ubar} \frac{\al \cdot O^2}{\lvert u \rvert} \hspace{.5mm} \text{d}\ubar^{\prime} \leq \frac{\al \cdot O^2}{\lvert u \rvert}  \leq 1.
			\end{align*}
		\end{itemize}Combining all the estimates above, we arrive at the desired conclusion.
		
	\end{proof}

	\section{Elliptic estimates for top-order derivatives of Ricci coefficients} 
	
	\subsection{General elliptic estimates for Hodge systems}
	
	We recall here the definition of divergence and curl for a symmetric, covariant tensor of arbitrary rank:
	
	\[ (\div \phi)_{A_1\dots A_r} = \nabla^B \phi_{BA_1\dots A_r}, \] \[    (\curl \phi)_{A_1\dots A_r} = \slashed{\epsilon}^{BC} \nabla_B \phi_{CA_1\dots A_r}. \]The trace of such a tensor is defined by
	
	\[    (\tr\phi)_{A_1\dots A_{r-1}} = (\gamma^{-1})^{BC}\phi_{BCA_1\dots A_{r-1}}.      \] The main elliptic estimate that will be used here is the following:
	
	\begin{proposition} \label{mainellipticlemma}
		Under the assumptions of Theorem \ref{main1} and the bootstrap assumptions \eqref{bootstrapbounds}, let $\phi$ be a totally symmetric $(r+1)-$covariant tensorfield on a metric $2-$sphere $\left(\mathbb{S}^2,\gamma\right)$, satisfying \[ \div \phi = f, \hspace{2mm} \curl \phi = g, \hspace{2mm}\tr\phi =h. \] Then, for $ 1\leq i \leq 11 $, we have \[ \scaletwoSu{\aln \phi} \lesssim \al \sum_{j=0}^{i-1} \scaletwoSu{(\al \nabla)^j (f,g)} + \sum_{j=0}^{i-1} \scaletwoSu{(\al \nabla)^j(\phi,h)}.\]
	\end{proposition}
	
	\begin{proof}
		Recall the following identity from Chapter 7 in [Christodoulou] that, for $\phi, f,g$ and $h$ as above, there holds \be \label{ellipticid} \int_{S_{u,\ubar}}\left( \lvert \nabla \phi \rvert^2 +(r+1)K \lvert \phi\rvert^2 \right) \hspace{.5mm}\text{d}\mu_{\gamma} =  \int_{S_{u,\ubar}} \left( \lvert f \rvert^2 + \lvert g \rvert^2 + r K \lvert h \rvert^2 \right)\hspace{.5mm}\text{d}\mu_{\gamma}.  \ee Here $K$ denotes the Gauss curvature of the  sphere. To prove the lemma for the case $i=1$ first, we need to control $K$ in $L^{\infty}$. To that end, we will prove the following stronger lemma:

		\begin{lemma} \label{Klemma}
			For $0\leq k \leq 7,$ there holds $\scaleinfinitySu{(\al \nabla)^k K} \lesssim 1.$
		\end{lemma}
		
		\begin{proof}
			We begin by recalling that \[ K = -\rho_{R} - \frac{1}{4}\tr\chi\hspace{.5mm}\tr\chibar + \frac{1}{2}\chihat \cdot \chibarhat = -\rho  - \frac{1}{4}\tr\chi\hspace{.5mm}\tr\chibar + \frac{1}{2}\chihat \cdot \chibarhat + \frac{1}{2}\left(\lvert \rho_F \rvert^2 + \lvert \sigma_F \rvert^2\right)  \]and $s_2(K)=1.$ By virtue of the scale-invariant version of the $L^2-L^{\infty}$ Sobolev embedding inequality from Proposition \ref{Sobolevembedding}, there holds \be \sum_{0 \leq k \leq 7} \scaleinfinitySu{(\al \nabla)^k K} \lesssim \sum_{0\leq j \leq 9} \scaletwoSu{(\al \nabla)^j K}. \ee We proceed to estimate, for a fixed $0 \leq i \leq 9$, the term $\scaletwoSu{(\al \nabla)^i K}$. We have 
			\begin{align*}
			\scaletwoSu{\aln K} \lesssim& \scaletwoSu{\aln \rho} + \sum_{i_1+i_2=i}\scaletwoSu{(\al)^i \nabla^{i_1} \tr\chi \nabla^{i_2}\tr\chibar}\\
			&+\sum_{i_1+i_2=i}\scaletwoSu{(\al)^i \nabla^{i_1} \chihat \nabla^{i_2}\chibarhat}+ \sum_{i_1+i_2=i}\scaletwoSu{(\al)^i \nabla^{i_1} \Y \nabla^{i_2}\Y}. 
			\end{align*}The first term above can be bounded by $1$, by Proposition \ref{L2curvature2}. For the second term, we have \begin{equation}
\begin{split}\label{Ksecondtermbound} 
			\sum_{i_1+i_2=i}\scaletwoSu{(\al)^i \nabla^{i_1} \tr\chi \nabla^{i_2}\tr\chibar} =& \frac{\lvert u \rvert^2}{a} \sum_{i_1+i_2=i}\ScaletwoSu{(\al)^i \nabla^{i_1} \tr\chi \nabla^{i_2}\left(\frac{a}{\lvert u \rvert^2}\tr\chibar \right)}\\ \lesssim& \frac{\lvert u \rvert^2}{a} \cdot \frac{1}{\lvert u \rvert} \cdot \left( O_{\infty}[\tr\chibar]\cdot O_2[\tr\chi] + O_{2}[\tr\chibar]\cdot O_{\infty}[\tr\chi] \right).
			\end{split}
\end{equation}In the above inequality we have conditioned on the number of derivatives that fall on $\tr\chi$ and those that fall on $\tr\chibar$. Notice that, from Proposition \ref{trchibarboundRicci}, there holds $O_{\infty}[\tr\chibar] +O_2[\tr\chibar] \lesssim 1$. For $O_2[\tr\chi],$ from Proposition \ref{trchiboundRicci}, we read off \eqref{Kusefulequation} that \[ O_2[\tr\chi] \leq \frac{a}{\lvert u \rvert}O[\chihat,\alpha_F]\cdot O[\chihat,\alpha_F] + \frac{a }{\lvert u \rvert^2}O^3+ \frac{\al}{\lvert u \rvert}O^2.  \] Plugging this inequality in \eqref{Ksecondtermbound} and using $O[\chihat,\alpha_F]\lesssim 1$ from Propositions \ref{chihatproposition} and \ref{alphaFproposition} (remember crucially that we work with up to 9 derivatives at most, so the top order terms $\mathcal{R}[\alpha]$ and $\underline{\mathcal{F}}[\rho_F,\sigma_F]$ used to estimate $\chihat$ and $\alpha_F$ are redundant on the right-hand side) we arrive at \[ \sum_{i_1+i_2=i}\scaletwoSu{(\al)^i \nabla^{i_1} \tr\chi \nabla^{i_2}\tr\chibar} \lesssim 1. \] For the third term, there holds
			
			\begin{align*}
			\sum_{i_1+i_2=i} \scaletwoSu{(\al)^{i} \nabla^{i_1}\chihat \nabla^{i_2}\chibarhat} =& \lvert u \rvert \sum_{i_1+i_2=i} \ScaletwoSu{(\al)^i \nabla^{i_1}\left( \frac{\chihat}{\al} \right) \nabla^{i_2} \left( \frac{\al}{\lvert u \rvert} \chibarhat \right) }\\ \lesssim& O_{\infty}[\chibarhat] \cdot O_{2}[\chihat] + O_{2}[\chibarhat] \cdot O_{\infty}[\chihat] \lesssim 1 \cdot 1 = 1. 
			\end{align*}Here we have used Proposition \ref{chibarhatproposition} and Proposition \ref{chihatproposition}, where we achieve the better bound $O_2[\chihat] \lesssim 1,$ given that we work up to 9 derivatives, not 10, which means that we can bound the curvature term in the $\mathcal{L}^2_{(sc)}(S_{u,\ubar})-$norm instead of the $\mathcal{L}^2_{(sc)}(H_u^{0,\ubar})-$ norm. The bound on $\alpha$ is then invoked by Proposition \ref{alphaproposition}. Finally, the fourth term is bounded by $O^2/\lvert u \rvert \lesssim 1$, using proposition \ref{usefulstatements}. This concludes the proof of the lemma.
		\end{proof} \par \noindent By applying the scale-invariant version of H\"older's inequality for $K \lvert h \rvert^2$ and using Lemma \ref{Klemma}, we obtain the result for $i=1$. For $i>1$, the symmetrized angular derivative of $\phi$ defined by \[ (\nabla \phi)^s_{BA_1\dots A_{r+1}} = \frac{1}{r+2}\left( \nabla_B \phi_{A_1\dots A_r} +\sum_{i=1}^{r+1} \nabla_{A_i}\phi_{A_1\dots \langle A_i  \rangle B\dots A_{r+1}} \right)    \]satisfies the div-curl system \begin{equation} \begin{cases}
		\div (\nabla\phi)^s = (\nabla f)^s - \frac{1}{r+2}(\Hodge{\nabla g})^s +(r+1)K\phi -\frac{2K}{r+1}(\gamma \otimes^s h), \\ \curl (\nabla \phi)^s = \frac{r+1}{r+2}(\nabla g)^s + (r+1) K (\Hodge{ \phi})^s, \\ \tr (\nabla \phi)^s = \frac{2}{r+2}f + \frac{r}{r+2}(\nabla h)^s, \end{cases} \end{equation}where \[ \gamma \otimes^s h := \gamma_{A_i A_j} \sum_{i\leq j=1,\dots,r+1} h_{A_1\dots \langle A_i \rangle \dots \langle A_j \rangle \dots A_{r+1}}\]and \[ (\Hodge{\phi})^s_{A_1\dots A_{r+1}} :=  \frac{1}{r+1} \sum_{i=1}^{r+1} {\slashed{\epsilon}_{A_i}}^B \phi_{A_1\dots \langle A_i \rangle B \dots A_{r+1}}. \]Using \eqref{ellipticid} and iterating, we obtain that for $i\leq 11$ there holds

		\begin{equation}
\begin{split} \label{ellipticineq} 
		&\twoSu{\nabla^i \phi}^2\\ \lesssim& \twoSu{\nabla^{i-1}(f,g)}^2+ \lVert K ( \lvert \nabla^{i-2}(f,g) \rvert^2  + \lvert \nabla^{i-1}(\phi,h) \rvert^2 ) \rVert_{L^1(S_{u,\ubar})} \\    &+ \bigg\lVert K \left(\sum_{i_1+2i_2+i_3 =i-3} \nabla^{i_1}K^{i_2+1} \nabla^{i_3}(\phi,h) \right)^2 \bigg\rVert_{L^1(S_{u,\ubar})} +   \bigg\lVert K \left(\sum_{i_1+2i_2+i_3=i-4} \nabla^{i_1}K^{i_2+1} \nabla^{i_3}f\right)^2 \bigg\rVert_{L^1(S_{u,\ubar})} \\ &+ \sum_{i_1+2i_2+i_3 = i-2} \twoSu{\nabla^{i_1}K^{i_2+1} \nabla^{i_3}(\phi,h)}^2    + \sum_{i_1+2i_2+i_3= i-3} \twoSu{\nabla^{i_1}K \nabla^{i_2}(f,g)}^2,   \end{split}
\end{equation}where we have adopted the convention that $\sum_{i\leq -1} = 0$. Whenever a $K$--term appears with at most $7$ derivatives, we estimate it in $L^{\infty}$ or equivalently in $\mathcal{L}^{\infty}_{(sc)}$. Whenever a $K$--term contains between 8 and 9 derivatives we shall estimate it in $L^2$ and the rest of the terms in $L^{\infty}$, noting that we can estimate terms of the form $\lVert \nabla^i (f,g,\phi, h) \rVert_{L^{\infty}}$ with $i\leq 7$ by the corresponding norms in $L^2$ through the standard Sobolev embedding. By Lemma \ref{Klemma}, after translating back to standard $L^p$ norms, there holds
		\[ \sum_{i \leq 7} \inftySu{\lvert u\rvert^i \nabla^i K} + \sum_{j\leq 9} \twoSu{\lvert u \rvert^j \nabla^j K}  \lesssim 1.         \]Therefore, for $i\leq 11$, we have \[   \twoSu{ \lvert u \rvert^i \nabla^i \phi}^2 \lesssim  \sum_{j\leq i-1}  \left( \twoSu{ \lvert u \rvert^{j+1} \nabla^{j} (f,g)}^2  + \twoSu{\lvert u \rvert^j \nabla^j (\phi,h)}^2 \right) .            \]Translating the above equation  into scale-invariant norms and then multiplying it by $\frac{\lvert u \rvert^{2s_2(\phi)}}{a^{s_2(\phi)}}$, we arrive at \begin{equation} \scaletwoSu{\aln \phi}^2 \lesssim \sum_{j\leq i-1} \left(\scaletwoSu{ (\al)^{j+1} \nabla^j (f,g) }^2 + \scaletwoSu{ (\al \nabla)^j (\phi,h) }^2\right). \end{equation}Taking square roots above yields Proposition \ref{mainellipticlemma}. \end{proof}\par \noindent Finally, for the special case where $\phi$ is a symmetric, traceless $2-$tensor, we need only know its divergence:
	
	\begin{proposition}\label{mainellipticlemma2}
		Suppose $\phi$ is a symmetric, traceless $2-$tensor satisfying \[ \div \phi = f. \] Then, under the assumptions of Theorem \ref{main1} and the bootstrap assumptions \eqref{bootstrapbounds}, for $1\leq i \leq 11$, there holds \[ \scaletwoSu{\aln \phi} \lesssim \sum_{j\leq i-1} \left( \scaletwoSu{ (\al)^{j+1} \nabla^j f }+ \scaletwoSu{ (\al \nabla)^j \phi }\right). \] \end{proposition}
	
	\begin{proof}
		This is a direct application of Proposition \ref{mainellipticlemma}, by noticing that \[ \curl \phi = \Hodge{f}.\] This is a straightforward calculation, using that the $2-$tensor $\phi$ is symmetric and traceless. 
	\end{proof}

	\subsection{Elliptic estimates for 11 derivatives of Ricci coefficients} 
	
	We start this section with the following auxiliary bootstrap assumption. Introduce the top-order quantity
	
\begin{equation}\begin{split}
	\mathcal{O}_{11,2}(u,\ubar) =& \frac{1}{a^{\frac{1}{2}}} \scaletwoHu{ ( a^{\frac{1}{2}})^{10}\nabla ^{11} \chihat}     +\scaletwoHu{ ( a^{\frac{1}{2}})^{10}\nabla^{11} (\tr\chi,\omega)} \\ &+ \scaletwoHbaru{a^5 \nabla^{11} \eta} + \frac{a}{\lvert u \rvert}\scaletwoHu{a^5 \nabla^{11}(\eta,\etabar)}  \\ &+ \intu \frac{a^{\frac{3}{2}}}{\lvert u \rvert^3 } \scaletwoSuprime{( a^{\frac{1}{2}})^{10}\nabla^{11} \chibarhat} \duprime  +\scaletwoHbaru{(\al)^{10} \nabla^{11}\omegabar}\\ &+  \intu \frac{a^2}{\lvert u\rvert^3 }\scaletwoSuprime{ a^5 \nabla^{11} \tr\chibar}\duprime. \end{split} \end{equation}Throughout this section we assume \begin{equation}\label{bootstrapelliptic} \mathcal{O}_{11,2} \leq O_{11} \lesssim a^{\frac{1}{320}}.\end{equation}We begin with estimates for $\tr\chi$ and $\chibarhat$.

	\begin{proposition}
		\label{trchichihatellipticprop} Under the assumptions of Theorem \ref{main1} and the bootstrap assumptions \eqref{bootstrapbounds}, there holds
		\[   \scaletwoSu{(\al)^{10} \nabla^{11} \tr\chi} \lesssim 1+ \mathcal{R}[\tbeta]+\mathcal{R}[\alpha] + \mathcal{F}[\alpha_F],  \]\[ \frac{1}{\al} \scaletwoSu{(\al)^{10}\nabla^{11} \chihat} \lesssim 1+ \mathcal{R}[\tbeta]+\mathcal{R}[\alpha].   \]
	\end{proposition}	
	
	\begin{proof}
		Consider the following equation: 
		
		\begin{equation}
		\nabla_4 \tr\chi + \frac{1}{2}(\tr\chi)^2 = -\lvert \chihat \rvert^2 -\lvert \alpha_F \rvert^2 -2 \omega \hspace{.5mm} \tr\chi.
		\end{equation}Commuting with angular derivatives $i$ times, we arrive at 	\begin{equation*}\begin{split}\nabla_4 \nabla^i \tr\chi =& \sum_{i_1+ i_2 + i_3 +i_4 =i} \nabla^{i_1}\psi^{i_2}\nabla^{i_3}\chihat \nabla^{i_4} \chihat  
		+ \sum_{i_1 + i_2 +i_3 + i_4 =i}  \nabla^{i_1}\psi^{i_2} \nabla^{i_3}\alpha_F \nabla^{i_4} \alpha_F  \\ 
	&+ \sum_{i_1 + i_2 +i_3 + i_4 =i} \nabla^{i_1} \psi^{i_2} \nabla^{i_3} (\psi, \chihat) \nabla^{i_4} \psi \\  =& \sum_{i_1+i_2 =i} \nabla^{i_1} \chihat \nabla^{i_2} \chihat + \sum_{i_1+i_2+i_3+i_4+1=i} \nabla^{i_1} \psi^{i_2+1} \nabla^{i_3}\chihat \nabla^{i_4}\chihat    \\ &+   \sum_{i_1+i_2 =i} \nabla^{i_1} \alpha_F \nabla^{i_2} \alpha_F + \sum_{i_1+i_2+i_3+i_4+1=i} \nabla^{i_1} \psi^{i_2+1} \nabla^{i_3}\alpha_F \nabla^{i_4}\alpha_F   \\  &+\sum_{i_1 + i_2 +i_3 + i_4 =i} \nabla^{i_1} \psi^{i_2} \nabla^{i_3} (\psi, \chihat) \nabla^{i_4} \psi. \end{split}\end{equation*} Passing the above to  scale-invariant norms and applying the triangle inequality, we have 	\begin{equation}\label{trchiellipticalign*1}
\begin{split}
		&\scaletwoSu{(\al)^{10}\nabla^{11} \tr\chi}\\
		\leq&   \int_{0}^{\ubar} a \big\lVert (\al)^{10} \left(\frac{\chihat}{\al}, \frac{\alpha_F}{\al}\right)\nabla^{11} \left(\frac{\chihat}{\al}, \frac{\alpha_F}{\al}\right)\big \rVert_{\mathcal{L}^2_{(sc)}(S_{u,\ubar^\prime})} \text{d}\ubar^\prime \\ &+\sum_{\substack{i_1+i_2=11 \\i_1,i_2 \leq 10}}  \int_{0}^{\ubar} a \big\lVert (\al)^{10} \nabla^{i_1}\left(\frac{\chihat}{\al}, \frac{\alpha_F}{\al}\right)\nabla^{i_2}\left(\frac{\chihat}{\al}, \frac{\alpha_F}{\al}\right)\big \rVert_{\mathcal{L}^2_{(sc)}(S_{u,\ubar^\prime})} \text{d}\ubar^\prime \\
		&+ \sum_{i_1 + i_2 +i_3 + i_4 +1 =11} \int_0^{\ubar} a \big\lVert  (\al)^{10} \nabla^{i_1} \psi^{i_2+1} \nabla^{i_3} \left(\frac{\chihat}{\al}, \frac{\alpha_F}{\al}\right)\nabla^{i_4}\left(\frac{\chihat}{\al}, \frac{\alpha_F}{\al}\right)\big \rVert_{\mathcal{L}^2_{(sc)}(S_{u,\ubar^\prime})} \text{d}\ubar^\prime \\
		&+  \sum_{i_1 + i_2 +i_3 + i_4 =11} \int_0^{\ubar}  \al \big\lVert (\al)^{10} \nabla^{i_1} \psi^{i_2} \nabla^{i_3} \left(\frac{\psi}{\al}, \frac{\chihat}{\al}\right) \nabla^{i_4} \psi\big\rVert_{\mathcal{L}^{2}_{(sc)}(S_{u,\ubar^\prime})}\text{d}\ubar^\prime \\
		\leq& \frac{a}{\lvert u \rvert} O[\chihat] \cdot \int_0^{\ubar} \ScaletwoSuubarprime{(\al)^{10} \nabla^{11} \left(\frac{\chihat}{\al} \right)} \hspace{.5mm} \text{d}\ubar^\prime \\ &+ \frac{\al}{\lvert u \rvert}O[\alpha_F]\cdot \mathcal{F}[\alpha_F]+   \frac{\al}{\lvert u \rvert}O_{2}[\chihat,\alpha_F]\cdot O_{\infty}[\chihat,\alpha_F] + \frac{\al }{\lvert u \rvert^2}O^3+ \frac{O^2}{\lvert u \rvert}+ \frac{\al}{\lvert u \rvert}O\cdot O_{11}\\ \leq& \frac{\al}{\lvert u \rvert} O[\chihat] \cdot \int_0^{\ubar} \scaletwoSuubarprime{(\al)^{10} \nabla^{11} \chihat} \hspace{.5mm} \text{d}\ubar^\prime + \frac{a}{\lvert u \rvert}O[\alpha_F]\cdot \mathcal{F}[\alpha_F] +
		\frac{\al}{\lvert u \rvert}(\mathcal{R}[\alpha]+1)(O[\alpha]+1)\\ \leq& \frac{\al}{\lvert u \rvert} O_{\infty}[\chihat] \cdot \int_0^{\ubar} \scaletwoSuubarprime{(\al)^{10} \nabla^{11} \chihat} \hspace{.5mm} \text{d}\ubar^\prime + \frac{a}{\lvert u \rvert}O_{\infty}[\alpha_F]\cdot \mathcal{F}[\alpha_F] +1. 
\end{split}
\end{equation}

For $\chihat$, we have \[ \div \chihat = \frac{1}{2}\nabla \tr\chi - \frac{1}{2}(\eta-\etabar)\cdot(\chihat - \frac{1}{2}\tr\chi \hspace{.5mm} \gamma)- \tbeta + \frac{1}{2}R_{a4}. \]Schematically, \[ \div \chihat- \frac{1}{2}\nabla \tr\chi + \tbeta = \psi \cdot \psi + \alpha_F \cdot \Y.   \]Applying Proposition \ref{mainellipticlemma2}, we arrive at \begin{equation}\begin{split}
		&\scaletwoSu{(\al)^{10} \nabla^{11} \chihat}\\ \lesssim& \sum_{i \leq 10}  \frac{1}{\al}\scaletwoSu{\aln \tr\chi}+ \sum_{i \leq 10}\scaletwoSu{\aln \tbeta} \\&+ \sum_{i \leq 10} \sum_{i_1+i_2=i} \scaletwoSu{(\al)^i \nabla^{i_1}\psi \nabla^{i_2}\psi}+  \sum_{i \leq 10} \sum_{i_1+i_2=i} \scaletwoSu{(\al)^i \nabla^{i_1}\alpha_F \nabla^{i_2}\Y}\\ &+ \frac{1}{\al} \sum_{i \leq 10} \scaletwoSu{\aln \chihat}.
		\end{split}\end{equation}By using the estimate on $\scaletwoSu{\aln \tr\chi}$ from Proposition \ref{trchiboundRicci} and applying Gr\"onwall's inequality, we get
		
		\begin{equation}\begin{split} \label{chihellipticalign*} 
		&\scaletwoSu{(\al)^{10} \nabla^{11} \chihat} \\ \lesssim& \frac{a}{\lvert u \rvert}O[\alpha_F]\cdot \mathcal{F}[\alpha_F] +
		\mathcal{R}[\alpha]+1+  \sum_{i \leq 10}\scaletwoSu{\aln \tbeta} \\&+  \sum_{i \leq 10} \sum_{i_1+i_2=i} \scaletwoSu{(\al)^i \nabla^{i_1}\psi \nabla^{i_2}\psi}+  \sum_{i \leq 10} \sum_{i_1+i_2=i} \scaletwoSu{(\al)^i \nabla^{i_1}\alpha_F \nabla^{i_2}\Y} \\&+ \frac{1}{\al} \sum_{i \leq 10} \scaletwoSu{\aln \chihat}.
		\end{split}\end{equation}Raising \eqref{chihellipticalign*} to the square and integrating along $\ubar$, we arrive at  
		\begin{equation}
\begin{split}
 \label{chihellipticalign*2}
		&\int_0^{\ubar} \scaletwoSuubarprime{(\al)^{10} \nabla^{11} \chihat}^2 \hspace{.5mm} \text{d}\ubar^{\prime}\\ \lesssim&  \left(\frac{a}{\lvert u \rvert}O[\alpha_F]\cdot \mathcal{F}[\alpha_F] +
		\mathcal{R}[\alpha]+1 \right)^2+ \int_0^{\ubar}  \sum_{i \leq 10}\scaletwoSuubarprime{\aln \tbeta}^2 \hspace{.5mm} \text{d}\ubar^{\prime}\\ &+  \int_0^{\ubar} \sum_{i \leq 10} \sum_{i_1+i_2=i} \scaletwoSuubarprime{(\al)^i \nabla^{i_1}\psi \nabla^{i_2}\psi}^2 \hspace{.5mm} \text{d}\ubar^{\prime} \\&+  \int_0^{\ubar} \sum_{i \leq 10} \sum_{i_1+i_2=i} \scaletwoSuubarprime{(\al)^i \nabla^{i_1}\alpha_F \nabla^{i_2}\Y}^2 \hspace{.5mm} \text{d}\ubar^{\prime} + \int_0^{\ubar} \sum_{i \leq 10} \ScaletwoSuubarprime{\aln \left(\frac{\chihat}{\al}\right)}^2 \hspace{.5mm} \text{d}\ubar^{\prime}.
		\end{split}
\end{equation}Taking the square roots of the above inequality, we arrive at
		
		\begin{equation}
		\frac{1}{\al}  \scaletwoHu{(\al \nabla)^{11} \chihat} \lesssim  \mathcal{R}[\tbeta] + \mathcal{R}[\alpha]+ \mathcal{F}[\alpha_F]+1.
		\end{equation}
		
		\par \noindent By plugging this back to \eqref{trchiellipticalign*1} and applying H\"older's inequality, we get 
		
		\begin{equation}\begin{split}
		&\scaletwoSu{a^5 \nabla^{11} \tr\chi} \\ \lesssim& \frac{\al}{\lvert u \rvert}(\mathcal{R}[\alpha]+\mathcal{R}[\tbeta]+1) + \frac{a}{\lvert u \rvert}O[\alpha_F]\mathcal{F}[\alpha_F]+\mathcal{R}[\alpha]+1 \\ \lesssim& 1+ \mathcal{R}[\tbeta]+\mathcal{R}[\alpha] + \mathcal{F}[\alpha_F].
		\end{split}\end{equation}Squaring and taking $\mathcal{L}^2_{(sc)}$ in the $\ubar$-direction, we arrive at
		
		\begin{equation}
		\scaletwoHu{a^5 \nabla^{11}\tr\chi} \lesssim \mathcal{R}[\alpha] + \mathcal{F}[\alpha_F]+1.
		\end{equation}
	\end{proof}\par \noindent	We now proceed with estimates for the highest number of derivatives in $\omega$. Define the following Hodge operators acting on the leaves $S_{u,\ubar}$ of our double null foliation
	
	\begin{itemize}
		\item The operator $\mathcal{D}_1$ maps a $1-$form $F$ to the pair of functions $(\text{div}\hspace{.5mm}F, \text{curl}\hspace{.5mm} F)$,
		
		\item The operator $\mathcal{D}_2$ maps an $S-$tangent, symmetric traceless tensor $F$  into the $S-$tangent one-form $\text{div}F$,
		
		\item The operator $\Hodge{\mathcal{D}}_1$ maps a pair of scalar functions $(F_1,F_2)$ to the $S-$tangent $1-$form $-\nabla F_1 + \Hodge{\nabla}F_2$,
		
		\item The operator $\Hodge{\mathcal{D}_2}$ maps a $1-$form $F$ to the $2-$covariant, symmetric, traceless tensor $ -\frac{1}{2}\widehat{\mathcal{L}_F \gamma}, $where  \[ \widehat{\mathcal{L}_F\gamma}_{ab} = \nabla_a F_b + \nabla_b F_a - (\text{div} \hspace{.5mm} F)\gamma_{ab}.   \]
	\end{itemize}
	
	\begin{proposition}\label{omegaellipticprop} Under the assumptions of Theorem \ref{main1} and the bootstrap assumptions \eqref{bootstrapbounds}- \eqref{bootstrapelliptic}, there holds \[ \scaletwoHu{a^5 \nabla^{11}\omega} \lesssim \mathcal{R}[\tbeta]+1. \]

	\end{proposition}
	\begin{proof}
		Introduce $\omega^{\dagger}$, defined as the solution to \[ \nabla_3 \omega^{\dagger} = \frac{1}{2} \sigma\] with zero initial data on $H_{u_{\infty}}$.  Introduce the pair of scalars $\langle \omega \rangle = (-\omega, \omegad)$ and define $\kappa$ by \[   \kappa := \Hodge{\mathcal{D}}_1\langle \omega\rangle - \frac{1}{2}\tbeta = \nabla \omega + \Hodge{\nabla} \omegad- \frac{1}{2}\tbeta.\]We need to derive a transport equation for $\Hodge{\mathcal{D}}_1\langle \omega \rangle$. To this end, recall the commutation formula \[  [\nabla_3, \nabla] f = - \frac{1}{2} \tr\chibar \nabla f - \chibarhat \cdot \nabla f+ \frac{1}{2}(\eta+\etabar) \nabla_3 f, \]\[ [\nabla_3,\Hodge{\nabla}]g = - \frac{1}{2}\tr\chibar \Hodge{\nabla}g + \chibarhat \cdot \Hodge{\nabla}g +\frac{1}{2}(\Hodge{\eta}+\Hodge{\etabar})\nabla_3 g. \]Therefore, \[ [\nabla_3, \Hodge{\mathcal{D}_1}](f,g) = - \frac{1}{2}\tr\chibar \Hodge{\mathcal{D}_1}(f,g) + \chibarhat \cdot (\nabla f +\Hodge{\nabla} g)  - \frac{1}{2}(\eta+ \etabar) \nabla_3 f + \frac{1}{2}(\Hodge{\eta} + \Hodge{\etabar})\nabla_3 g.  \]Now recall that \[ \nabla_3 \omega = \frac{1}{2}\rho + \psi \hspace{.5mm} \psi + \Y \hspace{.5mm} \Y := \frac{1}{2} \rho + \underline{F}. \]This means that \begin{align*}
		\nabla_3 \Hodge{\mathcal{D}_1}\langle \omega \rangle =& \Hodge{\mathcal{D}_1}\left(-\frac{1}{2}\rho - \underline{F},  \frac{1}{2}\sigma \right) + [\nabla_3,\Hodge{\mathcal{D}_1}]\langle \omega \rangle \\= &\frac{1}{2}\nabla \rho + \frac{1}{2}\Hodge{\nabla} \sigma + \nabla \underline{F} - \frac{1}{2}\tr\chibar \Hodge{\mathcal{D}_1}\langle \omega \rangle + \chibarhat \cdot (-\nabla \omega+ \Hodge{\nabla}\omegad)\\ &- \frac{1}{2}(\eta+\etabar)\left(\frac{1}{2}\rho + \underline{F}\right) + \frac{1}{4}(\Hodge{\eta}+\Hodge{\etabar})\sigma.
		\end{align*}Schematically, we reduce this to the equation \begin{equation}
\begin{split}  &\nabla_3 \Hodge{\mathcal{D}_1}\langle \omega \rangle + \frac{1}{2}\tr\chibar  \Hodge{\mathcal{D}_1}\langle \omega \rangle  - \frac{1}{2}(\nabla \rho + \Hodge{\nabla} \sigma)\\ =& \psi \nabla (\omega,\omegabar,\eta,\etabar) + \chibarhat \nabla(\omega,\omegad) + (\rho_F,\sigma_F) \nabla (\rho_F,\sigma_F) + \psi \cdot \Psi + \psi \cdot \psi \cdot \psi + \psi \cdot \Y \cdot \Y. \label{kappa1} \end{split}
\end{equation} Recall also the $\nabla_3$-direction schematic equation for $\tbeta$: \begin{equation}\begin{split} \label{kappa2}	\nabla_3 \tbeta + \tr\chibar \tbeta - \nabla \rho - \Hodge{\nabla} \sigma =&  (\psi,\chihat) \Psi +  \alpha_F \nabla \alphabar_F + \alphabar_F \nabla \alpha_F + (\rho_F,\sigma_F)\nabla (\rho_F,\sigma_F)\\ &+ (\psi,\chibarhat,\tr\chibar, \chih)\cdot (\alpha_F,\Y)\cdot \Y. \end{split}\end{equation} From \eqref{kappa1} and \eqref{kappa2} we see that $\kappa$ obeys the following schematic equation:
		\begin{equation}\begin{split}
		\label{kappaeq} \nabla_3 \kappa + \frac{1}{2}\tr\chibar \kappa =&(\psi,\chihat)\Psi+ (\psi,\chibarhat)\nabla \psi   +\Y \nabla (\Y,\alpha_F) \\&+(\alpha_F,\Y)\nabla \Y  + (\psi,\chibarhat,\tr\chibar, \chih)\cdot (\alpha_F,\Y)\cdot \Y + \psi\cdot \psi \cdot \psi. \end{split}\end{equation} By commuting \eqref{kappaeq} with $i\leq 10$ angular derivatives, we arrive at
		\begin{equation}\begin{split}
		\nabla_3 \nabla^i \kappa + \frac{i+1}{2}\tr\chibar \nabla^i \kappa =& \sum_{i_1+i_2+i_3+i_4=i\leq 10} \nabla^{i_1}\psi^{i_2}\nabla^{i_3}(\psi,\chihat)\nabla^{i_4}\Psi \\ &+ \sum_{i_1+i_2+i_3+i_4=i\leq 10} \nabla^{i_1}\psi^{i_2}\nabla^{i_3}(\psi,\chibarhat)\nabla^{i_4+1}\psi \\&+ \sum_{i_1+i_2+i_3+i_4=i\leq 10}\nabla^{i_1}\psi^{i_2}\nabla^{i_3} \Y \nabla^{i_4+1}(\Y,\alpha_F)\\ &+ \sum_{i_1+i_2+i_3+i_4=i\leq 10}\nabla^{i_1}\psi^{i_2}\nabla^{i_3} (\alpha_F,\Y) \nabla^{i_4+1}\Y \\ &+ \sum_{i_1+i_2+i_3+i_4+i_5=i\leq 10} \nabla^{i_1}\psi^{i_2} \nabla^{i_3}(\psi,\chibarhat,\tr\chibar,\chihat)\nabla^{i_4}(\alpha_F,\Y)\nabla^{i_5} \Y \\ &+ \sum_{i_1+i_2+i_3+i_4+i_5=i\leq 10} \nabla^{i_1}\psi^{i_2} \nabla^{i_3}\psi \nabla^{i_4}\psi \nabla^{i_5} \psi \\ &+ \sum_{i_1+i_2+i_3+i_4=i\leq 10} \nabla^{i_1} \psi^{i_2} \nabla^{i_3}(\psi,\chibarhat,\tildetr)\nabla^{i_4}\kappa \\ &+ \sum_{i_1+i_2+i_3+i_4+1=i\leq 10} \nabla^{i_1}\psi^{i_2+1}\nabla^{i_3}\tr\chibar \nabla^{i_4}\kappa :=G.    \end{split}\end{equation}Applying Proposition \ref{prop37} with $\lambda_0 = \frac{i+1}{2}$, we get
		
		\[   \lvert u \rvert^i \hspace{.5mm} \twoSu{\nabla^i \kappa} \lesssim \lvert u_{\infty} \rvert^i \hspace{.5mm} \lVert \nabla^i \kappa \rVert_{L^2(S_{u_{\infty},\ubar})} + \int_{u_{\infty}}^u \lvert u^\prime \rvert^i \hspace{.5mm} \lVert G \rVert_{L^2(S_{u^{\prime},\ubar})}\hspace{.5mm} \text{d}u^\prime  .           \]Now by definition, we have $s_2(\kappa) = s_2(\nabla \omega, \tbeta) = 0.5$. So $s_2(\nabla^i \kappa) = \frac{i+1}{2}$. This means that
		
		\begin{align*} \scaletwoSu{\aln \kappa} =&  a^{\frac{i-1}{2}} \lvert u \rvert^{i+1} \twoSu{\nabla^i \kappa} = (a^{\frac{i-1}{2}} \lvert u \rvert) \cdot \lvert u \rvert^i \twoSu{\nabla^i \kappa} \\ \lesssim&  (a^{\frac{i-1}{2}} \lvert u \rvert)  \cdot \left(  \lvert u_{\infty} \rvert^i \hspace{.5mm} \lVert \nabla^i \kappa \rVert_{L^2(S_{u_{\infty},\ubar})} + \int_{u_{\infty}}^u \lvert u^\prime \rvert^i \hspace{.5mm} \lVert G \rVert_{L^2(S_{u^{\prime},\ubar})}\hspace{.5mm} \text{d}u^\prime  \right) \\ \lesssim& a^{\frac{i-1}{2}} \lvert u_{\infty} \rvert^{i+1}  \lVert \nabla^i \kappa \rVert_{L^2(S_{u_{\infty},\ubar})} +  \int_{u_{\infty}}^u a^{\frac{i-1}{2}} \lvert u^{\prime} \rvert^{i+1} \lVert G \rVert_{L^2(S_{u^\prime, \ubar})} \hspace{.5mm} \text{d}u^{\prime}.  \end{align*}In the last inequality we have used the facts that $\lvert u \rvert \leq \lvert u_{\infty}\rvert$, $\lvert u \rvert \leq \lvert u^\prime \rvert$ for $\lvert u^\prime \rvert$ in the range given above. From this we conclude that
		
		\begin{align*}
		&\scaletwoSu{\aln \kappa}\\
		 \lesssim& \lVert \aln \kappa \rVert_{\mathcal{L}^{2}_{(sc)}(S_{u_{\infty},\ubar})} + \int_{u_{\infty}}^u \frac{a}{\lvert u^\prime \rvert^2} \scaletwoSuprime{(\al)^i G}\hspace{.5mm} \text{d}u^\prime \\ \lesssim& \lVert \aln \kappa \rVert_{\mathcal{L}^{2}_{(sc)}(S_{u_{\infty},\ubar})} + \int_{u_{\infty}}^u \frac{a}{\lvert u^\prime \rvert^2} \sum_{i_1+i_2+i_3+i_4=i} \scaletwoSuprime{(\al)^i \nabla^{i_1}\psi^{i_2}\nabla^{i_3}(\psi,\chih)\nabla^{i_4}\Psi }\hspace{.5mm} \text{d}u^\prime \\ &+ \int_{u_{\infty}}^u \frac{a}{\lvert u^\prime \rvert^2} \sum_{i_1+i_2+i_3+i_4=i} \scaletwoSuprime{(\al)^i \nabla^{i_1}\psi^{i_2}\nabla^{i_3}(\psi,\chibarhat)\nabla^{i_4+1}\psi }\hspace{.5mm} \text{d}u^\prime \\ &+ \int_{u_{\infty}}^u \frac{a}{\lvert u^\prime \rvert^2} \sum_{i_1+i_2+i_3+i_4=i} \scaletwoSuprime{(\al)^i \nabla^{i_1}\psi^{i_2}\nabla^{i_3}\Y \nabla^{i_4+1}(\Y,\alpha_F) }\hspace{.5mm} \text{d}u^\prime \\ &+ \int_{u_{\infty}}^u \frac{a}{\lvert u^\prime \rvert^2} \sum_{i_1+i_2+i_3+i_4=i} \scaletwoSuprime{(\al)^i \nabla^{i_1}\psi^{i_2}\nabla^{i_3}(\Y,\alpha_F) \nabla^{i_4+1}\Y }\hspace{.5mm} \text{d}u^\prime \\  &+ \int_{u_{\infty}}^u \frac{a}{\lvert u^\prime \rvert^2} \sum_{i_1+i_2+i_3+i_4+i_5=i} \scaletwoSuprime{(\al)^i \nabla^{i_1}\psi^{i_2}\nabla^{i_3}(\psi,\chibarhat,\tr\chibar,\chihat) \nabla^{i_4}(\alpha_F,\Y)\nabla^{i_5}\Y }\hspace{.5mm} \text{d}u^\prime \\  &+ \int_{u_{\infty}}^u \frac{a}{\lvert u^\prime \rvert^2} \sum_{i_1+i_2+i_3+i_4+i_5=i} \scaletwoSuprime{(\al)^i \nabla^{i_1}\psi^{i_2}\nabla^{i_3}\psi \nabla^{i_4}\psi \nabla^{i_5} \psi }\hspace{.5mm} \text{d}u^\prime \\   &+ \int_{u_{\infty}}^u \frac{a}{\lvert u^\prime \rvert^2} \sum_{i_1+i_2+i_3+i_4=i} \scaletwoSuprime{(\al)^i \nabla^{i_1}\psi^{i_2}\nabla^{i_3}(\psi,\chibarhat,\tildetr)\nabla^{i_4}\kappa }\hspace{.5mm} \text{d}u^\prime \\  &+ \int_{u_{\infty}}^u \frac{a}{\lvert u^\prime \rvert^2} \sum_{i_1+i_2+i_3+i_4+1=i} \scaletwoSuprime{(\al)^i \nabla^{i_1}\psi^{i_2+1}\nabla^{i_3}\tr\chibar \nabla^{i_4}\kappa }\hspace{.5mm} \text{d}u^\prime .
		\end{align*} By raising the above to the second power and integrating in $\ubar$ we get
		
		\begin{align*}
		&\intubar \scaletwoSuubarprime{\aln \kappa}^2 \hspace{.5mm} \text{d}\ubar^{\prime}\\ 
		\lesssim& \lVert \aln \kappa \rVert_{\mathcal{L}^{2}_{(sc)}(H_{u_{\infty}})}^2 + \intubar  \frac{a}{\lvert u \rvert} \intu \frac{a}{\lvert u^\prime \rvert^2}\scaletwoSuprimeubarprime{(\al)^i G}^2\duprime \dubarprime \\ =& \lVert \aln \kappa \rVert_{\mathcal{L}^{2}_{(sc)}(H_{u_{\infty}})}^2 + \frac{a}{\lvert u \rvert} \intu  \frac{a}{\lvert u^{\prime} \rvert^2} \left(\intubar \scaletwoSuprimeubarprime{(\al)^i G}^2\dubarprime \right) \duprime \\ \lesssim& \lVert \aln \kappa \rVert_{\mathcal{L}^{2}_{(sc)}(H_{u_{\infty}})}^2 + \intubar \frac{a}{\lvert u \rvert} \int_{u_{\infty}}^u \frac{a}{\lvert u^\prime \rvert^2} \sum_{i_1+i_2+i_3+i_4=i} \scaletwoSuprimeubarprime{(\al)^i \nabla^{i_1}\psi^{i_2}\nabla^{i_3}(\psi,\chih)\nabla^{i_4}\Psi }^2 \hspace{.5mm} \text{d}u^\prime \dubarprime \\ &+ \intubar \frac{a}{\lvert u \rvert} \int_{u_{\infty}}^u \frac{a}{\lvert u^\prime \rvert^2} \sum_{i_1+i_2+i_3+i_4=i} \scaletwoSuprimeubarprime{(\al)^i \nabla^{i_1}\psi^{i_2}\nabla^{i_3}(\psi,\chibarhat)\nabla^{i_4+1}\psi }^2 \hspace{.5mm} \text{d}u^\prime \dubarprime \\ &+ \intubar \frac{a}{\lvert u \rvert} \int_{u_{\infty}}^u \frac{a}{\lvert u^\prime \rvert^2} \sum_{i_1+i_2+i_3+i_4=i} \scaletwoSuprimeubarprime{(\al)^i \nabla^{i_1}\psi^{i_2}\nabla^{i_3}\Y \nabla^{i_4+1}(\Y,\alpha_F) }^2 \hspace{.5mm} \text{d}u^\prime \dubarprime \\ &+ \intubar \frac{a}{\lvert u \rvert} \int_{u_{\infty}}^u \frac{a}{\lvert u^\prime \rvert^2} \sum_{i_1+i_2+i_3+i_4=i} \scaletwoSuprimeubarprime{(\al)^i \nabla^{i_1}\psi^{i_2}\nabla^{i_3}(\Y,\alpha_F) \nabla^{i_4+1}\Y }^2 \hspace{.5mm} \text{d}u^\prime \dubarprime\\  &+\intubar \frac{a}{\lvert u \rvert} \int_{u_{\infty}}^u \frac{a}{\lvert u^\prime \rvert^2} \sum_{i_1+i_2+i_3+i_4+i_5=i} \scaletwoSuprimeubarprime{(\al)^i \nabla^{i_1}\psi^{i_2}\nabla^{i_3}(\psi,\chibarhat,\tr\chibar,\chihat) \nabla^{i_4}(\alpha_F,\Y)\nabla^{i_5}\Y }^2 \hspace{.5mm} \text{d}u^\prime \dubarprime \\  &+ \intubar \frac{a}{\lvert u \rvert}\int_{u_{\infty}}^u \frac{a}{\lvert u^\prime \rvert^2} \sum_{i_1+i_2+i_3+i_4+i_5=i} \scaletwoSuprimeubarprime{(\al)^i \nabla^{i_1}\psi^{i_2}\nabla^{i_3}\psi \nabla^{i_4}\psi \nabla^{i_5} \psi }^2\hspace{.5mm} \text{d}u^\prime \dubarprime \\   &+\intubar \frac{a}{\lvert u \rvert} \int_{u_{\infty}}^u \frac{a}{\lvert u^\prime \rvert^2} \sum_{i_1+i_2+i_3+i_4=i} \scaletwoSuprimeubarprime{(\al)^i \nabla^{i_1}\psi^{i_2}\nabla^{i_3}(\psi,\chibarhat,\tildetr)\nabla^{i_4}\kappa }^2\hspace{.5mm} \text{d}u^\prime \dubarprime\\  &+\intubar \frac{a}{\lvert u \rvert} \int_{u_{\infty}}^u \frac{a}{\lvert u^\prime \rvert^2} \sum_{i_1+i_2+i_3+i_4+1=i} \scaletwoSuprimeubarprime{(\al)^i \nabla^{i_1}\psi^{i_2+1}\nabla^{i_3}\tr\chibar \nabla^{i_4}\kappa }^2\hspace{.5mm} \text{d}u^\prime \dubarprime\\ :=&  T_1+ \dots +T_9.
		\end{align*}We bound each term separately. \begin{itemize}
			\item There holds \begin{align*} T_1 =  \lVert \aln \kappa \rVert_{\mathcal{L}^{2}_{(sc)}(H_{u_{\infty}})}^2\lesssim&\lVert \aln \tbeta \rVert_{\mathcal{L}^{2}_{(sc)}(H_{u_{\infty}})}^2 +\frac{1}{a}  \sum_{i \leq 11}\lVert \aln \omega \rVert_{\mathcal{L}^{2}_{(sc)}(H_{u_{\infty}})}^2 \\ \lesssim& \mathcal{R}[\tbeta]^2 + \left(\mathcal{I}^{(0)}\right)^2 +1 \lesssim \mathcal{R}[\tbeta]^2 +1 .     \end{align*}
			\item There holds \begin{align*}T_2=&\intubar \sum_{i_1+i_2+i_3+i_4=i} \scaletwoSuprimeubarprime{(\al)^i \nabla^{i_1}\psi^{i_2} \nabla^{i_3}(\psi,\chihat) \nabla^{i_4}\Psi}^2 \dubarprime \\ \lesssim& \frac{a \cdot O^4}{\lvert u^{\prime} \rvert^2} + \intubar a \ScaletwoSuprimeubarprime{\left( \frac{\psi}{\al},\frac{\chihat}{\al} \right) \cdot (\al\nabla)^{10} \Psi}^2 \dubarprime \lesssim \frac{a \cdot O^4}{\lvert u^{\prime} \rvert^2} + \frac{a \cdot O^2}{\lvert u^\prime \rvert^2 } \mathcal{R}[\Psi]^2.\end{align*} Therefore \begin{equation}
			T_2 \lesssim \frac{a}{\lvert u \rvert} \intu \frac{a}{\lvert u^\prime \rvert^2} \cdot \left(\frac{a \cdot O^4}{\lvert u^{\prime} \rvert^2} + \frac{a \cdot O^2}{\lvert u^\prime \rvert^2 } \mathcal{R}[\Psi]^2 \right)\duprime  \lesssim \frac{a^3 \cdot O^4}{\lvert u \rvert^4} + \frac{a^3 \cdot O^2 \cdot R^2}{\lvert u \rvert^4} \lesssim 1.
			\end{equation}
			\item For the third term \[\intubar \frac{a}{\lvert u \rvert} \int_{u_{\infty}}^u \frac{a}{\lvert u^\prime \rvert^2} \sum_{i_1+i_2+i_3+i_4=i} \scaletwoSuprimeubarprime{(\al)^i \nabla^{i_1}\psi^{i_2}\nabla^{i_3}(\psi,\chibarhat)\nabla^{i_4+1}\psi }^2 \hspace{.5mm} \text{d}u^\prime \dubarprime \]we bound separately the cases where $i_4=10$ and where not. In the former case, we need to distinguish the subcases where the $\psi$-term in $\nabla^{i_4+1}\psi$ belongs to those components that are bounded in the $\scaletwoHu{\cdot}$-norm and those bounded in the $\scaletwoHbaru{\cdot}$-norm in the bootstrap assumption \eqref{bootstrapelliptic}. \begin{itemize}
				\item When $i_4 < 10$ we can bound $T_3$ by $1$, using Proposition \ref{usefulstatements}.
				
				\item When $i_4 =10$ we have \begin{align*} &\intubar \frac{a}{\lvert u \rvert} \int_{u_{\infty}}^u \frac{a}{\lvert u^\prime \rvert^2} \cdot \frac{\upr^2}{a}  \ScaletwoSuprimeubarprime{(\al)^{10} \left( \frac{\al}{\upr} \psi, \frac{\al}{\upr}\chibarhat \right) \nabla^{11}\psi }^2 \hspace{.5mm} \text{d}u^\prime \dubarprime \\ =& \intubar \frac{a}{\lvert u \rvert} \int_{u_{\infty}}^u  \frac{O^2}{\lvert u^\prime \rvert^2} \scaletwoSuprimeubarprime{(\al)^{10} \nabla^{11}\psi }^2 \hspace{.5mm} \text{d}u^\prime \dubarprime \\ =& \intubar \frac{O^2}{ \lvert u \rvert} \int_{u_{\infty}}^u  \frac{a}{\lvert u^\prime \rvert^2} \scaletwoSuprimeubarprime{(\al)^{10} \nabla^{11}\psi }^2 \hspace{.5mm} \text{d}u^\prime \dubarprime \\ \lesssim& \frac{O^2}{\lvert u \rvert}\cdot O_{11}^2 + \frac{a}{\lvert u \rvert}\cdot \frac{O^2}{\lvert u \rvert}\cdot O_{11}^2.  \end{align*}In the last inequality we have distinguished the cases according to the exact form of $\psi$ in \eqref{bootstrapelliptic}, so that we are able to bound it by $O_{11}$.
				
			\end{itemize}
			\item There holds \begin{align*}
			T_4 =& \intubar \frac{a}{\lvert u \rvert} \int_{u_{\infty}}^u \frac{a}{\lvert u^\prime \rvert^2} \sum_{i_1+i_2+i_3+i_4=i} \scaletwoSuprimeubarprime{(\al)^i \nabla^{i_1}\psi^{i_2}\nabla^{i_3}\Y \nabla^{i_4+1}(\Y,\alpha_F) }^2 \hspace{.5mm} \text{d}u^\prime \dubarprime \\ \lesssim& \frac{a^2 \cdot O^4}{\lvert u \rvert^4} + \intubar \frac{a}{\lvert u \rvert} \int_{u_{\infty}}^u \frac{a^2}{\lvert u^\prime \rvert^2}  \ScaletwoSuprimeubarprime{(\al)^{10} \Y \nabla^{11}\left(\frac{\Y}{\al}, \frac{\alpha_F}{\al} \right) }^2 \hspace{.5mm} \text{d}u^\prime \dubarprime\\ \lesssim& \frac{a^2 \cdot O^4}{\lvert u \rvert^4} + \frac{a^3 \cdot O^2 \cdot \mathcal{F}^2[\alpha_F]}{\lvert u \rvert^4}+\frac{a^2 \cdot O^2 \cdot \mathcal{F}^2[\rho_F,\sigma_F]}{\lvert u \rvert^4}.
			\end{align*}Here we have calculated explicitly all the possible pairs that appear in the schematic $\Y \nabla (\Y,\alpha_F)$ and those are $\rho_F\nabla \rho_F, \sigma_F \nabla \sigma_F$ and $\alphabar_F \nabla \alpha_F$.
			\item Similarly, there holds \begin{align*}
			T_5 =& \intubar \frac{a}{\lvert u \rvert} \int_{u_{\infty}}^u \frac{a}{\lvert u^\prime \rvert^2} \sum_{i_1+i_2+i_3+i_4=i} \scaletwoSuprimeubarprime{(\al)^i \nabla^{i_1}\psi^{i_2}\nabla^{i_3}(\alpha_F,\Y) \nabla^{i_4}\Y }^2 \hspace{.5mm} \text{d}u^\prime \dubarprime \\ \lesssim& \frac{a^2 \cdot O^4}{\lvert u \rvert^4} + \intubar \frac{a}{\lvert u \rvert} \int_{u_{\infty}}^u \frac{a^2}{\lvert u^\prime \rvert^2}  \ScaletwoSuprimeubarprime{(\al)^{10} \left( \frac{\alpha_F}{\al}, \frac{\Y}{\al} \right) \nabla^{11}\Y }^2 \hspace{.5mm} \text{d}u^\prime \dubarprime \\ \lesssim& \frac{a^2 \cdot O^4}{\lvert u \rvert^4} + \frac{a^3 \cdot O^2 \cdot \underline{\mathcal{F}}^2\hspace{.5mm}[\alphabar_F]}{\lvert u \rvert^4}+\frac{a^2 \cdot O^2 \cdot \mathcal{F}^2[\rho_F,\sigma_F]}{\lvert u \rvert^4} .\end{align*}Here we have calculated explicitly all the possible pairs that appear in the schematic $(\alpha_F,\Y)\nabla\Y$ and those are $\alpha_F \nabla \alphabar_F, \rho_F\nabla \rho_F$ and $\sigma_F \nabla  \sigma_F$.
			
			\item There holds \begin{align*}
			T_6 =& \intubar \frac{a}{\lvert u \rvert} \int_{u_{\infty}}^u \frac{a}{\lvert u^\prime \rvert^2} \sum_{i_1+i_2+i_3+i_4+i_5=i} \scaletwoSuprimeubarprime{(\al)^i \nabla^{i_1}\psi^{i_2}\nabla^{i_3}(\psi,\chibarhat,\tr\chibar,\chihat) \nabla^{i_4}(\alpha_F,\Y)\nabla^{i_5}\Y }^2 \hspace{.5mm} \text{d}u^\prime \dubarprime \\ \lesssim& \intubar \frac{a}{\lvert u \rvert} \intu \frac{a}{\lvert u^\prime \rvert^2} \cdot \frac{O^6}{a}\duprime \dubarprime = \frac{a \cdot O^6}{\lvert u \rvert^2}.
			\end{align*}Here we have estimated \begin{align*}
			    &\int_{u_{\infty}}^u \frac{a}{\lvert u^\prime \rvert^2} \sum_{i_1+i_2+i_3+i_4+i_5=i} \scaletwoSuprimeubarprime{(\al)^i \nabla^{i_1}\psi^{i_2}\nabla^{i_3}(\psi,\chibarhat,\tr\chibar,\chihat) \nabla^{i_4}(\alpha_F,\Y)\nabla^{i_5}\Y }^2 \hspace{.5mm} \text{d}u^\prime \\ =&\intu \frac{a}{\upr^2}\cdot \frac{\upr^4}{a^2}\cdot a \cdot \ScaletwoSuprimeubarprime{(\al)^i \nabla^{i_1}\psi^{i_2}\nabla^{i_3}\left(\frac{a(\psi,\chibarhat,\tr\chibar,\chihat)}{\upr^2} \right) \nabla^{i_4}\left(\frac{\alpha_F}{\al}, \frac{\Y}{\al}\right)\nabla^{i_5}\Y }^2 \hspace{.5mm} \text{d}u^\prime \\ \lesssim& \intu \frac{a}{\upr^2}\cdot \frac{\upr^4}{a^2}\cdot a \cdot \frac{O^6}{\upr^4} \duprime \lesssim \frac{O^6}{u}.
			\end{align*}
			
			\item There holds \begin{align*}
			T_7 =& \intubar \frac{a}{\lvert u \rvert}\int_{u_{\infty}}^u \frac{a}{\lvert u^\prime \rvert^2} \sum_{i_1+i_2+i_3+i_4+i_5=i} \scaletwoSuprime{(\al)^i \nabla^{i_1}\psi^{i_2}\nabla^{i_3}\psi \nabla^{i_4}\psi \nabla^{i_5} \psi }^2\hspace{.5mm} \text{d}u^\prime \dubarprime \\ \lesssim& \frac{a}{\lvert u \rvert} \cdot \frac{a \cdot O^6}{\lvert u \rvert^5}.
			\end{align*} \item The final two terms can be absorbed to the left by Gr\"onwall's inequality, by virtue of schematically containing the term $\nabla^{i_4}\kappa$. 
		\end{itemize}
		
		\par \noindent From the following $\div-\curl$ system

		\[ \div \nabla \omega = \div \kappa + \frac{1}{2}\nabla \tbeta,     \]\[ \curl \nabla \omega = 0,    \] \[    \div \nabla \omegad = \curl \kappa + \frac{1}{2} \curl \tbeta, \] \[ \curl \nabla \omegad = 0  \]and Proposition \ref{mainellipticlemma} we have that  
		
		\begin{align*}
		&\scaletwoSu{a^5 \nabla^{11} (\omega, \omegad)}\\ \lesssim& \sum_{j=0}^{10} \scaletwoSu{(\al \nabla)^j \kappa }+ \scaletwoSu{(\al \nabla)^i \tbeta}+ \frac{1}{\al} \sum_{j=0}^{10}\scaletwoSu{(\al \nabla)^j (\omega,\omegad)}.
		\end{align*}Passing to $\scaletwoHu{\cdot}$-norms, we arrive at \be  \scaletwoHu{(\al)^{10}\nabla^{11}(\omega,\omegabard)}\lesssim \mathcal{R}[\tbeta]+1.           \ee
		
	\end{proof}
	\par \noindent	We move on to top order estimates for $\eta$. 
	
	\begin{proposition} \label{etaellipticprop}
		Under the assumptions of Theorem \ref{main1} and the bootstrap assumptions \eqref{bootstrapbounds}-\eqref{bootstrapelliptic}, there holds \[ \frac{a}{\lvert u \rvert}  \scaletwoHu{a^5 \nabla^{11} \eta} + \scaletwoHbaru{a^5 \nabla^{11} \eta} \lesssim \mathcal{R}+1.     \]
    \end{proposition}
	
	\begin{proof}
		Introduce the quantity \[ \mu = -\div \eta - \rho.  \]Our goal is to derive a $\nabla_4$-transport equation for $\mu$. Recall the commutation formula, for a $1-$form $U$ \begin{align*}
		[\nabla_4,\text{div}] \hspace{.5mm} U = - \frac{1}{2}\tr\chi \hspace{.5mm} \div U- \chihat \cdot \nabla U- \tbeta \cdot U + \frac{1}{2}(\eta+\etabar) \cdot \nabla_4 U - \etabar \cdot \chibarhat \cdot U- \frac{1}{2}\tr\chi \hspace{.5mm} \etabar \cdot U+ \tr\chi \hspace{.5mm} \etabar \cdot U.
		\end{align*} In particular \begin{align*}
		\nabla_4 \hspace{.5mm} \div \eta = \text{div}\left(\nabla_4 \eta\right) +	[\nabla_4,\text{div}] \hspace{.5mm} \eta = \div(\chihat, \tr \chi \hspace{.5mm}\gamma)\cdot (\eta-\etabar) + (\chihat,\tr\chi \hspace{.5mm} \gamma)\cdot \div (\eta-\etabar) - \div \tbeta\\ - \frac{1}{2}\tr\chi \hspace{.5mm} \div \eta- \chihat \cdot \nabla \eta- \tbeta \cdot \eta + \frac{1}{2}(\eta+\etabar) \cdot \nabla_4 \eta - \etabar \cdot \chibarhat \cdot \eta- \frac{1}{2}\tr\chi \hspace{.5mm} \etabar \cdot\eta+ \tr\chi \hspace{.5mm} \etabar \cdot \eta.
		\end{align*}Schematically, this rewrites as 
		
		\begin{equation}
		\nabla_4 (\text{div} \hspace{.5mm} \eta) + \div \tbeta = (\psi,\chihat) \cdot \nabla (\eta, \etabar) + \psi \nabla (\psi,\chihat) + \psi \cdot \Psi + \psi \cdot (\psi,\chihat,\chibarhat) \cdot \psi + \psi \cdot \alpha_F \cdot \Y. 
		\end{equation}
		Moreover, \[\nabla_4  \rho - \div \tbeta =  (\psi,\chibarhat)\cdot(\alpha,\Psi) + \alpha_F \cdot \nabla (\rho_F,\sigma_F) + (\rho_F,\sigma_F) \cdot \nabla (\alpha_F, \rho_F,\sigma_F)+ (\psi,\chibarhat,\tr\chibar, \chih)\cdot (\alpha_F,\Y)\cdot (\alpha_F,\Y)  . \]Consequently, $\mu$ satisfies the following transport equation:
		
		\begin{equation}
\begin{split}
		\nabla_4 \mu = (\psi,\chihat) \cdot \nabla(\eta,\etabar)+ \psi \nabla (\psi,\chihat)  + \alpha_F \cdot \nabla (\rho_F,\sigma_F) + (\rho_F,\sigma_F) \cdot \nabla \alpha_F + (\psi,\chibarhat)\cdot(\alpha,\Psi) \\+\psi \cdot (\psi,\chihat,\chibarhat) \cdot \psi + (\psi,\chibarhat,\tr\chibar, \chih)\cdot (\alpha_F,\Y)\cdot (\alpha_F,\Y).  \label{mueq}
		\end{split}
\end{equation}Commuting \eqref{mueq} with $i \leq 10$ angular derivatives we arrive at
		
		\begin{align*}
		&\nabla_4 \nabla^i \mu\\ =& \sum_{i_1+i_2+i_3+i_4=i} \nabla^{i_1}\psi^{i_2}\nabla^{i_3}(\psi,\chihat)\nabla^{i_4+1}(\eta,\etabar)+ \sum_{i_1+i_2+i_3+i_4=i} \nabla^{i_1}\psi^{i_2}\nabla^{i_3} \psi \nabla^{i_4+1}(\psi,\chihat) \\ &+ \sum_{i_1+i_2+i_3+i_4=i} \nabla^{i_1}\psi^{i_2}\nabla^{i_3} \alpha_F \nabla^{i_4+1}(\rho_F,\sigma_F) + \sum_{i_1+i_2+i_3+i_4=i} \nabla^{i_1}\psi^{i_2}\nabla^{i_3} (\rho_F,\sigma_F) \nabla^{i_4+1}(\alpha_F, \rho_F,\sigma_F) \\ &+ \sum_{i_1+i_2+i_3+i_4=i} \nabla^{i_1}\psi^{i_2}\nabla^{i_3}(\psi,\chibarhat)\nabla^{i_4}(\alpha,\Psi) + \sum_{i_1+i_2+i_3+i_4+i_5=i} \nabla^{i_1}\psi^{i_2}\nabla^{i_3}\psi \nabla^{i_4}(\psi,\chihat,\chibarhat)\nabla^{i_5}\psi\\ &+\sum_{i_1+i_2+i_3+i_4+i_5=i} \nabla^{i_1}\psi^{i_2}\nabla^{i_3}(\psi,\chibarhat,\tr\chibar,\chihat)\nabla^{i_4}(\alpha_F,\Y)\nabla^{i_5}(\alpha_F,\Y)\\ &+ \sum_{i_1+i_2+i_3+i_4=i} \nabla^{i_1}\psi^{i_2}\nabla^{i_3}(\psi,\chihat)\nabla^{i_4}\mu.
		\end{align*}We now pass to scale-invariant norms. Noticing that $\restri{(\nabla^{i}\mu)}{\Hbar_0} = 0$, we can apply Proposition \ref{prop36} to obtain
		
		\begin{align*}
		&\scaletwoSu{\aln \mu}\\ \lesssim& \intubar  \sum_{i_1+i_2+i_3+i_4=i} \ScaletwoSuubarprime{(\al)^i \nabla^{i_1}\psi^{i_2}\nabla^{i_3}(\psi,\chihat)\nabla^{i_4+1}(\eta,\etabar)} \dubarprime \\ &+\intubar \sum_{i_1+i_2+i_3+i_4=i} \scaletwoSuubarprime{(\al)^i\nabla^{i_1}\psi^{i_2} \nabla^{i_3}\psi \nabla^{i_4+1}(\psi,\chihat)} \dubarprime \\  &+\intubar \sum_{i_1+i_2+i_3+i_4=i} \scaletwoSuubarprime{(\al)^i\nabla^{i_1}\psi^{i_2} \nabla^{i_3}\alpha_F \nabla^{i_4+1}(\rho_F,\sigma_F)} \dubarprime \\  &+\intubar \sum_{i_1+i_2+i_3+i_4=i} \scaletwoSuubarprime{(\al)^i\nabla^{i_1}\psi^{i_2} \nabla^{i_3}(\rho_F,\sigma_F) \nabla^{i_4+1}(\alpha_F, \rho_F,\sigma_F)} \dubarprime \\  &+\intubar \sum_{i_1+i_2+i_3+i_4=i} \scaletwoSuubarprime{(\al)^i\nabla^{i_1}\psi^{i_2} \nabla^{i_3}(\psi,\chibarhat) \nabla^{i_4}(\alpha,\Psi)}\dubarprime \\ &+ \intubar \sum_{i_1+i_2+i_3+i_4+i_5=i} \scaletwoSuubarprime{(\al)^i\nabla^{i_1}\psi^{i_2} \nabla^{i_3}\psi \nabla^{i_4}(\psi,\chihat, \chibarhat) \nabla^{i_4}\psi} \dubarprime \\ &+  \intubar \sum_{i_1+i_2+i_3+i_4+i_5=i} \scaletwoSuubarprime{(\al)^i\nabla^{i_1}\psi^{i_2} \nabla^{i_3}(\psi,\chihat, \tr\chibar, \chibarhat) \nabla^{i_4}(\alpha_F,\Y) \nabla^{i_5} (\alpha_F,\Y)} \dubarprime \\ &+\intubar \sum_{i_1+i_2+i_3+i_4=i} \scaletwoSuubarprime{(\al)^i\nabla^{i_1}\psi^{i_2} \nabla^{i_3}(\psi,\chihat) \nabla^{i_4}\mu }\dubarprime\\
		 =& I_1+ \dots + I_8.
		\end{align*}\begin{itemize}
			\item  We have \begin{align*}
			I_1 \leq& \intubar \ScaletwoSuubarprime{(\al)^{i+1}\left( \frac{\psi}{\al}, \frac{\chihat}{\al}\right) \nabla^{i+1}(\eta,\etabar) } \dubarprime + \frac{\al \cdot O^2}{\lvert u \rvert} \\ \lesssim& \frac{\al \cdot O}{\lvert u \rvert}\cdot \scaletwoHu{a^5 \nabla^{11} (\eta,\etabar)} + \frac{\al \cdot O^2}{\lvert u \rvert} \leq \frac{O}{\al}\cdot O_{11}[\eta, \etabar] + \frac{\al \cdot O^2}{\lvert u \rvert} \lesssim 1.
			\end{align*}
			\item We have  \begin{equation}
			I_2 \lesssim \intubar \ScaletwoSuubarprime{(\al)^{11}\psi \nabla^{11}\left(\frac{\psi}{\al},\frac{\chihat}{\al}\right) } \dubarprime + \frac{\al \cdot O^2}{\lvert u \rvert}  \lesssim \frac{\al \cdot O \cdot O_{11}}{\lvert u \rvert} + \frac{\al \cdot O^2}{\lvert u \rvert}.
			\end{equation}
			
			\item We have \begin{equation}
			I_3 \lesssim \intubar \ScaletwoSuubarprime{\left( \frac{\alpha_F}{\al} \right) (\al \nabla)^{11} \Y} \dubarprime +  \frac{\al \cdot O^2}{\lvert u \rvert} \\ \lesssim \frac{\al \cdot O \cdot \mathcal{F}[\Y]}{\lvert u \rvert} +\frac{\al \cdot O^2}{\lvert u \rvert}.
			\end{equation}
			
			\item Similarly, we have \begin{equation}
			I_4 \lesssim \intubar \ScaletwoSuubarprime{\Y (\al \nabla)^{11} \left( \frac{\Y}{\al}, \frac{\alpha_F}{\al}    \right)} \dubarprime +  \frac{\al \cdot O^2}{\lvert u \rvert} \\ \lesssim \frac{\al \cdot O \cdot \mathcal{F}[\alpha_F]}{\lvert u \rvert} +\frac{\al \cdot O^2}{\lvert u \rvert}.
			\end{equation}
			\item There holds
			
			\begin{align*}
			I_5 \lesssim& \intubar \frac{\upr}{\al} \cdot  \ScaletwoSuubarprime{ \left (\frac{\al}{u}\psi,\frac{\al}{u}\chibarhat \right) \hspace{.5mm}\cdot (\al \nabla)^{10}(\alpha,\Psi)} \dubarprime +1\\ \lesssim&  \intubar \frac{u}{\al} \cdot \frac{O_{\infty}[\chibarhat]}{u} \cdot  \scaletwoSuubarprime{  (\al \nabla)^{10}(\alpha,\Psi)} \dubarprime + 1      \lesssim \mathcal{R}[\alpha]+1,
			\end{align*}since $O_{\infty}[\chibarhat]\lesssim 1$ by Proposition \ref{chibarhatproposition}.
			
			\item There holds \begin{equation}
			I_6 \lesssim \frac{O^6}{\lvert u \rvert \cdot \al}.
			\end{equation}
			\item There holds \begin{equation}
			I_7 \lesssim 1,
			\end{equation}just as in the term $J_6$ in the proof of Proposition \ref{L2curvature2}.
			\item Finally, the term $I_8,$ after expanding $\mu = -\div \eta - \rho$, can be controlled by $I_1 +I_5$.
		\end{itemize}

		\par \noindent	Consequently, \be \scaletwoSu{(\al \nabla)^i \mu} \lesssim \mathcal{R}[\alpha]+1. \ee
		
		\par \noindent	Now observe the div--curl system (the second equation is given schematically):
		
		\[ \div \eta = -\mu - \rho, \]
		\[  \curl \eta = \sigma + \chibarhat \wedge \chihat + \Y \cdot \Y. \]

		\par \noindent	So that, applying Proposition \ref{mainellipticlemma} , we have
		
		\begin{equation}
\begin{split}\label{etaresults}
		&\scaletwoSu{(\al \nabla)^{11}
			\eta}\\ \lesssim& \al \sum_{i\leq 10} \big( \scaletwoSu{(\al \nabla)^{i}\mu} + \scaletwoSu{(\al \nabla)^i (\rho,\sigma)}\\ &+ u \cdot \ScaletwoSu{(\al \nabla)^i \left(\frac{\al}{u}\chibarhat \cdot \frac{\chihat}{\al}\right)} + \scaletwoSu{(\al \nabla)^i (\Y \cdot \Y)} \big) + \sum_{i \leq 10} \scaletwoSu{(\al \nabla)^i \eta}.
		\end{split}
\end{equation}
		
		\par \noindent Integrating along the $\ubar$-direction and raising to the second power, we arrive at
		
		\begin{equation}
		\frac{a}{\lvert u \rvert} \scaletwoHu{a^5 \nabla^{11} \eta} \lesssim \mathcal{R} +1.
		\end{equation}In a similar way, using \eqref{etaresults}, we get \[ \scaletwoHbaru{a^5 \nabla^{11} \eta} \lesssim \mathcal{R} +1.     \]	\end{proof}
	\par \noindent 	We move on to estimates for $\etabar$.
	
	\begin{proposition} \label{etabarellipticprop}
		Under the assumptions of Theorem \ref{main1} and the bootstrap assumptions \eqref{bootstrapbounds}-\eqref{bootstrapelliptic}, there holds \[ \frac{a}{\lvert u \rvert}    \scaletwoHu{a^5 \nabla^{11}\etabar} \lesssim \mathcal{R}+1. \]
	\end{proposition}
	
	\begin{proof} 
		Introduce $\mubar$ defined by   \[  \mubar= -\div \etabar -\rho. \]We then have the Hodge system for $\etabar$
		
		\[   \div \etabar = -\mubar- \rho, \]
		\[     \curl \etabar = -\sigma - \frac{1}{2}\chibarhat \wedge \chihat. \]
		
		\par \noindent For a $1-$form $U_b$ we have
		
		\[ [\nabla_3,\text{div}]U = - \frac{1}{2}\tr\chibar \hspace{.5mm} \text{div}U -\chibarhat \cdot \nabla U - \tbetabar \cdot U    + \frac{1}{2}(\eta+\etabar)\nabla_3 U - \eta\cdot \chihat \cdot U - \frac{1}{2}\tr\chibar \eta \cdot U + \tr\chibar \eta \cdot U.   \]
		
		\vspace{3mm}
		
		\par \noindent Consequently, \begin{align*} \nabla_3 \hspace{.5mm} \div \etabar =& \div (\nabla_3 \hspace{.5mm} \etabar) + [\nabla_3,\text{div}]\hspace{.5mm} \etabar\\ =& \div\left( - \chibarhat \cdot (\etabar-\eta )  - \frac{1}{2}\tr\chibar \cdot (\etabar-\eta)+\tbetabar - \alphabar_F\cdot (\rho_F,\sigma_F)     \right) + [\nabla_3, \div]\hspace{.5mm}\etabar  \\=& -(\div \hspace{.5mm} \chibarhat) \cdot (\etabar-\eta) - \chibarhat \cdot \div(\etabar-\eta) - \frac{1}{2}\tr\chibar \div \etabar + \frac{1}{2}\tr\chibar \div \eta\\ &-\frac{1}{2}(\etabar-\eta)\div(\tildetr) + \div\tbetabar - \alphabar_F \div (\rho_F,\sigma_F)-(\rho_F,\sigma_F) \div \alphabar_F \\ &- \frac{1}{2}\tr\chibar \div \etabar -\chibarhat \cdot \nabla \etabar - \tbetabar \cdot \eta + \frac{1}{2}(\eta+\etabar)\cdot \left( \psi\cdot(\chibarhat,\tr\chibar) +\tbetabar + \Y \cdot \Y \right) + \frac{1}{2} tr\chibar\cdot \eta \cdot \etabar . \end{align*}We thus have the (semi--)schematic identity
		
		\begin{align*}
		\nabla_3 \hspace{.5mm} \div \hspace{.5mm} \etabar + \tr\chibar \hspace{.5mm} \div \hspace{.5mm} \etabar - \div \tbetabar =& \psi \cdot \div \chibarhat + \chibarhat \nabla (\eta,\etabar)+ \tr\chibar \nabla \eta + \psi \hspace{.5mm} \div(\tildetr) +  \Y \nabla (\rho_F,\sigma_F)\\ &+ \Y \hspace{.5mm} \nabla \alphabar_F + \psi \cdot \tbetabar + \psi \cdot (\chibarhat,\tr\chibar)\cdot \psi + \psi \cdot \Y \cdot \Y.
		\end{align*}

		\par \noindent 	Also,
		
		\begin{align*}
		\nabla_3 \rho + \tr\chibar \rho + \div \tbetabar =&  - \frac{1}{2}\tr\chibar \cdot \rho + \chihat \cdot \alphabar + \psi \cdot \tbetabar + (\rho_F,\sigma_F) \nabla \alphabar_F\\ &+ \alphabar_F\nabla (\rho_F,\sigma_F) +(\psi, \tr\chibar) \cdot \Y\cdot \Y + (\psi,\chibarhat) \cdot (\Y,\alpha_F)\cdot \Y.
		\end{align*}
		
		\par \noindent Combining the above two equations, $\mubar$ satisfies the following transport equation:
		
		\begin{equation}
\begin{split}
		&\nabla_3 \mubar + \tr\chibar \hspace{.5mm} \mubar\\ =&    \psi \cdot \nabla \chibarhat + \chibarhat \hspace{.5mm} \nabla(\eta,\etabar) + \tr\chibar \nabla \eta + \psi \nabla(\tildetr) + \alphabar_F \nabla(\rho_F,\sigma_F) \\ &+ (\rho_F,\sigma_F)\nabla \alphabar_F + \psi \cdot \tbetabar + \tr\chibar \cdot \rho + \psi \cdot (\chibarhat,\tr\chibar)\cdot \psi + (\psi,\tr\chibar)\cdot \Y \cdot \Y + (\psi,\chibarhat)\cdot (\Y,\alpha_F)\cdot \Y. \label{mubareq1}
		\end{split}
\end{equation}By commuting \eqref{mubareq1} with $i\leq 10$ angular derivatives, we arrive at 
		
		\begin{align*}
		&\nabla_3 \nabla^i \mubar + \frac{i+2}{2}\tr\chibar \hspace{.5mm} \mubar\\ =& \sum_{ i_1+i_2+i_3+i_4=i} \nabla^{i_1}\psi^{i_2} \nabla^{i_3}\psi \nabla^{i_4+1}\chibarhat +  \sum_{ i_1+i_2+i_3+i_4=i} \nabla^{i_1}\psi^{i_2} \nabla^{i_3} \chibarhat \nabla^{i_4+1} (\eta,\etabar)\\ &+ \sum_{ i_1+i_2+i_3+i_4=i} \nabla^{i_1}\psi^{i_2} \nabla^{i_3} \tr\chibar \nabla^{i_4+1}\eta +  \sum_{ i_1+i_2+i_3+i_4=i} \nabla^{i_1}\psi^{i_2} \nabla^{i_3} \psi \nabla^{i_4+1}\tildetr \\ &+  \sum_{ i_1+i_2+i_3+i_4=i} \nabla^{i_1}\psi^{i_2} \nabla^{i_3} \alpha_F \nabla^{i_4+1}(\rho_F,\sigma_F) +  \sum_{ i_1+i_2+i_3+i_4=i} \nabla^{i_1}\psi^{i_2} \nabla^{i_3} (\rho_F,\sigma_F) \nabla^{i_4+1}\alphabar_F \\ &+  \sum_{ i_1+i_2+i_3+i_4=i} \nabla^{i_1}\psi^{i_2} \nabla^{i_3} \psi \nabla^{i_4}\tbetabar+  \sum_{ i_1+i_2+i_3+i_4=i} \nabla^{i_1}\psi^{i_2} \nabla^{i_3} \tr\chibar \nabla^{i_4}\rho \\ &+  \sum_{ i_1+i_2+i_3+i_4+i_5=i} \nabla^{i_1}\psi^{i_2} \nabla^{i_3}\psi \nabla^{i_4}(\chibarhat,\tr\chibar) \nabla^{i_5}\psi + \sum_{ i_1+i_2+i_3+i_4+i_5=i} \nabla^{i_1}\psi^{i_2} \nabla^{i_3} (\psi,\tr\chibar) \nabla^{i_4}\Y\nabla^{i_5}\Y \\ &+\sum_{ i_1+i_2+i_3+i_4+i_5=i} \nabla^{i_1}\psi^{i_2} \nabla^{i_3} (\psi,\chibarhat) \nabla^{i_4}(\Y,\alpha_F)\nabla^{i_5}\Y + \sum_{i_1 + i_2 + 1=i} \nabla^{i_1+1}\tr\chibar \nabla^{i_2}\mubar\\ &+ \sum_{i_1+i_2+i_3+i_4+1=i} \nabla^{i_1}\psi^{i_2+1}\nabla^{i_3}\tr\chibar \nabla^{i_4}\mubar + \sum_{i_1+i_2+i_3+i_4=i } \nabla^{i_1}\psi^{i_2}\nabla^{i_3}(\psi,\chibarhat,\tildetr)\nabla^{i_4}\mubar\\ :=& G.
		\end{align*}   By Proposition \ref{prop37} we can bound \[ \lvert u \rvert^{i+1}\twoSu{\nabla^i \mubar} \lesssim \lvert u_{\infty} \rvert^{i+1} \lVert \nabla^i \mubar \rVert_{L^{2}(S_{u_\infty,\ubar})} + \intu \lvert u^\prime \rvert^{i+1} \lVert G \rVert_{L^{2}(S_{u^\prime,\ubar})} \duprime.       \]We have $s_2(\nabla^i \mubar)=\frac{i+2}{2}$ and $s_2(G) = \frac{i+4}{2}$. By passing to scale-invariant norms, we have \begin{align*} &\frac{a}{\lvert u \rvert} \scaletwoSu{\aln \mubar}\\ \lesssim& \frac{a}{\lvert u_\infty \rvert } \lVert \aln \mubar \rVert_{\mathcal{L}^2_{(sc)}(S_{u_\infty,\ubar})} + \intu \frac{a^2}{\lvert u^\prime \rvert^3}\scaletwoSuprime{(\al)^i G}\duprime\\ \lesssim&    \frac{a}{\lvert u_\infty \rvert } \lVert \aln \mubar \rVert_{\mathcal{L}^2_{(sc)}(S_{u_\infty,\ubar})} + \intu \frac{a^2}{\lvert u^\prime \rvert^3}\ScaletwoSuprime{(\al)^i \sum_{i_1+i_2+i_3+i_4=i}\nabla^{i_1}\psi^{i_2}\nabla^{i_3}\psi \nabla^{i_4+1}\chibarhat} \duprime   \\  &+ \intu \frac{a^2}{\lvert u^\prime \rvert^3}\ScaletwoSuprime{(\al)^i \sum_{i_1+i_2+i_3+i_4=i}\nabla^{i_1}\psi^{i_2}\nabla^{i_3}\chibarhat \nabla^{i_4+1}(\eta,\etabar)} \duprime \\ &+ \intu \frac{a^2}{\lvert u^\prime \rvert^3}\ScaletwoSuprime{(\al)^i \sum_{i_1+i_2+i_3+i_4=i}\nabla^{i_1}\psi^{i_2}\nabla^{i_3}\tr\chibar \nabla^{i_4+1}\eta} \duprime \\ &+ \intu \frac{a^2}{\lvert u^\prime \rvert^3}\ScaletwoSuprime{(\al)^i \sum_{i_1+i_2+i_3+i_4=i}\nabla^{i_1}\psi^{i_2}\nabla^{i_3}\psi \nabla^{i_4+1}\tildetr} \duprime \\  &+ \intu \frac{a^2}{\lvert u^\prime \rvert^3}\ScaletwoSuprime{(\al)^i \sum_{i_1+i_2+i_3+i_4=i}\nabla^{i_1}\psi^{i_2}\nabla^{i_3}\alphabar_F \nabla^{i_4+1}(\rho_F,\sigma_F)}   \duprime \\  &+ \intu \frac{a^2}{\lvert u^\prime \rvert^3}\ScaletwoSuprime{(\al)^i \sum_{i_1+i_2+i_3+i_4=i}\nabla^{i_1}\psi^{i_2}\nabla^{i_3}(\rho_F,\sigma_F) \nabla^{i_4+1}\alphabar_F} \duprime    \\  &+ \intu \frac{a^2}{\lvert u^\prime \rvert^3}\ScaletwoSuprime{(\al)^i \sum_{i_1+i_2+i_3+i_4=i}\nabla^{i_1}\psi^{i_2}\nabla^{i_3}\psi \nabla^{i_4}\tbetabar} \duprime  \\  &+ \intu \frac{a^2}{\lvert u^\prime \rvert^3}\ScaletwoSuprime{(\al)^i \sum_{i_1+i_2+i_3+i_4=i}\nabla^{i_1}\psi^{i_2}\nabla^{i_3}\tr\chibar\nabla^{i_4}\rho} \duprime \\  &+ \intu \frac{a^2}{\lvert u^\prime \rvert^3}\ScaletwoSuprime{(\al)^i \sum_{i_1+i_2+i_3+i_4+i_5=i}\nabla^{i_1}\psi^{i_2}\nabla^{i_3}\psi \nabla^{i_4}(\chibarhat,\tr\chibar) \nabla^{i_5}\psi} \duprime \\ &+ \intu \frac{a^2}{\lvert u^\prime \rvert^3}\ScaletwoSuprime{(\al)^i \sum_{i_1+i_2+i_3+i_4+i_5=i}\nabla^{i_1}\psi^{i_2}\nabla^{i_3}(\psi,\tr\chibar) \nabla^{i_4}\Y \nabla^{i_5}\Y} \duprime\\ &+ \intu \frac{a^2}{\lvert u^\prime \rvert^3}\ScaletwoSuprime{(\al)^i \sum_{i_1+i_2+i_3+i_4+i_5=i}\nabla^{i_1}\psi^{i_2}\nabla^{i_3}(\psi,\chihat) \nabla^{i_4}(\Y,\alpha_F) \nabla^{i_5}\Y} \duprime \\ &+ \intu \frac{a^2}{\lvert u^\prime \rvert^3}\ScaletwoSuprime{(\al)^i \sum_{i_1+i_2+1=i} \nabla^{i_1+1}\tr\chibar \nabla^{i_2}\mubar} \duprime \\ &+ \intu \frac{a^2}{\lvert u^\prime \rvert^3}\ScaletwoSuprime{(\al)^i \sum_{i_1+i_2+i_3+i_4+1=i}\nabla^{i_1}\psi^{i_2+1}\nabla^{i_3}\tr\chibar \nabla^{i_4}\mubar } \duprime \\ &+ \intu \frac{a^2}{\lvert u^\prime \rvert^3}\ScaletwoSuprime{(\al)^i \sum_{i_1+i_2+i_3+i_4=i}\nabla^{i_1}\psi^{i_2}\nabla^{i_3}(\psi,\chibarhat,\tildetr) \nabla^{i_4}\mubar} \duprime\\ :=& T_1 +\dots + T_{15}. \end{align*}We bound $T_1$ to $T_{15}$ individually.
		
		\begin{itemize}
			\item Given that $\mubar = -\div \etabar - \rho$ and the fact that $\mathcal{I}^{(0)}$ bounds up to 14 derivatives for $\etabar$ and $\rho$, there holds   \begin{equation}\begin{split}T_1 = \frac{a}{\lvert u_\infty \rvert } \lVert \aln \mubar \rVert_{\mathcal{L}^2_{(sc)}(S_{u_\infty,\ubar})}\lesssim \mathcal{I}^{(0)}\lesssim 1.\end{split}\end{equation}
			
			\item There holds \begin{align*} T_2 =&  \intu \frac{a^2}{\lvert u^\prime \rvert^3}\ScaletwoSuprime{(\al)^i \sum_{i_1+i_2+i_3+i_4=i}\nabla^{i_1}\psi^{i_2}\nabla^{i_3}\psi \nabla^{i_4+1}\chibarhat} \duprime \\ \lesssim& \intu \frac{a^2}{\upr^3} \scaletwoSuprime{a^5\cdot \psi \cdot \nabla^{11}\chibarhat} \duprime \\ &+ \intu \frac{a^{\f32}}{\upr^2} \ScaletwoSuprime{(\al)^i \sum_{\substack{i_1+i_2+i_3+i_4=i,\\ i_4\leq 9}} \nabla^{i_1}\psi^{i_2}\nabla^{i_3}\psi \nabla^{i_4+1} \left(\frac{\al}{\upr}\chibarhat \right) } \duprime \\   \lesssim& \intu \frac{a^2}{\upr^3} \scaletwoSuprime{a^5 \cdot \psi \cdot \nabla^{11}\chibarhat} \duprime \\ &+ \intu \frac{a^{\f32}}{\upr^2} \cdot \frac{1}{\al} \sum_{i_1+i_2=i\leq 10} \scaletwoSuprime{(\al)^i \nabla^{i_1}\psi^{i_2+2}} \duprime   \\       \lesssim& 1 + \frac{a}{\lvert u \rvert^{\frac{3}{2}}}\cdot O \cdot O_{11}[\chibarhat] \lesssim 1 .      \end{align*}Here we have made use of Proposition \ref{chibarhattrchibarelliptic}.
			\item There holds \begin{equation}\begin{split} T_3 \lesssim&  \frac{a^{\frac{3}{2}}O^2}{\lvert u \rvert^2}+ \intu \frac{a^{\frac{3}{2}}}{\lvert u^\prime \rvert^2} \cdot \frac{O}{\lvert u^\prime \rvert} \ScaletwoSuprime{a^5 \nabla^{11}(\eta,\etabar)} \duprime \\ \lesssim&  \intu \frac{a^{\frac{3}{2}}}{\lvert u^\prime \rvert^2} \cdot \frac{O}{\lvert u^\prime \rvert} \ScaletwoSuprime{a^5 \nabla^{11}(\eta,\etabar)} \duprime +1.   \end{split}\end{equation}This term is controlled by Propositions \ref{etaellipticprop} and \ref{etabarellipticprop}.
			\item There holds \begin{equation}\begin{split} T_4 &\lesssim \frac{\al \cdot O^2}{\lvert u \rvert}+ \intu \frac{a^2}{\lvert u^\prime \rvert^3} \cdot \frac{\lvert u^\prime \rvert^2}{a} \cdot \frac{O_{\infty}[\tr\chibar]}{\lvert u^\prime \rvert} \scaletwoSuprime{a^5 \nabla^{11}\eta} \\  &\lesssim \intu \frac{a\cdot O_{\infty}[\tr\chibar]}{\lvert u^\prime \rvert^2}  \scaletwoSuprime{a^5 \nabla^{11}\eta} \duprime +1 \\ &\lesssim   \intu \frac{a}{\lvert u^\prime \rvert^2}  \scaletwoSuprime{a^5 \nabla^{11}\eta} \duprime
			+1 \\ &\lesssim \frac{a}{\lvert u \rvert} \scaletwoHu{a^5 \nabla^{11} \eta} +1 .  \end{split}\end{equation}Here we have used the fact that $O_{\infty}[\tr\chibar] \lesssim 1$, shown in Section 4.
			\item There holds \begin{equation}\begin{split} T_5 &\lesssim \frac{\al \cdot O^2}{\lvert u \rvert^2}+ \intu \frac{a^2}{\lvert u^\prime \rvert^3} \cdot \frac{\lvert u^\prime \rvert}{a} \cdot \frac{O}{\lvert u^\prime \rvert} \scaletwoSuprime{a^5 \nabla^{11}\tildetr} \duprime \\ &\lesssim \frac{\al \cdot O^2}{\lvert u\rvert^2} + \frac{\al \cdot O}{\lvert u\rvert^{\frac{3}{2}}}\scaletwoHbaru{a^5 \nabla^{11}\tildetr} \\&\lesssim \frac{O}{\lvert u\rvert^{\frac{1}{2}} \cdot \al}\cdot O_{11}[\tildetr] +1. \end{split}\end{equation}
			
			\item There holds \begin{align*}
			T_6 \lesssim& \intu \frac{a^2}{\lvert u^\prime \rvert^3} \cdot \frac{1}{\al} \cdot \frac{O^2}{\lvert u^\prime \rvert} \duprime +\intu \frac{a^2}{\lvert u^\prime \rvert^3} \cdot \frac{O}{\lvert u^\prime \rvert} \scaletwoSuprime{a^5 \nabla^{11}(\rho_F,\sigma_F)}  \duprime \\ \lesssim& \frac{a^{\frac{3}{2}} \cdot O^2}{\lvert u \rvert^3} + \left( \intu \frac{a}{\lvert u^\prime \rvert^2} \scaletwoSuprime{a^5 \nabla^{11}(\rho_F,\sigma_F)}^2 \duprime \right)^{\frac{1}{2}} \left(\intu \frac{a^3 \cdot O^2}{\lvert u^\prime \rvert^6} \duprime \right)^{\frac{1}{2}}\\ \lesssim &1 + \frac{a^{\frac{3}{2}} \cdot O}{\lvert u \rvert^{\frac{5}{2}}} \cdot \scaletwoHbaru{a^5 \nabla^{11}(\rho_F,\sigma_F)} \lesssim 1 + \frac{a^2 \cdot O}{\lvert u \rvert^{\frac{5}{2}}} \cdot \underline{\mathcal{F}}[\rho_F,\sigma_F] \lesssim 1. \end{align*}
			
			\item  Similarly, there holds \begin{equation}
			T_7 \lesssim \frac{a^{\frac{3}{2}} \cdot O^2}{\lvert u \rvert^3} + \frac{a^{\frac{3}{2}} \cdot O}{\lvert u \rvert^{\frac{5}{2}}} \cdot \scaletwoHbaru{a^5 \nabla^{11}\alphabar_F} \lesssim \frac{a^{\frac{3}{2}} \cdot O^2}{\lvert u \rvert^3} + \frac{a^{\frac{3}{2}} \cdot O}{\lvert u \rvert^{\frac{5}{2}}} \cdot \underline{\mathcal{F}}\hspace{.5mm}[\alphabar_F] \lesssim 1. 
			\end{equation}
			
			\item There holds \begin{align*} T_8 \lesssim& \intu \frac{a^2}{\lvert u^\prime \rvert^3} \cdot \frac{O}{\lvert u^\prime \rvert} \cdot \sum_{k\leq 10} \scaletwoSuprime{(\al \nabla)^k \tbetabar}  \duprime \\ \lesssim&  \left(\sum_{k \leq 10} \intu   \frac{a}{\lvert u^\prime \rvert^2} \scaletwoSuprime{(\al \nabla)^k \tbetabar}^2 \duprime \right)^{\frac{1}{2}} \left( \intu \frac{a^3 \cdot O^2}{\lvert u^\prime \rvert^6}  \duprime \right)^{\frac{1}{2}} \lesssim  \frac{a^{\frac{3}{2}}\cdot O}{\lvert u \rvert^{\frac{5}{2}}} \cdot \underline{\mathcal{R}}\hspace{.5mm} [\tbetabar] \lesssim 1.   \end{align*}
			
			\item There holds (this is the most marginal term) \begin{align*}
			T_9 \lesssim& \intu \frac{a^2}{\lvert u^\prime \rvert^3} \cdot \frac{O}{\lvert u^\prime \rvert} \cdot \frac{\lvert u^\prime \rvert^2}{a} \cdot \sum_{k\leq 10} \scaletwoSuprime{(\al \nabla)^k \rho}  \duprime \\ \lesssim&  \left(\sum_{k \leq 10} \intu \frac{a}{\lvert u^\prime \rvert^2} \scaletwoSuprime{(\al \nabla)^k \rho}^2 \duprime \right)^{\frac{1}{2}} \left( \intu \frac{a \cdot O^2}{\lvert u^\prime \rvert^2}  \duprime \right)^{\frac{1}{2}} \lesssim&  \frac{a^{\frac{1}{2}}\cdot O[\tr\chibar]}{\lvert u \rvert^{\frac{1}{2}}} \cdot \underline{\mathcal{R}}\hspace{.5mm} [\rho] \lesssim \underline{\mathcal{R}}[\rho].   
			\end{align*}
			
			\item There holds \begin{align*} T_{11}+T_{12} =& \intu \frac{a^2}{\lvert u^\prime \rvert^3}\ScaletwoSuprime{(\al)^i \sum_{i_1+i_2+i_3+i_4+i_5=i}\nabla^{i_1}\psi^{i_2}\nabla^{i_3}(\psi,\tr\chibar) \nabla^{i_4}\Y \nabla^{i_5}\Y} \duprime\\ &+ \intu \frac{a^2}{\lvert u^\prime \rvert^3}\ScaletwoSuprime{(\al)^i \sum_{i_1+i_2+i_3+i_4+i_5=i}\nabla^{i_1}\psi^{i_2}\nabla^{i_3}(\psi,\chihat) \nabla^{i_4}(\Y,\alpha_F) \nabla^{i_5}\Y} \duprime \\ \lesssim& \intu \frac{a^2}{\lvert u^\prime \rvert^3}\cdot \frac{\upr^2}{a}\cdot \ScaletwoSuprime{(\al)^i \sum_{i_1+i_2+i_3+i_4+i_5=i}\nabla^{i_1}\psi^{i_2}\nabla^{i_3}\left(\frac{a}{\upr^2} \psi,\frac{a}{\upr^2}\tr\chibar   \right) \nabla^{i_4}\Y \nabla^{i_5}\Y} \duprime\\ &+ \intu \frac{a^3}{\lvert u^\prime \rvert^3}\ScaletwoSuprime{(\al)^i \sum_{i_1+i_2+i_3+i_4+i_5=i}\nabla^{i_1}\psi^{i_2}\nabla^{i_3}\left( \frac{\psi}{\al},\frac{\chihat}{\al}\right) \nabla^{i_4}\left(\frac{\Y}{\al},\frac{\alpha_F}{\al} \right) \nabla^{i_5}\Y} \duprime \\ \lesssim& \intu \left( \frac{a}{\upr} + \frac{a^3}{\upr^3} \right) \cdot \frac{O^3}{\upr^2}\duprime \lesssim \frac{a\cdot O^3}{u^2} \lesssim 1. \end{align*}
			
			\item The last three terms can be controlled by Gr\"onwall's inequality. Indeed, \begin{align*}
			   & \intu \frac{a^2}{\lvert u^\prime \rvert^3}\ScaletwoSuprime{(\al)^i \sum_{i_1+i_2+1=i} \nabla^{i_1+1}\tr\chibar \nabla^{i_2}\mubar} \duprime \\ &+ \intu \frac{a^2}{\lvert u^\prime \rvert^3}\ScaletwoSuprime{(\al)^i \sum_{i_1+i_2+i_3+i_4+1=i}\nabla^{i_1}\psi^{i_2+1}\nabla^{i_3}\tr\chibar \nabla^{i_4}\mubar } \duprime \\ &+ \intu \frac{a^2}{\lvert u^\prime \rvert^3}\ScaletwoSuprime{(\al)^i \sum_{i_1+i_2+i_3+i_4=i}\nabla^{i_1}\psi^{i_2}\nabla^{i_3}(\psi,\chibarhat,\tildetr) \nabla^{i_4}\mubar} \duprime \\ \lesssim& \intu \frac{a}{\lvert u^\prime \rvert^2}\ScaletwoSuprime{(\al)^i \sum_{i_1+i_2=i} \nabla^{i_1}\left(\frac{a}{\upr}\tildetr\right) \nabla^{i_2}\mubar} \duprime \\ &+ \intu \frac{a^2}{\lvert u^\prime \rvert^3}\cdot \frac{\upr^2}{a}\cdot \ScaletwoSuprime{(\al)^i \sum_{i_1+i_2+i_3+i_4+1=i}\nabla^{i_1}\psi^{i_2+1}\nabla^{i_3}\left( \frac{a}{\upr^2}\tr\chibar \right) \nabla^{i_4}\mubar } \duprime \\ &+ \intu \frac{a^2}{\lvert u^\prime \rvert^3}\cdot \frac{\upr}{\al}\cdot \ScaletwoSuprime{(\al)^i \sum_{i_1+i_2+i_3+i_4=i}\nabla^{i_1}\psi^{i_2}\nabla^{i_3}\left(\frac{\al}{\upr}\psi,\frac{\al}{\upr}\chibarhat,\frac{\al}{\upr}\tildetr \right) \nabla^{i_4}\mubar} \duprime \\ \lesssim& \intu \left(\frac{a\cdot O}{\upr^3} + \frac{a\cdot O^2}{\upr^3} +\frac{a^{\f32}\cdot }{\upr^3} \right)\cdot \ScaletwoSuprime{\sum_{i_1\leq i} (\al \nabla)^{i_1}\mubar} \duprime.
			\end{align*}When $i_1=i$ the three terms in the parenthesis of the line above are integrable with respect to $u$ and so Gr\"onwall's inequality allows us to control the term. When $i_1<i$, we can use the definition $\mubar= -\div \etabar -\rho$ and bound the terms by the already established estimates of the previous sections.

		\end{itemize}Consequently, there holds
	
\begin{equation}\label{etabarresults}
	\frac{a}{\lvert u \rvert}\scaletwoSu{\aln \mubar} \lesssim  \intu \frac{a^{\frac{3}{2}} \cdot O}{\lvert u^\prime \rvert^3} \scaletwoSuprime{a^5\nabla^{11}(\eta,\etabar)} \duprime + \frac{a}{\lvert u \rvert}\scaletwoHu{a^5 \nabla^{11}\eta} +\underline{\mathcal{R}}[\rho]+1.
\end{equation}Now observe the $\div-\curl$ system (the second equation is given schematically):

\begin{gather}
\div \hspace{.5mm} \etabar = -\mubar-\rho, \\ \curl \hspace{.5mm} \etabar = -\sigma -\chibarhat \wedge \chihat +\Y\cdot \Y.
\end{gather}Consequently,

\begin{align*}
\frac{a}{\lvert u \rvert} \lVert a^5 \nabla^{11} \etabar \rVert_{L^{2}_{(sc)}(S_{u,\ubar})} \lesssim& \sum_{i \leq 10} \big(\frac{a}{\lvert u \rvert} \scaletwoSu{\aln \mubar} + \scaletwoSu{\aln (\rho,\sigma)} \\ &+ \scaletwoSu{\aln (\chibarhat \cdot \chihat)} + \scaletwoSu{\aln (\Y\cdot \Y)} \big) \\&+ \frac{1}{\al}\sum_{i \leq 10} \scaletwoSu{\aln \etabar}. 
\end{align*}By raising the above to the second power, integrating along $\ubar$ and using \eqref{etabarresults} along with Gr\"onwall's inequality  we can get that \be \frac{a}{\lvert u \rvert} \scaletwoHu{a^5 \nabla^{11} \etabar} \lesssim \mathcal{R} +1. \ee

	\end{proof}

	\vspace{3mm} \par \noindent We now prove the highest order bounds for $\omegabar$.
	
	\begin{proposition} \label{omegabarellipticprop}
		Under the assumptions of Theorem \ref{main1} and the bootstrap assumptions \eqref{bootstrapbounds}-\eqref{bootstrapelliptic}, we have 
		\[ \scaletwoHbaru{a^5 \nabla^{11} \omegabar}  \lesssim 1+\mathcal{R} +\mathcal{F}. \]
	\end{proposition}
	\begin{proof}
		Define the auxiliary function $\omegabard$ by \[ \nabla_4 \omegabard = \frac{1}{2}\sigma \]and trivial initial data along $\Hb_{0}$. We then define $\kappabar$ by \[   \kappabar = -\nabla \omegabar + \Hodge{\nabla} \omegabard - \frac{1}{2}\tbetabar. \] We need to obtain a transport equation for $\kappabar$. Notice that if we set $\langle \omegabar \rangle = (\omegabar, \omegabard)$, then \[  \kappabar=  \Hodge{\mathcal{D}}_1\langle \omegabar \rangle - \frac{1}{2}\tbetabar.        \] Recall the commutation formulae
		
		\[ [\nabla_4, \nabla] f = - \frac{1}{2}\tr\chi \nabla f - \chihat \cdot \nabla f + \frac{1}{2}(\eta+\etabar)D_4 f,            \]   \[ [\nabla_4,\Hodge{\nabla}]g = - \frac{1}{2}\tr\chi \Hodge{\nabla} g + \chihat \cdot \Hodge{\nabla} g +\frac{1}{2}(\Hodge{\eta} + \Hodge{\etabar})D_4 g.    \] Thus, for a pair of scalars $(f,g)$, there holds \be [\nabla_4 , \Hodge{\mathcal{D}}_1](f,g) = -\frac{1}{2}\tr\chi \Hodge{\mathcal{D}}_1(f,g) + \chihat \cdot (\nabla f + \Hodge{\nabla} g) - \frac{1}{2}(\eta+\etabar) \nabla_4 f+ \frac{1}{2}(\Hodge{\eta} + \Hodge{\etabar})\nabla_4 g.     \ee Now recall that \[ \nabla_4 \omegabar = \frac{1}{2} \rho + \psi \cdot \psi + \Y \cdot \Y := \frac{1}{2}\rho + F.   \]Therefore,
		
		\begin{equation}
		\nabla_4 \Hodge{\mathcal{D}}_1 \langle \omegabar \rangle+ \frac{1}{2} \nabla \rho - \frac{1}{2}\Hodge{\nabla}\sigma  = (\psi,\chihat) \cdot \nabla \psi + \Y \cdot \nabla \Y + \psi \cdot \Psi + \psi \cdot \psi \cdot \psi + \psi \cdot \Y \cdot \Y.
		\end{equation}Moreover,
		
		\begin{equation}
		\nabla_4 \tbetabar+ \tr\chi \tbetabar + \nabla \rho - \Hodge{\nabla} \sigma = (\psi,\chibarhat) \cdot \Psi + (\alpha_F, \Y) \cdot \nabla \Y + (\psi,\chihat,\chibarhat,\tr\chibar) \cdot (\alpha_F,\Y) \cdot \Y.
		\end{equation}Consequently, 
		
		\begin{equation}
		\nabla_4 \kappabar = (\psi,\chibarhat)\cdot \Psi + (\alpha_F,\Y)\cdot \nabla \Y + (\psi,\chihat) \cdot \nabla \psi + (\psi,\chihat,\chibarhat,\tr\chibar)\cdot (\alpha_F,\Y)\cdot \Y + \psi \cdot \psi \cdot \psi.
		\end{equation}Commuting with $i\leq 10$ angular derivatives, we get 
		
		\begin{align*}
		&\nabla_4 \nabla^i \kappabar\\ =& \sum_{i_1+i_2+i_3+i_4=i}\nabla^{i_1}\psi^{i_2}\nabla^{i_3}(\psi,\chibarhat) \nabla^{i_4}\Psi + \sum_{i_1+i_2+i_3+i_4=i}\nabla^{i_1}\psi^{i_2}\nabla^{i_3}(\alpha_F,\Y ) \nabla^{i_4+1}\Y \\ &+ \sum_{i_1+i_2+i_3+i_4=i}\nabla^{i_1}\psi^{i_2}\nabla^{i_3}(\psi,\chihat) \nabla^{i_4+1}\psi  + \sum_{i_1+i_2+i_3+i_4+i_5=i}\nabla^{i_1}\psi^{i_2}\nabla^{i_3} (\psi,\chihat,\chibarhat,\tr\chibar)\nabla^{i_4}(\alpha_F,\Y)\nabla^{i_5}\Y \\ &+  \sum_{i_1+i_2+i_3+i_4+i_5=i}\nabla^{i_1}\psi^{i_2}\nabla^{i_3} \psi \nabla^{i_4}\psi\nabla^{i_5}\psi + \sum_{i_1+i_2+i_3+i_4=i}\nabla^{i_1}\psi^{i_2}\nabla^{i_3}(\psi,\chihat) \nabla^{i_4}\kappabar. 
		\end{align*}Passing to scale-invariant norms, there holds
		
		\begin{align*}
		&\scaletwoSu{\aln \kappabar}\\ \lesssim& \intubar \sum_{i_1 + i_2 + i_3+i_4 =i} \scaletwoSuubarprime{(\al)^i \nabla^{i_1}\psi^{i_2}\nabla^{i_3}(\psi,\chibarhat) \nabla^{i_4}\Psi}\dubarprime \\ &+ \intubar \sum_{i_1 + i_2 + i_3+i_4 =i} \scaletwoSuubarprime{(\al)^i \nabla^{i_1}\psi^{i_2}\nabla^{i_3}(\alpha_F,\Y) \nabla^{i_4+1}\Y}\dubarprime \\ &+ \intubar \sum_{i_1 + i_2 + i_3+i_4 =i} \scaletwoSuubarprime{(\al)^i \nabla^{i_1}\psi^{i_2}\nabla^{i_3}(\psi,\chihat) \nabla^{i_4+1}\psi}\dubarprime \\ &+ \intubar \sum_{i_1 + i_2 + i_3+i_4+i_5 =i} \scaletwoSuubarprime{(\al)^i \nabla^{i_1}\psi^{i_2}\nabla^{i_3} (\psi,\chihat,\chibarhat,\tr\chibar)\nabla^{i_4}(\alpha_F,\Y)\nabla^{i_5}\Y} \dubarprime \\ &+ \intubar \sum_{i_1 + i_2 + i_3+i_4+i_5 =i} \scaletwoSuubarprime{(\al)^i \nabla^{i_1}\psi^{i_2}\nabla^{i_3} \psi\nabla^{i_4}\psi \nabla^{i_5}\psi} \dubarprime \\ &+ \intubar \sum_{i_1+i_2+i_3+i_4=i}  \scaletwoSuubarprime{(\al)^i\nabla^{i_1}\psi^{i_2}\nabla^{i_3}(\psi,\chihat) \nabla^{i_4}\kappabar} \dubarprime\\ :=& J_1 +\dots + J_6.
		\end{align*}We again estimate term by term.
		\begin{itemize}
			\item We have \begin{equation}\begin{split} &\scaletwoSuubarprime{(\al)^i \nabla^{i_1}\psi^{i_2}\nabla^{i_3}(\psi,\chibarhat) \nabla^{i_4}\Psi}\\ \lesssim&    \ScaletwoSuubarprime{\left(\frac{\al}{u}\psi,\frac{\al}{u}\chibarhat\right)(\al\nabla)^{10}\Psi} + \frac{u}{\al} \ScaletwoSuubarprime{(\al)^i \nabla^{i_1}\psi^{i_2}\nabla^{i_3}\left(\frac{\al}{u}\psi,\frac{\al}{u}\chibarhat\right) \nabla^{i_4}\Psi}     \\ \lesssim& \frac{O^2}{\al} + \frac{O_{\infty}[\chibarhat]}{\al} \scaletwoSuubarprime{(\al \nabla)^{10}\Psi}.  \end{split}\end{equation} Consequently, \[     J_1 \lesssim \frac{O^2}{\al}+ \frac{O_{\infty}[\chibarhat]\mathcal{R}[\Psi]}{\al} \lesssim 1.  \]     
			\item We have \begin{equation}\begin{split} &\scaletwoSuubarprime{(\al)^i \nabla^{i_1}\psi^{i_2}\nabla^{i_3}(\alpha_F,\Y) \nabla^{i_4+1}\Y} \\ =& \al \cdot\ScaletwoSuubarprime{(\al)^i \nabla^{i_1}\psi^{i_2}\nabla^{i_3}\left(\frac{\alpha_F}{\al},\frac{\Y}{\al} \right) \nabla^{i_4+1}\Y} \\ \lesssim& \frac{\al \cdot O^2}{\lvert u \rvert} + \frac{\al}{u}\cdot \left(O[\alpha_F]+ \frac{O[\Y]}{\al}\right) \scaletwoSuubarprime{a^5 \nabla^{11} \Y}.   \end{split}\end{equation}Consequently,
			
			\[ J_2 \lesssim      \frac{\al \cdot O^2}{\lvert u \rvert} + \mathcal{F}[\Y]+ \frac{\al \cdot O_{\infty}[\alpha_F] \cdot \mathcal{F}[\Y]}{u}.       \]
			\item We have  \begin{equation}\begin{split} &\scaletwoSuubarprime{(\al)^i \nabla^{i_1}\psi^{i_2}\nabla^{i_3}(\psi,\chihat) \nabla^{i_4+1}\psi} \\ =& \al \cdot \ScaletwoSuubarprime{(\al)^i \nabla^{i_1}\psi^{i_2}\nabla^{i_3}\left(\frac\psi\al,\frac\chihat\al \right) \nabla^{i_4+1}\psi}\\ \lesssim& \frac{\al \cdot O^2}{\lvert u \rvert} + \frac{\al \cdot O \cdot O_{11}}{\lvert u \rvert}.   \end{split}\end{equation} Then $J_3$ satisfies the same bound.
			
			\item We have \begin{equation}\begin{split} &\scaletwoSuubarprime{(\al)^i \nabla^{i_1}\psi^{i_2}\nabla^{i_3} (\psi,\chihat,\chibarhat,\tr\chibar)\nabla^{i_4}(\alpha_F,\Y)\nabla^{i_5}\Y}  \\ =& \frac{u^2}{a} \cdot \al \cdot  \ScaletwoSuubarprime{(\al)^i \nabla^{i_1}\psi^{i_2}\nabla^{i_3} \left(\frac{a}{u^2}\psi,\frac{a}{u^2}\chihat,\frac{a}{u^2}\chibarhat,\frac{a}{u^2}\tr\chibar\right)\nabla^{i_4}\left(\frac{\alpha_F}{\al},\frac{Y}{\al} \right)\nabla^{i_5}\Y}  \\ \lesssim& \frac{u^2}{\al} \cdot \frac{O^3}{u^2} \lesssim  \frac{O^3}{\al}.\end{split}\end{equation} Thus $J_4$ satisfies, upon integration, the same bound.
			
			\item There holds  \[ \sum_{i_1 + i_2 + i_3+i_4+i_5 =i} \scaletwoSuubarprime{(\al)^i \nabla^{i_1}\psi^{i_2}\nabla^{i_3} \psi \nabla^{i_4}\psi \nabla^{i_5}\psi}  \lesssim \frac{O^3}{\lvert u \rvert^2}.\]
			
			\item The final term can be absorbed to the left by a Gr\"onwall-type argument.
			
		\end{itemize}Overall, \[ \scaletwoSu{\aln \kappabar} \lesssim 1+\mathcal{R}+\mathcal{F}. \] By the following div-curl system \[     \div \nabla \omegabar = - \div \kappabar - \frac{1}{2}\div \tbetabar, \] \[ \curl \nabla \omegabar=0,\] \[ \curl \omegabard= \curl \kappabar + \frac{1}{2} \curl \tbetabar,\] \[ \div \nabla \omegabard =0, \]applying Proposition \ref{mainellipticlemma}, we have \begin{align*}
		\scaletwoSu{(\al)^{10} \nabla^{11}(\omegabar,\omegabard)} \lesssim&  \sum_{j=0}^{10}\left( \scaletwoSu{(\al \nabla)^j \kappabar} + \scaletwoSu{(\al \nabla)^j \tbetabar}\right) \\ &+\frac{1}{\al} \sum_{j=0}^{10}\scaletwoSu{(\al \nabla)^j (\omegabar,\omegabard)}.
		\end{align*} Integrating in the $u$--direction, we get the desired result.
	\end{proof} 
	Finally, we prove top order estimates for the remaining Ricci coefficients $\tr\chibar, \chibarhat$. 
	
	\begin{proposition} \label{chibarhattrchibarelliptic}
		Under the assumptions of Theorem \ref{main1} and the bootstrap assumptions \eqref{bootstrapbounds}-\eqref{bootstrapelliptic}, there holds \[    \intu \frac{a^2}{\lvert u\rvert^3 }\scaletwoSuprime{ a^5 \nabla^{11} \tr \chibar}\duprime \lesssim \mathcal{R}+\underline{\mathcal{R}}+1, \] \[ \intu \frac{a^{\frac{3}{2}}}{\upr^3}\scaletwoSuprime{a^5 \nabla^{11}\chibarhat} \duprime \lesssim 1.\]
	
	\end{proposition}
	\begin{proof}
		We begin with the equation \[  \nabla_3 \tildetr + \tr\chibar \tildetr = \frac{2}{\lvert u \rvert^2}(\Omega^{-1}-1)+\tildetr \tildetr +2 \omegabar \tr\chibar - \lvert \chibarhat \rvert^2 - \lvert \alphabar_F \rvert^2.\]Commuting this equation with $i$ angular derivatives, we get 	\begin{align*}  \nabla_3 \nabla^i \tildetr + \frac{i+2}{2}&\tr\chibar \tildetr = \sum_{i_1 + i_2 +i_3 = i}\nabla^{i_1} \psi^{i_2} \nabla^{i_3} \left(\frac{2}{\lvert u \rvert^2}(\Omega^{-1}-1)+\tildetr \tildetr +2 \omegabar \tr\chibar - \lvert \chibarhat \rvert^2 - \lvert \alphabar_F \rvert^2 \right)  \\ 
		&+ \sum_{i_1 + i_2 +i_3 + i_4 =i} \nabla^{i_1} \psi^{i_2} \nabla^{i_3}(\psi,\chibarhat, \tildetr)\nabla^{i_4} \tildetr       + \sum_{i_1 + i_2 +i_3 + i_4 +1 =i} \nabla^{i_1} \psi^{i_2+1} \nabla^{i_3} \tr\chibar \nabla^{i_4}\tildetr\\ &\quad \quad \quad:= \tilde{F}_i .        \end{align*}Rewriting in terms of scale-invariant norms, \begin{align*} \frac{a}{\lvert u \rvert} \scaletwoSu{(\al)^{i-1} \nabla^i \tildetr } \leq& \frac{a}{\lvert u_\infty \rvert} \rvert \lVert (\al)^{i-1} \nabla^i \tildetr  \rVert_{\mathcal{L}^2_{(sc)}(S_{u_{\infty},u})} +\int_{u_{\infty}}^{u}  \frac{a^2}{\lvert  u^\prime \rvert^3    } \lVert (\al)^{i-1} \tilde{F}_i\rVert_{\mathcal{L}^2_{(sc)}(S_{u^\prime, \ubar})}\text{d}u^\prime\\
		=&\frac{a}{\lvert u_\infty \rvert} \lVert (\al)^{i-1} \nabla^i \tildetr  \rVert_{\mathcal{L}^2_{(sc)}(S_{u_{\infty},u})}  + I_1+I_2+I_3,       \end{align*}where \[  \frac{a}{\lvert u \rvert} \scaletwoSu{(\al)^{i-1} \nabla^i \tildetr } \leq \frac{a}{\lvert u_\infty \rvert} \rvert \lVert (\al)^{i-1} \nabla^i \tildetr  \rVert_{\mathcal{L}^2_{(sc)}(S_{u_{\infty},u})} \lesssim 1.  \]We proceed to estimate $I_1$ to $I_3$.
		
	\begin{itemize}
		\item We can rewrite $I_1= I_{11}+I_{12}+I_{13}+I_{14}+I_{15}$ in the obvious way. We further decompose $I_{11}=I_{111}+I_{112}$. There holds \begin{align*}
		 I_{111} = &	\intu \frac{a^2}{\lvert u^\prime \rvert^3} \ScaletwoSuprime{(\al)^{i-1} \sum_{i_1 + i_2 + i_3=i-1} \nabla^{i_1}\psi^{i_2}\nabla^{i_3+1}\frac{2}{\lvert u \rvert^2}(\Omega^{-1}-1)} \duprime \\  =&\intu \frac{a^2}{\lvert u^\prime \rvert^5} \ScaletwoSuprime{(\al)^{i-1} \sum_{i_1 + i_2 =i-1} \nabla^{i_1}\p^{i_2+1}} \duprime \lesssim\f{a^2\cdot O}{|u|^4}. \end{align*}  
	Also,
	\begin{equation*}
\begin{split}
I_{112}=&\int_{u_{\infty}}^u \f{a^2}{|u'|^3}\|(\at)^i \sum_{i_1+i_2+i_3=i}\nab^{i_1}\p^{i_2}\nab^{i_3}(\f{\O^{-1}-1}{|u'|^2})\|_{L^2_{sc}{(S_{u',\ub})}}du'\\
=&\int_{u_{\infty}}^u |u'|^{i+1} \|\sum_{i_1+i_2+i_3=i}\nab^{i_1}\p^{i_2}\nab^{i_3}(\f{\O^{-1}-1}{|u'|^2})\|_{L^2{(S_{u',\ub})}}du' \quad (\mbox{in standard norms})\\
=&\int_{u_{\infty}}^u |u'|^{i+1} \|\sum_{i_1+i_2+i_3=i}\nab^{i_1}\p^{i_2}\nab^{i_3}(\f{\O^{-1}-1}{|u'|^2})\|_{L^2{(S_{u',\ub})}}du' \,\, (\mbox{Using} \,\, \f{\partial}{\partial \ub}\O^{-1}=2\o \Leftrightarrow \nab_4\O^{-1}=2\O^{-1}\o)\\
=&\int_{u_{\infty}}^u |u'|^{i+1} \|\sum_{i_1+i_2+i_3=i}\nab^{i_1}\p^{i_2}\nab^{i_3}[\f{1}{|u'|^2}\cdot \int_0^{\ub}2\o(u',\ub',\theta^1, \theta^2)d\ub']\|_{L^2{(S_{u',\ub})}}du' \\
=&\int_{u_{\infty}}^u |u'|^{i+1} \|\sum_{i_1+i_2+i_3=i}\nab^{i_1}\p^{i_2}[\f{1}{|u'|^2}\cdot \int_0^{\ub}2\nab^{i_3}\o(u',\ub',\theta^1, \theta^2)d\ub']\|_{L^2{(S_{u',\ub})}}du' \\
\leq& \int_{u_{\infty}}^u |u'|^{i+1} \sum_{i_1+i_2+i_3=i}\f{1}{|u'|^{i_1+i_2}}\cdot\f{1}{|u'|^2}\cdot\f{1}{|u'|^{i_3}}\cdot \f{\at}{|u'|^{\f12}}\cdot \big(\underline{\mathcal{R}}[\rho]+1\big)\, du'    (\mbox{by Proposition }  \ref{omegaprop})\\
\leq& \int_{u_{\infty}}^u \f{\at}{|u'|^{\f32}} \big(\underline{\mathcal{R}}[\rho]+1\big)\,du' \ls \underline{\mathcal{R}}[\rho]+1.
\end{split}
\end{equation*}

		 Similarly, there holds \begin{align*}
	 I_{12} =&	\intu \frac{a^2}{\lvert u^\prime \rvert^3} \ScaletwoSuprime{(\al)^{i-1} \sum_{i_1 + i_2 + i_3+i_4 =i} \nabla^{i_1}\psi^{i_2}\nabla^{i_3}\tildetr \nabla^{i_4}\tildetr} \duprime \\ =& \intu \frac{a^2}{\lvert u^\prime \rvert^3} \cdot \frac{\upr^2}{a^2} \ScaletwoSuprime{(\al)^{i-1} \sum_{i_1 + i_2 + i_3+i_4 =i} \nabla^{i_1}\psi^{i_2}\nabla^{i_3}\left(\frac{a}{\upr}\tildetr \right) \nabla^{i_4}\left(\frac{a}{\upr}\tildetr\right)} \duprime\\ \lesssim& \intu  \frac{a^{\frac{3}{2}}}{\lvert u^\prime \rvert^3} \cdot \frac{\lvert u^\prime \rvert^2}{a^2} \cdot\frac{O^2}{\lvert u^\prime \rvert} \duprime + \intu \frac{a^2}{\lvert u^\prime \rvert^3} \cdot \frac{\lvert u^\prime \rvert}{a} \cdot \frac{O}{\lvert u^\prime \rvert} \cdot \scaletwoSuprime{a^5 \nabla^{11}\tildetr} \duprime .
		\end{align*}Also, \begin{equation}\begin{split}
		 I_{13} &= \intu \frac{a^2}{\lvert u^\prime \rvert^3} \ScaletwoSuprime{(\al)^{i-1} \sum_{i_1 + i_2 + i_3+i_4 =i} \nabla^{i_1}\psi^{i_2}\nabla^{i_3}\omegabar \nabla^{i_4}\tr\chibar} \duprime \\  &= \intu \frac{a^2}{\lvert u^\prime \rvert^3} \cdot \frac{\upr^2}{a}\cdot  \ScaletwoSuprime{(\al)^{i-1} \sum_{i_1 + i_2 + i_3+i_4 =i} \nabla^{i_1}\psi^{i_2}\nabla^{i_3}\omegabar \nabla^{i_4}\left(\frac{a}{\upr^2}\tr\chibar\right)} \duprime \\&\lesssim \intu \frac{a^{\frac{3}{2}}}{\lvert u^\prime \rvert^3} \cdot \frac{\lvert u^\prime \rvert^2}{a}\cdot \frac{O^2}{\upr} \duprime  +\intu \frac{a^2}{\upr^3} \cdot \frac{\upr^2}{a}\scaletwoSu{a^5 \nabla^{11}\omegabar}\cdot \frac{O_{\infty}[\tr\chibar]}{\upr} \duprime\\ &\,\,+\intu \frac{a^2}{\upr^3} \cdot \frac{\upr}{a} \cdot \frac{O_{\infty}[\omegabar]}{\upr}\ScaletwoSuprime{a^5 \nabla^{11}\left(\frac{a}{\upr}\tr\chibar\right)}	\duprime \\ &\lesssim \frac{\al \cdot O^2}{\lvert u \rvert} + \intu \frac{a}{\lvert u^{\prime} \rvert^2} \scaletwoSuprime{a^5\nabla^{11}\omegabar} \duprime + (\text{a term handled by Gr\"onwall's inequality}).	\end{split}\end{equation}There also holds \begin{equation}\begin{split}\label{I13eq} I_{14} &=
		 \intu \frac{a^2}{\upr^3} \ScaletwoSuprime{(\al)^{i-1} \sum_{i_1+i_2+i_3+i_4=i} \nabla^{i_1}\psi^{i_2}\nabla^{i_3}\chibarhat \nabla^{i_4}\chibarhat}\duprime \\ &= \intu \frac{a^2}{\upr^3} \cdot\frac{\upr^2}{a}\cdot \ScaletwoSuprime{(\al)^{i-1} \sum_{i_1+i_2+i_3+i_4=i} \nabla^{i_1}\psi^{i_2}\nabla^{i_3}\left( \frac{\al}{\upr}\chibarhat \right) \nabla^{i_4}\left( \frac{\al}{\upr}\chibarhat\right)}\duprime \\&\lesssim \intu \frac{a^{\frac{3}{2}}}{\upr^3} \cdot \frac{\upr^2}{a}\cdot \frac{O^2}{\upr}\duprime +\intu \frac{a^2}{\upr^3} \cdot \frac{\upr}{\al}\cdot \frac{O_{\infty}[\chibarhat]}{ \upr}\scaletwoSuprime{a^5 \nabla^{11}\chibarhat}\duprime \\ &\lesssim \frac{\al\cdot O^2}{\lvert u \rvert^2}+ \intu \frac{a^{\frac{3}{2}}}{\upr^3} \scaletwoSuprime{a^5 \nabla^{11}\chibarhat}\duprime.
		 \end{split}\end{equation}Finally, for the term $I_{15}$ we have
		 
		 \begin{equation}\begin{split}
		  &\intu \frac{a^2}{\upr^3} \ScaletwoSuprime{(\al)^{i-1}\sum_{i_1+i_2+i_3 +i_4=i} \nabla^{i_1}\psi^{i_2}\nabla^{i_3}\alphabar_F \nabla^{i_4}\alphabar_F}\duprime \\ \lesssim &\intu \frac{a^{\frac{3}{2}}}{\upr^3}\cdot \frac{O^2}{\upr}\duprime +  \intu \frac{a^2}{\upr^3}\cdot \frac{O}{\upr}\cdot \scaletwoSuprime{a^5\nabla^{11}\alphabar_F} \duprime \\\lesssim  &1 + \frac{a^{\frac{3}{2}}\cdot O}{\lvert u \rvert^{\frac{5}{2}}} \cdot \scaletwoHbaru{a^5 \nabla^{11}\alphabar_F} \lesssim 1.
		 \end{split}\end{equation}
		 
		 \item There holds  \begin{align*}
		  &\intu \frac{a^2}{\upr^3} \ScaletwoSuprime{(\al)^{i-1}\sum_{i_1+i_2+i_3 +i_4=i}  \nabla^{i_1}\psi^{i_2}\nabla^{i_3}(\psi, \chibarhat,\tildetr) \nabla^{i_4}\tildetr}\duprime \\ \lesssim& \intu \frac{a^2}{\upr^3} \ScaletwoSuprime{(\al)^{i-1}\sum_{i_1+i_2=i} \nabla^{i_1}(\psi, \chibarhat,\tildetr) \nabla^{i_2}\tildetr}\duprime \\ &+ \intu \frac{a^2}{\upr^3} \ScaletwoSuprime{(\al)^{i-1}\sum_{i_1+i_2+i_3 +i_4+1=i} \nabla^{i_1}\psi^{i_2+1}\nabla^{i_3}(\psi, \chibarhat,\tildetr) \nabla^{i_4}\tildetr}\duprime \\ \lesssim& \intu \frac{a^2}{\upr^3} \ScaletwoSuprime{a^5(\psi, \chibarhat,\tildetr) \nabla^{11}\tildetr}\duprime \\ &+ \intu \frac{a^2}{\upr^3} \ScaletwoSuprime{a^5\nabla^{11} (\psi, \chibarhat,\tildetr) \tildetr}\duprime +1 \\ \lesssim& \intu \frac{a^2}{\upr^3} \cdot \frac{\upr}{\al} \cdot \frac{\upr}{a} \cdot \frac{O_{\infty}[\chibarhat]}{\upr} \cdot \ScaletwoSu{a^5 \nabla^{11} \left( \frac{a}{\upr}\tildetr \right)} \duprime \\ &+ \intu \frac{a^2}{\upr^3} \cdot \frac{\upr}{a} \cdot \frac{O_{\infty}[\tildetr]}{\upr} \cdot  \scaletwoSu{a^5 \nabla^{11}(\psi, \chibarhat, \tildetr)} \duprime +1 \\  \lesssim& \intu \frac{\al}{\upr^2} \cdot \ScaletwoSu{a^5 \nabla^{11} \left( \frac{a}{\upr}\tildetr \right)} \duprime \\ &+ \intu \frac{a}{\upr^3}  \cdot  \scaletwoSu{a^5 \nabla^{11}(\psi, \chibarhat, \tildetr)} \duprime +1.		  \end{align*}\par \noindent Noting that $\psi \in \begin{Bmatrix} \eta,\etabar \end{Bmatrix}$ and recalling Propositions \ref{etaellipticprop} and \ref{etabarellipticprop}, using Gr\"onwall's inequality, we can bound this term by \[   \intu \frac{a}{\upr^3}  \cdot  \scaletwoSu{a^5 \nabla^{11} \chibarhat} \duprime +1	. \]

		  \item For the last term there holds 
		  \begin{equation}\begin{split}
		       &\intu \frac{a^2}{\upr^3} \ScaletwoSuprime{(\al)^{i-1}\sum_{i_1+i_2+i_3 +i_4+1=i} \nabla^{i_1}\psi^{i_2+1}\nabla^{i_3}\tr\chibar \nabla^{i_4}\tildetr }\duprime \\  =   &\intu  \ScaletwoSuprime{(\al)^{i-1}\sum_{i_1+i_2+i_3 +i_4+1=i} \nabla^{i_1}\psi^{i_2+1}\nabla^{i_3}\left(\frac{a}{\upr^2}\tr\chibar\right) \nabla^{i_4}\left( \frac{a}{\upr}\tildetr\right) }\duprime      \\\lesssim &\intu \frac{a^2}{\upr^3}\cdot \frac{\upr^2}{a} \cdot \frac{\upr}{a} \ScaletwoSuprime{(\al)^{i-1}\sum_{i_1+i_2+i_3 +i_4+1=i} \nabla^{i_1}\psi^{i_2+1}\nabla^{i_3}\left( \frac{a}{\upr^2}\tr\chibar \right) \nabla^{i_4}\left( \frac{a}{\upr}\tildetr \right) }\duprime \\ \lesssim &\frac{O^3}{\al \cdot \lvert u \rvert} \lesssim 1.
		  \end{split}\end{equation}
	\end{itemize}We thus have, using Gr\"onwall's inequality, \be \label{trchiintermediate} \frac{a}{\lvert u \rvert}\scaletwoSu{a^5 \nabla^{11}\tildetr} \lesssim  1+\underline{\mathcal{R}}[\rho]+\intu \frac{a^{\frac{3}{2}}}{\upr^3} \scaletwoSuprime{a^5 \nabla^{11}\chibarhat}\duprime + \intu \frac{a}{\lvert u^{\prime} \rvert^2} \scaletwoSuprime{a^5\nabla^{11}\omegabar} \duprime .   \ee
		For $\chibarhat$, we have the constraint equation \[ \div\chibarhat = \frac{1}{2}\nabla \tildetr - \frac{1}{2}(\etabar-\eta)(\chibarhat-\frac{1}{2}\tr\chibar \gamma) + \tbetabar. \]Consequently, \begin{align*}
		\scaletwoSu{a^5 \nabla^{11}\chibarhat} \lesssim& \sum_{j \leq 10} \big( \scaletwoSu{(\al)^j \nabla^{j+1}\tildetr} + \scaletwoSu{(\al \nabla)^j \tbetabar} \\ &+ \ScaletwoSu{(\al)^j \sum_{j_1 +j_2=j} \nabla^{j_1}(\eta,\etabar)\nabla^{j_2}(\chibarhat,\tr\chibar)} \big) \\&+ \sum_{j \leq 10} \frac{1}{\al} \cdot \scaletwoSu{\aln \chibarhat}. 
		\end{align*}Integrating this along the incoming direction, we have
		
		\begin{equation}\begin{split} &\intu \frac{a^{\frac{3}{2}}}{\upr^3}\scaletwoSuprime{a^5 \nabla^{11}\chibarhat} \duprime \\ \lesssim \,\, & \intu \frac{a^{\frac{3}{2}}}{\upr^3}   \sum_{j \leq 10} \big( \scaletwoSuprime{(\al)^j \nabla^{j+1}\tildetr} + \scaletwoSuprime{(\al \nabla)^j \tbetabar} \\+ &\frac{\upr^2}{a} \ScaletwoSuprime{(\al)^j \sum_{j_1 +j_2=j} \nabla^{j_1}(\eta,\etabar)\nabla^{j_2}\left(\frac{a}{\upr^2}\chibarhat,\frac{a}{\upr^2}\tr\chibar\right)} \big) \duprime \\+ &\intu \frac{a}{\upr^3} \sum_{i\leq 10} \scaletwoSuprime{\aln \chibarhat} \duprime  \lesssim \intu \frac{a^{\frac{3}{2}}}{\upr^3}\scaletwoSuprime{a^5 \nabla^{11}\tildetr}\duprime + \frac{a^{\f32}}{\lvert u \rvert^2} \cdot R[\tbetabar] \\ + &\intu \frac{a^{\f32}}{\upr^3} \cdot \frac{\upr^2}{a} \cdot \frac{O^2}{\upr} \duprime + \intu \frac{a}{\upr^3} \cdot \frac{\upr}{\al} \duprime \\ \lesssim \,\, &\intu \frac{a^{\frac{3}{2}}}{\upr^3}\scaletwoSuprime{a^5 \nabla^{11}\tildetr}\duprime + 1. \end{split}\end{equation} Plugging this back to \eqref{trchiintermediate} and using Gr\"onwall's inequality, we get
		
		\begin{align*} \frac{a}{\lvert u \rvert} \scaletwoSu{a^5\nabla^{11}\tildetr} \lesssim& 1+\underline{\mathcal{R}}[\rho] + \intu \frac{a}{\lvert u^{\prime} \rvert^2} \scaletwoSuprime{a^5\nabla^{11}\omegabar} \duprime \\ \lesssim& 1 +\underline{\mathcal{R}}[\rho]+ \frac{a}{\lvert u \rvert}\scaletwoHbaru{a^5\nabla^{11}\omegabar} \lesssim \mathcal{R}+\underline{\mathcal{R}}+1. \end{align*}

\par \noindent Integrating in the $u$--direction we obtain

\[ \intu \frac{a^2}{\upr^3} \scaletwoSuprime{a^5 \nabla^{11}\tildetr} \lesssim \mathcal{R}+\underline{\mathcal{R}}+1. \] 
	
\begin{remark}
    In a similar fashion, we can obtain the following estimates:
            \be \left(\intu \frac{a^3}{\upr^4} \scaletwoSuprime{a^5 \nabla^{11} \tildetr}^2 \duprime \right)^{\f12} +\left(\intu \frac{a^2}{\upr^4} \scaletwoSuprime{a^5 \nabla^{11} \chibarhat}^2 \duprime \right)^{\f12}  \lesssim \mathcal{R} +\underline{\mathcal{R}}+1. \label{chibarhattrchibartraditionalformestimates} \ee
    
\end{remark}

\end{proof}

	\section{Energy estimates}
	
	In this section with scale invariant norms we will derive energy estimates for curvature components
	and their angular derivatives. Our goal is to show that \[ \mathcal{R}+\underline{\mathcal{R}}+\mathcal{F}+\underline{\mathcal{F}} \lesssim \left(\mathcal{I}^{(0)}\right)^4 +\left(\mathcal{I}^{(0)}\right)^2+ \mathcal{I}^{(0)}+1.     \]
	
\par \noindent 	We begin with the integration by parts formula.
	
	\subsection{Integration by parts}
	The following holds. Define $\mathcal{D}_{u,\ubar}:= (u_\infty, u)\times (0,\ubar)$. A direct computation yields the following:
	
	\begin{proposition} \label{byparts}
		Suppose $\phi_1, \phi_2$ are $r-$tensorfields. Then there holds 
		
		\begin{equation*}
		\int_{\mathcal{D}_{u,\ubar}} \phi_1 \nabla_4 \phi_2 +	\int_{\mathcal{D}_{u,\ubar}} \phi_2 \nabla_4 \phi_1  = \int_{\Hu}\phi_1 \phi_2 - \int_{\Hbar_0^{(u_{\infty}, u)}} \phi_1 \phi_2 + \int_{\mathcal{D}_{u,\ubar}} (2\omega-\tr\chi)\phi_1 \phi_2.
		\end{equation*}
	\end{proposition}

\begin{proposition} \label{energyprop2}
	Given an $r-$tensorfield $~^{(1)}\!\phi$ and an $(r-1)$-tensorfield $~^{(2)}\!\phi$, there holds \begin{equation*}
	\int_{\mathcal{D}_{u,\ubar}} ~^{(1)}\!\phi^{A_1\dots A_r} \nabla_{A_r}~^{(2)}\!\phi^{A_1\dots A_{r-1}} + 	\int_{\mathcal{D}_{u,\ubar}} \nabla^{A_r}~^{(1)}\!\phi_{A_1\dots A_r}~^{(2)}\!\phi^{A_1\dots A_{r-1}} \\ = -(\eta+\etabar) ~^{(1)}\hspace{.5mm}\!\phi^{(2)}\!\phi.
	\end{equation*} 
\end{proposition}Moreover, we shall require the following bound:

\begin{proposition}\label{energyprop3}
	Suppose $\phi$ is an $r-$tensorfield and let $\lambda_1 = 2\lambda_0 -1$. Then
	
	\begin{equation*}
		2\int_{\mathcal{D}_{u,\ubar}}\upr^{2\lambda_1}\phi (\nabla_3 + \lambda_0 \tr\chibar)\phi = \int_{\Hu} \upr^{2\lambda_1}\phi^2 -\int_{H_{u_{\infty}}^{(0,\ubar)}} \lvert u_{\infty} \rvert^{2\lambda_1}\phi^2 + \int_{\mathcal{D}_{u,\ubar}} \upr^{2\lambda_1} f\phi^2,
	\end{equation*}where $f$ satisfies the bound \[ f\lesssim \frac{O}{\lvert u \rvert^2}.  \] 
\end{proposition}

\begin{proof}
	There holds \begin{align*}
	&\frac{\text{d}}{\text{d}u} \left( \int_{\S} \lvert u \rvert^{2\lambda_1} \Omega \lvert \phi \rvert^2  \right)= \Lbar  \left( \int_{\S} \lvert u \rvert^{2\lambda_1} \Omega \lvert \phi \rvert^2  \right) \\ =& \int_{\S} \Omega^2 \left(2 \lvert u \rvert^{2\lambda_1}\langle \phi, \nabla_3 \phi + \lambda_0 \tr\chibar  \phi \rangle \right) + \int_{\S}\Omega^2 \left(  \lvert u \rvert^{2\lambda_1}\left(-2\frac{\lambda_1(e_3(u))}{\lvert u\rvert} +(1-2\lambda_0)\tr\chibar -2 \omegabar\right) \lvert \phi\rvert^2  \right).
	\end{align*}Here we have used that $\Lbar = \Omega e_3 = \frac{\partial}{\partial u} +b^A \frac{\partial}{\partial \theta^A}$. Immediate calculations imply that \[\big \lvert -2\frac{\lambda_1(e_3(u))}{\lvert u\rvert} +(1-2\lambda_0)\tr\chibar -2 \omegabar \big\rvert \lesssim \frac{O}{\lvert u \rvert^2} .   \]The proposition then follows by integrating in the slab $\mathcal{D}_{u,\ubar}$ and using the fundamental theorem of calculus.
\end{proof}
\subsection{The Hodge structure as an aid for energy estimates}

Observe that for $(\mathfrak{Y}_1, \mathfrak{Y}_2) \in \begin{Bmatrix} (\alpha, \tbeta), (\tbeta, (\rho,\sigma)) , ((\rho,\sigma),\tbetabar), (\tbetabar, \alphabar) \end{Bmatrix} \cup \begin{Bmatrix} (\alpha_F, (\rho_F, \sigma_F)), ((\rho_F,\sigma_F),\alphabar_F) \end{Bmatrix}$ we can write the equations for $\mathfrak{Y}_1$ and $\mathfrak{Y}_2$ in the following form:

\begin{gather}
\nabla_3 \mathfrak{Y}_1 + \left(\frac{1}{2}+s_2(\mathfrak{Y}_1) \right)\tr\chibar \mathfrak{Y}_1 - \mathcal{D}\mathfrak{Y}_2 = P_0, \\ \nabla_4 \mathfrak{Y}_2 - \Hodge{\mathcal{D}}\hspace{.5mm} \mathfrak{Y}_1 = Q_0.
\end{gather}	Here by $\mathcal{D}$ we denote a differential operator on $\S$ and by $\Hodge{\mathcal{D}}$ its $L^2$-adjoint. By commuting the above equations $i$ times, we arrive at

\begin{gather}\label{ref1}
\nabla_3 \nabla^i \mathfrak{Y}_1 + \left(\frac{i+1}{2}+s_2(\mathfrak{Y}_1) \right)\tr\chibar \mathfrak{Y}_1 - \mathcal{D}\nabla^i \mathfrak{Y}_2 = P_i, \\\label{ref2} \nabla_4 \nabla^i \mathfrak{Y}_2 - \Hodge{\mathcal{D}}\hspace{.5mm} \nabla^i \mathfrak{Y}_1 = Q_i.
\end{gather}

\par \noindent The purpose of this section is to prove the following:

\begin{proposition}
	Under the assumptions of Theorem \ref{main1} and the bootstrap assumptions \eqref{bootstrapbounds} and given a pair $(\mathfrak{Y}_1, \mathfrak{Y}_2)$ satisfying 
	\begin{gather*}
	\nabla_3\nabla^i \mathfrak{Y}_1 + \left( \frac{i+1}{2}+ s_2(\Y_1)\right) \tr\chibar \nabla^i \mathfrak{Y}_1 - \mathcal{D}\nabla^i \mathfrak{Y}_2 = P, \\ 
	\nabla_4 \nabla^i \mathfrak{Y}_2 - \Hodge{\mathcal{D}} \nabla^i \mathfrak{Y}_1 = Q,
	\end{gather*}the following inequality holds:

	\begin{align*}
	&\int_{H_u^{(0,\ubar)}} \scaletwoSuubarprime{\nabla^i \mathfrak{Y}_1}^2 \hspace{.5mm} \text{d}\ubar^\prime + \int_{\Hbar_{\ubar}^{(u_\infty,u)}} \frac{a}{\lvert u^\prime \rvert^2}\hspace{.5mm} \scaletwoSuprime{\nabla^i \mathfrak{Y}_2}^2 \hspace{.5mm} \text{d}u^\prime \\ 
	\lesssim& \int_{H_{u_\infty}^{(0,\ubar)}} \lVert{\nabla^i \mathfrak{Y}_1} \rVert^2_{\mathcal{L}^{2}_{(sc)}(S_{u_{\infty},\ubar^\prime})} \hspace{.5mm} \text{d}\ubar^\prime + \int_{\Hbar_{0}^{(u_\infty,u)}} \frac{a}{\lvert u^\prime \rvert^2}\hspace{.5mm} \lVert{\nabla^i \mathfrak{Y}_2}\rVert^2_{\mathcal{L}^2_{(sc)}(S_{u^\prime,0})} \hspace{.5mm} \text{d}u^\prime \\ 
	&+ \iint_{\mathcal{D}_{u,\ubar}} \frac{a}{\lvert u^\prime \rvert} \lVert \nabla^i \mathfrak{Y}_1 \cdot P \rVert_{\mathcal{L}^1_{(sc)}(S_{u^\prime, \ubar^\prime})} \hspace{.5mm} \text{d}u^\prime \text{d}\ubar^\prime + \iint_{\mathcal{D}_{u,\ubar}} \frac{a}{\lvert u^\prime \rvert} \lVert \nabla^i \mathfrak{Y}_2 \cdot Q \rVert_{\mathcal{L}^1_{(sc)}(S_{u^\prime, \ubar^\prime})} \hspace{.5mm} \text{d}u^\prime \text{d}\ubar^\prime.
	\end{align*}
\end{proposition}
\begin{proof}
	The Hodge structure will play a crucial role: For a pair $(\mathfrak{Y}_1, \mathfrak{Y}_2)$ or a pair $(\nab^i\mathfrak{Y}_1, \nab^i\mathfrak{Y}_2)$, the angular derivative operator $\M D$ and its $L^2$ adjoint operator $\M D^*$ form a Hodge system. Through Proposition \ref{energyprop2}, we have
	\begin{equation}\label{snb 5}
	\begin{split}
	&\int_{\S}\mathfrak{Y}_1\,\M D \mathfrak{Y}_2+\mathfrak{Y}_2\,\M D^*\mathfrak{Y}_1=-\int_{\S}(\eta+\etb)\mathfrak{Y}_1 \mathfrak{Y}_2,\\
	&\int_{\S}\nab^i\mathfrak{Y}_1\,\M D \nab^i\mathfrak{Y}_2+\nab^i\mathfrak{Y}_2\,\M D^*\nab^i\mathfrak{Y}_1=-\int_{\S}(\eta+\etb)\nab^i\mathfrak{Y}_1\nab^i\mathfrak{Y}_2.
	\end{split}
	\end{equation}
	
\par \noindent We now move forward and apply Proposition \ref{energyprop3} for $\nab^i\mathfrak{Y}_1$.  With 
$$\lambda_0=\f{1+i}{2}+s_2(\mathfrak{Y}_1), \quad \lambda_1:=2\lambda_0-1=i+2s_2(\mathfrak{Y}_1), \mbox{ we get}$$
\begin{equation}\label{6.6}
\begin{split}
&2\int_{D_{u,\ub}} |u'|^{2i+4s_2(\mathfrak{Y}_1)}\nab^i\mathfrak{Y}_1 \bigg(\nabla_3+\big(\f{1+i}{2}+s_2(\mathfrak{Y}_1)\big)\trchb\bigg)\nab^i\mathfrak{Y}_1\\\
=& \int_{H_u^{(0,\ub)}} |u|^{2i+4s_2(\mathfrak{Y}_1)}|\nab^i\mathfrak{Y}_1|^2-\int_{H_{\ui}^{(0,\ub)}} |\ui|^{2i+4s_2(\mathfrak{Y}_1)}|\nab^i\mathfrak{Y}_1|^2+\int_{D_{u,\ub}}|u'|^{2i+4s_2(\mathfrak{Y}_1)}f|\nab^i\mathfrak{Y}_1|^2,
\end{split}
\end{equation}
where $|f|\leq O/|u'|^2$. 
\vspace{3mm}

\par\noindent We also use Proposition \ref{byparts}, plugging in $\phi_1=\phi_2=|u|^{i+2s_2(\mathfrak{Y}_1)}\nab^i\mathfrak{Y}_2$
\begin{equation}\label{6.7}
\begin{split}
&2\int_{D_{u,\ub}} |u'|^{2i+4s_2(\mathfrak{Y}_1)} \nab^i  \mathfrak{Y}_2 \nabla_4\nab^i\mathfrak{Y}_2\\
=& \int_{\Hb_{\ub}^{(\ui,u)}} |u'|^{2i+4s_2(\mathfrak{Y}_1)}|\nab^i\mathfrak{Y}_2|^2-\int_{\Hb_0^{(\ui,u)}} |u'|^{2i+4s_2(\mathfrak{Y}_1)}|\nab^i\mathfrak{Y}_2|^2\\
&+\int_{D_{u,\ub}}|u'|^{2i+4s_2(\mathfrak{Y}_1)}(2\omega-\trch)|\nab^i\mathfrak{Y}_2|^2.
\end{split}
\end{equation}
Adding (\ref{6.6}) and (\ref{6.7}), we obtain
\begin{equation*}
\begin{split}
&2\int_{D_{u,\ub}} |u'|^{2i+4s_2(\mathfrak{Y}_1)}\nab^i\mathfrak{Y}_1 \bigg(\nabla_3+\big(\f{1+i}{2}+s_2(\mathfrak{Y}_1)\big)\trchb\bigg)\nab^i\mathfrak{Y}_1\\
&+2\int_{D_{u,\ub}} |u'|^{2i+4s_2(\mathfrak{Y}_1)} \nab^i\mathfrak{Y}_2 \nabla_4\nab^i\mathfrak{Y}_2\\
=& \int_{H_u^{(0,\ub)}} |u|^{2i+4s_2(\mathfrak{Y}_1)}|\nab^i\mathfrak{Y}_1|^2-\int_{H_{\ui}^{(0,\ub)}} |\ui|^{2i+4s_2(\mathfrak{Y}_1)}|\nab^i\mathfrak{Y}_1|^2+\int_{D_{u,\ub}}|u'|^{2i+4s_2(\mathfrak{Y}_1)}f|\nab^i\mathfrak{Y}_1|^2\\
&+\int_{\Hb_{\ub}^{(\ui,u)}}|u'|^{2i+4s_2(\mathfrak{Y}_1)}|\nab^i\mathfrak{Y}_2|^2-\int_{\Hb_0^{(\ui,u)}} |u'|^{2i+4s_2(\mathfrak{Y}_1)}|\nab^i\mathfrak{Y}_2|^2\\
&+\int_{D_{u,\ub}}|u'|^{2i+4s_2(\mathfrak{Y}_1)}(2\omega-\trch)|\nab^i\mathfrak{Y}_2|^2.
\end{split}
\end{equation*}
We now take (\ref{ref1}) and (\ref{ref2}) into account. With the help of (\ref{snb 5}), we then arrive at 
\begin{equation*}
\begin{split}
&\int_{H_u^{(0,\ub)}} |u|^{2i+4s_2(\mathfrak{Y}_1)}|\nab^i\mathfrak{Y}_1|^2+\int_{\Hb_{\ub}^{(\ui,u)}}|u'|^{2i+4s_2(\mathfrak{Y}_1)}|\nab^i\mathfrak{Y}_2|^2\\
=&\int_{H_{\ui}^{(0,\ub)}} |\ui|^{2i+4s_2(\mathfrak{Y}_1)}|\nab^i\mathfrak{Y}_1|^2+\int_{\Hb_{0}^{(\ui,u)}}|u'|^{2i+4s_2(\mathfrak{Y}_1)}|\nab^i\mathfrak{Y}_2|^2\\
&+2\int_{D_{u,\ub}} |u'|^{2i+4s_2(\mathfrak{Y}_1)}\nab^i\mathfrak{Y}_1 \cdot P+2\int_{D_{u,\ub}} |u'|^{2i+4s_2(\mathfrak{Y}_1)}\nab^i\mathfrak{Y}_2 \cdot Q\\
&-2\int_{D_{u,\ub}}|u'|^{2i+4s_2(\mathfrak{Y}_1)}(\eta+\etb)\nab^i\mathfrak{Y}_1\nab^i\mathfrak{Y}_2\\
&+\int_{D_{u,\ub}}|u'|^{2i+4s_2(\mathfrak{Y}_1)}f|\nab^i\mathfrak{Y}_1|^2+\int_{D_{u,\ub}}|u'|^{2i+4s_2(\mathfrak{Y}_1)}(2\omega-\trch)|\nab^i\mathfrak{Y}_2|^2.
\end{split}
\end{equation*}
Using $|(\eta+\etb)\nab^i\mathfrak{Y}_1\nab^i\mathfrak{Y}_2|\leq |\eta+\etb| \cdot \left(\lvert \nab^i\mathfrak{Y}_1\rvert^2+\lvert\nab^i\mathfrak{Y}_2\rvert^2\right),$ and the fact
$$|\eta+\etb|\leq \at O/{|u'|^2}, \quad |f|\leq O/|u'|^2, \quad |2\o-\tr\chi|\leq {O}/{|u'|},$$
by applying Gr\"onwall's inequality twice (one for $du$, one for $d\ub$), we obtain
\begin{equation*}
\begin{split}
&\int_{H_u^{(0,\ub)}} |u|^{2i+4s_2(\mathfrak{Y}_1)}|\nab^i\mathfrak{Y}_1|^2+\int_{\Hb_{\ub}^{(\ui,u)}}|u'|^{2i+4s_2(\mathfrak{Y}_1)}|\nab^i\mathfrak{Y}_2|^2\\
\ls&\int_{H_{\ui}^{(0,\ub)}} |\ui|^{2i+4s_2(\mathfrak{Y}_1)}|\nab^i\mathfrak{Y}_1|^2+\int_{\Hb_{0}^{(\ui,u)}}|u'|^{2i+4s_2(\mathfrak{Y}_1)}|\nab^i\mathfrak{Y}_2|^2\\
&+2\int_{D_{u,\ub}} |u'|^{2i+4s_2(\mathfrak{Y}_1)}\nab^i\mathfrak{Y}_1 \cdot P+2\int_{D_{u,\ub}}|u'|^{2i+4s_2(\mathfrak{Y}_1)}\nab^i\mathfrak{Y}_2 \cdot Q.
\end{split}
\end{equation*}
Multiplying by $a^{-i-2s_2(\mathfrak{Y}_1)}$ on both sides, we get
\begin{equation}\label{6.8}
\begin{split}
&\int_{H_u^{(0,\ub)}} a^{-i-2s_2(\mathfrak{Y}_1)}|u|^{2i+4s_2(\mathfrak{Y}_1)}|\nab^i\mathfrak{Y}_1|^2+\int_{\Hb_{\ub}^{(\ui,u)}}a^{-i-2s_2(\mathfrak{Y}_1)}|u'|^{2i+4s_2(\mathfrak{Y}_1)}|\nab^i\mathfrak{Y}_2|^2\\
\ls&\int_{H_{\ui}^{(0,\ub)}} a^{-i-2s_2(\mathfrak{Y}_1)}|\ui|^{2i+4s_2(\mathfrak{Y}_1)}|\nab^i\mathfrak{Y}_1|^2+\int_{\Hb_{0}^{(\ui,u)}}a^{-i-2s_2(\mathfrak{Y}_1)}|u'|^{2i+4s_2(\mathfrak{Y}_1)}|\nab^i\mathfrak{Y}_2|^2\\
&+2\int_{D_{u,\ub}} a^{-i-2s_2(\mathfrak{Y}_1)}|u'|^{2i+4s_2(\mathfrak{Y}_1)}\nab^i\mathfrak{Y}_1 \cdot P+2\int_{D_{u,\ub}}a^{-i-2s_2(\mathfrak{Y}_1)}|u'|^{2i+4s_2(\mathfrak{Y}_1)}\nab^i\mathfrak{Y}_2 \cdot Q.
\end{split}
\end{equation}
Taking into account the signature identities 

\begin{equation*}
\begin{split}
s_2(\nab^i \mathfrak{Y}_1)=\f{i}{2}+s_2(\mathfrak{Y}_1), \quad &s_2(\nab^i \mathfrak{Y}_2)=\f{i+1}{2}+s_2(\mathfrak{Y}_1),\\
s_2(P)=s_2(\nab_3\nab^i\mathfrak{Y}_1)=\f{i+2}{2}+s_2(\mathfrak{Y}_1), \quad &s_2(Q)=s_2(\M D^*\nab^i\mathfrak{Y}_1)=\f{i+1}{2}+s_2(\mathfrak{Y}_1),
\end{split}
\end{equation*}
and definitions
$$\|\phi\|_{L^2_{sc}(\S)}=a^{-s_2(\phi)}|u|^{2s_2(\phi)}\|\phi\|_{L^2(\S)},$$
$$\|\phi\|_{L^1_{sc}(\S)}=a^{-s_2(\phi)}|u|^{2s_2(\phi)-1}\|\phi\|_{L^1(\S)},$$
we rewrite (\ref{6.8}) as   
\begin{equation*}
\begin{split}
&\int_{H_u^{(0,\ub)}} \|\nab^i\mathfrak{Y}_1\|^2_{L^2_{sc}(\S)}+\int_{\Hb_{\ub}^{(\ui,u)}}\f{a}{|u'|^2}\|\nab^i\mathfrak{Y}_2\|^2_{L^2_{sc}(S_{u',\ub})}\\
\ls&\int_{H_{\ui}^{(0,\ub)}} \|\nab^i\mathfrak{Y}_1\|^2_{L^2_{sc}(S_{\ui,\ub})}+\int_{\Hb_{0}^{(\ui,u)}}\f{a}{|\ui|^2}\|\nab^i\mathfrak{Y}_2\|^2_{L^2_{sc}(S_{\ui,\ub})}\\
&+2\int_{D_{u,\ub}} \f{a}{|u'|}\|\nab^i\mathfrak{Y}_1 \cdot P\|_{L^1_{sc}(S_{u',\ub'})}+2\int_{D_{u,\ub}}\f{a}{|u'|}\|\nab^i\mathfrak{Y}_2 \cdot Q\|_{L^1_{sc}(S_{u',\ub'})}.
\end{split}
\end{equation*}
Recalling the definitions
\begin{equation*}
\begin{split}
\|\phi\|^2_{\mathcal{L}^2_{sc}(H_u^{(0,\underline{u})})}:=&
\int_0^{\underline{u}}\|\phi\|^2_{\mathcal{L}^2_{sc}(S_{u,\underline{u}'})}d\underline{u}',\\
\|\phi\|^2_{\mathcal{L}^2_{sc}(\underline{H}_{\underline{u}}^{(u_{\infty},u)})}:=&
\int_{u_{\infty}}^{u}{\frac{a}{|u'|^2}}\|\phi\|^2_{\mathcal{L}^2_{sc}(S_{u',\underline{u}})}du',
\end{split}
\end{equation*}and substituting them in the above, we arrive at the desired result.
\end{proof}
 
	\subsection{Energy estimates on the Maxwell components}
	
	Recall the null Maxwell equations
	
	\begin{gather}
	\nabla_4 \alphabar_F+ \frac{1}{2}\tr\chi \alphabar_F = - \nabla \rho_F - \Hodge{\nabla} \sigma_F -2 \Hodge{\etabar}\cdot \sigma_F -2 \Hodge{\eta}\cdot \rho_F+2\omega \hspace{.5mm} \alphabar_F- \chibarhat \cdot \alpha_F, \\ 
	\nabla_3 \alpha_F+ \frac{1}{2} \tr\chibar \alpha_F = - \nabla \rho_F+ \Hodge{\nabla}\sigma_F -2\Hodge{\etabar} \cdot \sigma_F + 2 \etabar \cdot \rho_F +2\omegabar \alpha_F -\chihat \cdot \alphabar_F,\\
	\nabla_4 \rho_F = -\text{div}\hspace{.5mm} \alpha_F - \tr\chi \rho_F - (\eta-\etabar) \cdot \alpha_F, \\
	\nabla_4 \sigma_F= - \text{curl}\hspace{.5mm} \alpha_F  -\tr\chi \sigma_F +(\eta-\etabar) \cdot \Hodge{\alpha_F}, \\ \nabla_3 \rho_F +\tr\chibar \rho_F =\text{div} \alphabar_F+(\eta-\etabar)\cdot \alphabar_F, \\
	\nabla_3 \sigma_F +\tr\chibar \sigma_F = -\text{curl} \alphabar_F + (\eta-\etabar) \cdot \Hodge{\alphabar_F}.
	\end{gather}Notice that for the pair $(\Y_1, \Y_2) = \begin{Bmatrix} \alpha_F, (\rho_F, \sigma_F) \end{Bmatrix}$ we have
	
	\begin{gather} \nabla_3 \Y_1 + \left(\frac{1}{2}+s_2(\Y_1)\right)\tr\chibar  \Y_1    - \Hodge{\mathcal{D}}_1 \Y_2 = (\psi,\chihat) \cdot \Y + \psi \cdot (\Y,\alpha_F), \\  
	\nabla_4 \Y_2 +\mathcal{D}_1 \Y_1 = \psi \cdot(\alpha_F, \Y),
	\end{gather}  while for the pair $(\Y_1, \Y_2)=\begin{Bmatrix}
	(\rho_F, -\sigma_F), \alphabar_F
	\end{Bmatrix}$ we have
	
	\begin{gather}
	\nabla_3 \Y_1 + \left(\frac{1}{2}+s_2(\Y_1)\right)\tr\chibar  \Y_1    - {\mathcal{D}}_1 \Y_2 =  \psi \cdot \Y, \\ \nabla_4 \Y_2 - \Hodge{\mathcal{D}}_1 \Y_1 = (\psi,\chibarhat)\cdot (\Y, \alpha_F).
	\end{gather}We introduce the following proposition
	
	\begin{proposition}
		Under the assumptions of Theorem \ref{main1} and the bootstrap assumptions \eqref{bootstrapbounds} and given a pair $(\Y_1, \Y_2)$ satisfying 
		\begin{gather*}
		\nabla_3\nabla^i \Y_1 + \left( \frac{i+1}{2}+ s_2(\Y_1)\right) \tr\chibar \nabla^i \Y_1 - \mathcal{D}\nabla^i \Y_2 = P, \\ 
		\nabla_4 \nabla^i \Y_2 - \Hodge{\mathcal{D}} \nabla^i \Y_1 = Q,
		\end{gather*}the following inequality holds:

		\begin{align*}
		&\int_{H_u^{(0,\ubar)}} \scaletwoSuubarprime{\nabla^i \Y_1}^2 \hspace{.5mm} \text{d}\ubar^\prime + \int_{\Hbar_{\ubar}^{(u_\infty,u)}} \frac{a}{\lvert u^\prime \rvert^2}\hspace{.5mm} \scaletwoSuprime{\nabla^i \Y_2}^2 \hspace{.5mm} \text{d}u^\prime \\ 
		\lesssim& \int_{H_{u_\infty}^{(0,\ubar)}} \lVert{\nabla^i \Y_1} \rVert^2_{\mathcal{L}^{2}_{(sc)}(S_{u_{\infty},\ubar^\prime})} \hspace{.5mm} \text{d}\ubar^\prime + \int_{\Hbar_{0}^{(u_\infty,u)}} \frac{a}{\lvert u^\prime \rvert^2}\hspace{.5mm} \lVert{\nabla^i \Y_2}\rVert^2_{\mathcal{L}^2_{(sc)}(S_{u^\prime,0})} \hspace{.5mm} \text{d}u^\prime \\ 
		&+ \iint_{\mathcal{D}_{u,\ubar}} \frac{a}{\lvert u^\prime \rvert} \lVert \nabla^i \Y_1 \cdot P \rVert_{\mathcal{L}^1_{(sc)}(S_{u^\prime, \ubar^\prime})} \hspace{.5mm} \text{d}u^\prime \text{d}\ubar^\prime + \iint_{\mathcal{D}_{u,\ubar}} \frac{a}{\lvert u^\prime \rvert} \lVert \nabla^i \Y_2 \cdot Q \rVert_{\mathcal{L}^1_{(sc)}(S_{u^\prime, \ubar^\prime})} \hspace{.5mm} \text{d}u^\prime \text{d}\ubar^\prime.
		\end{align*}
	\end{proposition}

	\par \noindent We begin with the pair $(\alpha_F, (-\rho_F, \sigma_F))$.
	
	\begin{proposition}
		Under the assumptions of Theorem \ref{main1} and the bootstrap assumptions \eqref{bootstrapbounds}, for $i\leq 11$, we have
		\begin{align*}
		&\frac{1}{\al} \scaletwoHu{(\al)^{i-1} \nabla^i \alpha_F} + \frac{1}{\al} \scaletwoHbaru{(\al)^{i-1} \nabla^i (\rho_F, \sigma_F)} \\
		\leq& \frac{1}{\al} \lVert (\al)^{i-1} \nabla^i \alpha_F \rVert_{\mathcal{L}^2_{(sc)}(H_{u_\infty}^{(0,\underline{u})})} + \frac{1}{\al} \lVert (\al)^{i-1} \nabla^i (\rho_F,\sigma_F) \rVert_{\mathcal{L}^2_{(sc)}(\Hbar_{0}^{(u_{\infty},\underline{u})})} + \frac{1}{a^{\frac{1}{3}}}.
		\end{align*}
	\end{proposition}
	
	\begin{proof}
		We schematically have 
		\begin{gather*} \nabla_3 \alpha_F + \frac{1}{2}\tr\chibar \alpha_F - \mathcal{D}(-\rho_F, \sigma_F)  = (\psi, \chihat) \cdot \Y + \psi \cdot (\Y, \alpha_F),    \\ 
		\nabla_4 (-\rho_F, \sigma_F) - \Hodge{\mathcal{D}}\alpha_F = \psi \cdot (\Y, \alpha_F).   \end{gather*}

		\par \noindent Commuting with $i$ angular derivatives we arrive at
		
		\begin{align*}
		&\nabla_3 \nabla^i \alpha_F + \frac{i+1}{2} \tr\chibar \alpha_F - \mathcal{D} \nabla^i (\rho_F, \sigma_F)\\ 
		=& \sum_{i_1 + i_2 + i_3+i_4 =i}  \nabla^{i_1}\psi^{i_2}\nabla^{i_3}(\eta,\etabar)\nabla^{i_4}(\rho_F,\sigma_F)  +  \sum_{i_1 + i_2 + i_3+i_4 =i}  \nabla^{i_1}\psi^{i_2}\nabla^{i_3} \omegabar \nabla^{i_4} \alpha_F  \\ &+ \sum_{i_1 + i_2 + i_3+i_4 =i}  \nabla^{i_1}\psi^{i_2}\nabla^{i_3}\chihat \nabla^{i_4}\alphabar_F +\sum_{ i_1+i_2+i_3+i_4+1=i}\nabla^{i_1}\psi^{i_2+1}\nabla^{i_3}\tr\chibar\nabla^{i_4}\alpha_F\\ &+  \sum_{i_1 + i_2 +i_3 + i_4 =i} \nabla^{i_1} \psi^{i_2} \nabla^{i_3} (\psi,\chibarhat, \widetilde{\tr\chibar}) \nabla^{i_4} \alpha_F \\ :=& P_1,
		\end{align*}	while for $(\rho_F,\sigma_F)$ we similarly obtain 
		
		\begin{align*}
		&\nabla_4 \nabla^i (\rho_F, \sigma_F) - \Hodge{\mathcal{D}}\nabla^{i} \alpha_F \\ =& \sum_{i_1+i_2+i_3+i_4=i} \nabla^{i_1}\psi^{i_2}\nabla^{i_3}(\eta,\etabar)\nabla^{i_4}\alpha_F + \sum_{i_1+i_2+i_3+i_4=i}\nabla^{i_1}\psi^{i_2}\nabla^{i_3}\tr\chi \nabla^{i_4}(\rho_F,\sigma_F) \\ &+\sum_{i_1 + i_2 + i_3+i_4 =i} \nabla^{i_1}\psi^{i_2} \nabla^{i_3}(\eta,\etabar,\chihat)\nabla^{i_4}(\rho_F,\sigma_F) \\ :=& Q_1.
		\end{align*}We arrive at
		
		\begin{align*}
		&\scaletwoHu{(\al)^{i-1} \nabla^i \alpha_F} + \scaletwoHbaru{(\al)^{i-1} \nabla^i (\rho_F, \sigma_F) } \\
		\leq& \lVert (\al)^{i-1} \nabla^i \alpha_F \rVert_{\mathcal{L}^2_{(sc)}(H_{u_\infty}^{(0,\ubar)})} + \lVert (\al)^{i-1} \nabla^i (\rho_F, \sigma_F) \rVert_{\mathcal{L}^2_{(sc)}(\Hbar_0^{(u_\infty,u)})}+N_1+M_1,
		\end{align*}where \begin{gather*}
		N_1 = \int_0^{\ubar} \int_{u_\infty}^u  \frac{a}{\lvert u^\prime \rvert} \scaleoneSuprimeubarprime{(\al)^{i-1} P_1  \cdot (\al)^{i-1} \nabla^i \alpha_F} \hspace{.5mm} \text{d}u^\prime \hspace{.5mm} \text{d}\ubar^\prime, \\
		M_1 = \int_0^{\ubar} \int_{u_\infty}^u  \frac{a}{\lvert u^\prime \rvert} \scaleoneSuprimeubarprime{(\al)^{i-1} Q_1  \cdot (\al)^{i-1} \nabla^i \alpha_F} \hspace{.5mm} \text{d}u^\prime \hspace{.5mm} \text{d}\ubar^\prime.
		\end{gather*} Let us focus on the term $N_1$ first. Using the scale-invariant version of H\"older's inequality we get
		
		\begin{align*}
		N_1 \leq&  \int_0^{\ubar} \int_{u_\infty}^u  \frac{a}{\lvert u^\prime \rvert^2} \scaletwoSuprimeubarprime{(\al)^{i-1} P_1} \scaletwoSuprimeubarprime{(\al)^{i-1} \nabla^i \alpha_F} \hspace{.5mm} \text{d}u^\prime \hspace{.5mm} \text{d}\ubar^\prime \\ 
		\leq&    \int_{u_\infty}^u  \frac{a}{\lvert u^\prime \rvert^2}  \left( \int_0^{\ubar} \scaletwoSuprimeubarprime{(\al)^{i-1} P_1}^2   \hspace{.5mm} \text{d}\ubar^\prime   \right)^{\frac{1}{2}} \hspace{.5mm} \text{d}u^\prime \cdot \sup_{u^\prime} \lVert (\al)^{i-1} \nabla^i \alpha_F \rVert_{\mathcal{L}^2_{(sc)}(H_{u^\prime}^{(0,\ubar)})     },
		\end{align*}where 
		\begin{align*}
         P_1	=& \sum_{i_1 + i_2 + i_3+i_4 =i}  \nabla^{i_1}\psi^{i_2}\nabla^{i_3}(\eta,\etabar)\nabla^{i_4}(\rho_F,\sigma_F)  +  \sum_{i_1 + i_2 + i_3+i_4 =i}  \nabla^{i_1}\psi^{i_2}\nabla^{i_3} \omegabar \nabla^{i_4} \alpha_F  \\ &+ \sum_{i_1 + i_2 + i_3+i_4 =i}  \nabla^{i_1}\psi^{i_2}\nabla^{i_3}\chihat \nabla^{i_4}\alphabar_F +\sum_{ i_1+i_2+i_3+i_4+1=i}\nabla^{i_1}\psi^{i_2+1}\nabla^{i_3}\tr\chibar\nabla^{i_4}\alpha_F\\ &+  \sum_{i_1 + i_2 +i_3 + i_4 =i} \nabla^{i_1} \psi^{i_2} \nabla^{i_3} (\psi,\chibarhat, \widetilde{\tr\chibar}) \nabla^{i_4} \alpha_F \\ :=& \sum_{j=1}^5 P_{1j}.
	\end{align*}Denote \[H_1 = \intubar \scaletwoSuprimeubarprime{(\al)^{i-1} P_1}^2 \hspace{.5mm} \text{d}\ubar^\prime.      \]We further have

		\begin{equation}
		H_1 = \int_{0}^{\ubar} \scaletwoSuprimeubarprime{(\al)^{i-1} P_1}^2 \hspace{.5mm} \text{d}\ubar^\prime \lesssim \sum_{j=1}^5 \intubar \scaletwoSuprimeubarprime{(\al)^{i-1}P_{1j}}^2 \dubarprime
		\end{equation}We treat each of those five terms separately.
		
		\begin{itemize}
			\item There holds \begin{align*} &\intubar \scaletwoSuprimeubarprime{(\al)^{i-1}P_{11}}^2 \dubarprime \\ =& \intubar \scaletwoSuprimeubarprime{(\al)^{i-1} \sum_{i_1+i_2+i_3+i_4=i} \nabla^{i_1} \psi^{i_2}\nabla^{i_3}(\eta,\etabar)\nabla^{i_4}(\rho_F,\sigma_F)}^2 \dubarprime \\ \lesssim& \frac{O^4}{a\cdot \lvert u^\prime \rvert^2} +  \frac{O^2}{\upr^2}\cdot \intubar \scaletwoSuprimeubarprime{a^5 \nabla^{11}(\eta,\etabar)}^2 \dubarprime + \frac{O^2}{\upr^2}\cdot \intubar \scaletwoSuprimeubarprime{a^5 \nabla^{11}(\rho_F,\sigma_F)}^2 \dubarprime \\ \lesssim& \frac{O^4}{a\cdot \lvert u^\prime \rvert^2} +  \frac{O^2}{a^2}\cdot \intubar (1+\mathcal{R}+\underline{\mathcal{R}})^2 \dubarprime + \frac{O^2}{\upr^2} F[\rho_F,\sigma_F]^2. 
			\end{align*}Here we have used Propositions \ref{etaellipticprop} and \ref{etabarellipticprop} as well as the bootstrap bounds \eqref{bootstrapbounds}.
			
			\item  There holds \begin{align*} &\intubar \scaletwoSuprimeubarprime{(\al)^{i-1}P_{12}}^2 \dubarprime \\ =& \intubar \ScaletwoSuprimeubarprime{(\al)^{i} \sum_{i_1+i_2+i_3+i_4=i} \nabla^{i_1} \psi^{i_2}\nabla^{i_3}\omegabar\nabla^{i_4}\left(\frac{\alpha_F}{\al}\right)}^2 \dubarprime \\ \lesssim& \frac{O^4}{ \lvert u^\prime \rvert^2} +  \frac{a\cdot O^2}{\upr^2}\cdot \intubar \scaletwoSuprimeubarprime{a^5 \nabla^{11}\omegabar}^2 \dubarprime + \frac{a\cdot O^2}{\upr^2}\cdot \intubar \scaletwoSuprimeubarprime{a^5 \nabla^{11}a_F}^2 \dubarprime \\ \lesssim& \frac{O^4}{ \lvert u^\prime \rvert^2} +  \frac{a\cdot O^2}{\upr^2}\cdot \intubar \scaletwoSuprimeubarprime{a^5 \nabla^{11}\omegabar}^2 \dubarprime + \frac{a\cdot O^2}{\upr^2} F[\alpha_F]^2. 
			\end{align*}Here we have made use of the bootstrap bounds \eqref{bootstrapbounds}.
			
			\item There holds
			
			\begin{align*}
			&\intubar \scaletwoSuprimeubarprime{(\al)^{i-1}P_{13}}^2 \dubarprime \\ =& \intubar \scaletwoSuprimeubarprime{(\al)^{i} \sum_{i_1+i_2+i_3+i_4=i} \nabla^{i_1} \psi^{i_2}\nabla^{i_3}\left(\frac{\chihat}{\al}\right)\nabla^{i_4}\alphabar_F}^2 \dubarprime \\ \lesssim& \frac{O^4}{\upr^2} + \frac{a\cdot O^2}{\upr^2} \cdot \intubar \scaletwoSuprimeubarprime{a^5 \nabla^{11}\left(\frac{\chihat}{\al}\right)}^2 \dubarprime  + \frac{a\cdot O^2}{\upr^2} \cdot \intubar \scaletwoSuprimeubarprime{a^5 \nabla^{11}\alphabar_F}^2 \dubarprime \\ \lesssim& 1 + \frac{a\cdot O^2}{\upr^2} \cdot \intubar \scaletwoSuprimeubarprime{a^5 \nabla^{11}\alphabar_F}^2 \dubarprime. 
			\end{align*} Here we have made used of Proposition \ref{trchichihatellipticprop}. \item The last two terms can be bounded above  by \[  F^2[\alpha_F] \cdot O[\chibarhat]^2+ 1,    \]using Gr\"onwall's inequality and the elliptic estimates.
		\end{itemize}
		
		\par \noindent We arrive at the bound
		
		\[  H_1 \lesssim 1+ F^2[\alpha_F]\cdot O[\chibarhat]^2+ \frac{a \cdot O^2}{\upr^2} \cdot \intubar \scaletwoSuprimeubarprime{a^5 \nabla^{11}\omegabar}^2 \duprime + \frac{a \cdot O^2}{\upr^2} \cdot \intubar \scaletwoSuprimeubarprime{a^5 \nabla^{11}\alphabar_F}^2 \duprime .    \]The two integrals above cannot be estimated along the $\ubar$-direction, but only along the $u-$direction.	Integrating along the $u$-direction, these bounds translate to a bound on $N_1$
		
		\be \label{N1bound}  N_1 \lesssim \left(    F[\alpha_F]\cdot O[\chibarhat]  + 1 \right) \sup_{u^\prime}\lVert (\al)^{i-1} \nabla^i \alpha_F \rVert_{\mathcal{L}^{2}_{(sc)}(H_{u^\prime}^{(0,\ubar)})}.     \ee For the term $M_1$ we follow the same procedure. We have
		
		\begin{align*}
		M_1 =& \int_0^{\ubar} \int_{u_\infty}^u  \frac{a}{\lvert u^\prime \rvert} \scaleoneSuprimeubarprime{(\al)^{i-1} Q_1  \cdot (\al)^{i-1} \nabla^i (\rho_F,\sigma_F)} \hspace{.5mm} \text{d}u^\prime \hspace{.5mm} \text{d}\ubar^\prime \\
		\leq&\int_0^{\ubar} \int_{u_\infty}^u  \frac{a}{\lvert u^\prime \rvert^2} \scaletwoSuprimeubarprime{(\al)^{i-1} Q_1} \scaletwoSuprimeubarprime{(\al)^{i-1} \nabla^i (\rho_F,\sigma_F)}\hspace{.5mm} \text{d}u^{\prime} \text{d}{\ubar}^{\prime} \\ 
		\leq& \int_0^{\ubar} \left( \int_{u_{\infty}}^u  \frac{a}{\lvert u^\prime \rvert^2} \scaletwoSuprimeubarprime{(\al)^{i-1} Q_1}^2  \hspace{.5mm} \text{d}u^\prime     \right)^{\frac{1}{2}}  \lVert (\al)^{i-1} \nabla^i(\rho_F,\sigma_F) \rVert_{\mathcal{L}^2_{(sc)}(\Hbar_{\ubar}^{(u_\infty, u)})}\hspace{.5mm} \text{d}\ubar^\prime \\
		\leq& \left( \int_0^{\ubar} \int_{u_\infty}^u  \frac{a}{\lvert u^\prime \rvert^2} \scaletwoSuprimeubarprime{(\al)^{i-1} Q_1}^2  \hspace{.5mm} \text{d}u^\prime  \dubarprime   \right)^{\frac{1}{2}} \cdot \sup_{\ubar^\prime} \lVert (\al)^{i-1} \nabla^i(\rho_F,\sigma_F) \rVert_{\mathcal{L}^2_{(sc)}(\Hbar_{\ubar}^{(u_\infty, u)})}.
		\end{align*}Here we recall that \begin{align*}Q_1 =& \sum_{i_1+i_2+i_3+i_4=i} \nabla^{i_1}\psi^{i_2}\nabla^{i_3}(\eta,\etabar)\nabla^{i_4}\alpha_F + \sum_{i_1+i_2+i_3+i_4=i}\nabla^{i_1}\psi^{i_2}\nabla^{i_3}\tr\chi \nabla^{i_4}(\rho_F,\sigma_F) \\ &+\sum_{i_1 + i_2 + i_3+i_4 =i} \nabla^{i_1}\psi^{i_2} \nabla^{i_3}(\eta,\etabar,\chihat)\nabla^{i_4}(\rho_F,\sigma_F)  \\ :=& Q_{11}+Q_{12}+Q_{13}.  \end{align*}
		
		\par \noindent Let \[  J_1 =  \int_0^{\ubar} \int_{u_\infty}^u  \frac{a}{\lvert u^\prime \rvert^2} \scaletwoSuprimeubarprime{(\al)^{i-1} Q_1}^2  \hspace{.5mm} \text{d}u^\prime \dubarprime   := J_{11}+J_{12}+J_{13} .         \]We have 
		
	\begin{equation} J_1 \lesssim \sum_{j=1}^3 \intubar\intu \frac{a}{\upr^2} \scaletwoSuprimeubarprime{(\al)^{i-1}Q_{1j}}^2 \duprime \dubarprime.
	\end{equation}Then, separating the cases where $i_4\neq i$ and $i_4 = i$ and treating each of the three terms separately, we get 
	
	\begin{itemize}
		\item There holds \begin{align*}
		J_{11} =& \intubar \intu \frac{a}{\upr^2} \scaletwoSuprimeubarprime{(\al)^{i-1} \sum_{i_1+i_2+i_3+i_4=i} \nabla^{i_1}\psi^{i_2}\nabla^{i_3}(\eta,\etabar)\nabla^{i_4}\alpha_F}^2 \duprime \dubarprime \\ =& \intu \frac{a}{\upr^2} \left( \intubar\scaletwoSuprimeubarprime{(\al)^{i} \sum_{i_1+i_2+i_3+i_4=i} \nabla^{i_1}\psi^{i_2}\nabla^{i_3}(\eta,\etabar)\nabla^{i_4}\left(\frac{\alpha_F}{\al}\right)}^2 \dubarprime \right) \duprime \\ \lesssim& \intu \frac{a}{\upr^2} \left( \frac{O^4}{\upr^2} + \frac{a\cdot O^2 \cdot F[\alpha_F]^2}{\upr^2} + \frac{O^2 \cdot (1+R)^2 }{a} \right) \duprime \lesssim 1.
		\end{align*}We have made use of Propositions \ref{etaellipticprop} and \ref{etabarellipticprop} here.
	     \item There holds  \begin{align*}
	     J_{12} =& \intubar \intu \frac{a}{\upr^2} \scaletwoSuprimeubarprime{(\al)^{i-1} \sum_{i_1+i_2+i_3+i_4=i} \nabla^{i_1}\psi^{i_2}\nabla^{i_3}\tr\chi \nabla^{i_4}(\rho_F,\sigma_F)}^2 \duprime \dubarprime \\ =& \intu \frac{a}{\upr^2} \left( \intubar\scaletwoSuprimeubarprime{(\al)^{i-1} \sum_{i_1+i_2+i_3+i_4=i} \nabla^{i_1}\psi^{i_2}\nabla^{i_3}\tr\chi \nabla^{i_4}(\rho_F,\sigma_F)}^2 \dubarprime \right) \duprime \\ \lesssim& \intu \frac{a}{\upr^2} \left( \frac{O^4}{a\cdot \upr^2} + \frac{ O^2 \cdot F[\rho_F,\sigma_F]^2}{\upr^2} + \frac{O^2 \cdot (1+R)^2 }{\upr^2} \right) \duprime \lesssim 1.
	     \end{align*}We have made use of Propositions \ref{etaellipticprop} and \ref{etabarellipticprop} as well as the bootstrap assumptions \eqref{bootstrapbounds}.
	     
	       \item There holds  \begin{align*}
	     J_{13} =& \intubar \intu \frac{a}{\upr^2} \scaletwoSuprimeubarprime{(\al)^{i-1} \sum_{i_1+i_2+i_3+i_4=i} \nabla^{i_1}\psi^{i_2}\nabla^{i_3}(\eta,\etabar,\chihat)\nabla^{i_4}(\rho_F,\sigma_F)}^2 \duprime \dubarprime \\ =& \intu \frac{a}{\upr^2} \left( \intubar\scaletwoSuprimeubarprime{(\al)^{i} \sum_{i_1+i_2+i_3+i_4=i} \nabla^{i_1}\psi^{i_2}\nabla^{i_3}\left(\frac{\eta,\etabar,\chihat}{\al}\right) \nabla^{i_4}(\rho_F,\sigma_F)}^2 \dubarprime \right) \duprime \\ \lesssim& \intu \frac{a}{\upr^2} \left( \frac{O^4}{\upr^2} + \frac{a \cdot  O^2 \cdot F[\rho_F,\sigma_F]^2}{\upr^2} + \frac{a \cdot O^2 \cdot (1+R)^2 }{\upr^2} \right) \duprime \lesssim 1.
	     \end{align*}Here we have made use of Propositions \ref{trchichihatellipticprop}, \ref{etaellipticprop} and \ref{etabarellipticprop} as well as the bootstrap bounds \ref{bootstrapbounds}.
	\end{itemize}

\par \noindent	Hence
		
		\be \label{M1bound} M_1 \leq  \sup_{\ubar^\prime} \lVert (\al)^{i-1} \nabla^i (\rho_F,\sigma_F) \rVert_{\mathcal{L}^2_{(sc)}(\Hbar_{\ubar^\prime}^{(u_\infty, u)})}.     \ee Taking the bounds \eqref{N1bound} and \eqref{M1bound} into account, we get
		
		\begin{align*}
		&a^{-1} \scaletwoHu{(\al)^{i-1} \nabla^i \alpha_F}^2 + a^{-1} \scaletwoHbaru{(\al)^{i-1} \nabla^i (\rho_F, \sigma_F)}^2 \\ 
		\leq& a^{-1} \scaletwoHzero{(\al)^{i-1} \nabla^i \alpha_F}^2 + a^{-1} \scaletwoHbarzero{(\al)^{i-1} \nabla^i (\rho_F, \sigma_F)}^2 + a^{-1}(N_1+M_1) \\
		\leq&  a^{-1} \scaletwoHzero{(\al)^{i-1} \nabla^i \alpha_F}^2 + a^{-1} \scaletwoHbarzero{(\al)^{i-1} \nabla^i (\rho_F, \sigma_F)}^2 \\ 
		&+ a^{-\frac{1}{2}} (F\cdot O + 1)\cdot F + a^{-\frac{1}{2}}  F\\ 
		\leq&  a^{-1} \scaletwoHzero{(\al)^{i-1} \nabla^i \alpha_F}^2 + a^{-1} \scaletwoHbarzero{(\al)^{i-1} \nabla^i (\rho_F, \sigma_F)}^2 + a^{-\frac{1}{4}}. 
		\end{align*} Thus
		
		\[
		\mathcal{F}^2[\alpha_F]+ \underline{\mathcal{F}}^2[\rho_F,\sigma_F] \leq  \mathcal{F}^2_0[\alpha_F] + \underline{\mathcal{F}}^2_0[\rho_F,\sigma_F] + \frac{1}{a^{\frac{1}{4}}}, 
		\]which translates to the desired energy bound
		
		\be \label{energymaxwell1} \Rightarrow \mathcal{F}[\alpha_F]+ \underline{\mathcal{F}}[\rho_F,\sigma_F] \leq  2\mathcal{F}_0[\alpha_F] + 2\underline{\mathcal{F}}_0[\rho_F,\sigma_F] + \frac{1}{a^{\frac{1}{8}}}. \ee \end{proof}
	\par \noindent  Continuing the estimates for the Maxwell components, we shift attention to the pair $\left((-\rho_F, \sigma_F), \alphabar_F\right)$.  
	
	\begin{proposition}
		Under the assumptions of Theorem \ref{main1} and the bootstrap assumptions \eqref{bootstrapbounds}, we have
		
		\begin{equation*}
		{\color{black}\scaletwoHu{(\al)^{i-1} \nabla^i (-\rho_F, \sigma_F)}^2 + \scaletwoHbaru{(\al)^{i-1} \nabla^i \alphabar_F}^2 \lesssim \mathcal{R}^4 + \underline{\mathcal{R}}^4 +\left( \mathcal{I}^{(0)} \right)^4 + \left( \mathcal{I}^{(0)} \right)^2 +1. }
		\end{equation*} 
		
	\end{proposition}
	
	\begin{proof}
		We recall the following schematic equations for the pair $(\Y_1, \Y_2) = \left((-\rho_F, \sigma_F), \alphabar_F\right)$:
		
		\begin{gather}
		\nabla_3 \Y_1 + \left(\frac{1}{2}+s_2(\Y_1)\right)\tr\chibar  \Y_1    - {\mathcal{D}}_1 \Y_2 =  (\eta,\etabar) \cdot \alphabar_F, \\ \nabla_4 \Y_2 - \Hodge{\mathcal{D}}_1 \Y_1 = (\eta,\etabar)\cdot (\rho_F,\sigma_F) + \omega\cdot \alphabar_F + \chibarhat\cdot \alpha_F.
		\end{gather}
		Commuting these equations $i$ times with angular derivatives $\nabla$ we arrive at the equation

		\begin{align*}
		&\nabla_3\nabla^i \Y_1 + \left( \frac{i+1}{2}+ s_2(\Y_1)\right) \tr\chibar \nabla^i \Y_1 - \mathcal{D}\nabla^i \Y_2  \\
		=&  \sum_{i_1+i_2+i_3=i} \nabla^{i_1} \psi^{i_2} \nabla^{i_3}(\eta,\etabar) \nabla^{i_4} \alphabar_F+ \sum_{i_1+ i_2 + i_3 +i_4+1 =i} \nabla^{i_1} \psi^{i_2+1} \nabla^{i_3}(\chibarhat,\tr\chibar) \nabla^{i_4}(\rho_F,\sigma_F)\\ &+ \sum_{i_1 + i_2 +i_3 + i_4 =i} \nabla^{i_1} \psi^{i_2} \nabla^{i_3} (\eta,\etabar,\chibarhat,\widetilde{\tr\chibar}) \nabla^{i_4} (\rho_F,\sigma_F) \\ :=&P_2,\\
		\end{align*}as well as the equation
		
		\begin{align*} &\nabla_4 \nabla^{i} \Y_2 - \Hodge{\mathcal{D}}_1 \nabla^i \Y_1 \\
		=& \sum_{i_1+i_2+i_3+ i_4 =i} \nabla^{i_1} \psi^{i_2} \nabla^{i_3} (\eta,\etabar) \nabla^{i_4}(\rho_F,\sigma_F) + \sum_{i_1+i_2+i_3 = i}\nabla^{i_1} \psi^{i_2} \nabla^{i_3} \omega \nabla^{i_4}\alphabar_F +\\ &+\sum_{i_1+i_2+i_3 = i}\nabla^{i_1} \psi^{i_2} \nabla^{i_3} \chibarhat \nabla^{i_4}\alpha_F+
		\sum_{i_1 + i_2 +i_3 + i_4 =i} \nabla^{i_1} \psi^{i_2} \nabla^{i_3}(\eta,\etabar,\chihat)\nabla^{i_4} \alphabar_F \\ :=&Q_2. 
		\end{align*}Applying the proposition, we arrive at
		
		\begin{align*}
		\scaletwoHu{(\al)^{i-1} \nabla^i (\rho_F,\sigma_F)}^2 + \scaletwoHbaru{(\al)^{i-1} \nabla^i \alphabar_F}^2 \\
		\leq \scaletwoHzero{(\al)^{i-1} \nabla^i (\rho_F,\sigma_F)}^2 + \scaletwoHbarzero{(\al)^{i-1} \nabla^i \alphabar_F}^2 + N_2 +M_2,
		\end{align*}where
		
		\begin{gather}
		N_2 = \int_0^{\ubar} \int_{u_\infty}^u  \frac{a}{\lvert u^\prime \rvert} \scaleoneSuprimeubarprime{(\al)^{i-1} P_2 \cdot(\al)^{i-1} \nabla^i (\rho_F,\sigma_F)} \hspace{.5mm} \text{d}u^\prime \hspace{.5mm} \text{d}\ubar^\prime, \\
		M_2 = \int_0^{\ubar} \int_{u_\infty}^u \frac{a}{\lvert u^\prime \rvert }\scaleoneSuprimeubarprime{(\al)^{i-1} Q_2 \cdot(\al)^{i-1} \nabla^i \alphabar_F} \hspace{.5mm} \text{d}u^\prime \hspace{.5mm} \text{d}\ubar^\prime.
		\end{gather}We focus on $N_2$ first. Using the same reasoning as for $N_1$, we have
		
		\begin{align*}
		N_2 =& \int_0^{\ubar} \int_{u_\infty}^u  \frac{a}{\lvert u^\prime \rvert} \scaleoneSuprimeubarprime{(\al)^{i-1} P_2 \cdot(\al)^{i-1} \nabla^i (\rho_F,\sigma_F)} \hspace{.5mm} \text{d}u^\prime \hspace{.5mm} \text{d}\ubar^\prime \\ 
		\leq& \int_0^{\ubar} \int_{u_\infty}^u  \frac{a}{\lvert u^\prime \rvert^2} \scaletwoSuprimeubarprime{(\al)^{i-1} P_2} \cdot \scaletwoSuprimeubarprime{(\al)^{i-1} \nabla^i (\rho_F,\sigma_F)} \hspace{.5mm} \text{d}u^\prime \hspace{.5mm} \text{d}\ubar^\prime .  
		\end{align*}
		Recall at this point the form that $P_2$ assumes: 
		
		\begin{align*} P_2 =&  	
		 \sum_{i_1+i_2+i_3+i_4=i} \nabla^{i_1} \psi^{i_2} \nabla^{i_3}(\eta,\etabar) \nabla^{i_4} \alphabar_F 
		+ \sum_{i_1+ i_2 + i_3 +i_4+1 =i} \nabla^{i_1} \psi^{i_2+1} \nabla^{i_3}(\chibarhat,\tr\chibar) \nabla^{i_4}(\rho_F,\sigma_F) \\ &+ \sum_{i_1 + i_2 +i_3 + i_4 =i} \nabla^{i_1} \psi^{i_2} \nabla^{i_3} (\eta,\etabar,\chibarhat,\widetilde{\tr\chibar}) \nabla^{i_4} (\rho_F,\sigma_F) := P_{21} + P_{22} + P_{23} . \end{align*}Consequently,we have the bound
		
		\begin{equation} \begin{split} N_2 &\leq  \sum_{j=1}^{3} \int_0^{\ubar} \int_{u_\infty}^u  \frac{a}{\lvert u^\prime \rvert^2} \scaletwoSuprimeubarprime{(\al)^{i-1} P_{2j}} \cdot \scaletwoSuprimeubarprime{(\al)^{i-1} \nabla^i (\rho_F,\sigma_F)} \hspace{.5mm} \text{d}u^\prime \hspace{.5mm} \text{d}\ubar^\prime \\&= N_{21} +N_{22}+N_{23}. \end{split} \end{equation}\par \noindent		We estimate each term separately.
		
		\begin{itemize}
			\item There holds \begin{align*}N_{21} &= \intu \frac{a}{\upr^2}  \intubar \scaletwoSuprimeubarprime{(\al)^{i-1} P_{21}} \scaletwoSuprimeubarprime{(\al)^{i-1} \nabla^i (\rho_F,\sigma_F)} \dubarprime  \duprime  \\ &\lesssim \intu \frac{a}{\upr^2} \left( \frac{O^4}{a \cdot \upr^2} + \frac{O^2}{\upr^2}\lVert a^5\nabla^{11}(\eta,\etabar) \rVert^2_{\mathcal{L}^2_{(sc)}(H_{u^\prime}^{(0,\ubar)})  }  \right)^{\frac{1}{2}} \cdot \lVert (\al)^{i-1}(\rho_F, \sigma_F) \rVert_{\mathcal{L}^2_{(sc)}(H_{u^{\prime}}^{(0,\ubar)})}	\duprime \\ &+ \intubar \left( \intu \frac{a}{\upr^2} \cdot  \frac{O^2}{\upr^2} \scaletwoSuprimeubarprime{a^5 \nabla^{11}\alphabar_F}^2 \duprime \right)^{\f12} \cdot \lVert (\al)^{i-1}(\rho_F, \sigma_F) \rVert_{\mathcal{L}^2_{(sc)}(\Hbar_{\ubar^{\prime}}^{(u_{\infty},u)})} \dubarprime \\ &\lesssim \intu \frac{a}{\upr^2} \left( \frac{O^4}{a \cdot \upr^2} + \frac{O^2}{\upr^2}\lVert a^5\nabla^{11}(\eta,\etabar) \rVert^2_{\mathcal{L}^2_{(sc)}(H_{u^\prime}^{(0,\ubar)})  }  \right)^{\frac{1}{2}} \cdot F \duprime \\ &+ \intubar \left( \intu \frac{a}{\upr^2} \cdot  \frac{O^2}{\upr^2} \scaletwoSuprimeubarprime{a^5 \nabla^{11}\alphabar_F}^2 \duprime \right)^{\f12} \cdot \al \cdot F \dubarprime
			\end{align*}Here $F$ is the bootstrap constant appearing in \eqref{bootstrapbounds}. Notice that what we have done in the above is to separate between three cases. The first is when neither $(\eta,\etabar)$ nor $\alphabar_F$ have $11$ derivatives, the second is when  11 derivatives fall on $(\eta, \etabar)$ and finally the third case is for when $11$ derivatives fall on $\alphabar_F$. The reason for this distinction is that we use H\"older's inequality in different directions, depending on what elliptic estimates and Maxwell norms we have. Making use of Propositions \ref{etaellipticprop} and \ref{etabarellipticprop}, we conclude that the last two terms can be bounded above by $1$, using the section on elliptic estimates. In particular, \[ N_{21} \lesssim 1. \]
			\item There holds 
			\begin{align*}N_{22}=& \intu \frac{a \cdot F}{\upr^2} \left( \intubar \scaletwoSuprimeubarprime{(\al)^{i-1}\sum_{i_1 + i_2 + i_3+i_4+1 =i}\nabla^{i_1}\psi^{i_2+1}\nabla^{i_3}(\chibarhat,\tr\chibar)\nabla^{i_4}(\rho_F,\sigma_F)}^2  \dubarprime     \right)^{\frac{1}{2}} \duprime  \\  =& \intu \left( \intubar \ScaletwoSuprimeubarprime{(\al)^{i-1}\sum_{i_1 + i_2 + i_3+i_4+1 =i}\nabla^{i_1}\psi^{i_2+1}\nabla^{i_3}\left(\frac{a}{\upr^2}\chibarhat,\frac{a}{\upr^2}\tr\chibar\right)\nabla^{i_4}(\rho_F,\sigma_F)}^2  \dubarprime     \right)^{\frac{1}{2}} \duprime \cdot F  \\\lesssim& \intu   \frac{1}{\al}\cdot \frac{O^3}{\upr^2} \duprime \cdot F \lesssim 1.
			\end{align*}
				\item There holds 
			\begin{align*}N_{23} &= \intu \intubar \frac{a }{\upr^2} \scaletwoSuprimeubarprime{(\al)^{i-1} P_{23}} \scaletwoSuprimeubarprime{(\al)^{i-1} \nabla^{i} (\rho_F, \sigma_F)} \duprime.  \\ 
			\end{align*} We first distinguish between the cases where $i <11$ and $i=11$. For the case $i< 11$, there holds
		
		\[ \scaletwoSu{ (\al)^{i-1} \sum_{i_1+i_2+i_3+i_4=i} \nabla^{i_1}\psi^{i_2}\nabla^{i_3}(\eta,\etabar,\chibarhat, \tildetr) \nabla^{i_4}(\rho_F, \sigma_F)} \lesssim \frac1\al \cdot \frac{\lvert u \rvert}{\al} \cdot \frac{O^2}{\lvert u\rvert} \lesssim \frac{O^2}{a}.    \]Moreover, since $i<11$, we can bound \[ \scaletwoSu{(\al)^{i-1}\nabla^i (\rho_F, \sigma_F)} \lesssim \frac{O}{\al}. \]Hence, when $i<11$, it is easy to establish that $N_{23} \lesssim 1.$ When $i=11$, we distinguish between four cases. The first one is when neither $(\eta, \etabar, \chibarhat, \tildetr)$ nor $(\rho_F, \sigma_F)$ have $11$ derivatives. The second is when $11$ derivatives fall on $(\rho_F, \sigma_F)$. The third is when $11$ derivatives fall on $(\eta, \etabar)$ and the fourth is when $11$ derivatives fall on $(\chibarhat, \tildetr)$. We treat these cases below:
		
		\begin{equation}
		    \begin{split}
		        N_{23} &\lesssim \intu \intubar \frac{a}{\upr^2} \cdot \frac{O^2}{a} \cdot \scaletwoSuprimeubarprime{ a^5 \nabla^{11} (\rho_F, \sigma_F)} \dubarprime \duprime \\ &+ \intu \intubar \frac{a}{\upr^2}\cdot \frac{\upr}{\al} \cdot \frac{O}{\upr} \cdot \scaletwoSuprimeubarprime{a^5 \nabla^{11}(\rho_F , \sigma_F)}^2 \dubarprime \duprime \\ &+  \intu \intubar \frac{a}{\upr^2}\cdot \frac{O}{\upr} \cdot \scaletwoSuprimeubarprime{a^5 \nabla^{11} (\eta,\etabar)} \scaletwoSuprimeubarprime{a^5 \nabla^{11} (\rho_F, \sigma_F)} \dubarprime \duprime \\  &+  \intu \intubar \frac{a}{\upr^2}\cdot \frac{O}{\upr} \cdot \scaletwoSuprimeubarprime{a^5 \nabla^{11} (\chibarhat,\tildetr)} \scaletwoSuprimeubarprime{a^5 \nabla^{11} (\rho_F, \sigma_F)} \dubarprime \duprime \\ &\lesssim 1 + \intu \frac{a\cdot O}{\upr^3} \lVert a^5 \nabla^{11}(\eta,\etabar) \rVert_{\mathcal{L}^2_{(sc)}(H_{u^{\prime}}^{(0,\ubar)})} \cdot \lVert a^5 \nabla^{11}(\rho_F,\sigma_F)  \rVert_{\mathcal{L}^2_{(sc)}(H_{u^{\prime}}^{(0,\ubar)})}\duprime \\ &+ \intubar O \intu \frac{a}{\upr^3} \scaletwoSuprimeubarprime{a^5 \nabla^{11} (\chibarhat, \tildetr)}  \scaletwoSuprimeubarprime{a^5 \nabla^{11}(\rho_F, \sigma_F)} \duprime \dubarprime \\ &\lesssim 1 + \intu \frac{a\cdot O}{\upr^3 }\cdot \frac{\upr}{a} \cdot R  \cdot F \duprime \\ &+ \sup_{\ubar}  \left( \intu \frac{a^2 O_{\infty}^2[\rho_F,\sigma_F]}{\upr^4}\scaletwoSuprimeubarprime{a^5 \nabla^{11}(\chibarhat,\tildetr)}^2 \duprime \right)^{\f12} \cdot a^{- \f12} \cdot \scaletwoHbaru{a^5 \nabla^{11} (\rho_F, \sigma_F)} \\ &\lesssim 1 + \mathcal{I}^{(0)} \cdot (\mathcal{R} +\underline{\mathcal{R}} +1) \lesssim (\mathcal{I}^{(0)})^2 + \mathcal{R}^2 +\underline{\mathcal{R}}^2 +1.
		    \end{split}
		\end{equation}Here we have used the fact that $O_{\infty}[\rho_F,\sigma_F] \lesssim 1$ from Proposition \ref{rhoFsphereestimates}, the energy estimates on $\underline{\mathcal{F}}[\rho_F, \sigma_F]$ to bound the term by the initial data, as well as Proposition \ref{chibarhattrchibarelliptic} and in particular \eqref{chibarhattrchibartraditionalformestimates}.

		 Combining these estimates, we arrive at

			\[N_{23}\lesssim (\mathcal{I}^{(0)})^2 + \mathcal{R}^2 +\underline{\mathcal{R}}^2 +1.\]
		\end{itemize}Putting everything together, there holds 
	
	\[ N_{2}\lesssim (\mathcal{I}^{(0)})^2 + \mathcal{R}^2 +\underline{\mathcal{R}}^2 +1.   \]

		\par \noindent We move on to $M_2$. We have 
		
		\begin{align*}
		&\int_0^{\ubar} \int_{u_\infty}^u \frac{a}{\lvert u^\prime \rvert }\scaleoneSuprimeubarprime{(\al)^{i-1} Q_2 \cdot(\al)^{i-1} \nabla^i \Y_2} \hspace{.5mm} \text{d}u^\prime \hspace{.5mm} \text{d}\ubar^\prime \\ 
		\leq& \int_0^{\ubar} \int_{u_\infty}^u \frac{a}{\lvert u^\prime \rvert^2} \scaletwoSuprimeubarprime{(\al)^{i-1} Q_2} \scaletwoSuprimeubarprime{(\al)^{i-1} \nabla^i \Y_2} \hspace{.5mm} \text{d}u^\prime \hspace{.5mm} \text{d}\ubar^\prime \\ \leq&  \int_{0}^{\ubar} \left( \int_{u_{\infty}}^u \frac{a}{\lvert u^\prime \rvert^2} \scaletwoSuprimeubarprime{(\al)^{i-1} Q_2}^2 \hspace{.5mm} \text{d} u^\prime \right)^{\frac{1}{2}}\left( \int_{u_{\infty}}^u \frac{a}{\lvert u^\prime \rvert^2} \scaletwoSuprimeubarprime{(\al)^{i-1} \nabla^i \Y_2}^2 \hspace{.5mm} \text{d} u^\prime \right)^{\frac{1}{2}} \hspace{.5mm} \text{d} \ubar^\prime \\ \leq&  
		\left( \int_0^{\ubar} \int_{u_{\infty}}^u \frac{a}{\lvert u^\prime \rvert^2} \scaletwoSuprimeubarprime{(\al)^{i-1} Q_2}^2 \hspace{.5mm} \text{d} u^\prime \hspace{.5mm} \text{d}\ubar^\prime \right)^{\frac{1}{2}} \cdot \sup_{\ubar^\prime} \lVert (\al)^{i-1} \nabla^i \Y_2 \rVert_{\mathcal{L}^2_{(sc)}(\Hbar_{\ubar^\prime}^{(u_\infty,u)})     } .  	\end{align*} Denote \[         H_2 = \int_0^{\ubar} \int_{u_\infty}^u  \frac{a}{\lvert u^\prime \rvert^2} \scaletwoSuprimeubarprime{(\al)^{i-1} Q_2}^2 \hspace{.5mm} \text{d}u^\prime\hspace{.5mm} \text{d}\ubar^\prime.        \]Recall at this point that
		
		\begin{align*}  Q_2 =&  \sum_{i_1+i_2+i_3+ i_4 =i} \nabla^{i_1} \psi^{i_2} \nabla^{i_3} (\eta,\etabar) \nabla^{i_4}(\rho_F,\sigma_F) + \sum_{i_1+i_2+i_3 = i}\nabla^{i_1} \psi^{i_2} \nabla^{i_3} \omega \nabla^{i_4}\alphabar_F\\ &+\sum_{i_1+i_2+i_3 = i}\nabla^{i_1} \psi^{i_2} \nabla^{i_3} \chibarhat \nabla^{i_4}\alpha_F +
		\sum_{i_1 + i_2 +i_3 + i_4 =i} \nabla^{i_1} \psi^{i_2} \nabla^{i_3}(\eta,\etabar,\chihat)\nabla^{i_4} \alphabar_F := Q_{21}+Q_{22}+Q_{23}+Q_{24}.     \end{align*} Thus, 
		
		\begin{equation}
		H_2 \leq \sum_{j=1}^4 \int_0^{\ubar} \int_{u_\infty}^u  \frac{a}{\lvert u^\prime \rvert^2} \scaletwoSuprimeubarprime{(\al)^{i-1}Q_{2j}}^2 \duprime \dubarprime . \end{equation}	We estimate term by term.
		
		\begin{itemize}
			\item The first two terms 
			
			\[\int_0^{\ubar} \int_{u_\infty}^u  \frac{a}{\lvert u^\prime \rvert^2} \scaletwoSuprimeubarprime{(\al)^{i-1}(Q_{21}, Q_{22})}^2 \duprime \dubarprime \]can be bounded by $1$ as before.
			
			\item For the third term, there holds
			\begin{align*}
			&\int_0^{\ubar} \int_{u_\infty}^u  \frac{a}{\lvert u^\prime \rvert^2} \scaletwoSuprimeubarprime{(\al)^{i-1}Q_{23}}^2 \duprime \dubarprime \\=&\int_0^{\ubar} \int_{u_\infty}^u  \frac{a}{\lvert u^\prime \rvert^2} \scaletwoSuprimeubarprime{(\al)^{i-1}\sum_{i_1+i_2+i_3+i_4=i} \nabla^{i_1}\psi^{i_2}\nabla^{i_3}\chibarhat\nabla^{i_4}\alpha_F }^2 \duprime \dubarprime \\ \lesssim&  \int_0^{\ubar} \int_{u_\infty}^u  \frac{a}{\lvert u^\prime \rvert^2} \cdot \frac{1}{a}\cdot  \frac{\upr^2}{a}\cdot a \cdot \frac{O[\chibarhat]^2O[\alpha_F]^2}{\upr^2} \duprime +  \int_{u_\infty}^u  \frac{a}{\lvert u^\prime \rvert^2} \cdot \frac{\upr^2}{a} \cdot \frac{O^2}{\upr^2} \left( \int_0^{\ubar} \scaletwoSuprimeubarprime{a^5 \nabla^{11}\alpha_F}^2 \dubarprime \right) \duprime \\ &+ \intubar \left( \intu \frac{a}{\upr^2} \cdot \frac{a\cdot O^2}{\upr^2}\cdot \frac{\upr^2}{a} \cdot \scaletwoSuprimeubarprime{a^5\nabla^{11}\left(\frac{\al}{\upr}\chibarhat  \right) }^2 \duprime        \right)\dubarprime \\
					\lesssim& 1	 + (\mathcal{R}^2 + \underline{\mathcal{R}}^2+1)\cdot (\mathcal{F}^2[\alpha_F] + \underline{\mathcal{F}}^2[\rho_F,\sigma_F]+1) + \intubar \intu \frac{a^2 \cdot O_{\infty}[\alpha_F]^2}{\upr^4} \cdot\scaletwoSuprimeubarprime{a^5 \nabla^{11} \chibarhat}^2 \duprime \dubarprime				\end{align*}We focus on the term \[ \intu \frac{a^2\cdot O_{\infty}[\alpha_F]^2}{\upr^4} \scaletwoSuprime{a^5\nabla^{11}\chibarhat}^2 \duprime \lesssim  \intu \frac{a^2}{\upr^4} \scaletwoSuprime{a^5\nabla^{11}\chibarhat}^2 \duprime.\]
Recall that, from the proof of Proposition \ref{chibarhattrchibarelliptic}, there holds

\begin{align*}
		\scaletwoSu{a^5 \nabla^{11}\chibarhat} \lesssim& \sum_{j \leq 10} \big( \scaletwoSu{(\al)^j \nabla^{j+1}\tildetr} + \scaletwoSu{(\al \nabla)^j \tbetabar} \\ &+ \ScaletwoSu{(\al)^j \sum_{j_1 +j_2=j} \nabla^{j_1}(\eta,\etabar)\nabla^{j_2}(\chibarhat,\tr\chibar)} \big) + \sum_{j \leq 10} \frac{1}{\al} \cdot \scaletwoSu{\aln \chibarhat}. 
		\end{align*}This implies that 
		
		\begin{align*}
		&\scaletwoSu{a^5 \nabla^{11}\chibarhat}^2\\ \lesssim& \sum_{j \leq 10} \big( \scaletwoSu{(\al)^j \nabla^{j+1}\tildetr}^2 + \scaletwoSu{(\al \nabla)^j \tbetabar}^2 \\ &+ \frac{u^2}{a}\ScaletwoSu{(\al)^j \sum_{j_1 +j_2=j} \nabla^{j_1}(\eta,\etabar)\nabla^{j_2}\left(\frac{a}{\lvert u \rvert^2}\chibarhat,\frac{a}{\lvert u \rvert ^2}\tr\chibar\right)}^2 \big) + \sum_{j \leq 10}  \cdot \ScaletwoSu{\aln \left(\frac{\chibarhat}{\al}\right)}^2\\ \lesssim& \frac{\lvert u \rvert^2}{a^2}(\mathcal{R}+\underline{\mathcal{R}}+1)^2  + \scaletwoSu{(\al \nabla)^j \tbetabar}^2 + \frac{\lvert u \rvert^4 }{a^2 }\cdot \frac{O^4}{\lvert u \rvert^2}+ \frac{1}{a}\cdot \frac{\lvert u\rvert^2}{a} \\ \lesssim& \frac{\lvert u \rvert^2 \cdot R^2}{a^2}  + \scaletwoSu{(\al \nabla)^j \tbetabar}^2 + \frac{\lvert u \rvert^4 }{a^2 }\cdot \frac{O^4}{\lvert u \rvert^2}+ \frac{1}{a}\cdot \frac{\lvert u\rvert^2}{a}.
		\end{align*}Multiplying the above by $\frac{a^2}{\lvert u \rvert^4}$ and taking the integral in the incoming direction, we have
		
		\begin{align*}
		    \intu \frac{a^2}{\upr^4}\scaletwoSuprime{a^5\nabla^{11}\chibarhat}^2 \duprime \lesssim& \intu \frac{R^2}{\upr^2}\duprime + \intu \frac{a^2}{\upr^4}\scaletwoSuprime{(\al \nabla)^j \tbetabar}^2 \duprime  \\\ &+ \intu \frac{O^4+1}{\upr^2} \duprime \lesssim \frac{R^2+O^4+1}{\lvert u \rvert} + \frac{16}{a^2} \hspace{.5mm}  \underline{\mathcal{R}}[\tbetabar]^2\lesssim 1.
		\end{align*}

			Here we have used Propositions \ref{alphaFproposition} and \ref{chibarhattrchibarelliptic} as well as the bootstrap bounds \eqref{bootstrapbounds}.
			\item The fourth term can be bounded by $1$ using the same procedures as above.
		\end{itemize}
		\par \noindent Consequently,

		\[ H_2 \leq  1	 + (\mathcal{R}^2 + \underline{\mathcal{R}}^2+1)\cdot (\mathcal{F}^2[\alpha_F] + \underline{\mathcal{F}}^2[\rho_F,\sigma_F]+1) \]and also

		\[ M_2 \leq H_2^{\frac{1}{2}} \scaletwoHbaru{(\al)^{i-1} \nabla^i \Y_2}     \leq H_2 + \frac{1}{4} \scaletwoHbaru{(\al)^{i-1} \nabla^i \Y_2}^2.  \]We also know that $N_2 \leq 1$. Putting everything together, we have

		\begin{align*}
		&\scaletwoHu{(\al)^{i-1} \nabla^i \Y_1}^2 + \scaletwoHbaru{(\al)^{i-1} \nabla^i \Y_2}^2 \\ 
		\leq& \scaletwoHzero{(\al)^{i-1} \nabla^i \Y_1}^2 + \scaletwoHbarzero{(\al)^{i-1} \nabla^i \Y_2}^2 + N_2 +M_2 \\ 
		\leq& \scaletwoHzero{(\al)^{i-1} \nabla^i \Y_1}^2 + \scaletwoHbarzero{(\al)^{i-1} \nabla^i \Y_2}^2 + N_2 +H_2 + \frac{1}{4} \scaletwoHbaru{(\al)^{i-1} \nabla^i \Y_2}^2.
		\end{align*}	The last term can be absorbed by the left-hand side. Thus,

		\begin{align*}
		&\scaletwoHu{(\al)^{i-1} \nabla^i \Y_1}^2 + \scaletwoHbaru{(\al)^{i-1} \nabla^i \Y_2}^2\\
		 \lesssim& \scaletwoHzero{(\al)^{i-1} \nabla^i \Y_1}^2 + \scaletwoHbarzero{(\al)^{i-1} \nabla^i \Y_2}^2
		\\&+ \left( (\mathcal{I}^{(0)})^2 + \mathcal{R}^2 +\underline{\mathcal{R}}^2 +1 \right) + (1+ \mathcal{R}^2+\underline{\mathcal{R}}^2) (\underline{\mathcal{F}}^2[\rho_F,\sigma_F]+\mathcal{F}^2[\alpha_F]+1).
		\end{align*}Thus, we arrive at the energy inequality
		
		\begin{equation}
\begin{split}
		&\mathcal{F}^2[\rho_F,\sigma_F] + \underline{\mathcal{F}}^2[\alphabar_F]  \\
		\leq& \mathcal{F}_0^2[\rho_F,\sigma_F] + \underline{\mathcal{F}}_0^2[\alphabar_F] + (1+ \mathcal{R}^2+\underline{\mathcal{R}}^2)(\mathcal{F}_0^2[\alpha_F]+\underline{\mathcal{F}}^2[\rho_F,\sigma_F]+a^{-\frac{1}{4}}) \\
		\leq& (\mathcal{I}^{(0)})^2 + (1+ \mathcal{R}^2+\underline{\mathcal{R}}^2)((\mathcal{I}^{(0)})^2  + \frac{1}{a^{\frac{1}{4}}}) \lesssim (\mathcal{I}^{(0)})^2  + (\mathcal{R}^4+\underline{\mathcal{R}}^4+1) + ((\mathcal{I}^{(0)})^4 + \frac{1}{\al})\\
		\lesssim& \mathcal{R}^4 + \underline{\mathcal{R}}^4+ (\mathcal{I}^{(0)})^4 + (\mathcal{I}^{(0)})^2  +1. \label{energymaxwell2}
		\end{split}
\end{equation}Combining \eqref{energymaxwell1} and \eqref{energymaxwell2} we get the following

		\begin{theorem}
			Under the assumptions of Theorem \ref{main1} and the bootstrap bounds \eqref{bootstrapbounds}, there holds
			
			\[ \mathcal{F}+ \underline{\mathcal{F}} \lesssim \mathcal{R}^2 + \underline{\mathcal{R}}^2+ (\mathcal{I}^{(0)})^2  + (\mathcal{I}^{(0)}) +1.     \]
		\end{theorem}
	\end{proof}	
	\subsection{Energy estimates for curvature}
	
	Again, for $(\Psi_1, \Psi_2) \in \begin{Bmatrix}
	(\alpha, \tbeta), (\tbeta, (\rho,\sigma)), ((\rho,\sigma),\tbetabar), (\tbetabar,\alphabar) 
	\end{Bmatrix}	$	we have the following:
	
	\begin{proposition}
		Under the assumptions of Theorem \ref{main1} and the bootstrap assumptions \eqref{bootstrapbounds}, assuming we have a pair $(\Psi_1,\Psi_2)$ satisfying 
		\begin{gather}
		\nabla_3 \nabla^i \Psi_1 + \left( \frac{1+i}{2}+s_2(\Psi_1)\right) \tr\chibar \nabla^i \Psi_1 - \mathcal{D}\nabla^i \Psi_2 = P, \\
		\nabla_4 \nabla^i \Psi_2 - \Hodge{\mathcal{D}} \nabla^i \Psi_1 = Q,
		\end{gather}with $(\mathcal{D}, \Hodge{\mathcal{D}})$ forming a Hodge dual, it follows

		\begin{align*}
		&\int_{H_u^{(0,\ubar)}} \scaletwoSuubarprime{\nabla^i \Psi_1}^2 \hspace{.5mm} \text{d}\ubar^\prime + \int_{\Hbar_{\ubar}^{(u_\infty,u)}} \frac{a}{\lvert u^\prime \rvert^2}\hspace{.5mm} \scaletwoSuprime{\nabla^i \Psi_2}^2 \hspace{.5mm} \text{d}u^\prime \\
		\lesssim& \int_{H_{u_\infty}^{(0,\ubar)}} \lVert{\nabla^i \Psi_1} \rVert^2_{\mathcal{L}^{2}_{(sc)}(S_{u_{\infty},\ubar^\prime})} \hspace{.5mm} \text{d}\ubar^\prime + \int_{\Hbar_{0}^{(u_\infty,u)}} \frac{a}{\lvert u^\prime \rvert^2}\hspace{.5mm} \lVert{\nabla^i \Psi_2}\rVert^2_{\mathcal{L}^2_{(sc)}(S_{u^\prime,0})} \hspace{.5mm} \text{d}u^\prime \\
		&+ \iint_{\mathcal{D}_{u,\ubar}} \frac{a}{\lvert u^\prime \rvert} \lVert \nabla^i \Psi_1 \cdot P \rVert_{\mathcal{L}^1_{(sc)}(S_{u^\prime, \ubar^\prime})} \hspace{.5mm} \text{d}u^\prime \text{d}\ubar^\prime + \iint_{\mathcal{D}_{u,\ubar}} \frac{a}{\lvert u^\prime \rvert} \lVert \nabla^i \Psi_2 \cdot Q \rVert_{\mathcal{L}^1_{(sc)}(S_{u^\prime, \ubar^\prime})} \hspace{.5mm} \text{d}u^\prime \text{d}\ubar^\prime.
		\end{align*}
	\end{proposition}With this in mind, we begin by considering the pair $(\alpha,\tbeta)$:
	
	\begin{proposition}
		Under the assumptions of Theorem \ref{main1} and the bootstrap assumptions \eqref{bootstrapbounds}, we have for $i\leq 10$ the following:
		\begin{equation}\begin{split}
		&\frac{1}{\al} \scaletwoHu{\aln \alpha} + \frac{1}{\al}\scaletwoHbaru{\aln \tbeta} \\
		\leq &\frac{1}{\al} \lVert \aln \alpha \rVert_{\mathcal{L}^{2}_{(sc)}(H_{u_\infty}^{(0,\ubar)})}+ \frac{1}{\al} \lVert \aln \tbeta \rVert_{\mathcal{L}^2_{(sc)}(\underline{H}_{0}^{(u_{\infty},u)})} + \frac{1}{a^{\frac{1}{3}}}.
		\end{split}\end{equation}

	\end{proposition}		
	
	\begin{proof}
		We have the schematic equations
		
		\begin{equation}
		\nabla_4 \tbeta - \Hodge{\mathcal{D}}\alpha = \psi(\tbeta,\alpha) + \alpha_F \cdot \nabla \alpha_F + (\psi,\chihat)\cdot(\Y,\alpha_F) + \psi \cdot (\Y,\alpha_F), \alpha_F\end{equation}\begin{align*}
		\nabla_3 \alpha + \frac{1}{2}\tr\chibar \alpha - \mathcal{D} \tbeta =& (\psi,\chihat)\cdot(\Psi,\beta,\alpha) + (\rho_F,\sigma_F) \cdot \nabla(\rho_F,\sigma_F, \alpha_F) \\ &+ \alpha_F \cdot \nabla (\rho_F,\sigma_F) +(\psi,\chibarhat,\tr\chibar) \cdot (\Y,\alpha_F) \cdot (\Y,\alpha_F).
		\end{align*}Commuting the above equations $i$ times with $\nabla$ we arrive at 
		
		\begin{align*}
		&\nabla_3 \nabla^i \alpha + \frac{1+i}{2}\tr\chibar \nabla^i \alpha - \mathcal{D} \nabla^i \tbeta \\
		=& \sum_{i_1+i_2+i_3+i_4+1 =i} \nabla^{i_1} \psi^{i_2+1} \nabla^{i_3}(\chibarhat,\tr\chibar) \nabla^{i_4}\alpha + \sum_{i_1+i_2+i_3+i_4=i} \nabla^{i_1} \psi^{i_2} \nabla^{i_3} (\psi,\chihat,\chibarhat,\widetilde{\tr\chibar}) \nabla^{i_4} (\Psi,\beta,\alpha) \\ 
		&+ \sum_{i_1+i_2+i_3=i} \nabla^{i_1}\psi^{i_2} \nabla^{i_3}( \Y \cdot \nabla(\Y, \alpha_F) + \alpha_F \cdot \nabla \Y +(\psi,\chibarhat,\tr\chibar) \cdot (\Y,\alpha_F) \cdot (\Y,\alpha_F) ) \\
		=& \sum_{i_1+i_2+i_3+i_4+1 =i} \nabla^{i_1} \psi^{i_2+1} \nabla^{i_3}(\chibarhat,\tr\chibar) \nabla^{i_4}\alpha + \sum_{i_1+i_2+i_3+i_4=i} \nabla^{i_1} \psi^{i_2} \nabla^{i_3} (\psi,\chihat,\chibarhat,\widetilde{\tr\chibar}) \nabla^{i_4} (\Psi,\beta,\alpha)  \\ &
		+ \sum_{i_1+i_2+i_3+i_4=i} \nabla^{i_1}\psi^{i_2} \nabla^{i_3}(\rho_F,\sigma_F, \alpha_F) \nabla^{i_4+1} (\rho_F,\sigma_F)	+ \sum_{i_1+i_2+i_3+i_4=i} \nabla^{i_1}\psi^{i_2} \nabla^{i_3}(\rho_F,\sigma_F) \nabla^{i_4+1} (\alpha_F)\\ &
		+ \sum_{i_1+i_2+i_3+i_4+i_5=i} \nabla^{i_1}\psi^{i_2} \nabla^{i_3}(\psi,\chibarhat,\tr\chibar)\nabla^{i_4}(\Y,\alpha_F)\nabla^{i_5}(\Y,\alpha_F) \\ :=& F_1,
		\end{align*}
		
		\begin{align*}
		&\nabla_4 \nabla^i \tbeta -\Hodge{\mathcal{D}}\nabla^i \alpha \\
		=& \sum_{i_1+i_2+i_3+i_4=i} \nabla^{i_1}\psi^{i_2} \nabla^{i_3} (\psi,\chihat) \nabla^{i_4} (\beta,\alpha)\\
		&+ \sum_{i_1+i_2+i_3+i_4=i} \nabla^{i_1}\psi^{i_2} \nabla^{i_3} \alpha_F \nabla^{i_4+1} \alpha_F + \sum_{i_1+i_2+i_3+i_4=i} \nabla^{i_1}\psi^{i_2} \nabla^{i_3} (\psi,\chihat) \nabla^{i_4} (\Y,\alpha_F)\\
		&+ \sum_{i_1+i_2+i_3+i_4+i_5=i} \nabla^{i_1} \psi^{i_2} \nabla^{i_3}\psi \nabla^{i_4}\alpha_F \nabla^{i_5}\alpha_F \\ :=&G_1.
		\end{align*}This gives us  
		
		\begin{align*}
		&\scaletwoHu{\aln \alpha}^2 + \scaletwoHbaru{\aln \tbeta}^2 \\
		\leq&  \lVert \aln \alpha \rVert^2_{\mathcal{L}^{2}_{(sc)}(H_{u_\infty}^{(0,\ubar)})}  +  \lVert \aln \tbeta \rVert^2_{\mathcal{L}^2_{(sc)}(\underline{H}_{0}^{(u_{\infty},u)})} +N_1+M_1,
		\end{align*}where \begin{gather}
		N_1 = \int_{0}^{\ubar} \int_{u_{\infty}}^u \frac{a}{\lvert u^\prime \rvert} \scaleoneSuprimeubarprime{(\al)^i F_1 \cdot \aln \alpha}\hspace{.5mm} \text{d}u^\prime \hspace{.5mm} \text{d}\ubar^{\prime},\\
		M_1 = \int_{0}^{\ubar} \int_{u_{\infty}}^u \frac{a}{\lvert u^\prime \rvert} \scaleoneSuprimeubarprime{(\al)^i G_1 \cdot \aln \tbeta}\hspace{.5mm} \text{d}u^\prime \hspace{.5mm} \text{d}\ubar^{\prime}.
		\end{gather}By H\"{o}lder's inequality, as per the previous subsection, we can obtain
		
		\begin{align*}
		N_1 =& \int_{0}^{\ubar} \int_{u_{\infty}}^u \frac{a}{\lvert u^\prime \rvert} \scaleoneSuprimeubarprime{(\al)^i F_1 \cdot \aln \alpha}\hspace{.5mm} \text{d}u^\prime \hspace{.5mm} \text{d}\ubar^{\prime} \\
		\leq& \int_{u_\infty}^{u} \frac{a}{\lvert u^\prime \rvert^2} \left(\int_{0}^{\ubar} \scaletwoSuprimeubarprime{(\al)^i F_1}^2 \hspace{.5mm} \text{d}\ubar^\prime \right)^{\frac{1}{2}} \hspace{.5mm} \text{d} u^\prime \cdot \sup_{u^\prime} \lVert \aln \alpha \rVert_{\mathcal{L}^2_{(sc)}(H_{u^\prime}^{(0,\underline{u})})} ,
		\end{align*}where we recall that 
		
		\begin{align*}
		F_1 =&\sum_{i_1+i_2+i_3+i_4+1 =i} \nabla^{i_1} \psi^{i_2+1} \nabla^{i_3}(\chibarhat,\tr\chibar) \nabla^{i_4}\alpha + \sum_{i_1+i_2+i_3+i_4=i} \nabla^{i_1} \psi^{i_2} \nabla^{i_3} (\psi,\chihat,\chibarhat,\widetilde{\tr\chibar}) \nabla^{i_4} (\Psi,\beta,\alpha)  \\
		&+ \sum_{i_1+i_2+i_3+i_4=i} \nabla^{i_1}\psi^{i_2} \nabla^{i_3}(\rho_F,\sigma_F, \alpha_F) \nabla^{i_4+1} (\rho_F,\sigma_F)	+ \sum_{i_1+i_2+i_3+i_4=i} \nabla^{i_1}\psi^{i_2} \nabla^{i_3}(\rho_F,\sigma_F) \nabla^{i_4+1} (\alpha_F)\\
		&+ \sum_{i_1+i_2+i_3+i_4+i_5=i} \nabla^{i_1}\psi^{i_2} \nabla^{i_3}(\psi,\chibarhat,\tr\chibar)\nabla^{i_4}(\Y,\alpha_F)\nabla^{i_5}(\Y,\alpha_F). 
		\end{align*}Denote \[  H_1 = \int_0^{\ubar} \scaletwoSuprimeubarprime{(\al)^i F_1}^2 \hspace{.5mm} \text{d}\ubar^\prime.    \]We work on the term $H_1$.
		
		\begin{align*}
		H_1 =& \int_0^{\ubar} \scaletwoSuprimeubarprime{(\al)^i F_1}^2 \hspace{.5mm} \text{d}\ubar^\prime \\
		&\leq \int_0^{\ubar} a^{-1} \scaleinfinitySuprimeubarprime{\frac{\al}{\lvert u^\prime \rvert}\psi,\frac{\al}{\lvert u^\prime \rvert} \chihat, \frac{\al}{\lvert u^\prime \rvert} \chibarhat, \frac{\al}{\lvert u^\prime \rvert} \widetilde{\tr\chibar}}^2 \scaletwoSuprimeubarprime{\aln (\Psi,\beta,\alpha)}^2 \hspace{.5mm} \text{d} \ubar^\prime \\ &+ \sum_{ i_1+i_2+i_3+i_4+1=i} \int_0^{\ubar} \frac{\lvert u^\prime \rvert^4}{a^2} \scaletwoSuprimeubarprime{(\al)^{i+1} \nabla^{i_1} \psi^{i_2+1} \nabla^{i_3} (\frac{a}{\lvert u^\prime \rvert^2}\chibarhat, \frac{a}{\lvert u^\prime \rvert^2}\tr\chibar )\nabla^{i_4}(\frac{\alpha}{\al})}^2 \hspace{.5mm} \text{d}\ubar^\prime \\ 
		&+ \sum_{\substack{i_1+i_2+i_3+i_4=i, \\ 
				i_4 \leq i-1}} \int_0^{\ubar} \frac{\lvert u^\prime \rvert^2}{a} \ScaletwoSuprimeubarprime{ (\al)^{i+1} \nabla^{i_1} \psi^{i_2} \nabla^{i_3} \left(\frac{\al}{\lvert u^\prime \rvert} \psi,\frac{\al}{\lvert u^\prime \rvert}\chihat, \frac{\al}{\lvert u^\prime \rvert}\chibarhat, \frac{\al}{\lvert u^\prime \rvert}\widetilde{\tr\chibar} \right) \nabla^{i_4}\left(\frac{\Psi}{\al},\frac{\beta}{\al},\frac{\alpha}{\al} \right)   }^2 \hspace{.5mm} \text{d}\ubar^\prime \\
		&+ \sum_{i_1+i_2+i_3+i_4=i} \int_0^{\ubar} a \scaletwoSuprimeubarprime{(\al)^i \nabla^{i_1}\psi^{i_2}\nabla^{i_3}(\frac{\Y}{\al},\frac{\alpha_F}{\al})\nabla^{i_4+1}(\rho_F,\sigma_F)}^2 \hspace{.5mm} \text{d}\ubar^\prime \\
		&+ \sum_{i_1+i_2+i_3+i_4=i} \int_0^{\ubar} a \scaletwoSuprimeubarprime{(\al)^i \nabla^{i_1}\psi^{i_2}\nabla^{i_3}(\rho_F,\sigma_F)\nabla^{i_4+1}\left( \frac{\alpha_F}{\al}     \right)}^2 \hspace{.5mm} \text{d}\ubar^\prime  \\ 
		&+ \sum_{i_1+ i_2 + i_3 +i_4 +i_5 =i} \int_0^{\ubar} {\lvert u^\prime \rvert^4} \scaletwoSuprimeubarprime{(\al)^{i} \nabla^{i_1} \psi^{i_2} \nabla^{i_3}( \frac{a}{\lvert u^\prime \rvert^2}\psi, \frac{a}{\lvert u^\prime \rvert^2}\chibarhat, \frac{a}{\lvert u^\prime \rvert^2} \tr\chibar) \nabla^{i_4}(\frac{\Y}{\al}, \frac{\alpha_F}{\al}) \nabla^{i_5} (\frac{\Y}{\al}, \frac{\alpha_F}{\al})}^2 \hspace{.5mm} \text{d}\ubar^\prime \\ &:= J_1+J_2+J_3+J_4+J_5+J_6. 
		\end{align*}
		\begin{itemize}
			\item The sum of the terms $J_1+J_2+J_3$ can be controlled just as in the vacuum case by  \[\mathcal{R}^2[\alpha]\cdot\left( O^2[\chihat]+O^2[\chibarhat] \right) +O^6+O^4. \]
			
			\item For the terms $J_4$ and $J_5$ we have
			
			\begin{align*}
			J_4 + J_5 =& \sum_{i_1+i_2+i_3+i_4=i} \int_0^{\ubar} a \scaletwoSuprimeubarprime{(\al)^i \nabla^{i_1}\psi^{i_2}\nabla^{i_3}(\frac{\Y}{\al},\frac{\alpha_F}{\al})\nabla^{i_4+1}(\rho_F,\sigma_F)}^2 \hspace{.5mm} \text{d}\ubar^\prime \\
			&+ \sum_{i_1+i_2+i_3+i_4=i} \int_0^{\ubar} a \scaletwoSuprimeubarprime{(\al)^i \nabla^{i_1}\psi^{i_2}\nabla^{i_3}(\rho_F,\sigma_F)\nabla^{i_4+1}\left( \frac{\alpha_F}{\al}     \right)}^2 \hspace{.5mm} \text{d}\ubar^\prime \\ \lesssim& \frac{O^4}{\upr^2}+ \frac{a\cdot O^2}{\upr^2} \cdot \left( \mathcal{F}[\alpha_F]^2+\mathcal{F}[\rho_F,\sigma_F]^2 \right) \lesssim \frac{O^4}{\upr^2} + \frac{a\cdot O^2}{\upr^2}\cdot (R^4  + 1).
			\end{align*}where we have used Theorem 6.1 in the last inequality.
			\item The final term $J_6$ can be bounded by $O^6$.
		\end{itemize}Therefore, putting everything together, we have 
		
		\[ H_1 \leq   \mathcal{R}^2[\alpha]\cdot\left( O^2[\chihat]+O^2[\chibarhat] \right)+ \frac{a \cdot O^2}{\lvert u^\prime \rvert^2}\cdot(R^4 + (\mathcal{I}^{(0)})^4 + (\mathcal{I}^{(0)})^2 +1) +O^6+O^4,  \]which translates to \begin{equation}
		N_1 \leq \left(\mathcal{R}[\alpha] \cdot\left(O[\chihat]+O[\chibarhat] \right) +\frac{\al\cdot O}{\lvert u^\prime \rvert} \cdot(R^2+(\mathcal{I}^{(0)})^2 + \mathcal{I}^{(0)} +1)+O^3+O^2\right)\cdot \sup_{u^\prime}\lVert \aln \alpha \rVert_{\mathcal{L}^2_{(sc)}(H_{u^\prime}^{(0,\ubar)}}. \label{N1curvaturealphabound}
		\end{equation}We continue with the term $M_1$ in the same way. We have
		
		\begin{align*}
		M_1 =& \int_{0}^{\ubar} \int_{u_{\infty}}^u \frac{a}{\lvert u^\prime \rvert} \scaleoneSuprimeubarprime{(\al)^i G_1 \cdot \aln \tbeta}\hspace{.5mm} \text{d}u^\prime \hspace{.5mm} \text{d}\ubar^{\prime} \\ 
		\leq& \left( \int_0^{\ubar} \int_{u_{\infty}}^u \frac{a}{\lvert u^\prime \rvert^2} \scaletwoSuprimeubarprime{(\al)^i G_1}^2 \hspace{.5mm} \text{d}u^\prime \hspace{.5mm} \text{d}\ubar^{\prime} \right)^{\frac{1}{2}}\cdot \sup_{\ubar^\prime} \lVert \aln \tbeta \rVert_{\mathcal{L}^2_{(sc)}(\Hbar_{\ubar^\prime}^{(u_{\infty},u)})}.
		\end{align*}At this point we recall that
		\begin{align*}
		G_1 =& \sum_{i_1+i_2+i_3+i_4=i} \nabla^{i_1}\psi^{i_2} \nabla^{i_3} (\psi,\chihat) \nabla^{i_4} (\tbeta,\alpha)\\ 
		&+ \sum_{i_1+i_2+i_3+i_4=i} \nabla^{i_1}\psi^{i_2} \nabla^{i_3} \alpha_F \nabla^{i_4+1} \alpha_F + \sum_{i_1+i_2+i_3+i_4=i} \nabla^{i_1}\psi^{i_2} \nabla^{i_3} (\psi,\chihat) \nabla^{i_4} (\Y,\alpha_F)\\
		&+ \sum_{i_1+i_2+i_3+i_4+i_5=i} \nabla^{i_1} \psi^{i_2} \nabla^{i_3}(\psi,\chihat) \nabla^{i_4}(\Y,\alpha_F) \nabla^{i_5}\alpha_F.
		\end{align*}Define \[ K_1 :=  \int_0^{\ubar} \int_{u_{\infty}}^u \frac{a}{\lvert u^\prime \rvert^2} \scaletwoSuprimeubarprime{(\al)^i G_1}^2 \hspace{.5mm} \text{d}u^\prime \hspace{.5mm} \text{d}\ubar^{\prime}.    \]We then have
		
		\begin{align*}
		K_1 \leq& \int_0^{\ubar} \int_{u_{\infty}}^u  \frac{a^2}{\lvert u^\prime \rvert^2} \ScaletwoSuprimeubarprime{(\al)^i \left(\frac\psi\al,\frac\chihat\al\right)\nabla^{i}\left(\frac\tbeta\al,\f\alpha\al\right)}^2 \hspace{.5mm} \text{d}u^\prime \hspace{.5mm} \text{d}\ubar^\prime \\
		&+ \int_0^{\ubar} \int_{u_{\infty}}^u  \frac{a}{\lvert u^\prime \rvert^2 }  \ScaletwoSuprimeubarprime{\sum_{\substack{i_1 + i_2 + i_3+i_4 =i \\ i_4 \leq i-1}}(\al)^i \nabla^{i_1} \psi^{i_2} \nabla^{i_3}\left(\frac\psi\al,\frac\chihat\al\right)\nabla^{i_4}\left(\frac\tbeta\al,\f\alpha\al\right)}^2 \hspace{.5mm} \text{d}u^\prime \hspace{.5mm} \text{d}\ubar^\prime \\
		&+  \int_0^{\ubar} \int_{u_{\infty}}^u  \frac{a}{\lvert u^\prime \rvert^2 } \scaletwoSuprimeubarprime{(\al)^i \alpha_F \cdot \nabla^{i+1}\alpha_F}^2 \hspace{.5mm} \text{d}u^\prime \hspace{.5mm} \text{d}\ubar^\prime \\ 
		&+ \int_0^{\ubar} \int_{u_{\infty}}^u  \frac{a}{\lvert u^\prime \rvert^2 } \scaletwoSuprimeubarprime{ \sum_{\substack{i_1+ i_2 + i_3 +i_4 =i\\ i_4 \leq i-1}}(\al)^i \nabla^{i_1} \psi^{i_2} \nabla^{i_3}\alpha_F \nabla^{i_4+1}\alpha_F     }^2 \hspace{.5mm} \text{d}u^\prime \hspace{.5mm} \text{d}\ubar^\prime \\ 
		&+ \int_0^{\ubar} \int_{u_{\infty}}^u  \frac{a}{\lvert u^\prime \rvert^2 } \scaletwoSuprimeubarprime{\sum_{i_1+i_2+i_3+i_4=i}(\al)^i \nabla^{i_1} \psi^{i_2} \nabla^{i_3}(\psi,\chihat)\nabla^{i_4}(\Y,\alpha_F)   }^2  \hspace{.5mm} \text{d}u^\prime \hspace{.5mm} \text{d}\ubar^\prime \\ 
		&+ \int_0^{\ubar} \int_{u_{\infty}}^u  \frac{a}{\lvert u^\prime \rvert^2 } \scaletwoSuprimeubarprime{\sum_{i_1+i_2+i_3+i_4+i_5=i}(\al)^i \nabla^{i_1} \psi^{i_2} \nabla^{i_3}(\psi,\chihat)\nabla^{i_4}(\Y,\alpha_F) \nabla^{i_5}\alpha_F   }^2  \hspace{.5mm} \text{d}u^\prime \hspace{.5mm} \text{d}\ubar^\prime  \\
		:=& I_1+I_2+I_3+I_4+I_5+I_6.
		\end{align*}
		
		\begin{itemize}
			\item The sum $I_1+I_2$ can be bounded as in the vacuum case by $R^2\cdot O^2 + O^4$.

			\item We have \begin{equation}\begin{split}
			I_3+I_4  &= \int_0^{\ubar} \int_{u_{\infty}}^u  \frac{a}{\lvert u^\prime \rvert^2 } \ScaletwoSuprimeubarprime{(\al)^i \left(\f{\alpha_F}{\al}\right) \cdot \nabla^{i+1}\alpha_F}^2 \hspace{.5mm} \text{d}u^\prime \hspace{.5mm} \text{d}\ubar^\prime \\ 
		&\,\,+ \int_0^{\ubar} \int_{u_{\infty}}^u  \frac{a}{\lvert u^\prime \rvert^2 } \scaletwoSuprimeubarprime{ \sum_{\substack{i_1+ i_2 + i_3 +i_4 =i\\ i_4 \leq i-1}}(\al)^i \nabla^{i_1} \psi^{i_2} \nabla^{i_3}\left(\f{\alpha_F}{\al}\right)  \nabla^{i_4+1}\left(\f{\alpha_F}{\al}\right)      }^2 \hspace{.5mm} \text{d}u^\prime \hspace{.5mm} \text{d}\ubar^\prime\\
			&\leq  \intu \frac{a}{\upr^2}\cdot \frac{a\cdot O^2}{\upr^2}\left( \intubar \scaletwoSuprimeubarprime{(\al)^i \nabla^{i+1}\alpha_F}^2 \dubarprime \right) \duprime + \intu \intubar \frac{a}{\upr^2}\cdot \frac{a\cdot O^4}{\upr^2}\dubarprime \duprime \\ &\lesssim \frac{a^3\cdot O^2}{\upr^3}\cdot \mathcal{F}[\alpha_F]^2 + 1 \lesssim O^2 \cdot (\mathcal{I}^{(0)})^2 +1.
			\end{split}\end{equation}	
			\item We have 
			
			\begin{equation}\begin{split}
			I_5 =& \int_0^{\ubar} \int_{u_{\infty}}^u  \frac{a}{\lvert u^\prime \rvert^2 } \scaletwoSuprimeubarprime{\sum_{i_1+i_2+i_3+i_4=i}(\al)^i \nabla^{i_1} \psi^{i_2} \nabla^{i_3}(\psi,\chihat)\nabla^{i_4}(\Y,\alpha_F)   }^2  \hspace{.5mm} \text{d}u^\prime \hspace{.5mm} \text{d}\ubar^\prime \\
			=& \int_0^{\ubar} \int_{u_{\infty}}^u  \frac{a^3}{\lvert u^\prime \rvert^2 } \scaletwoSuprimeubarprime{\sum_{i_1+i_2+i_3+i_4=i} (\al)^i \nabla^{i_1}\psi^{i_2} \nabla^{i_3}(\frac{\psi}{\al}, \frac{\chihat}{\al})\nabla^{i_4}(\frac{\Y}{\al}, \frac{\alpha_F}{\al})    }^2 \hspace{.5mm} \text{d}u^\prime \hspace{.5mm} \text{d}\ubar^{\prime}\\
			\leq& \frac{a^3\cdot O^4}{\lvert u \rvert^3} \leq O^4.
			\end{split}\end{equation}
			\item For the last term $I_6$ we have
			
			\begin{equation}\begin{split}
			I_6 &= \int_0^{\ubar} \int_{u_{\infty}}^u  \frac{a}{\lvert u^\prime \rvert^2 } \scaletwoSuprimeubarprime{\sum_{i_1+i_2+i_3+i_4+i_5=i}(\al)^i \nabla^{i_1} \psi^{i_2} \nabla^{i_3}(\psi,\chihat)\nabla^{i_4}(\Y,\alpha_F) \nabla^{i_5}\alpha_F   }^2  \hspace{.5mm} \text{d}u^\prime \hspace{.5mm} \text{d}\ubar^\prime \\ &=
			\int_0^{\ubar} \int_{u_{\infty}}^u  \frac{a^4}{\lvert u^\prime \rvert^2 } \scaletwoSuprimeubarprime{\sum_{i_1+i_2+i_3+i_4+i_5=i}(\al)^i \nabla^{i_1} \psi^{i_2} \nabla^{i_3}\left(\frac{\psi}{\al},\frac{\chihat}{\al}\right)\nabla^{i_4}\left(\frac{\Y}{\al},\frac{\alpha_F}{\al}\right) \nabla^{i_5}\left(\frac{\alpha_F}{\al}\right)   }^2  \hspace{.5mm} \text{d}u^\prime \hspace{.5mm} \text{d}\ubar^\prime \\ 
			&\leq \int_0^{\ubar} \int_{u_{\infty}}^u  \frac{a^4}{\lvert u^\prime \rvert^2}\cdot \frac{O^6}{\lvert u^\prime \rvert^4} \duprime \hspace{.5mm} \text{d}\ubar^\prime = \frac{a^4 \cdot O^6}{\lvert u \rvert^5} \leq \frac{O^6}{a}.
			\end{split}\end{equation}\end{itemize}Putting everything together, we have
		
		\begin{equation}
		K_1 \leq  R^2\cdot O^2 + O^4 + O^6 + \frac{O^2}{\lvert u \rvert^3}(R^4 +  (\mathcal{I}^{(0)})^4 + (\mathcal{I}^{(0)})^2  +1),  
		\end{equation}so that

		\be M_1 \leq \left( R\cdot O + O^2 + O^3 + \frac{O}{\lvert u \rvert^{\frac{3}{2}}}(R^2 +  (\mathcal{I}^{(0)})^2 +\mathcal{I}^{(0)}  +1) \right)\cdot \sup_{\ubar^{\prime}}   \lVert \aln \tbeta \rVert_{\mathcal{L}^2_{(sc)}(\Hbar_{\ubar^\prime}^{(u_{\infty},u)})}. \label{M1curvaturealphabound}  \ee
		Putting \eqref{N1curvaturealphabound} and \eqref{M1curvaturealphabound} together, we finally get
		
		\begin{align*}
		&\frac{1}{\al} \scaletwoHu{\aln \alpha} + \frac{1}{\al}\scaletwoHbaru{\aln \tbeta} \\
		\leq& \frac{1}{\al} \lVert \aln \alpha \rVert_{\mathcal{L}^{2}_{(sc)}(H_{u_\infty}^{(0,\ubar)})}  + \frac{1}{\al} \lVert \aln \tbeta \rVert_{\mathcal{L}^2_{(sc)}(\underline{H}_{0}^{(u_{\infty},u)})} +\frac{1}{\al}N_1+ \frac{1}{\al}M_1\\
		\leq&  \frac{1}{\al} \lVert \aln \alpha \rVert_{\mathcal{L}^{2}_{(sc)}(H_{u_\infty}^{(0,\ubar)})}  + \frac{1}{\al} \lVert \aln \tbeta \rVert_{\mathcal{L}^2_{(sc)}(\underline{H}_{0}^{(u_{\infty},u)})} \\
		&+ \frac{1}{\al} \left(\mathcal{R}[\alpha] \cdot\left(O[\chihat]+O[\chibarhat] \right) +\frac{\al\cdot O}{\lvert u^\prime \rvert} \cdot(R^2+(\mathcal{I}^{(0)})^2 + \mathcal{I}^{(0)} +1)+O^3+O^2\right)\cdot \sup_{u^\prime}\lVert \aln \alpha \rVert_{\mathcal{L}^2_{(sc)}(H_{u^\prime}^{(0,\ubar)}}\\
		&+  \frac{1}{\al}\cdot \left( R\cdot O + O^2 + O^3 + \frac{O}{\lvert u \rvert^{\frac{3}{2}}}(R^2 +  (\mathcal{I}^{(0)})^2 +\mathcal{I}^{(0)}  +1) \right)\cdot \sup_{\ubar^{\prime}}   \lVert \aln \tbeta \rVert_{\mathcal{L}^2_{(sc)}(\Hbar_{\ubar^\prime}^{(u_{\infty},u)})} 
		\end{align*}from which our result follows.
		
	\end{proof}		\subsubsection{Estimates on the remaining components}
	
	We proceed to show the estimates for the pair $(\Psi_1, \Psi_2)= (\tbeta, (\rho,\sigma))$. The other two pairs are similar.
	
	\begin{proof}
		We have the schematic equation  
		
		\begin{gather}
		\nabla_3 \tbeta + \tr\chibar \tbeta - \mathcal{D}(-\rho,\sigma) = (\psi,\chihat) \Psi + \alpha_F \nabla(\alphabar_F,\rho_F,\sigma_F)+ (\rho_F,\sigma_F)\nabla(\rho_F,\sigma_F,\alpha_F)+  (\psi,\chibarhat, \tr\chibar,\chihat)\cdot (\alpha_F,\Y) \cdot \Y,\\
		\nabla_4 ((-\rho,\sigma))- \Hodge{\mathcal{D}}\tbeta = (\psi,\chihat)(\Psi,\alpha) +\alpha_F \nabla \Y + \Y \nabla(\Y,\alpha_F) + (\psi,\chibarhat,\tr\chibar,\chihat)(\alpha_F,\Y)(\alpha_F,\Y).
		\end{gather}Commuting with $i$ angular derivatives we get
		\begin{align*}
		&\nabla_3 \nabla^i \tbeta + \frac{i+2}{2}\tr\chibar \nabla^i \tbeta - \mathcal{D} \nabla^i ((-\rho,\sigma)) \\
		=& \sum_{ i_1+i_2+i_3+i_4+1=i} \nabla^{i_1}\psi^{i_2+1} \nabla^{i_3}(\chibarhat,\tr\chibar)\nabla^{i_4} \Psi + \sum_{ i_1+i_2+i_3+i_4=i} \nabla^{i_1} \psi^{i_2} \nabla^{i_3}(\psi,\chihat,\chibarhat,\widetilde{\tr \chibar})\nabla^{i_4}\Psi \\ 
		&+ \sum_{ i_1+i_2+i_3+i_4=i} \nabla^{i_1}\psi^{i_2} \nabla^{i_3}(\alpha_F,\Y) \nabla^{i_4+1} \Y + \sum_{ i_1+i_2+i_3+i_4=i} \nabla^{i_1}\psi^{i_2} \nabla^{i_3} \Y \nabla^{i_4+1}(\Y,\alpha_F)\\ 
		&+ \sum_{ i_1+i_2+i_3+i_4+i_5=i} \nabla^{i_1}\psi^{i_2} \nabla^{i_3}(\psi,\chibarhat,\tr\chibar)\nabla^{i_4}(\alpha_F,\Y) \nabla^{i_5} \Y \\ :=& F_2,
		\end{align*}while
		
		\begin{align*}
		&\nabla_4 \nabla^i ((-\rho,\sigma)) - \Hodge{\mathcal{D}} \nabla^i \tbeta\\
		 =& \sum_{ i_1+i_2+i_3+i_4=i} \nabla^{i_1}\psi^{i_2} \nabla^{i_3}(\psi,\chibarhat,\chihat) \nabla^{i_4}(\Psi,\alpha) \\
		&+ \sum_{ i_1+i_2+i_3+i_4=i} \nabla^{i_1}\psi^{i_2} \nabla^{i_3}(\alpha_F,\Y) \nabla^{i_4+1} \Y + \sum_{ i_1+i_2+i_3+i_4=i} \nabla^{i_1}\psi^{i_2} \nabla^{i_3} \Y \nabla^{i_4+1}(\Y,\alpha_F)\\ 
		&+ \sum_{ i_1+i_2+i_3+i_4+i_5=i} \nabla^{i_1}\psi^{i_2} \nabla^{i_3}(\psi,\chibarhat,\tr\chibar)\nabla^{i_4}(\alpha_F,\Y) \nabla^{i_5} (\alpha_F,\Y) \\ :=& G_2.
		\end{align*}We have
		
		\begin{align*}
		&\scaletwoHu{\aln \tbeta}^2 + \scaletwoHbaru{\aln (\rho,\sigma)}^2 \\
		\leq&  \lVert \aln \tbeta \rVert^2_{\mathcal{L}^{2}_{(sc)}(H_{u_\infty}^{(0,\ubar)})}  +  \lVert \aln (\rho,\sigma) \rVert^2_{\mathcal{L}^2_{(sc)}(\underline{H}_{0}^{(u_{\infty},u)})} +N_2+M_2,
		\end{align*}where \begin{gather}
		N_2 = \int_{0}^{\ubar} \int_{u_{\infty}}^u \frac{a}{\lvert u^\prime \rvert} \scaleoneSuprimeubarprime{(\al)^i F_2 \cdot \aln \tbeta}\hspace{.5mm} \text{d}u^\prime \hspace{.5mm} \text{d}\ubar^{\prime},\\
		M_2 = \int_{0}^{\ubar} \int_{u_{\infty}}^u \frac{a}{\lvert u^\prime \rvert} \scaleoneSuprimeubarprime{(\al)^i G_2 \cdot \aln (\rho,\sigma)}\hspace{.5mm} \text{d}u^\prime \hspace{.5mm} \text{d}\ubar^{\prime}.
		\end{gather}We have 
		
		\begin{align*}
		N_2 =&  \int_{0}^{\ubar} \int_{u_{\infty}}^u \frac{a}{\lvert u^\prime \rvert} \scaleoneSuprimeubarprime{(\al)^i F_2 \cdot \aln \tbeta}\hspace{.5mm} \text{d}u^\prime \hspace{.5mm} \text{d}\ubar^{\prime} \\
		\leq& \int_{u_{\infty}}^u  \frac{a}{\lvert u^\prime \rvert^2} \left( \int_0^{\ubar} \scaletwoSuprimeubarprime{(\al)^i F_2}^2 \hspace{.5mm} \text{d}\ubar^\prime \right)^{\frac{1}{2}} \hspace{.5mm} \text{d}u^\prime \cdot R,
		\end{align*}where we recall 
		
		\begin{align*}
		F_2 =& \sum_{ i_1+i_2+i_3+i_4+1=i} \nabla^{i_1}\psi^{i_2+1} \nabla^{i_3}(\chibarhat,\tr\chibar)\nabla^{i_4} \Psi + \sum_{ i_1+i_2+i_3+i_4=i} \nabla^{i_1} \psi^{i_2} \nabla^{i_3}(\psi,\chihat,\chibarhat,\widetilde{\tr \chibar})\nabla^{i_4}\Psi \\
		&+ \sum_{ i_1+i_2+i_3+i_4=i} \nabla^{i_1}\psi^{i_2} \nabla^{i_3}(\alpha_F,\Y) \nabla^{i_4+1} \Y + \sum_{ i_1+i_2+i_3+i_4=i} \nabla^{i_1}\psi^{i_2} \nabla^{i_3} \Y \nabla^{i_4+1}(\Y,\alpha_F)\\
		&+ \sum_{ i_1+i_2+i_3+i_4+i_5=i} \nabla^{i_1}\psi^{i_2} \nabla^{i_3}(\psi,\chibarhat,\tr\chibar)\nabla^{i_4}(\alpha_F,\Y) \nabla^{i_5} \Y.
		\end{align*}Define $H_2= \int_0^{\ubar} \scaletwoSuprimeubarprime{(\al)^i F_2}^2 \hspace{.5mm} \text{d}\ubar^\prime$. The first two terms are handled like in the vacuum case. For the last three terms, we have
		
		\begin{itemize}
			\item For the first of the last three terms, we control \begin{equation}\begin{split}
		&	\int_0^{\ubar} \ScaletwoSuprimeubarprime{(\al)^{i+1} \sum_{ i_1+i_2+i_3+i_4=i} \nabla^{i_1} \psi^{i_2} \nabla^{i_3}\left(\f{\alpha_F}{\al},\f{\Y}{\al}\right) \nabla^{i_4+1} \Y}^2 \hspace{.5mm} \text{d}\ubar^\prime \\
			\leq& \int_0^{\ubar} \ScaletwoSuprimeubarprime{ (\al)^{i+1} \left(\f{\alpha_F}{\al},\f{\Y}{\al}\right) \nabla^{i+1} \Y}^2 \hspace{.5mm} \text{d}\ubar^\prime \\
			&+ \int_0^{\ubar} \ScaletwoSuprimeubarprime{(\al)^{i+1} \sum_{\substack{ i_1+i_2+i_3+i_4=i\\ i_4 \leq i-1}} \nabla^{i_1} \psi^{i_2} \nabla^{i_3}\left(\f{\alpha_F}{\al},\f{\Y}{\al}\right) \nabla^{i_4+1} \Y}^2 \hspace{.5mm} \text{d}\ubar^\prime\\
			\leq& \frac{O^2}{\lvert u^\prime \rvert^2}\int_0^{\ubar} \scaletwoSuprimeubarprime{(\al \nabla)^{i+1}\Y}^2 + \int_0^{\ubar} \frac{O^4}{\lvert u^\prime \rvert^2} \text{d}\ubar^\prime  
			\leq \frac{a \cdot O^2}{\lvert u^\prime \rvert^2} (\mathcal{F}[\Y])^2 + \frac{O^4}{\lvert u^\prime \rvert^2}.
			\end{split}\end{equation}
			
			\item For the middle term we have 
			
			\begin{equation}\begin{split}
		& \hspace{5mm}	\int_0^{\ubar} \ScaletwoSuprimeubarprime{(\al)^{i+1} \sum_{ i_1+i_2+i_3+i_4=i} \nabla^{i_1} \psi^{i_2} \nabla^{i_3}\Y \nabla^{i_4+1} \left(\f{\alpha_F}{\al},\f\Y\al \right)}^2 \hspace{.5mm} \text{d}\ubar^\prime \\ 
			&\leq \int_0^{\ubar} \ScaletwoSuprimeubarprime{ (\al)^{i+1}  \Y \nabla^{i+1} \left(\f{\alpha_F}{\al},\f\Y\al \right)}^2 \hspace{.5mm} \text{d}\ubar^\prime  \\&\,\,+ \int_0^{\ubar} \scaletwoSuprimeubarprime{(\al)^{i+1} \sum_{\substack{ i_1+i_2+i_3+i_4=i\\ i_4 \leq i-1}} \nabla^{i_1} \psi^{i_2} \nabla^{i_3}\Y \nabla^{i_4+1} \left(\f{\alpha_F}{\al},\f\Y\al \right)}^2 \hspace{.5mm} \text{d}\ubar^\prime \\ &\leq \frac{a \cdot O^2}{\lvert u^\prime \rvert^2} (\mathcal{F}[\alpha_F])^2 + \frac{O^4}{\lvert u^\prime \rvert^2}.
			\end{split}\end{equation}
			\item For the last term we have
			
			\begin{align*}
			&\int_0^{\ubar} \scaletwoSuprimeubarprime{(\al)^i \sum_{ i_1+i_2+i_3+i_4+i_5=i} \nabla^{i_1}\psi^{i_2} \nabla^{i_3}(\psi,\chibarhat,\tr\chibar)\nabla^{i_4}(\alpha_F,\Y)\nabla^{i_5}\Y  }^2 \hspace{.5mm} \text{d}\ubar^{\prime} \\
			\leq& \int_0^{\ubar} \frac{\lvert u^\prime \rvert^4}{a} \scaletwoSuprimeubarprime{(\al)^i \sum_{ i_1+i_2+i_3+i_4+i_5=i} \nabla^{i_1} \psi^{i_2} \nabla^{i_3}(\frac{a}{\lvert u^\prime \rvert^2}\psi,\frac{a}{\lvert u^\prime \rvert^2}\chibarhat,\frac{a}{\lvert u^\prime \rvert^2}\tr\chibar)\nabla^{i_4}(\frac{\alpha_F}{\al}, \frac{\Y}{\al})\nabla^{i_5}\Y   }^2 \hspace{.5mm}\text{d}\ubar^\prime \\
			\leq& \int_0^{\ubar} \frac{\lvert u^\prime \rvert^4}{a}\cdot \frac{O^6}{\lvert u^\prime \rvert^4} = \frac{O^6}{a}.
			\end{align*}
		\end{itemize}This completes the bounds on $N_2$. For $M_2$, we have 
		
		\begin{align*}
		M_2 =& \int_0^{\ubar}\int_{u_\infty}^u  \frac{a}{\lvert u^\prime \rvert} \scaleoneSuprimeubarprime{(\al)^i G_2 \cdot \aln \Psi}\hspace{.5mm} \text{d}u^\prime \text{d}\ubar^{\prime} \\ 
		\leq& \int_0^{\ubar} \int_{u_{\infty}}^u \frac{a}{\lvert u^\prime \rvert^2} \scaletwoSuprimeubarprime{(\al)^i G_2} \scaletwoSuprimeubarprime{\aln \Psi} \hspace{.5mm} \text{d}u^\prime \hspace{.5mm} \text{d}\ubar^\prime \\
		\leq& \int_0^{\ubar} \left( \int_{u_{\infty}}^u \frac{a}{\lvert u^\prime \rvert^2}\scaletwoSuprimeubarprime{(\al)^i G_2}^2 \hspace{.5mm} \text{d}u^\prime \right)^{\frac{1}{2}} \left( \int_{u_{\infty}}^u \frac{a}{\lvert u^\prime \rvert^2}\scaletwoSuprimeubarprime{\aln \Psi}^2 \hspace{.5mm} \text{d}u^\prime \right)^{\frac{1}{2}} \hspace{.5mm} \text{d}\ubar^{\prime}  \\
		=&  \int_0^{\ubar} \left( \int_{u_{\infty}}^u \frac{a}{\lvert u^\prime \rvert^2}\scaletwoSuprimeubarprime{(\al)^i G_2}^2 \hspace{.5mm} \text{d}u^\prime \right)^{\frac{1}{2}} \lVert \aln \Psi \rVert_{\mathcal{L}^2_{(sc)}(\Hbar_{\ubar^\prime}^{(u_{\infty},u)})} \hspace{.5mm} \text{d}\ubar^{\prime} \\
		\leq& \left( \int_{0}^{\ubar} \int_{u_{\infty}}^u \frac{a}{\lvert u^\prime \rvert^2} \scaletwoSuprimeubarprime{(\al)^i G_2}^2 \hspace{.5mm} \text{d}u^\prime \hspace{.5mm} \text{d}\ubar^{\prime}\right)^{\frac{1}{2}} \left( \int_0^{\ubar} \lVert \aln \Psi \rVert^2_{\mathcal{L}^2_{(sc)}(\Hbar_{\ubar^\prime}^{(u_{\infty},u)})} \hspace{.5mm} \text{d}\ubar^\prime \right)^\frac{1}{2} \\
		\leq& \int_{0}^{\ubar} \int_{u_{\infty}}^u \frac{a}{\lvert u^\prime \rvert^2} \scaletwoSuprimeubarprime{(\al)^i G_2}^2 \hspace{.5mm} \text{d}u^\prime \hspace{.5mm} \text{d}\ubar^{\prime} + \frac{1}{4} \int_0^{\ubar} \lVert \aln \Psi \rVert^2_{\mathcal{L}^2_{(sc)}(\Hbar_{\ubar^\prime}^{(u_{\infty},u)})} \hspace{.5mm} \text{d}\ubar^\prime.  \end{align*}It is already clear that the last term above will eventually be handled by Gr\"onwall's inequality. The following lemma will allow this:
		
		\begin{lemma}[Lemma 14.8 in \cite{KR:Trapped}]
			Let $f(x,y), g(x,y)$ be positive functions defined on the rectangle $U := \begin{Bmatrix} (x,y) \hspace{.5mm} \mid \hspace{.5mm} 0\leq x\leq x_0, \hspace{.5mm} 0\leq y \leq y_0 \end{Bmatrix}$. Suppose there exist nonnegative constants $J,c_1,c_2$ such that $f$ and $g$ verify the inequality
			
			\[ f(x,y)+g(x,y) \lesssim J + c_1\int_0^x f(x',y) \hspace{.5mm}\text{d}x^\prime + c_2\int_0^y g(x,y^\prime)\hspace{.5mm} \text{d}y^\prime     \]for all $(x,y) \in U$. Then there holds
			
			\[f(x,y) + g(x,y) \lesssim J \e^{c_1x+c_2y}, \hspace{5mm} \forall (x,y)\in U.      \]
		\end{lemma}Before applying Gr\"onwall, we first define \[ K_2 = \int_0^{\ubar}\int_{u_{\infty}}^u \frac{a}{\lvert u^\prime \rvert^2} \scaletwoSuprimeubarprime{(\al)^i G_2}^2 \hspace{.5mm} \text{d}u^\prime \hspace{.5mm} \text{d}\ubar^{\prime} .   \]We then have

		\begin{align*}
		K_2 \leq& \int_0^{\ubar}\int_{u_{\infty}}^u \frac{a}{\lvert u^\prime \rvert^2} \scaletwoSuprimeubarprime{(\al)^i \sum_{i_1+i_2+i_3+i_4=i} \nabla^{i_1} \psi^{i_2} \nabla^{i_3}(\psi,\chihat,\chibarhat)\nabla^{i_4}(\Psi,\alpha)}^2 \hspace{.5mm} \text{d}u^\prime \hspace{.5mm} \text{d}\ubar^{\prime}\\
		&+ \int_0^{\ubar}\int_{u_{\infty}}^u \frac{a}{\lvert u^\prime \rvert^2} \scaletwoSuprimeubarprime{(\al)^i \sum_{i_1+i_2+i_3 +i_4=i}\nabla^{i_1}\psi^{i_2}\nabla^{i_3}(\alpha_F,\Y) \nabla^{i_4+1}\Y}^2 \hspace{.5mm} \text{d}u^\prime \hspace{.5mm} \text{d}\ubar^{\prime} \\
		&+ \int_0^{\ubar}\int_{u_{\infty}}^u \frac{a}{\lvert u^\prime \rvert^2} \scaletwoSuprimeubarprime{(\al)^i \sum_{i_1+i_2+i_3 +i_4=i}\nabla^{i_1}\psi^{i_2}\nabla^{i_3}\Y \nabla^{i_4+1}\alpha_F}^2 \hspace{.5mm} \text{d}u^\prime \hspace{.5mm} \text{d}\ubar^{\prime} \\
		&+ \int_0^{\ubar}\int_{u_{\infty}}^u \frac{a}{\lvert u^\prime \rvert^2} \scaletwoSuprimeubarprime{(\al)^i \sum_{i_1+i_2+i_3 +i_4+i_5=i}\nabla^{i_1}\psi^{i_2}\nabla^{i_3}(\psi,\chibarhat,\tr\chibar) \nabla^{i_4}(\alpha_F,\Y) \nabla^{i_5}(\alpha_F,\Y)}^2 \hspace{.5mm} \text{d}u^\prime \hspace{.5mm} \text{d}\ubar^{\prime} \\
		:=& K_{21}+K_{22}+K_{23}+K_{24}.
		\end{align*}	
		
		\begin{itemize}
			\item The term $K_{21}$ can be controlled as below
			
\begin{equation*}
\begin{split}
K_{21}\leq& \int_0^{\ub}\int_{u_{\infty}}^u \frac{a}{|u'|^2}\|(\at)^i (\p,\chih, \chibh)\nab^i(\Psi, \a)\|^2_{L^2_{sc}(S_{u',\ub})}du'd\ub'\\
&+\int_0^{\ub}\int_{u_{\infty}}^u \f{a}{|u'|^2} \|(\at)^i\sum_{\substack{i_1+i_2+i_3+i_4=i\\i_4\leq i-1}}\nabla^{i_1}\p^{i_2}\nabla^{i_3}(\p, \chih, \chibh)\nabla^{i_4} (\Psi, \a)\|^2_{L^2_{sc}(S_{u',\ub})}  du'd\ub'\\
\leq& \int_0^{\ub}\int_{u_{\infty}}^u \frac{a}{|u'|^4}\|\p, \chih, \chibh\|^2_{L^{\infty}_{sc}(S_{u',\ub})} \|(\at)^i\nab^i (\Psi, \a)\|^2_{L^2_{sc}(S_{u',\ub})}du'd\ub'\\
&+\int_0^{\ub}\int_{u_{\infty}}^u \|(\at)^{i+1}\sum_{\substack{i_1+i_2+i_3+i_4=i\\i_4\leq i-1}}\nabla^{i_1}\p^{i_2}\nabla^{i_3}(\f{\at}{|u'|}\p,  \f{\at}{|u'|}\chih,\f{\at}{|u'|}\chibh)\nabla^{i_4} (a^{-\f12}\Psi, a^{-\f12}\a)\|^2_{L^2_{sc}(S_{u',\ub})}  du'd\ub'
\end{split}
\end{equation*}

\begin{equation*}
\begin{split}
\leq& \int_0^{\ub}\int_{u_{\infty}}^u a^{-1}\|\f{\at}{|u'|}\p,  \f{\at}{|u'|}\chih, \f{\at}{|u'|}\chibh\|^2_{L^{\infty}_{sc}(S_{u',\ub})} \f{a}{|u'|^2} \|(\at)^i\nab^i\Psi\|^2_{L^2_{sc}(S_{u',\ub})}du'd\ub'\\
&+ \int_0^{\ub}\int_{u_{\infty}}^u \|\f{\at}{|u'|} \p,  \f{\at}{|u'|}\chih, \f{\at}{|u'|} \chibh\|^2_{L^{\infty}_{sc}(S_{u',\ub})} \f{a}{|u'|^2} \|(\at)^i\nab^i (a^{-\f12}\a)\|^2_{L^2_{sc}(S_{u',\ub})}du'd\ub'\\
&+ \int_0^{\ub}\int_{u_{\infty}}^u \bigg(\f{a}{|u'|^2}O^2[\chibh]O^2[\a]+\f{a^{-\f12}\cdot a\cdot O^4}{|u'|^2}\bigg) du' d\ub' \\
\leq& \int_0^{\ub} a^{-1}\sup_{u'}\bigg(\|\f{\at}{|u'|}\p,  \f{\at}{|u'|}\chih, \f{\at}{|u'|}\chibh\|^2_{L^{\infty}_{sc}(S_{u',\ub})}\bigg) \bigg(\int_{u_{\infty}}^u  \f{a}{|u'|^2} \|(\at)^i\nab^i\Psi\|^2_{L^2_{sc}(S_{u',\ub})}du'\bigg)d\ub'\\
&+\int_{u_{\infty}}^u \sup_{u'}\bigg(\|\f{\at}{|u'|} \p,  \f{\at}{|u'|}\chih, \f{\at}{|u'|} \chibh\|^2_{L^{\infty}_{sc}(S_{u',\ub})}\bigg) \f{a}{|u'|^2} \bigg(\int_0^{\ub}\|(\at)^i\nab^i (a^{-\f12}\a)\|^2_{L^2_{sc}(S_{u',\ub})}d\ub'\bigg) du'\\
&+\f{a}{|u|}\bigg(O^2[\chibh]+O^2[\chih]+1\bigg)O^2[\a]+\f{\at}{|u|}\cdot O^4\\
\leq&a^{-1}\cdot O^2 \cdot R^2+\bigg(O^2[\chibh]+O^2[\chih]+1\bigg) \cdot \big(\Rb^2[\b]+\f{1}{a}\big)\\
&+\f{a}{|u|}\bigg(O^2[\chibh]+O^2[\chih]+1\bigg) \cdot \bigg(\M R^2[\a]+O^2[\a]\bigg)+\f{\at}{|u|}\cdot O^4\\
\ls& \f{1}{a^{\f13}}+(1+\M R^2[\a])\cdot(\Rb^2[\b]+\M R^2[\a]+1)\leq \left( \mathcal{I}^{(0)} \right)^4 +1.
\end{split}
\end{equation*}

			\item The sum $K_{22} +K_{23}$ can be controlled by

			\begin{equation}\begin{split} &K_{22}+K_{23} \\ =& \int_0^{\ubar}\int_{u_{\infty}}^u \frac{a}{\lvert u^\prime \rvert^2} \scaletwoSuprimeubarprime{(\al)^i \sum_{i_1+i_2+i_3 +i_4=i}\nabla^{i_1}\psi^{i_2}\nabla^{i_3}(\alpha_F,\Y) \nabla^{i_4+1}\Y}^2 \hspace{.5mm} \text{d}u^\prime \hspace{.5mm} \text{d}\ubar^{\prime} \\
		&+ \int_0^{\ubar}\int_{u_{\infty}}^u \frac{a}{\lvert u^\prime \rvert^2} \scaletwoSuprimeubarprime{(\al)^i \sum_{i_1+i_2+i_3 +i_4=i}\nabla^{i_1}\psi^{i_2}\nabla^{i_3}\Y \nabla^{i_4+1}\alpha_F}^2 \hspace{.5mm} \text{d}u^\prime \hspace{.5mm} \text{d}\ubar^{\prime}     \\ \lesssim& \int_0^{\ubar}\int_{u_{\infty}}^u \frac{a^2}{\lvert u^\prime \rvert^2} \ScaletwoSuprimeubarprime{a^5 \left(\f{\alpha_F}{\al},\f\Y\al\right)\nabla^{11}\Y}^2  \\&+ \int_0^{\ubar}\int_{u_{\infty}}^u \frac{a}{\lvert u^\prime \rvert^2} \ScaletwoSuprimeubarprime{(\al)^{i+1} \sum_{\substack{i_1+i_2+i_3 +i_4=i,\\ i_4<10}}\nabla^{i_1}\psi^{i_2}\nabla^{i_3}\left(\f{\alpha_F}{\al},\f\Y\al \right) \nabla^{i_4+1}\Y}^2 \hspace{.5mm} \text{d}u^\prime \hspace{.5mm} \text{d}\ubar^{\prime}   \\ &+ \int_0^{\ubar}\int_{u_{\infty}}^u \frac{a^2}{\lvert u^\prime \rvert^2} \ScaletwoSuprimeubarprime{a^5 \Y \nabla^{11}\alpha_F}^2  \\&+ \int_0^{\ubar}\int_{u_{\infty}}^u \frac{a}{\lvert u^\prime \rvert^2} \ScaletwoSuprimeubarprime{(\al)^{i+1} \sum_{\substack{i_1+i_2+i_3 +i_4=i,\\ i_4<10}}\nabla^{i_1}\psi^{i_2}\nabla^{i_3} \Y \nabla^{i_4+1}\alpha_F}^2 \hspace{.5mm} \text{d}u^\prime \hspace{.5mm} \text{d}\ubar^{\prime}         \\ \leq& \int_{0}^{\ubar} \int_{u_{\infty}}^u \frac{a}{\lvert u^\prime \rvert^2}\cdot \frac{a(O^4+O^2\cdot F^2)}{\lvert u^\prime \rvert^2}\duprime \hspace{.5mm} \text{d}\ubar^\prime \lesssim 1.   \end{split}\end{equation}
			
			\item The final term $K_{24}$ can be controlled in the same way
		    \begin{equation}\begin{split}
		        &\int_0^{\ubar}\int_{u_{\infty}}^u \frac{a}{\lvert u^\prime \rvert^2} \scaletwoSuprimeubarprime{(\al)^i \sum_{i_1+i_2+i_3 +i_4+i_5=i}\nabla^{i_1}\psi^{i_2}\nabla^{i_3}(\psi,\chibarhat,\tr\chibar) \nabla^{i_4}(\alpha_F,\Y) \nabla^{i_5}(\alpha_F,\Y)}^2 \hspace{.5mm} \text{d}u^\prime \hspace{.5mm} \text{d}\ubar^{\prime} \\ = &\int_0^{\ubar}\int_{u_{\infty}}^u a\cdot \lvert u^{\prime} \rvert^2 \cdot \ScaletwoSuprimeubarprime{(\al)^i \sum_{i}\nabla^{i_1}\psi^{i_2}\nabla^{i_3}\left (\frac{a(\psi,\chibarhat,\tr\chibar)}{\upr^2} \right) \nabla^{i_4}\left(\frac{(\alpha_F,\Y)}{\al}\right) \nabla^{i_5}\left(\frac{(\alpha_F,\Y)}{\al}\right)}^2 \\ &\text{d}u^\prime \hspace{.5mm} \text{d}\ubar^{\prime}. 
		    \end{split}\end{equation}Clearly, in the above, the worst bound that can be obtained is when there is a triple anomaly. This only holds when the Ricci coefficient is $\tr\chibar$ and both Maxwell components are $\alpha_F$. Also, importantly, the number $i$ of angular derivatives is at most $10$ so that no elliptic estimates are required. In this case,
			
			\begin{equation}
			    \begin{split}
			        &\int_0^{\ubar}\int_{u_{\infty}}^u a\cdot \lvert u^{\prime} \rvert^2 \cdot \ScaletwoSuprimeubarprime{(\al)^i \sum_{i}\nabla^{i_1}\psi^{i_2}\nabla^{i_3}\left (\frac{a \hspace{.5mm} \tr\chibar}{\upr^2} \right) \nabla^{i_4}\left(\frac{\alpha_F}{\al}\right) \nabla^{i_5}\left(\frac{\alpha_F}{\al}\right)}^2 \duprime \dubarprime \\ &\lesssim \sup_{\ubar} \intu  a\cdot \upr^2 \cdot \frac{O^2[\tr\chibar] O^4[\alpha_F]}{\upr^4} \duprime \lesssim O^2[\tr\chibar] O^4[\alpha_F].
			    \end{split}
			\end{equation} Now $O[\tr\chibar]\lesssim 1$ by Proposition \ref{trchibarboundRicci} and $O[\alpha_F] \lesssim \underline{\mathcal{F}}[\rho_F,\sigma_F]+1 \lesssim 2\underline{\mathcal{F}}_0[\rho_F,\sigma_F]+ \frac{1}{a^{\frac{1}{8}}}$, where in the first inequality we use the scale--invariant $L^2$--estimates on the Maxwell components and in the second inequality the energy estimates.  In particular, the inequality can be traced back to the initial data and this is a key point.
		\end{itemize}
		\vspace{3mm}
		
		\par \noindent Finally, we use Gr\"onwall for $M_2$ and conclude. In a similar way, we obtain
		
		\begin{equation*}
		{\color{black} \scaletwoHu{\aln(\rho,\sigma, \tbetabar)} + \scaletwoHbaru{\aln (\tbetabar, \alphabar)} \lesssim (\mathcal{I}^{(0)})^2+ \mathcal{I}^{(0)}+1. }
		\end{equation*}The result follows.
	\end{proof}

	\section{The formation of trapped surfaces}\label{TSF}

In this section, we prove

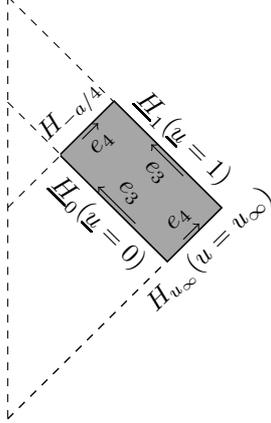
\begin{minipage}[!t]{0.27\textwidth}
\begin{tikzpicture}[scale=0.70]
\draw [white](3,-1)-- node[midway, sloped, below,black]{$H_{u_{\infty}}(u=u_{\infty})$}(4,0);
\draw [white](1,1)-- node[midway,sloped,above,black]{$H_{-a/4}$}(2,2);
\draw [white](2,2)--node [midway,sloped,above,black] {$\Hb_{1}(\ub=1)$}(4,0);
\draw [white](1,1)--node [midway,sloped, below,black] {$\Hb_{0}(\ub=0)$}(3,-1);
\draw [dashed] (0, 4)--(0, -4);
\draw [dashed] (0, -4)--(4,0)--(0,4);
\draw [dashed] (0,0)--(2,2);
\draw [dashed] (0,-4)--(4,0);
\draw [dashed] (0,2)--(3,-1);
\draw [very thick] (1,1)--(3,-1)--(4,0)--(2,2)--(1,1);
\fill[black!35!white] (1,1)--(3,-1)--(4,0)--(2,2)--(1,1);
\draw [->] (3.3,-0.6)-- node[midway, sloped, above,black]{$e_4$}(3.6,-0.3);
\draw [->] (1.4,1.3)-- node[midway, sloped, below,black]{$e_4$}(1.7,1.6);
\draw [->] (3.3,0.6)-- node[midway, sloped, below,black]{$e_3$}(2.7,1.2);
\draw [->] (2.4,-0.3)-- node[midway, sloped, above,black]{$e_3$}(1.7,0.4);
\end{tikzpicture}
\end{minipage}
\hspace{0.01\textwidth} 
\begin{minipage}[!t]{0.63\textwidth}

{\bf Theorem \ref{main2} }
Given $\mathcal{I}^{(0)}$, there exists a sufficiently large $a_0=a_0(\mathcal{I}^{(0)})$ such that the following holds. For any $0< a_0<a$, the unique smooth solution $(\mathcal{M}, g)$ of the Einstein-Maxwell equations from Theorem \ref{main1} with initial data satisfying
		\begin{itemize}
			\item $\sum_{ i \leq 10, k\leq 3}a^{-\frac{1}{2}} \lVert \nabla_4^k \left(\lvert u_\infty\rvert \nabla\right)^i (\chihat,\alpha_F)\rVert_{L^\infty(S_{u_\infty,\ubar})}\leq \mathcal{I}^{(0)}$ along $u=u_\infty$,
			\item Minkowskian initial data along $\ubar=0$,
			\item $\int_0^1 \lvert u_\infty \rvert^2 \left( \lvert \chihat_0 \rvert^2 + \lvert \alpha_{F0}\rvert^2 \right) (u_\infty, \ubar^\prime)\hspace{.5mm} \text{d}\ubar^\prime \geq a$ uniformly for every direction along $u=u_\infty$
		\end{itemize}has a trapped surface at $S_{-\alpha/4,1}$.
\end{minipage}

Proof. We first derive pointwise estimates for $|\chih|^2_{\gamma}$. 
Fix $(\theta^1, \theta^2)\in S^2$. We consider the following null structure equation 
\begin{equation*}
 \nab_3\chih+\frac 12 \tr\chib \chih-2\omegab \chih=\nab\widehat{\otimes} \eta-\frac 12 \tr\chi \chibh +\eta\widehat{\otimes} \eta.
\end{equation*}
We contract this $2$-tensor with another $2$-tensor $\chih$ and get
\begin{equation}\label{chih square}
\f12 \nab_3|\chih|^2_{\gamma}+\f12\tr\chib |\chih|^2_{\gamma}-2\omb|\chih|^2_{\gamma}=\chih(\nab\widehat{\otimes}\eta-\f12\tr\chi\chibh+\eta\widehat{\otimes}\eta).
\end{equation}
Employing the fact $\omb=-\f12\nab_3(\log \Omega)=-\f12\O^{-1}\nab_3\O$, we rewrite (\ref{chih square}) as 
\begin{equation*}
\begin{split} 
\nab_3(\O^2|\chih|^2_{\gamma})+\O^2\tr\chib|\chih|^2_{\gamma}=2\O^2\chih(\nab\widehat{\otimes}\eta-\f12\tr\chi\chibh+\eta\widehat{\otimes}\eta).
\end{split}
\end{equation*}
Using $\nab_3=\f{1}{\O}(\f{\partial}{\partial u}+b^A\f{\partial}{\partial \theta^A})$, we rewrite the above equation as
\begin{equation*}
\begin{split} 
\f{\partial}{\partial u}(\O^2|\chih|^2_{\gamma})+\O\tr\chib\cdot \O^2|\chih|^2_{\gamma}=&2\O^3\chih(\nab\widehat{\otimes}\eta-\f12\tr\chi\chibh+\eta\widehat{\otimes}\eta)-b^A\f{\partial}{\partial \theta^A}(\O^2|\chih|^2_{\gamma}).
\end{split}
\end{equation*}
Substitute $\O\tr\chib$ with
$$\Omega \tr\chib=\Omega(\tr\chib+\f{2}{|u|})-\Omega\f{2}{|u|}=\Omega(\tr\chib+\f{2}{|u|})-(\Omega-1)\f{2}{|u|}-\f{2}{|u|}$$
we have
\begin{equation*}
\begin{split} 
\f{\partial}{\partial u}(\O^2|\chih|^2_{\gamma})-\f{2}{|u|}\O^2|\chih|^2_{\gamma}=&2\O^3\chih(\nab\widehat{\otimes}\eta-\f12\tr\chi\chibh+\eta\widehat{\otimes}\eta)-b^A\f{\partial}{\partial \theta^A}(\O^2|\chih|^2_{\gamma})\\
&-\O(\tr\chib+\f{2}{|u|})(\O^2|\chih|^2_{\gamma})+(\O-1)\cdot\f{2}{|u|}\cdot(\O^2|\chih|^2_{\gamma}).
\end{split}
\end{equation*}
This gives
\begin{equation}\label{chih square 2}
\begin{split} 
\f{\partial}{\partial u}\bigg(u^2\O^2|\chih|^2_{\gamma}\bigg)=&2\cdot|u|^2\cdot\O^3\chih(\nab\widehat{\otimes}\eta-\f12\tr\chi\chibh+\eta\widehat{\otimes}\eta)-|u|^2\cdot b^A\f{\partial}{\partial \theta^A}(\O^2|\chih|^2_{\gamma})\\
&-|u|^2\cdot\O(\tr\chib+\f{2}{|u|})(\O^2|\chih|^2_{\gamma})+|u|^2\cdot(\O-1)\cdot\f{2}{|u|}\cdot(\O^2|\chih|^2_{\gamma}).
\end{split}
\end{equation}
For $b^{A}$, we have equation
\begin{equation*}
\f{\partial b^{A}}{\partial \ub}=-4\Omega^2\zeta^{A},
\end{equation*}
which is from 
\begin{equation*}
[L,\Lb]=\f{\partial b^{A}}{\partial \ub}\f{\partial}{\partial \theta^{A}}.
\end{equation*}
Applying the identity $\zeta_A=\f12\eta_A-\f12\etb_A$, Proposition \ref{Omega} and the derived estimates on $\eta, \etb$, we conclude that there holds in $D_{u,\ub}$
\begin{equation*}
\|b^{A}\|_{L^{\infty}(\S)}\leq \f{\at}{|u|^2}.
\end{equation*}
For the right hand side of (\ref{chih square 2}), we have
$$\|2\cdot|u|^2\cdot\O^3\chih(\nab\widehat{\otimes}\eta-\f12\tr\chi\chibh+\eta\widehat{\otimes}\eta)\|_{L^{\infty}(\S)}\leq |u|^2\cdot\f{\at}{|u|}\cdot(\f{\at}{|u|^3}+\f{a}{|u|^4})\leq\f{a}{|u|^2},$$
$$\||u|^2\cdot b^A\f{\partial}{\partial \theta^A}(\O^2|\chih|^2_{\gamma})\|_{L^{\infty}(\S)}\leq
|u|^2\cdot\f{\at}{|u|^2}\cdot\f{a}{|u|^2}\leq\f{a^{\f32}}{|u|^2},$$
$$\|-|u|^2\cdot\O(\tr\chib+\f{2}{|u|})(\O^2|\chih|^2_{\gamma})\|_{L^{\infty}(\S)}\leq
|u|^2\cdot\f{1}{|u|^2}\cdot\f{a}{|u|^2}\leq\f{a}{|u|^2},$$
$$\||u|^2\cdot(\O-1)\cdot\f{2}{|u|}\cdot(\O^2|\chih|^2_{\gamma}) \|_{L^{\infty}(\S)}\leq
|u|^2\cdot\f{1}{|u|}\cdot\f{2}{|u|}\cdot\f{a}{|u|^2}\leq\f{a}{|u|^2}.$$
In summary, we have
$$\f{\partial}{\partial u}\bigg(u^2\O^2|\chih|^2_{\gamma}\bigg)=M, \mbox{ and } |M|\ls \f{a^{\f32}}{|u|^2}\ll \f{a^{\f74}}{|u|^2},$$
which implies 
\begin{equation*}
-\f{a^{\f74}}{|u|}+\f{a^{\f74}}{|\ui|}  \leq |u|^2\O^2|\chih|^2_{\gamma}(u,\ub,\theta^1,\theta^2)-|u_{\infty}|^2\O^2|\chih|^2_{\gamma}(u_{\infty},\ub,\theta^1,\theta^2). 
\end{equation*}
Recall $\O(u_{\infty},\ub,\theta^1,\theta^2)=1$. We hence have
\begin{equation*}
|u|^2\O^2|\chih|^2_{\gamma}(u,\ub,\theta^1,\theta^2)\geq|u_{\infty}|^2|\chih|^2_{\gamma}(u_{\infty},\ub,\theta^1,\theta^2)-\f{a^{\f74}}{|u|}.
\end{equation*}
Integrating with respect to $\ub$, we further have for $\ui\leq u\leq -a/4$
\begin{equation}\label{ts chih}
\begin{split}
\int_0^1 |u|^2\O^2|\chih|^2_{\gamma}(u,\ub',\theta^1,\theta^2)d\ub'\geq \int_0^1 |u_{\infty}|^2|\chih|^2_{\gamma}(u_{\infty},\ub',\theta^1,\theta^2)d\ub'-\f{a^{\f74}}{|u|}.  
\end{split}
\end{equation}

\par \noindent In the same fashion, we derive pointwise estimates for $|\a_F|^2_{\gamma}$. Consider the null Maxwell equation
$$\nabla_3 \alpha_F+ \frac{1}{2} \tr\chibar \alpha_F-2\omb \a_F = - \nabla \rho_F+ \Hodge{\nabla}\sigma_F -2\Hodge{\etabar} \cdot \sigma_F + 2 \etabar \cdot \rho_F -\chihat \cdot \alphabar_F.$$

\par \noindent We contract this $1$-form with another $1$-form $\a_F$ and get
\begin{equation}\label{aF square}
\f12 \nab_3|\a_F|^2_{\gamma}+\f12\tr\chib |\a_F|^2_{\gamma}-2\omb|\a_F|^2_{\gamma}=\a_F(- \nabla \rho_F+ \Hodge{\nabla}\sigma_F -2\Hodge{\etabar} \cdot \sigma_F + 2 \etabar \cdot \rho_F -\chihat \cdot \alphabar_F).
\end{equation}
Employing the fact $\omb=-\f12\nab_3(\log \Omega)=-\f12\O^{-1}\nab_3\O$, we rewrite (\ref{aF square}) as 
\begin{equation*}
\begin{split} 
\nab_3(\O^2|\a_F|^2_{\gamma})+\O^2\tr\chib|\a_F|^2_{\gamma}=2\O^2\a_F(- \nabla \rho_F+ \Hodge{\nabla}\sigma_F -2\Hodge{\etabar} \cdot \sigma_F + 2 \etabar \cdot \rho_F -\chihat \cdot \alphabar_F).
\end{split}
\end{equation*}
Using $\nab_3=\f{1}{\O}(\f{\partial}{\partial u}+b^A\f{\partial}{\partial \theta^A})$, we rewrite the above equation as
\begin{equation*}
\begin{split} 
\f{\partial}{\partial u}(\O^2|\a_F|^2_{\gamma})+\O\tr\chib\cdot \O^2|\a_F|^2_{\gamma}=&2\O^3\a_F(- \nabla \rho_F+ \Hodge{\nabla}\sigma_F -2\Hodge{\etabar} \cdot \sigma_F + 2 \etabar \cdot \rho_F -\chihat \cdot \alphabar_F)-b^A\f{\partial}{\partial \theta^A}(\O^2|\a_F|^2_{\gamma}).
\end{split}
\end{equation*}
Substitute $\O\tr\chib$ with
$$\Omega \tr\chib=\Omega(\tr\chib+\f{2}{|u|})-\Omega\f{2}{|u|}=\Omega(\tr\chib+\f{2}{|u|})-(\Omega-1)\f{2}{|u|}-\f{2}{|u|}$$
we have
\begin{equation*}
\begin{split} 
\f{\partial}{\partial u}(\O^2|\a_F|^2_{\gamma})-\f{2}{|u|}\O^2|\a_F|^2_{\gamma}=&2\O^3\a_F(- \nabla \rho_F+ \Hodge{\nabla}\sigma_F -2\Hodge{\etabar} \cdot \sigma_F + 2 \etabar \cdot \rho_F -\chihat \cdot \alphabar_F)-b^A\f{\partial}{\partial \theta^A}(\O^2|\a_F|^2_{\gamma})\\
&-\O(\tr\chib+\f{2}{|u|})(\O^2|\a_F|^2_{\gamma})+(\O-1)\cdot\f{2}{|u|}\cdot(\O^2|\a_F|^2_{\gamma}).
\end{split}
\end{equation*}
This gives
\begin{equation}\label{chih square 2}
\begin{split} 
\f{\partial}{\partial u}\bigg(u^2\O^2|\a_F|^2_{\gamma}\bigg)=&2\cdot|u|^2\cdot\O^3\a_F(- \nabla \rho_F+ \Hodge{\nabla}\sigma_F -2\Hodge{\etabar} \cdot \sigma_F + 2 \etabar \cdot \rho_F -\chihat \cdot \alphabar_F)-|u|^2\cdot b^A\f{\partial}{\partial \theta^A}(\O^2|\a_F|^2_{\gamma})\\
&-|u|^2\cdot\O(\tr\chib+\f{2}{|u|})(\O^2|\a_F|^2_{\gamma})+|u|^2\cdot(\O-1)\cdot\f{2}{|u|}\cdot(\O^2|\a_F|^2_{\gamma}).
\end{split}
\end{equation}
For $b^{A}$, we have the equation
\begin{equation*}
\f{\partial b^{A}}{\partial \ub}=-4\Omega^2\zeta^{A},
\end{equation*}
which is from 
\begin{equation*}
[L,\Lb]=\f{\partial b^{A}}{\partial \ub}\f{\partial}{\partial \theta^{A}}.
\end{equation*}
Applying the identity $\zeta_A=\f12\eta_A-\f12\etb_A$, Propositions \ref{Omega}, derived estimates of $\eta, \etb$, it holds in $D_{u,\ub}$
\begin{equation*}
\|b^{A}\|_{L^{\infty}(\S)}\leq \f{\at}{|u|^2}.
\end{equation*}
For the right hand side of (\ref{chih square 2}), we have
$$\|2\cdot|u|^2\cdot\O^3\a_F(- \nabla \rho_F+ \Hodge{\nabla}\sigma_F -2\Hodge{\etabar} \cdot \sigma_F + 2 \etabar \cdot \rho_F -\chihat \cdot \alphabar_F)\|_{L^{\infty}(\S)}\leq |u|^2\cdot\f{\at}{|u|}\cdot(\f{\at}{|u|^3}+\f{a}{|u|^4}+\f{a^{\f32}}{|u|^4})\leq\f{a}{|u|^2}+\f{a^2}{|u|^3},$$
$$\||u|^2\cdot b^A\f{\partial}{\partial \theta^A}(\O^2|\a_F|^2_{\gamma})\|_{L^{\infty}(\S)}\leq
|u|^2\cdot\f{\at}{|u|^2}\cdot\f{a}{|u|^2}\leq\f{a^{\f32}}{|u|^2},$$
$$\|-|u|^2\cdot\O(\tr\chib+\f{2}{|u|})(\O^2|\a_F|^2_{\gamma})\|_{L^{\infty}(\S)}\leq
|u|^2\cdot\f{1}{|u|^2}\cdot\f{a}{|u|^2}\leq\f{a}{|u|^2},$$
$$\||u|^2\cdot(\O-1)\cdot\f{2}{|u|}\cdot(\O^2|\a_F|^2_{\gamma}) \|_{L^{\infty}(\S)}\leq
|u|^2\cdot\f{1}{|u|}\cdot\f{2}{|u|}\cdot\f{a}{|u|^2}\leq\f{a}{|u|^2}.$$
In summary, we have
$$\f{\partial}{\partial u}\bigg(u^2\O^2|\a_F|^2_{\gamma}\bigg)=M, \mbox{ and } |M|\ls \f{a^{\f32}}{|u|^2}\ll \f{a^{\f74}}{|u|^2},$$
which implies
\begin{equation*}
-\f{a^{\f74}}{|u|}+\f{a^{\f74}}{|\ui|}  \leq |u|^2\O^2|\a_F|^2_{\gamma}(u,\ub,\theta^1,\theta^2)-|u_{\infty}|^2\O^2|\a_F|^2_{\gamma}(u_{\infty},\ub,\theta^1,\theta^2). 
\end{equation*}
Recall $\O(u_{\infty},\ub,\theta^1,\theta^2)=1$. We hence have
\begin{equation*}
|u|^2\O^2|\a_F|^2_{\gamma}(u,\ub,\theta^1,\theta^2)\geq|u_{\infty}|^2|\a_F|^2_{\gamma}(u_{\infty},\ub,\theta^1,\theta^2)-\f{a^{\f74}}{|u|}.
\end{equation*}
Integrating with respect to $\ub$, we further have for $\ui\leq u\leq -a/4$
\begin{equation*}
\begin{split}
\int_0^1 |u|^2\O^2|\a_F|^2_{\gamma}(u,\ub',\theta^1,\theta^2)d\ub'\geq \int_0^1 |u_{\infty}|^2|\a_F|^2_{\gamma}(u_{\infty},\ub',\theta^1,\theta^2)d\ub'-\f{a^{\f74}}{|u|}.
\end{split}
\end{equation*}
Together with (\ref{ts chih})
\begin{equation*}
\begin{split}
\int_0^1 |u|^2\O^2|\chih|^2_{\gamma}(u,\ub',\theta^1,\theta^2)d\ub'\geq \int_0^1 |u_{\infty}|^2|\chih|^2_{\gamma}(u_{\infty},\ub',\theta^1,\theta^2)d\ub'-\f{a^{\f74}}{|u|}.  
\end{split}
\end{equation*}
We conclude that
\begin{equation*}
\begin{split}
\int_0^1 |u|^2\O^2(|\chih|^2_{\gamma}+|\a_F|^2_{\gamma})(u,\ub',\theta^1,\theta^2)d\ub'\geq& \int_0^1 |u_{\infty}|^2(|\chih_F|^2_{\gamma}+|\a_F|^2_{\gamma})(u_{\infty},\ub',\theta^1,\theta^2)d\ub'-\f{2a^{\f74}}{|u|}\\
\geq& a-\f{2a^{\f74}}{|u|}\geq a-\f{8a^{\f74}}{a} \geq \f{7a}{8}.  
\end{split}
\end{equation*}
Pick $u=-a/4$. With the fact $\|\O-1\|_{L^{\infty}(\S)}\ls {1}/{a}$, for sufficiently large $a$, we hence have
\begin{equation*}
\begin{split}
(-\f{a}{4})^2\int_0^1 (|\chih|^2_{\gamma}+|\a_F|^2_{\gamma})(-\f{a}{4},\ub',\theta^1,\theta^2)d\ub'\geq& \f67\cdot
\int_0^1 (-\f{a}{4})^2\O^2(|\chih|^2_{\gamma}+|\a_F|^2_{\gamma})(-\f{a}{4},\ub',\theta^1,\theta^2)d\ub'\\
\geq&\f67\cdot \f{7a}{8}=\f{3a}{4}.   
\end{split}
\end{equation*}
This implies
\begin{equation}\label{int chih square}
\begin{split}
\int_0^1 (|\chih|^2_{\gamma}+|\a_F|^2_{\gamma})(-\f{a}{4},\ub',\theta^1,\theta^2)d\ub'\geq&\f{3a}{4}\cdot \f{16}{a^2}= \f{12}{a}
\end{split}
\end{equation}

\par \noindent We now consider the outgoing null structure equation for $\tr\chi$,
$$\nab_4 \tr\chi+\f12(\tr\chi)^2=-|\chih|^2_{\gamma}-2\o\tr\chi-|\a_F|_{\gamma}^2.$$
Using $\o=-\f12\nab_4(\log \Omega)$, we have
\begin{equation*}
\begin{split}
\nab_4 \tr\chi+\f12(\tr\chi)^2=&-|\chih|^2-2\o\tr\chi-|\a_F|_{\gamma}^2\\
=&-|\chih|^2_{\gamma}+\nab_4(\log\Omega)\tr\chi-|\a_F|_{\gamma}^2=-|\chih|^2_{\gamma}+\f{1}{\Omega}\nab_4\Omega\cdot \tr\chi-|\a_F|_{\gamma}^2.
\end{split}
\end{equation*}
Hence,
\begin{equation*}
\begin{split}
\nab_4(\Omega^{-1} \tr\chi)=&-\O^{-2}\nab_4\O\cdot\tr\chi+\O^{-1}\nab_4\tr\chi\\
=&\O^{-1}(\nab_4\tr\chi-\O^{-1}\cdot\nab_4\O\cdot\tr\chi)=\O^{-1}\bigg(-\f12(\tr\chi)^2-|\chih|^2_{\gamma}-|\a_F|_{\gamma}^2\bigg).
\end{split}
\end{equation*}
With the fact $e_4=\Omega^{-1}\f{\partial}{\partial \ub}$, we have
\begin{equation}\label{O trchi}
\begin{split}
\f{\partial}{\partial \ub}(\Omega^{-1} \tr\chi)=&-\f12(\tr\chi)^2-|\chih|^2_{\gamma}-|\a_F|_{\gamma}^2.
\end{split}
\end{equation}
For every $(\theta^1,\theta^2)\in \mathbb{S}^2$, along $\Hb_0$ we have 
$$(\Omega^{-1}\tr\chi)(-\f{a}{4}, 0, \theta^1, \theta^2)=1^{-1}\cdot \f{2}{a/4}=\f{8}{a}.$$
We then integrate (\ref{O trchi}). Using (\ref{int chih square}) we obtain
\begin{equation*}
\begin{split}
&(\Omega^{-1}\tr\chi)(-\f{a}{4},1, \theta^1, \theta^2)\\
\leq & (\Omega^{-1}\tr\chi)(-\f{a}{4}, 0, \theta^1, \theta^2)-\int_0^{1}(|\chih|^2_{\gamma}+|\a_F|_{\gamma}^2)(-\f{a}{4},\ub',\theta^1,\theta^2)d\ub'\\
\leq & \f{8}{a}-\f{12}{a}<0.
\end{split}
\end{equation*}
Recall, finally, that in $D_{u,\ub}$ the following estimate holds
\begin{equation*}
\|\tr\chib+\f{2}{|u|}\|_{L^{\infty}(S_{u,\ub})}\leq \f{1}{|u|^2}.
\end{equation*}
In particular, this implies
$$\tr\chib(-\f{a}{4},1, \theta^1, \theta^2)<0 \mbox{ for every} (\theta^1, \theta^2)\in \mathbb{S}^2.$$
Therefore, we conclude that $S_{-\f{a}{4},1}$ is a trapped surface.

{\color{black}

\section{A Scaling Argument}\label{sec rescale} 
In this article, we use coordinate system $(u, \ub, \theta^1, \theta^2)$ based on double null foliations, where $(\theta^1, \theta^2)$ are stereographic coordinates on $\mathbb{S}^2$. In these coordinates, we study spacetime region
$$u_{\infty} \leq u \leq -\f{a}{4}, \, \, \quad 0\leq \ub \leq 1. $$
The Lorentzian metric $g$ satisfies ansatz
$$g=-2\O^2(du\otimes d\ub+d\ub\otimes du)+\gamma_{AB}(d\theta^A-d^A du)\otimes(d\theta^B-d^B du).$$

\subsection{A Spacetime Rescaling}  
Following \cite{An:2019}, we use a new coordinate system $(u', \ub', \theta^{1'}, \theta^{2'})$, where
\begin{equation}\label{rescale}
u'=\d u, \, \, \, \ub'=\d\ub, \, \, \, \theta^{1'}=\d \theta^1, \, \, \,\theta^{2'}=\d \theta^2. 
\end{equation}
Note that coordinates $(\theta^1, \theta^2)$ on $\S$ are set up through stereographic projection. Assume $(x_1, x_2, x_3)$  satisfying $x_1^2+x_2^2+x_3^2=a^2$ and lying on the upper hemisphere of $S_{-a,0}$ (with radius $a$). It then has stereographic projection $(\zeta_1, \zeta_2)=(\f{ax_1}{a+x_3}, \f{ax_2}{a+x_3})$. Scale down the length by a factor $\d$, we then have $x_1'=\d x_1, \, x_2'=\d x_2, \, x_3'=\d x_3, \,\, (x_1')^2+(x_2')^2+(x_2')^2=\d^2 a^2$ and $(x_1', x_2', x_3')$ has stereographic projection 
$$(\zeta_1', \zeta_2')=(\f{\d a x_1'}{\d a+x_3'}, \f{\d ax_2'}{\d a+x_3'})=(\f{\d a\cdot \d x_1}{\d a+\d x_3}, \f{\d a \cdot \d x_2}{\d a+\d x_3})=(\f{\d ax_1}{a+x_3}, \f{\d ax_2}{a+x_3})=(\d\zeta_1, \d\zeta_2).$$ Therefore, the rescaled coordinates $(\theta^{1'}, \theta^{2'})=(\d \theta^1, \d \theta^2)$ on $S_{u', \ub'}$ make perfect sense  since $2$-sphere $S_{u',\ub'}=S_{\d u, \d \ub}$ is scaled down from $\S$ by a factor $\delta$.

Under the rescaling (\ref{rescale}), it follows
$$g'(u',\ub', \theta^{1'}, \theta^{2'})=\d^2\cdot g(u,\ub, \theta^1, \theta^2).$$ 
In $(u',\ub',\theta^{1'}, \theta^{2'})$ coordinates, we let
$$g'(u', \ub', \theta^1, \theta^2)=-2{\O'}^2(du'\otimes d\ub'+d\ub'\otimes du')+\gamma'_{A'B'}(d\theta^{A'}-d^{A'} du')\otimes(d\theta^{B'}-d^{B' }du').$$
Compare with
$$g(u, \ub, \theta^1, \theta^2)=-2\O^2(du\otimes d\ub+d\ub\otimes du)+\gamma_{AB}(d\theta^A-d^A du)\otimes(d\theta^B-d^B du).$$
Here we have
$$du'=\d\cdot du, \quad d\ub'=\d\cdot d\ub, \quad d\theta^{A'}=\d \cdot d \theta^A \,\mbox{  for } A=1, 2 , $$
$${\O'}^2(u',\ub', \theta^{1'}, \theta^{2'})=\O^2(u,\ub, \theta^1, \theta^2), \quad \gamma'_{A'B'}(u',\ub', \theta^{1'}, \theta^{2'})=\gamma_{AB}(u,\ub, \theta^1, \theta^2),$$ 
$$d^{A'}(u',\ub', \theta^{1'}, \theta^{2'})=d^A(u,\ub, \theta^1, \theta^2),$$
$$e'_3(u',\ub', \theta^{1'}, \theta^{2'})={\O'}^{-1}(\f{\partial}{\partial u'}+d^{A'}\f{\partial}{\partial \ub'})=\d^{-1}{\O}^{-1}(\f{\partial}{\partial u}+d^{A}\f{\partial}{\partial \ub})=\d^{-1}\cdot e_3(u, \ub, \theta^1, \theta^2),$$
\begin{equation}\label{e4 e4'}
e'_4(u',\ub', \theta^{1'}, \theta^{2'})={\O'}^{-1}\f{\partial}{\partial\ub'}=\d^{-1}{\O}^{-1}\f{\partial}{\partial\ub}=\d^{-1}\cdot e_4(u,\ub, \theta^1, \theta^2),
\end{equation}
\begin{equation}\label{eA eA'}
e'_A(u',\ub', \theta^{1'}, \theta^{2'})=\d^{-1}\cdot e_A(u,\ub, \theta^1, \theta^2), \mbox{ for } A=1,2. 
\end{equation}

\subsection{Rescaled Geometric Quantities}
As usual, with frame $\{e'_3, e'_4, e'_A,  e'_B\}$, we define 
\begin{equation*}
\begin{split}
&\chi'_{A'B'}=g'(D'_{A'} e'_4,e'_B),\, \,\, \quad \chib'_{A'B'}=g'(D'_{A'} e'_3,e'_B),\\
&\eta'_{A'}=-\frac 12 g'(D'_{3'} e'_A,e'_4),\quad \etab'_{A'}=-\frac 12 g'(D'_{4'} e'_A,e'_3),\\
&\omega'=-\frac 14 g'(D'_{4'} e'_3,e'_4),\quad\,\,\, \omegab'=-\frac 14 g'(D'_{3'} e'_4,e'_3),\\
&\zeta'_{A'}=\frac 1 2 g'(D'_{A'} e'_4,e'_3).
\end{split}
\end{equation*}
With $\gamma'_{A'B'}$ being the induced metric on $\S'$, we further decompose $\chi', \chib'$ into
$$\chi'_{A'B'}=\f12\tr\chi'\cdot \gamma'_{A'B'}+\chih'_{A'B'}, \quad \chib'_{A'B'}=\f12\tr\chib'\cdot \gamma'_{A'B'}+\chibh'_{A'B'}.$$
$$\mbox{Here} \quad D'_{e'_{\mu}}e'_{\nu}:=\Gamma'^{\lambda}_{\mu'\nu'}e'_{\lambda} \quad \mbox{ and } \quad \Gamma'^{\lambda}_{\mu'\nu'}:=\f12g'^{\lambda'\kappa'}(\f{\partial g'_{\kappa'\mu'}}{\partial x'_{\nu}}+\f{\partial g'_{\kappa' \nu'}}{\partial x'_{\mu}}-\f{\partial g'_{\mu'\nu'}}{\partial x'_{\kappa}}).$$
In \cite{An:2019}, we have
\begin{proposition}\label{Prop rescale 1}
For $\Gamma\in \{\chih, \tr\chi, \chibh, \tr\chib, \eta, \etb, \zeta, \o,\omb\}$ written in two different coordinates $(u',\ub',\theta^{1'}, \theta^{2'})$ and $(u, \ub, \theta^1, \theta^2)$, it holds that
$$\Gamma'(u',\ub',\theta^{1'}, \theta^{2'})=\d^{-1}\cdot \Gamma (u,\ub,\theta^{1}, \theta^{2}).$$
\end{proposition}
And
\begin{proposition}\label{Prop rescale 2}
For $\Psi\in \{\a, \b, \rho, \sigma, \beb, \ab\}$ written in coordinates $(u',\ub',\theta^{1'}, \theta^{2'})$ and $(u, \ub, \theta^1, \theta^2)$, the following identity is true
$$\Psi'(u',\ub',\theta^{1'}, \theta^{2'})=\d^{-2}\cdot\Psi(u,\ub,\theta^{1}, \theta^{2}).$$
\end{proposition}

Besides above geometric quantities, under the rescaling (\ref{rescale}), for the Maxwell field we have 
\begin{equation}\label{Maxwell rescaling}
F'_{\mu\nu}(u',\ub', \theta^{1'}, \theta^{2'})=\d \cdot F_{\mu\nu}(u,\ub, \theta^1, \theta^2).
\end{equation}
And we define 
$$(\a_{F'}')_{A'}={F'}_{e_{A}' e_4'}, \quad (\ab_{F'}')_{A'}={F'}_{e_{A}' e_3'}, \quad \rho_{F'}'=\f12 {F'}_{e'_3 e'_4}, \quad \sigma_{F'}'={F'}_{e'_1 e'_2}'$$
It holds that
\begin{proposition}\label{Prop rescale 3}
For $\Gamma_F\in \{\a_F, \ab_F, \rho_F, \sigma_F\}$ written in two different coordinates $(u',\ub',\theta^{1'}, \theta^{2'})$ and $(u, \ub, \theta^1, \theta^2)$, it holds that
$$\Gamma'_F(u',\ub',\theta^{1'}, \theta^{2'})=\d^{-1}\cdot \Gamma_F (u,\ub,\theta^{1}, \theta^{2}).$$
\end{proposition}
\begin{proof}
We first calculate $(\a'_{F'})_{A'}$.  Using definition of $(\a'_{F'})_{A'}$ and (\ref{Maxwell rescaling}), we have
\begin{equation*}
\begin{split}
(\a'_{F'})_{A'}(u',\ub', \theta^{1'}, \theta^{2'})=&{F'}_{e_{A}' e_4'}=\d\cdot {F}_{e_{A}' e_4'}\\
=&\d\cdot \d^{-1}\cdot \d^{-1}\cdot \d^{-1}  {F}_{e_{A} e_4}=\d^{-1}\cdot (\a_{F})_{A}(u, \ub, \theta^1, \theta^2).
\end{split}
\end{equation*}
The rest Maxwell components are treated in the same way. 
\end{proof}

\subsection{Rescaled Uniform Bounds}\label{rescale bounds}

Applying Proposition \ref{Prop rescale 1} and Proposition \ref{Prop rescale 2}, next we establish the connection to \cite{A-L}.  Take $\chih$ as an example. With Proposition \ref{Prop rescale 1}, estimates derived for $\mathcal{O}_{i,\infty}[\chih]$ and $u'=\d u$, we have 
\begin{equation*}
\begin{split}
&|\chih'_{A'B'}(u',\ub', \theta^{1'}, \theta^{2'})|=\d^{-1}\cdot |\chih_{AB}(u, \ub, \theta^1, \theta^2)|\leq \d^{-1}\cdot \f{\at}{|u|}=\f{\at}{\d|u|}=\f{\at}{|u'|}. 
\end{split}
\end{equation*}
In the same fashion, we have
\begin{equation*}
|\chibh'_{A'B'}(u',\ub', \theta^{1'}, \theta^{2'})|=\d^{-1}\cdot |\chibh_{AB}(u, \ub, \theta^1, \theta^2)|\leq \d^{-1}\cdot \f{\at}{|u|^2}=\f{\d\at}{\d^2|u|^2}=\f{\d\at}{|u'|^2}, 
\end{equation*}
\begin{equation*}
|\tr\chi'(u',\ub', \theta^{1'}, \theta^{2'})|=\d^{-1}\cdot |\tr\chi(u, \ub, \theta^1, \theta^2)|\leq \d^{-1}\cdot \f{1}{|u|}=\f{1}{\d|u|}=\f{1}{|u'|}, 
\end{equation*}
\begin{equation*}
|\eta'_{A'}(u',\ub', \theta^{1'}, \theta^{2'})|=\d^{-1}\cdot |\eta_{A}(u, \ub, \theta^1, \theta^2)|\leq \d^{-1}\cdot \f{\at}{|u|^2}=\f{\d\at}{\d^2|u|^2}=\f{\d\at}{|u'|^2}, 
\end{equation*}
\begin{equation*}
|\etb'_{A'}(u',\ub', \theta^{1'}, \theta^{2'})|=\d^{-1}\cdot |\etb_{A}(u, \ub, \theta^1, \theta^2)|\leq \d^{-1}\cdot \f{\at}{|u|^2}=\f{\d\at}{\d^2|u|^2}=\f{\d\at}{|u'|^2}, 
\end{equation*}
\begin{equation*}
|\o'(u',\ub', \theta^{1'}, \theta^{2'})|=\d^{-1}\cdot |\o(u, \ub, \theta^1, \theta^2)|\leq \d^{-1}\cdot \f{1}{|u|}=\f{1}{\d|u|}=\f{1}{|u'|}, 
\end{equation*}
\begin{equation*}
|\omb'(u',\ub', \theta^{1'}, \theta^{2'})|=\d^{-1}\cdot |\omb(u, \ub, \theta^1, \theta^2)|\leq \d^{-1}\cdot \f{a}{|u|^3}=\f{\d^2 a}{\d^3|u|^3}=\f{\d^2 a}{|u'|^3}\boxed{\leq \f{\d\at}{|u'|^2}} \, ,
\end{equation*}
\begin{equation*}
|\tr\chib'(u',\ub', \theta^{1'}, \theta^{2'})+\f{2}{|u'|}|=\d^{-1}\cdot |\tr\chib(u, \ub, \theta^1, \theta^2)+\f{2}{|u|}|\leq \d^{-1}\cdot \f{a}{|u|^3}=\f{\d^2 a}{\d^3|u|^3}=\f{\d^2 a}{|u'|^3}\boxed{\leq \f{\d\at}{|u'|^2}} \, .
\end{equation*}
For the estimates of $\omb'$ and $\tr\chib'$, we use $|u'|\geq \d\at$. In the same manner, by Proposition \ref{Prop rescale 2} and with the help that $|u'|\geq  \d a/4$ we have
\begin{equation*}
\begin{split}
&|\b'_{A'}(u',\ub', \theta^{1'}, \theta^{2'})|=\d^{-2}\cdot |\b_{A}(u, \ub, \theta^1, \theta^2)|\leq \d^{-2}\cdot \f{\at}{|u|^2}=\f{\at}{\d^2|u|^2}=\f{\at}{|u'|^2}, 
\end{split}
\end{equation*}
\begin{equation*}
\begin{split}
&|\rho'(u',\ub', \theta^{1'}, \theta^{2'})|=\d^{-2}\cdot |\rho(u, \ub, \theta^1, \theta^2)|\leq \d^{-2}\cdot \f{a}{|u|^3}=\f{\d a}{\d^3|u|^3}=\f{\d a}{|u'|^3}, 
\end{split}
\end{equation*}
\begin{equation*}
\begin{split}
&|\sigma'(u',\ub', \theta^{1'}, \theta^{2'})|=\d^{-2}\cdot |\sigma(u, \ub, \theta^1, \theta^2)|\leq \d^{-2}\cdot \f{a}{|u|^3}=\f{\d a}{\d^3|u|^3}=\f{\d a}{|u'|^3}, 
\end{split}
\end{equation*}
\begin{equation*}
\begin{split}
&|\beb'_{A'}(u',\ub', \theta^{1'}, \theta^{2'})|=\d^{-2}\cdot |\beb_{A}(u, \ub, \theta^1, \theta^2)|\leq \d^{-2}\cdot \f{a^{\f32}}{|u|^4}=\f{\d^2 a^{\f32}}{\d^4|u|^4}=\f{\d^2 a^{\f32}}{|u'|^4}\boxed{\leq \f{\d\at}{|u'|^3}} \, , 
\end{split}
\end{equation*}
\begin{equation}\label{rescale ab}
\begin{split}
&|\ab'_{A'B'}(u',\ub', \theta^{1'}, \theta^{2'})|=\d^{-2}\cdot |\ab_{AB}(u, \ub, \theta^1, \theta^2)|\leq \d^{-2}\cdot \f{a^2}{|u|^5}=\f{\d^3 a^2}{\d^5|u|^5}=\f{\d^3 a^2}{|u'|^5}, 
\end{split}
\end{equation}
\begin{equation}\label{rescale alpha}
\begin{split}
&|\a'_{A'B'}(u',\ub', \theta^{1'}, \theta^{2'})|=\d^{-2}\cdot |\a_{AB}(u, \ub, \theta^1, \theta^2)|\leq \d^{-2}\cdot \f{\at}{|u|}=\f{\d^{-1}\at}{\d |u|}=\f{\d^{-1}\at}{|u'|}. 
\end{split}
\end{equation}

\noindent Similarly by Proposition \ref{Prop rescale 2} and with the help that $|u'|\geq  \d a/4$ we have
\begin{equation*}
\begin{split}
&|({\a'}_{F'})_{A'}(u',\ub', \theta^{1'}, \theta^{2'})|=\d^{-1}\cdot |(\a_{F})_{A}(u, \ub, \theta^1, \theta^2)|\leq \d^{-1}\cdot \f{\at}{|u|}=\f{\at}{\d|u|}=\f{\at}{|u'|}, 
\end{split}
\end{equation*}

\begin{equation*}
\begin{split}
&|({\ab'}_{F'})_{A'}(u',\ub', \theta^{1'}, \theta^{2'})|=\d^{-1}\cdot |(\ab_{F})_{A}(u, \ub, \theta^1, \theta^2)|\leq \d^{-1}\cdot \f{a}{|u|^3}=\f{\d^2 a}{\d^3|u|^3}=\f{\d^2 a}{|u'|^3}\leq \f{\d \at}{|u'|^2}, 
\end{split}
\end{equation*}

\begin{equation*}
\begin{split}
&|({\rho'}_{F'})_{A'}(u',\ub', \theta^{1'}, \theta^{2'})|=\d^{-1}\cdot |(\rho_{F})_{A}(u, \ub, \theta^1, \theta^2)|\leq \d^{-1}\cdot \f{\at}{|u|^2}=\f{\d\at}{\d^2|u|^2}=\f{\d\at}{|u'|^2}, 
\end{split}
\end{equation*}

\begin{equation*}
\begin{split}
&|({\sigma'}_{F'})_{A'}(u',\ub', \theta^{1'}, \theta^{2'})|=\d^{-1}\cdot |(\sigma_{F})_{A}(u, \ub, \theta^1, \theta^2)|\leq \d^{-1}\cdot \f{\at}{|u|^2}=\f{\d\at}{\d^2|u|^2}=\f{\d\at}{|u'|^2}.
\end{split}
\end{equation*}

\noindent By repeating the arguments as in Section \ref{TSF}, we therefore obtain Theorem \ref{main3}. 

}

\end{document}